\RequirePackage{fix-cm}
\documentclass[smallextended]{svjour3}       
\renewcommand\qed{\hfill \ensuremath{\Box}}

\usepackage{graphicx}

\usepackage{amsmath} 
\usepackage{amsfonts}
\usepackage{amssymb} 
\usepackage{bm}           
\usepackage{mathrsfs} 
\usepackage{mathabx}
\usepackage{amsxtra}
\numberwithin{equation}{section}
\usepackage[numbers,sort&compress]{natbib} 
\usepackage[perpage,symbol*]{footmisc}  

\usepackage[titletoc,title]{appendix}

\DeclareMathOperator{\dive}{div}
\DeclareMathOperator{\supp}{supp}

\DeclareMathOperator{\RE}{Re}
\DeclareMathOperator{\IM}{Im}

\def\d{\,\mathrm{d}}
\def\p{\partial}

\newcommand*\xbar[1]{%
	\hbox{%
		\vbox{%
			\hrule height 0.5pt 
			\kern0.5ex
			\hbox{%
				\kern-0.1em
				\ensuremath{#1}%
				\kern-0.1em
			}%
		}%
	}%
}

\newcommand{\VERTiii}[1]{{\left\vert\kern-0.3ex\left\vert\kern-0.3ex\left\vert #1
		\right\vert\kern-0.3ex\right\vert\kern-0.3ex\right\vert}}
\newcommand{\VERT}{\vert\kern-0.3ex\vert\kern-0.3ex\vert}
\newcommand{\VERTl}{\left\vert\kern-0.3ex\left\vert\kern-0.3ex\left\vert}
\newcommand{\VERTr}{\right\vert\kern-0.3ex\right\vert\kern-0.3ex\right\vert}
\newcommand{\VERTbig}{\big\vert\kern-0.3ex\big\vert\kern-0.3ex\big\vert}
\newcommand{\VERTBig}{\Big\vert\kern-0.3ex\Big\vert\kern-0.3ex\Big\vert}

\makeatletter
\DeclareFontFamily{OMX}{MnSymbolE}{}
\DeclareSymbolFont{MnLargeSymbols}{OMX}{MnSymbolE}{m}{n}
\SetSymbolFont{MnLargeSymbols}{bold}{OMX}{MnSymbolE}{b}{n}
\DeclareFontShape{OMX}{MnSymbolE}{m}{n}{
	<-6>  MnSymbolE5 <6-7>  MnSymbolE6 <7-8>  MnSymbolE7 <8-9>  MnSymbolE8 <9-10> MnSymbolE9 <10-12> MnSymbolE10 <12->   MnSymbolE12
}{}
\DeclareFontShape{OMX}{MnSymbolE}{b}{n}{
	<-6>  MnSymbolE-Bold5 <6-7>  MnSymbolE-Bold6 <7-8>  MnSymbolE-Bold7 <8-9>  MnSymbolE-Bold8 <9-10> MnSymbolE-Bold9 <10-12> MnSymbolE-Bold10 <12->   MnSymbolE-Bold12
}{}

\let\llangle\@undefined
\let\rrangle\@undefined
\DeclareMathDelimiter{\llangle}{\mathopen}%
{MnLargeSymbols}{'164}{MnLargeSymbols}{'164}
\DeclareMathDelimiter{\rrangle}{\mathclose}%
{MnLargeSymbols}{'171}{MnLargeSymbols}{'171}
\makeatother

\numberwithin{equation}{section}
\numberwithin{theorem}{section}
\numberwithin{lemma}{section}
\numberwithin{proposition}{section}
\numberwithin{corollary}{section}
\numberwithin{definition}{section}
\numberwithin{remark}{section}
\numberwithin{example}{section}

\begin{document}
	
	\title{Nonlinear Stability of Relativistic Vortex Sheets in Three-Dimensional Minkowski Spacetime
	}
	
	\titlerunning{Relativistic Vortex Sheets}        
	
	\author{Gui-Qiang G. Chen  \and \\
		Paolo Secchi \and\\
		Tao Wang
	}
	
	\authorrunning{G.-Q.~G.~Chen \and P.~Secchi \and T.~Wang} 
	
	\institute{Gui-Qiang G. Chen  \at
		Mathematical Institute, University of Oxford, Oxford, OX2 6GG, UK\\
		AMSS \& UCAS, Chinese Academy of Sciences, Beijing 100190, China                 \\
		\email{chengq@maths.ox.ac.uk}           
		\and
		Paolo Secchi \at
		DICATAM, Mathematical Division, University of Brescia, 25133 Brescia, Italy\\
		\email{paolo.secchi@unibs.it}
		\and
		Tao Wang \at
		School of Mathematics and Statistics, Wuhan University, Wuhan 430072, China\\
		DICATAM, Mathematical Division, University of Brescia, 25133 Brescia, Italy\\
		Hubei Key Laboratory of Computational Science (Wuhan University), Wuhan 430072, China\\
		\email{tao.wang@whu.edu.cn}
	}
	
	\date{Received: July 5, 2017 / Accepted: October 5, 2018}

	\maketitle
	
	\begin{abstract}
		We are concerned with the nonlinear stability of vortex sheets for the relativistic Euler equations in three-dimensional Minkowski spacetime. This is a nonlinear hyperbolic problem with a characteristic free boundary. In this paper, we introduce a new symmetrization by choosing appropriate functions as primary unknowns. A necessary and sufficient condition for the weakly linear stability of relativistic vortex sheets is obtained by analyzing the roots of the Lopatinski\u{\i} determinant associated to the constant coefficient linearized problem. Under this stability condition, we show that the variable coefficient linearized problem obeys an energy estimate with a loss of derivatives. The construction of certain weight functions plays a crucial role in absorbing the error terms caused by microlocalization. Based on the weakly linear stability result, we establish the existence and nonlinear stability of relativistic vortex sheets under small initial perturbations by a Nash--Moser iteration scheme.
	\end{abstract}
	
	\tableofcontents
\section{Introduction}
We are concerned with the nonlinear stability of relativistic vortex sheets for the Euler equations describing the evolution of a relativistic compressible fluid. Relativistic vortex sheets arise as a very important feature in several models of phenomena occurring in astrophysics, plasma physics, and nuclear physics. Vortex sheets are interfaces between two incompressible or compressible flows across which there is a discontinuity in fluid velocity. {In particular, a}cross a vortex sheet, the tangential velocity field has a jump, while the normal component of the flow velocity is continuous. The discontinuity in the tangential velocity field creates a concentration of vorticity along the interface. {Moreover,} compressible vortex sheets are characteristic discontinuities to the Euler equations for compressible fluids and as such they are fundamental waves which play an important role in the study of general entropy solutions to multidimensional hyperbolic systems of conservation laws (\emph{cf.}\;Chen--Feldman \cite{CF18}).

It was observed in \cite{M58MR0097930,FM63MR0154509}, by the normal mode analysis, that rectilinear vortex sheets for non-relativistic isentropic compressible fluids in two space dimensions are linearly stable when the Mach number $\textsf{M}>\sqrt{2}$ and are violently unstable when $\textsf{M}<\sqrt{2}$, while planar vortex sheets are always violently unstable in three space dimensions. This kind of instabilities is the analogue of the Kelvin--Helmholtz instability for incompressible fluids. Artola--Majda \cite{AM87MR914450} studied certain instabilities of two-dimensional supersonic vortex sheets by analyzing the interaction with highly oscillatory waves through geometric optics. A rigorous mathematical theory on nonlinear stability and local-in-time existence of two-dimensional non-relativistic supersonic vortex sheets was first established by Coulombel--Secchi \cite{CS08MR2423311,CS09MR2505379} based on their linear stability results in \cite{CS04MR2095445} and a Nash--Moser iteration scheme.

Motivated by the earlier results in \cite{CS04MR2095445,CS08MR2423311,CS09MR2505379}, we aim to establish the nonlinear stability of relativistic vortex sheets in three-dimensional Minkowski spacetime under the necessary condition for the linear stability on the piecewise constant background state. This problem is a nonlinear hyperbolic problem with a characteristic free boundary. The so-called Lopatinski\u{\i} condition holds only in a weak sense, which yields a loss of derivatives.

We first reformulate the relativistic Euler equations into a symmetrizable hyperbolic system by choosing appropriate functions as primary unknowns. Our symmetrization is purely algebraic and different from those obtained by Makino--Ukai in \cite{MU95IIMR1346915} and Trakhinin \cite{T09MR2560044}. As in Francheteau--M\'{e}tivier \cite{FM00MR1787068}, we straighten the unknown front by lifting functions $\Phi^{\pm}$ that satisfy the \emph{eikonal} equations \eqref{Phi.eq.a} on the whole domain. Consequently, the original problem can be transformed into a nonlinear problem in a half-space for which the boundary matrix has constant rank on the whole half-space. This constant rank property is essential to derive energy estimates for the variable coefficient linearized problem by developing the Kreiss' symmetrizers technique from \cite{K70MR0437941,C04MR2069632,CS04MR2095445}.

Then we consider the constant coefficient linearized problem around the piecewise constant background state. By computing the roots of the associated Lopatinski\u{\i} determinant, we deduce the necessary stability condition ({\it cf}.\;\eqref{H2}):
$$
\textsf{M}>\textsf{M}_c:=\frac{\sqrt{2}}{\sqrt{1+\epsilon^2\bar{c}^2}},
$$
where $\epsilon^{-1}$ is the speed of light and $\bar{c}$ is the sound speed of the background state. In the non-relativistic limit $\epsilon\to 0$, this stability condition is reduced to $\textsf{M}>\sqrt{2}$, the well-known fact studied in \cite{M58MR0097930,CS04MR2095445}. The critical Mach number $\textsf{M}_c$ of the relativistic stability condition is always strictly smaller than $\sqrt{2}$, which means that the relativistic vortex sheets are stable in a larger physical regime.
Moreover, when the sound speed $\bar{c}$ is arbitrarily close to the light speed $\epsilon^{-1}$, the critical Mach number $\textsf{M}_c$ approaches $1$ so that the stability holds precisely for supersonic relativistic  flows. The symbol associated to the unknown front is elliptic, which enables us to eliminate the front and to consider a standard boundary value problem. We prove that the constant coefficient linearized problem obeys an \emph{a priori} energy estimate, which exhibits a loss of derivatives with respect to the source terms, owing to the failure of the uniform Kreiss--Lopatinski\u{\i} condition. Since the boundary is characteristic, there exists a loss of control on the trace of the solution.

After that, we study the effective linear problem, which is deduced from the linearized problem around a perturbation of the background state by using  the ``good unknown'' of Alinhac \cite{A89MR976971} and neglecting some zero-th order terms. The dropped terms will be considered as the error terms at each Nash--Moser iteration step in the subsequent nonlinear analysis. We first prove for small perturbations that the solution satisfies the same \emph{a priori} estimate as the constant coefficient case. The energy estimate is deduced by the technique applied earlier to weakly stable shock waves in \cite{C04MR2069632} and isentropic compressible vortex sheets in \cite{CS04MR2095445}. It consists of the paralinearization of the linearized problem, analysis of the Lopatinski\u{\i} determinant, microlocalization, and construction of the Kreiss' symmetrizers. In particular, we introduce certain weight functions, vanishing \emph{only} on the bicharacteristic curves starting from the critical set, to absorb the error terms caused by microlocalization. Based on this basic energy estimate, we establish a well-posedness result for the effective linear problem in the usual Sobolev space $H^s$ with $s$ large enough. This is achieved by means of a duality argument and higher order energy estimates. Although our problem is a hyperbolic problem with characteristic boundary that yields a natural loss of normal derivatives, we manage to compensate this loss by estimating missing normal derivatives through the equations of the linearized vorticity. With the well-posedness and tame estimate for the effective linear problem in hand, we prove the local existence theorem for relativistic vortex sheets (see Theorem {\rm\ref{thm}}) by a Nash--Moser iteration scheme. We emphasize that our choice of new primary unknowns is essential for three main reasons: The system becomes symmetrizable hyperbolic; it has an appropriate form for the analysis of the Lopatinski\u{\i} determinant; and, most of all, it is suitable for getting a vorticity-type equation.

Characteristic discontinuities, especially vortex sheets, arise in a broad range of physical problems in fluid mechanics, oceanography, aerodynamics, plasma physics, astrophysics, and elastodynamics. The linear results in \cite{CS04MR2095445} have been generalized to cover the two-dimensional nonisentropic flows \cite{MT08MR2441089}, the three-dimensional compressible steady flows \cite{WY13MR3065290,WYuan15MR3327369}, and the two-dimensional two-phase flows \cite{RWWZ16MR3474128}. It is worth mentioning that a key ingredient in all of these proofs is the constant rank property of the boundary matrix. Recently, the methodology in \cite{CS04MR2095445} has been developed to deal with several constant coefficient linearized problems arising in two-dimensional compressible magnetohydrodynamics (MHD) and elastic flows; {\it cf.}\;\cite{WY13ARMAMR3035981,CDS16MR3527627,CHW17Adv}. {See also the very recent preprint \cite{CHW18} for the linear stability of elastic vortex sheets in the variable coefficient case.} For three-dimensional MHD, Chen--Wang \cite{CW08MR2372810, CW12MR3289359} and Trakhinin \cite{T09MR2481071} proved {independently} the nonlinear stability of compressible current-vortex sheets, which indicates that non-paralleled magnetic fields stabilize the motion of three-dimensional compressible vortex sheets. Moreover, the modified Nash--Moser iteration scheme developed in \cite{H76MR0602181,CS08MR2423311} {has} 
been successfully applied to the compressible liquids in vacuum \cite{T09MR2560044}, the plasma-vacuum interface problem \cite{ST14MR3151094},  MHD contact discontinuities \cite{MTT18MR3766987}, {and vortex sheets for three-dimensional steady flow \cite{WY15MR3328144} and two-dimensional two-phase flow \cite{HWY18}}.

Let us also mention some earlier works on the relativistic fluids. The global existence of discontinuous solutions to the relativistic Euler equations in one space dimension was first investigated by Smoller--Temple \cite{ST93MR1234105}. Also, Makino--Ukai \cite{MU95IIMR1346915} showed the existence of local smooth solutions in three space dimensions when the initial data is away from the vacuum. The stability of relativistic compressible flows with vacuum was addressed in \cite{T09MR2560044,JLM16MR3448785}. Moreover, the blow-up in finite time of smooth solutions for the relativistic Euler equations was shown in Pan--Smoller \cite{PS06MR2202310}. Also see Christodoulou \cite{C07,C18} for the formation and development of shocks in the multidimensional relativistic compressible fluids.

The plan of this paper is as follows: In \S\,\ref{sec.2}, after introducing the free boundary problem for relativistic vortex sheets, we reformulate the relativistic Euler equations and reduce our nonlinear problem to that in a fixed domain. Then we state the main result in this paper and introduce the weighted spaces and norms. Section \ref{sec.3} is mainly devoted to proving Theorem {\rm\ref{thm.1}}, {\it i.e.}  an energy estimate for the constant coefficient linearized problem. More precisely, after some reductions, we compute the roots of the associated Lopatinski\u{\i} determinant and deduce the criterion for weakly linear stability in \S\,\ref{sec.Lop1}. Then we adopt the argument developed recently by Chen--Hu--Wang \cite{CHW17Adv}  to prove the energy estimate for the constant coefficient case. In \S\,\ref{sec.4}, we introduce the effective linear problem and its reformulation. Section \ref{sec.5} is devoted to the proof of Theorem {\rm\ref{thm.2}}, the energy estimate for the effective linear problem. After deriving a weighted energy estimate with certain weights vanishing \emph{only} on the bicharacteristic curves starting from the critical set,  we can absorb the error terms caused by microlocalization and complete the proof of Theorem {\rm\ref{thm.2}}. In \S\,\ref{sec.6}, we prove a well-posedness result of the effective linear problem in the usual Sobolev space $H^s$ with $s$ large enough.
In \S\,\ref{secCA}, we  obtain the smooth ``approximate solution'' by imposing necessary compatible conditions on the initial data. Then the original problem \eqref{RE0} and \eqref{Phi.eq} is reduced into a nonlinear problem with zero initial data. In \S\,\ref{sec.NM}, by using a modification of the Nash--Moser iteration scheme, we show the existence of solutions to the reduced problem and conclude the proof of our main result, Theorem {\rm\ref{thm}}. Appendix \ref{AppA} concerns the motivation of introducing new primary unknowns and the derivation of the new symmetrization.

\section{Nonlinear Problems and the Main Theorem}\label{sec.2}

In this section, we first introduce the free boundary problem for relativistic vortex sheets, then reformulate the relativistic Euler equations and reduce our nonlinear problem to that in a fixed domain, and finally state the main theorem of this paper and introduce the weighted spaces and norms.

\subsection{Relativistic Vortex Sheets}\label{sec2.1}
We consider the equations of relativistic perfect fluid dynamics in the three-dimensional Minkowski spacetime $\mathbb{R}^{2+1}$, that is, the {relativistic}
Euler equations (see Lichnerowicz \cite{L67Relativistic}):
\begin{align}\label{RE1}
\partial_{\alpha}T^{\alpha\beta}=0,
\end{align}
where $T$ denotes the energy-momentum stress tensor with components
\begin{align} \notag
T^{\alpha\beta}=(p+\rho \epsilon^{-2})u^{\alpha}u^{\beta}+p g^{\alpha\beta}.
\end{align}
Here $p$ is the pressure, $\rho$ is the energy-mass density, $\epsilon^{-1}$ is the speed of light,
$g^{\alpha\beta}=\mathrm{diag}\,(-1,1,1)$ is the flat Minkowski metric,
and $\mathbf{u}=(u^0,u^1,u^2)^{\mathsf{T}}$ is the flow velocity satisfying
\begin{align}
\label{u.con}g^{\alpha\beta}u^{\alpha}u^{\beta}=-1.
\end{align}
The notation, $\partial_{\alpha}$, denotes the differentiation with respect to variable $x_{\alpha}$, and the {Greek} indices ``$\alpha$'' and ``$\beta$'' run from 0 to 2.
Throughout this paper, we use the Einstein summation convention whereby a repeated index
in a term implies the summation over all the values of that index.

We introduce the coordinate velocity $v=(v_1,v_2)^{\mathsf{T}}:=(u^1,u^2)^{\mathsf{T}}/(\epsilon u^0)$.
By virtue of \eqref{u.con}, the physical constraint is:
\begin{align} \label{v.con}
|v|<\epsilon^{-1}.
\end{align}
We also introduce the spacetime coordinates $(t,x)$ with $t:=\epsilon x_0$ and $x:=(x_1,x_2)$.
Then system \eqref{RE1} can be equivalently rewritten as
\begin{subequations} \label{RE2}
	\begin{alignat}{2}  \label{RE2.a}
	&\partial_t\left((\rho+\epsilon^2 p)\varGamma^2-\epsilon^2 p\right)
	+\nabla_x\cdot\left((\rho+\epsilon^2 p)\varGamma^2 v\right)=0,\\
	&\partial_t\left((\rho+\epsilon^2 p)\varGamma^2 v\right)
	+\nabla_x\cdot\left((\rho+\epsilon^2 p)\varGamma^2 v\otimes v\right)+\nabla_x  p=0, \label{RE2.b}
	\end{alignat}
\end{subequations}
where $\partial_t=\frac{\partial}{\partial t}$, $\nabla_x=(\partial_1,\partial_2)^{\mathsf{T}}$ with $\partial_j=\frac{\partial}{\partial x_j}$,
matrix $v\otimes v$ has $(i, j)$-entry $v_iv_j$, and
\begin{align}\label{Gamma.def}
\varGamma=\varGamma(v):=\frac{1}{\sqrt{1-\epsilon^2|v|^2}}
\end{align}
is the Lorentz factor.
The fluid is assumed to be barotropic, which means that pressure $p$ is given by an explicit function of $\rho$.
We also assume that $p=p(\rho)$ is a $C^{\infty}$ function defined on $(\rho_{*},\rho^{*})$ and satisfies
\begin{align}
\label{p.con}
{
	0<p'(\rho)<\epsilon^{-2}\qquad\textrm{for all }\rho\in(\rho_{*},\rho^{*}),
}
\end{align}
where $\rho_{*}$ and $\rho^{*}$ are some constants such that $0\leq \rho_{*}<\rho^{*}\leq \infty$.
Consequently, density $\rho$ is a strictly increasing function of $p$ defined on $(p(\rho_{*}),p(\rho^{*}))$,
and system \eqref{RE2} is closed with three unknowns $(\rho, v_1, v_2)$.
Barotropic fluids arise in many physical situations such as very cold matter, nuclear matter and ultrarelativistic fluids ({\it cf}.\;\cite[Chapter\;II]{zbMATH05040442}, \cite[Chapter\;IX]{C-B09MR2473363}, and \cite[\S\,1]{ST93MR1234105}).
%

Let $(\rho,v)(t,x)$ be smooth functions on either side of a smooth hypersurface $\Sigma(t):=\{x_2=\varphi(t,x_1)\}$.
Then $(\rho, v)$ is a weak solution of \eqref{RE2} if and only if $(\rho, v)$ is a classical solution of \eqref{RE2}
on each side of $\Sigma(t)$ and satisfies the Rankine--Hugoniot conditions at every point of $\Sigma(t)$:
\begin{align} \label{RH}
\left\{\begin{aligned}
&\partial_t\varphi\left[(\rho+\epsilon^2 p)\varGamma^2-\epsilon^2 p\right]-\left[(\rho+\epsilon^2 p)\varGamma^2 v\cdot \nu\right]=0,\\[1mm]
&\partial_t\varphi\left[(\rho+\epsilon^2 p)\varGamma^2 v\right]-\left[(\rho+\epsilon^2 p)\varGamma^2 (v\cdot\nu)v\right]-\left[p\right]\nu=0,
\end{aligned}\right.
\end{align}
where $\nu:=(-\partial_{1}\varphi,1)$ is a spatial normal vector to $\Sigma(t)$.
As usual, for any function $g$, we denote by $g^{\pm}$ the value of $g$ in ${\{\pm (x_2- \varphi(t,x_1))>0\}}$,
and $[g]:=g^+|_{\Sigma(t)}-g^-|_{\Sigma(t)}$ the jump across $\Sigma(t)$.

In this paper, we are interested in weak solutions $(\rho,v)$ of \eqref{RE2} such that the tangential velocity (with respect to $\Sigma(t)$)
is the only jump experienced by the solution $(\rho,v)$ across $\Sigma(t)$.
Then the Rankine--Hugoniot conditions \eqref{RH} are reduced to
\begin{align}\label{RH1}
\partial_t\varphi=v^+\cdot\nu=v^-\cdot\nu,
\quad p^+=p^-\qquad \mathrm{on}\ \Sigma(t).
\end{align}
A piecewise smooth weak solution $(\rho,v)$ of \eqref{RE2} with discontinuities across $\Sigma(t)$ is
called a  \emph{relativistic vortex sheet} if its trace on $\Sigma(t)$ satisfies \eqref{RH1}.

We note that system \eqref{RE2} admits trivial vortex-sheet solutions that consist of two constant states
separated by a rectilinear front:
\begin{align} \label{RVS1}
(\rho,v)(t,x_1,x_2)=\left\{\begin{aligned}
&(\bar{\rho},\bar{v},0)\quad &\mathrm{if}\ x_2>0,\\
&(\bar{\rho},-\bar{v},0)\quad &\mathrm{if}\ x_2<0,
\end{aligned}\right.
\end{align}
where $\bar{\rho}$ and $\bar{v}$ are suitable positive constants.
Every rectilinear relativistic vortex sheet is of this form by changing the observer if necessary.
In view of \eqref{v.con} and \eqref{p.con}, we may assume without loss of generality that
\begin{align} \label{H1}
\bar{\rho}\in (\rho_{*},\rho^*), \qquad  \bar{v}\in(0,\epsilon^{-1}).
\end{align}

The aim of this paper is to study the {local-in}-time existence and nonlinear stability of relativistic vortex sheets
with initial data close to the piecewise constant state \eqref{RVS1}.

\subsection{Reformulation and the Main Theorem}\label{sec2.2}
Let us first reformulate the relativistic Euler equations \eqref{RE2}
by choosing appropriate functions as primary unknowns.
To this end, we define the particle number density $N=N(\rho)$, the sound speed $c=c(\rho)$,
and $h=h(\rho)$ by
\begin{align}\label{N.c.def}
N(\rho):=\exp\left(\int_{\bar{\rho}}^{\rho}\frac{\mathrm{d}s}{s+\epsilon^2p(s)}\right),\ \
c(\rho):=\sqrt{p'(\rho)},\ \  h(\rho):=\frac{\rho+\epsilon^2 p(\rho)}{N(\rho)}.
\end{align}
We also introduce
\begin{align} \label{w.def}
w:=\varGamma v=\frac{v}{\sqrt{1-\epsilon^2|v|^2}},
\end{align}
so that
\begin{align}\label{Gamma.v.1}
\varGamma=\sqrt{1+\epsilon^2|w|^2}, \qquad v=\frac{w}{\sqrt{1+\epsilon^2|w|^2}}.
\end{align}
Then we discover that smooth solutions to system \eqref{RE2} satisfy
\begin{align} \label{hw.eq}
\varGamma (\partial_t+v\cdot\nabla_x)(hw) +N^{-1}\nabla_x p=0.
\end{align}
Let us take $U:=(p,hw_1,hw_2)^{\mathsf{T}}$ as primary unknowns and define the following matrices:
\setlength{\arraycolsep}{4pt}
\begin{align}
\label{A0.def} A_0(U)&:=\begin{pmatrix}\varGamma( 1-{\epsilon}^{4}c^2|v|^{2} ) & 2\epsilon^2 N c^2 v_1 & 2\epsilon^2 N c^2 v_2\\[0.5mm]
0 & \varGamma (1-\epsilon^2 v_1^2) & -\epsilon^2\varGamma v_1v_2\\[0.5mm]
0 &-\epsilon^2\varGamma v_1v_2 & \varGamma (1-\epsilon^2 v_2^2)
\end{pmatrix},     \\[1mm]
\label{A1.def}A_1(U)&:=\begin{pmatrix}
\varGamma v_1( 1-{\epsilon}^{4}c^2|v|^{2} )  & Nc^2(1+\epsilon^2 v_1^2) &\epsilon^2 v_1v_2 Nc^2\\[0.5mm]
N^{-1} (1-\epsilon^2 v_1^2)& \varGamma v_1(1-\epsilon^2 v_1^2)&-\epsilon^2\varGamma v_1^2v_2\\[0.5mm]
-\epsilon^2 v_1v_2 N^{-1} &-\epsilon^2\varGamma v_1^2v_2&\varGamma v_1(1-\epsilon^2 v_2^2)
\end{pmatrix},\\[1mm]
\label{A2.def}A_2(U)&:=\begin{pmatrix}
\varGamma v_2( 1-{\epsilon}^{4}c^2|v|^{2} ) & \epsilon^2 v_1v_2 Nc^2&Nc^2(1+\epsilon^2 v_2^2)\\[0.5mm]
-\epsilon^2 v_1v_2 N^{-1} & \varGamma v_2(1-\epsilon^2 v_1^2)&-\epsilon^2\varGamma v_1v_2^2\\[0.5mm]
N^{-1} (1-\epsilon^2 v_2^2)&-\epsilon^2\varGamma v_1v_2^2& \varGamma v_2(1-\epsilon^2 v_2^2)
\end{pmatrix}.
\end{align}
When the solution is in $C^1$, system \eqref{RE2} equivalently reads
\begin{align}    \label{RE4}
A_0(U)\partial_t U+A_1(U)\partial_{1} U+A_2(U)\partial_{2} U=0.
\end{align}
We postpone proving the equivalence of systems \eqref{RE2} and \eqref{RE4} to Appendix \ref{AppA}.
The choice of the new unknowns $U$ has several advantages.
First, system \eqref{RE4} is symmetrizable hyperbolic
in region $\{\rho_{*}<\rho<\rho^{*},|v|<\epsilon^{-1}\}$ (see Appendix \ref{AppA}
for the precise expression of the Friedrichs symmetrizer).
Second, we will see in the sequel that the form of \eqref{RE4} is appropriate for computing the roots of the Lopatinski\u{\i} determinant.
Third, equations \eqref{hw.eq} will enable us to obtain the linearized vorticity equation through
which the loss of derivatives can be compensated in the higher-order energy estimates.

Note that the first two identities in \eqref{RH1} are the \emph{eikonal} equations:
\begin{align}   \notag
\partial_t\varphi+\lambda_2(U^+,\partial_{1}\varphi)=0,  \qquad\,
\partial_t\varphi+\lambda_2(U^-,\partial_{1}\varphi)=0,
\end{align}
where $\lambda_2(U,\xi):=v_1\xi-v_2$, and $\xi\in\mathbb{R}$ is the second characteristic field of \eqref{RE4}
with the corresponding eigenvector:
\begin{align} \notag
r_2(U,\xi):=(0,\, 1-\epsilon^2 v_2^2+\epsilon^2 v_1v_2\xi,\,(1-\epsilon^2v_1^2)\xi+\epsilon^2v_1v_2)^{\mathsf{T}}.
\end{align}
It follows from \eqref{N.c.def} and \eqref{Gamma.v.1} that $\nabla_U \lambda_2(U,\xi) \cdot r_2(U,\xi)\equiv0$, {\it i.e.}
the characteristic field $\lambda_2$ is linearly degenerate in the sense of Lax \cite{L57MR0093653}.
As a consequence, a relativistic vortex sheet is a characteristic discontinuity.

Function $\varphi$ describing the discontinuity front is a part of the unknowns,
and thus the relativistic vortex sheet problem is a free boundary problem.
To reformulate this problem in a fixed domain, we replace unknowns $U$, which are smooth on either side of $\Sigma(t)$,
by
\begin{align} \label{transform}
U^{\pm}_{\sharp}(t,x):=U(t,x_1,\Phi^{\pm}(t,x)),
\end{align}
where $\Phi^{\pm}$ are smooth functions satisfying the constraints:
\begin{align} \notag
\Phi^{\pm}(t,x_1,0)=\varphi(t,x_1),\quad \pm\partial_2\Phi^{\pm}(t,x)\geq \kappa>0 \quad {\rm if}\ x_2\geq 0  .
\end{align}
Then the existence of relativistic vortex sheets amounts to constructing solutions $U^{\pm}_{\sharp}$,
which are smooth in the fixed domain $\{x_2>0\}$, to the following initial-boundary value problem:
\begin{subequations}   \label{RE0}
	\begin{alignat}{2} \label{RE0.a}
	&\mathbb{L}(U^{\pm},\Phi^{\pm}) =0&  \qquad &\mathrm{if}\  x_2>0,\\ \label{RE0.b}
	&\mathbb{B}(U^{+},U^{-},\varphi)=0& \qquad  &\mathrm{if}\ x_2=0,\\  \label{RE0.c}
	&(U^{\pm},\varphi)|_{t=0}=(U^{\pm}_0,\varphi_0),& \qquad &
	\end{alignat}
\end{subequations}
where index ``$\sharp$'' has been dropped for notational simplicity.
According to transformation \eqref{transform}, operators $\mathbb{L}$ and $\mathbb{B}$ take the forms:
\begin{align}
\label{L.def} &\left\{
\begin{aligned}&\mathbb{L}(U,\Phi)=L(U,\Phi)U \\
&\mbox{with}\quad
L(U,\Phi):=A_0(U)\partial_t+A_1(U)\partial_1+\widetilde{A}_2(U,\Phi)\partial_2,
\end{aligned}\right.
\\ \label{B.def}
&\mathbb{B}(U^+,U^-,\varphi):=
\begin{pmatrix}
[v_1] \partial_1\varphi-[v_2]\\[0.5mm]
\partial_t\varphi+v_1^+\partial_1\varphi-v_2^+\\[0.5mm]
p^+-p^-
\end{pmatrix},
\end{align}
where $A_j(U)$, $j=0,1,2$, are defined by \eqref{A0.def}, \eqref{A1.def}, \eqref{A2.def}, respectively, and
\begin{align}  \notag 
\widetilde{A}_2(U, \Phi):=
\frac{1}{\partial_2\Phi}\big(A_2(U)-\partial_t\Phi A_0(U)-\partial_1\Phi A_1(U)\big).
\end{align}

As in Francheteau--M\'{e}tivier \cite{FM00MR1787068}, we choose the change of variables $\Phi^{\pm}$  such that
\begin{subequations} \label{Phi.eq}
	\begin{alignat}{2}
	\label{Phi.eq.a}&\partial_t\Phi^{\pm}+v_1^{\pm}\partial_1\Phi^{\pm}-v_2^{\pm}=0\qquad &&\mathrm{if}\ x_2\geq 0,\\
	\label{Phi.eq.b}&\pm\partial_2\Phi^{\pm}\geq \kappa>0\qquad &&\mathrm{if}\ x_2\geq 0,\\
	\label{Phi.eq.c}&\Phi^{+}=\Phi^{-}=\varphi\qquad &&\mathrm{if}\ x_2= 0.
	\end{alignat}
\end{subequations}
Not only does this choice simplify much the expression of system \eqref{RE0.a},
but it also implies that the boundary matrix for problem \eqref{RE0}:
$$\mathrm{diag}\,(-\widetilde{A}_2(U^+, \Phi^+),\,-\widetilde{A}_2(U^-, \Phi^-)),$$
has constant rank on the whole closed half-space $\{x_2\geq 0\}$.
This will play a crucial role in deriving the energy estimates for the variable coefficient linearized problem
by developing further the Kreiss' symmetrizers argument from \cite{K70MR0437941,C04MR2069632,CS04MR2095445}.

In the new variables, the rectilinear vortex sheet \eqref{RVS1} corresponds to the following stationary
solution of \eqref{RE0.a}--\eqref{RE0.b} and \eqref{Phi.eq}:
\begin{align}    \label{RVS0}
\widebar{U}^{\pm}:=\big(\bar{p},\pm\bar{h}\widebar{w},0\big)^{\mathsf{T}},\quad\, \widebar{\varphi}:=0,\quad\,
\widebar{\Phi}^{\pm}:=\pm x_2,
\end{align}
where $\bar{p}:=p(\bar{\rho})$, $\bar{h}:=h(\bar{\rho})$, and $\widebar{w}:=\widebar{\varGamma} \bar{v}$ with
$\widebar{\varGamma}^{-1}:=\sqrt{1-\epsilon^2\bar{v}^2}$.

Imposing the smooth initial data $(U_0^{\pm},\varphi_0)$ close to \eqref{RVS0},
we aim to show the existence of solutions to the nonlinear problem \eqref{RE0} and \eqref{Phi.eq} under
the necessary condition for the linear stability on the background state \eqref{RVS0}.
The main result is stated as follows:

\begin{theorem}\ \label{thm}
	Let $T>0$ be any fixed constant and $\mu\in\mathbb{N}$ with $\mu\geq 13$.
	Assume that the background state \eqref{RVS0} satisfies the physical constraints \eqref{H1}
	and the necessary stability condition:
	\begin{align}\label{H2}
	\emph{$\textsf{M}$}:=\frac{|\bar{v}|}{\bar{c}}>\frac{\sqrt{2}}{\sqrt{1+\epsilon^2\bar{c}^2}}
	\qquad \mathrm{with}\quad  \bar{c}:=c(\bar{\rho}).
	\end{align}
	Assume further that the initial data $U^{\pm}_0$ and $\varphi_0$ satisfy the compatibility conditions
	up to order $\mu$ {\rm(}see {\rm\S\,\ref{secCA}}{\rm)},
	and that $(U^{\pm}_0-\widebar{U}^{\pm},\varphi_0)\in H^{\mu+1/2}(\mathbb{R}^2_+)\times H^{\mu+1}(\mathbb{R})$ has a compact support.
	Then there exists a positive constant $\varepsilon$ such that,
	if $\|U^{\pm}_0-\widebar{U}^{\pm}\|_{H^{\mu+1/2}(\mathbb{R}^2_+)}+\|\varphi_0\|_{H^{\mu+1}(\mathbb{R})}\leq \varepsilon$,
	problem \eqref{RE0} and \eqref{Phi.eq} has a solution $(U^{\pm},\Phi^{\pm},\varphi)$ on the time interval $[0,T]$
	satisfying
	$$
	(U^{\pm}-\widebar{U}^{\pm},\Phi^{\pm}-\widebar{\Phi}^{\pm})\in H^{\mu-7}((0,T)\times\mathbb{R}^2_+),\qquad
	\varphi\in H^{\mu-6}((0,T)\times\mathbb{R}).
	$$
\end{theorem}

\begin{remark}\
	In the non-relativistic limit $\epsilon\to 0$, from \eqref{H2}, one obtains the classical stability
	condition $\textsf{M}>\sqrt{2}$ for compressible vortex sheets.
	The critical Mach number
	$$
	\textsf{M}_c:=\frac{\sqrt{2}}{\sqrt{1+\epsilon^2\bar{c}^2}}
	$$
	of the relativistic stability condition is always
	strictly smaller than $\sqrt{2}$, which means that the relativistic vortex sheets are stable
	in a larger physical regime of the parameters.
	When  $\bar{c}$ is arbitrarily close to the light speed $\epsilon^{-1}$,  the critical Mach number $\textsf{M}_c$ approaches $1$
	so that the stability holds precisely for supersonic relativistic flows.
\end{remark}


\subsection{Weighted Sobolev Spaces and Norms}\label{sec2.w}
We are going to introduce certain weighted Sobolev spaces in order to prove Theorem {\rm\ref{thm}}.
Let $\Omega$ denote the half-space $\{(t,x_1,x_2)\in\mathbb{R}^3:x_2>0\}$.
Boundary $\partial\Omega$ is identified to $\mathbb{R}^2$.
For all $s\in\mathbb{R}$ and $\gamma\geq 1$,
the usual Sobolev space $H^s(\mathbb{R}^2)$ is equipped with the following norm:
\begin{align*}
\|v\|_{s,\gamma}^2:=\frac{1}{(2\pi)^2} \int_{\mathbb{R}^2}\lambda^{2s,\gamma}(\xi) |\widehat{v}(\xi)|^2\d \xi,\qquad
\lambda^{s,\gamma}(\xi):=(\gamma^2+|\xi|^2)^{\frac{s}{2}},
\end{align*}
where $\widehat{v}$ is the Fourier transform of $v$.
We equip space $L^2(\mathbb{R}_+;H^{s}(\mathbb{R}^2))$ with the norm:
\begin{align*}
\VERT v \VERT_{s,\gamma}^2:=\int_{\mathbb{R}_+}\|v(\cdot,x_2)\|_{s,\gamma}^2\d x_2.
\end{align*}
We will abbreviate the usual norms of $L^2(\mathbb{R}^2)$ and $L^2(\Omega)$ as
\begin{align*}
\|\cdot\|:=\|\cdot\|_{0,\gamma}\quad\, \mathrm{and}\quad\,  \VERT\cdot\VERT  :=\VERT\cdot\VERT_{0,\gamma}.
\end{align*}
The scalar products in $L^2(\mathbb{R}^2)$  and $L^2(\Omega)$ are denoted as follows:
\begin{align*}
\langle a,b\rangle:=\int_{\mathbb{R}^2} a(x)\overline{b(y)}\d y,\qquad
\llangle a, b\rrangle:=\int_{\Omega} a(y)\overline{b(y)}\d y,
\end{align*}
where $\overline{b(y)}$ is the complex conjugation of $b(y)$.

For $s\in\mathbb{R}$ and $\gamma\geq 1$, we introduce the weighted Sobolev
space $H^{s}_{\gamma}(\mathbb{R}^2)$ as
\begin{align*}
H^{s}_{\gamma}(\mathbb{R}^2)&:=\left\{
u\in\mathcal{D}'(\mathbb{R}^2)\,:\, \mathrm{e}^{-\gamma t}u(t,x_1)\in
H^{s}(\mathbb{R}^2) \right\},
\end{align*}
and its norm $\|u\|_{H^{s}_{\gamma}(\mathbb{R}^2)}:=\|\mathrm{e}^{-\gamma t}u\|_{s,\gamma}$.
We write $L^2_{\gamma}(\mathbb{R}^2):=H^0_{\gamma}(\mathbb{R}^2)$ and $\|u\|_{L^2_{\gamma}(\mathbb{R}^2)}:=\|\mathrm{e}^{-\gamma t}u\|$.

We define $L^2(\mathbb{R}_+;H^{s}_{\gamma}(\mathbb{R}^2))$, briefly denoted by $L^2(H^s_{\gamma})$,
as the space of distributions with finite $L^2(H^s_{\gamma})$--norm, where
\begin{align*}
\|u\|_{L^2(H^s_{\gamma})}^2:=\int_{\mathbb{R}_+}\|u(\cdot,x_2)\|_{H^s_{\gamma}(\mathbb{R}^2)}^2\d x_2
=\VERT \mathrm{e}^{-\gamma t}u\VERT_{s,\gamma}^2.
\end{align*}
We set $L^2_{\gamma}(\Omega):=L^2(H^0_{\gamma})$ and $\|u\|_{L^2_{\gamma}(\Omega)}:=\VERT \mathrm{e}^{-\gamma t}u\VERT$.

For all $k\in\mathbb{N}$ and $\gamma\geq 1$, we define the weighted Sobolev space $H^k_{\gamma}(\Omega)$ as
\begin{align*}
H^{k}_{\gamma}(\Omega)&:=\left\{ u\in\mathcal{D}'(\Omega)\,:\, \mathrm{e}^{-\gamma t}u\in   H^{k}(\Omega) \right\}.
\end{align*}

Throughout the paper, we introduce the notation: $A\lesssim B$ ($B\gtrsim A$) if $A\leq CB$ holds uniformly
for some positive constant $C$ that is \emph{independent} of $\gamma$.
The notation, $A\sim B$, means that both $A\lesssim B$ and $B\lesssim A$.
Then, for $k\in\mathbb{N}$, one has
\begin{align}\label{sim}
\|u\|_{k,\gamma}\sim \sum_{|\alpha|\leq k}\gamma^{k-|\alpha|}\|\partial^{\alpha} u\|
\qquad\,\textrm{for all }u\in H^{k}(\mathbb{R}^2).
\end{align}

For any real number $T$, we introduce $\omega_T:=(-\infty,T)\times\mathbb{R}$ and $\Omega_T:=\omega_T\times\mathbb{R}_+$.
For all $k\in\mathbb{N}$ and $\gamma\geq 1$, we define the weighted space $H^k_{\gamma}(\Omega_T)$
as
$$
H^k_{\gamma}(\Omega_T):=\left\{u\in\mathcal{D}'(\Omega_T)\,:\, \mathrm{e}^{-\gamma t}u\in H^{k}(\Omega_T)\right\}.
$$
In view of relation \eqref{sim}, we introduce the norm on $H^k_{\gamma}(\Omega_T)$ as
\begin{align} \label{norm.def}
\|u\|_{H^k_{\gamma}(\Omega_T)}:=\sum_{|\alpha|\leq k}\gamma^{k-|\alpha|}\|\mathrm{e}^{-\gamma t} \partial^{\alpha}u\|_{L^2(\Omega_T)}.
\end{align}
The norm on $H^k_{\gamma}(\omega_T)$ is defined in the same way.
For all $k\in\mathbb{N}$ and $\gamma\geq 1$, we define space $L^2(\mathbb{R}_+;H^k_{\gamma}(\omega_T))$,
briefly denoted by $L^2(H^k_{\gamma}(\omega_T))$, as the space of distributions with
finite $L^2(H^k_{\gamma}(\omega_T))$--norm, where
\begin{align*}
\|u\|_{L^2(H^k_{\gamma}(\omega_T))}^2:=\,&\int_{\mathbb{R}_+}\|u(\cdot,x_2)\|_{H^k_{\gamma}(\omega_T)}^2\d x_2\\
=\,&\sum_{\alpha_0+\alpha_1\leq k}
\gamma^{k-\alpha_0-\alpha_1}\|\mathrm{e}^{-\gamma t} \partial_t^{\alpha_0}\partial_1^{\alpha_1}u\|_{L^2(\Omega_T)}.
\end{align*}
This is an anisotropic Sobolev space for measuring only the tangential
regularity (with respect to boundary $\partial\Omega$).
We write $L^2_{\gamma}(\Omega_T):=L^2(H^0_{\gamma}(\omega_T))$
and $\|u\|_{L^2_{\gamma}(\Omega_T)}:=\|u\|_{L^2(H^0_{\gamma}(\omega_T))}$.

\section{Constant Coefficient Linearized Problem}\label{sec.3}

In order to deduce the necessary condition for the linear stability of the background state \eqref{RVS0},
in this section, we consider
the following linearized problem of \eqref{RE0} and \eqref{Phi.eq} around \eqref{RVS0}:
\begin{subequations} \label{P1}
	\begin{alignat}{3} \label{P1.a}
	&\mathbb{L}_{\pm}'V^{\pm}:=\left.\frac{\mathrm{d}}{\mathrm{d}\theta}\mathbb{L}
	\big(U^{\pm}_{\theta},\Phi^{\pm}_{\theta}\big)\right|_{\theta=0} =f^{\pm}\qquad &\mathrm{if}\ x_2>0,\\
	&\mathbb{B}'(V^{+},V^-,\psi):=\left.\frac{\mathrm{d}}{\mathrm{d}\theta}\mathbb{B}(U^{+}_{\theta},U^{-}_{\theta},\varphi_{\theta})\right|_{\theta=0}
	=g\qquad &\mathrm{if}\ x_2=0,
	\label{P1.b}
	\end{alignat}
\end{subequations}
where $U^{\pm}_{\theta}:=\widebar{U}^{\pm}+\theta V^{\pm}$, $\Phi^{\pm}_{\theta}:=\widebar{\Phi}^{\pm}+\theta\Psi^{\pm}$,
and $\varphi_{\theta}$ (resp.\;$\psi$) denotes the common trace of $\Phi^{\pm}_{\theta}$ (resp.\;$\Psi^{\pm}$)
on the boundary $\{x_2=0\}$.
The differential operators $\mathbb{L}_{+}'$ and $\mathbb{L}_{-}'$ are given by
\setlength{\arraycolsep}{1pt}
\begin{align*}
&\mathbb{L}_{\pm}':=A_0\big(\widebar{U}^{\pm}\big)\partial_t+A_1\big(\widebar{U}^{\pm}\big)\partial_1 \pm    A_2\big(\widebar{U}^{\pm}\big)\partial_2,
\end{align*}
which are both constant coefficient differential operators.
It follows from \eqref{p.con} and \eqref{N.c.def} that
\begin{align}\label{H1a}
N(\bar{\rho})=1,\qquad   \bar{c}=c(\bar{\rho})\in(0,\epsilon^{-1}).
\end{align}
To derive the boundary operator $\mathbb{B}'$, we infer from \eqref{Gamma.v.1} that
\begin{align}\label{Gamma.v}
\varGamma(U)=\frac{\sqrt{h(U_1)^2+\epsilon^2 U_2^2+\epsilon^2 U_3^2}}{h(U_1)},\quad
v_j(U)=\frac{U_{j+1}}{\varGamma(U)h(U_1)},\quad j=1,2.
\end{align}
Utilizing the identity: $h'(U_1)=\epsilon^2/N(U_1)$ yields
\begin{align} \label{v.partial}
\frac{\p v_j}{\p U_1 }=-\frac{\epsilon^2 v_j}{N h\varGamma^2},\quad\,
\frac{\p v_j}{\p U_{j+1} }=\frac{1-\epsilon^2 v_j^2}{h\varGamma},\quad\,
\frac{\p v_2}{\p U_2 }=\frac{\p v_1}{\p U_3 }=-\frac{\epsilon^2 v_1 v_2}{h\varGamma},
\end{align}
for $ j=1,2.$
The second component of $\mathbb{B}(U^{+}_{\theta},U^{-}_{\theta},\varphi_{\theta})$ is
$\theta (\p_t\psi+v_{1}(U^{+}_{\theta})\p_1\psi)-v_{2}(U^{+}_{\theta})$.
Then we use  \eqref{v.partial} to obtain
\begin{align*}
\left.\frac{\d}{\d \theta}\left(\mathbb{B}(U^{+}_{\theta},U^{-}_{\theta},\varphi_{\theta})\right)_2\right|_{\theta=0}
=\p_t\psi+\bar{v}^+\p_1\psi-(\bar{h}\widebar{\varGamma})^{-1}  V_3^+.
\end{align*}
After similar argument to the other components of $\mathbb{B}(U^{+}_{\theta},U^{-}_{\theta},\varphi_{\theta})$,
we have
\setlength{\arraycolsep}{3pt}
\begin{align}\notag
\mathbb{B}'(V^{+},V^-,\psi)=\mathbb{B}'(V^{\mathrm{nc}},\psi):=
\bar{b}\nabla\psi+\hat{B} V^{\mathrm{nc}},
\end{align}
where $\nabla\psi:=(\partial_t\psi,\partial_1\psi)^{\mathsf{T}}$,
$V^{\mathrm{nc}}:=(V_1^+,V_3^+,V_1^-,V_3^-)^{\mathsf{T}}$ denotes the ``noncharacteristic part''
of $V:=(V^+,V^-)^{\mathsf{T}}$, and coefficients $\bar{b}$ and $\hat{B}$ are given by
\begin{align}\label{b.M.bar}
\bar{b}:=\begin{pmatrix}
0  &2\bar{v}\\[0.5mm]
1  &\bar{v}\\[0.5mm]
0  &0
\end{pmatrix},
\qquad
\hat{B}:=\begin{pmatrix}
0  & -\big(\bar{h}\widebar{\varGamma}\big)^{-1} &0 & \big(\bar{h}\widebar{\varGamma}\big)^{-1} \\[0.5mm]
0 & -\big(\bar{h}\widebar{\varGamma}\big)^{-1}  &0 & 0\\[0.5mm]
1  &0  &-1  &0
\end{pmatrix}.
\end{align}

We are now ready to state the main result for the constant coefficient case.

\begin{theorem}\  \label{thm.1}
	Assume that the stationary solution defined by \eqref{RVS0} satisfies \eqref{H1} and \eqref{H2}.
	Then, for all $\gamma\geq 1$ and for all $(V ,\psi)\in H^2_{\gamma}(\Omega)\times H^2_{\gamma}(\mathbb{R}^2)$,
	the following estimate holds{\rm :}
	\begin{align} \notag
	&\gamma \|V\|_{L^2_{\gamma}(\Omega)}^2+\|{V^{\mathrm{nc}}}|_{x_2=0}\|_{L^2_{\gamma}(\mathbb{R}^2)}^2
	+\|\psi\|_{H^1_{\gamma}(\mathbb{R}^2)}^2
	\\&\quad \lesssim \gamma^{-3}\|\mathbb{L}_{\pm}'V^{\pm}\|_{L^2(H_{\gamma}^1)}^2
	+\gamma^{-2}\|\mathbb{B}'({V^{\mathrm{nc}}}|_{x_2=0},\psi)\|_{H^1_{\gamma}(\mathbb{R}^2)}^2.
	\label{thm.1e}
	\end{align}
\end{theorem}

\begin{remark}\ \label{rem3.1}
	In the case of $\textsf{M}<\textsf{M}_{\rm c}:=\frac{\sqrt{2}}{\sqrt{1+\epsilon^2\bar{c}^2}}$,
	the relativistic vortex sheet \eqref{RVS0} is violently unstable, \emph{i.e.}
	the Lopatinski\u{\i} determinant admits the roots in the interior of frequency space.
	On the other hand, when $\textsf{M}\ge \textsf{M}_{\rm c}$,
	all the roots of the Lopatinski\u{\i} determinant are localized on the boundary of frequency space.
	In particular, if $\textsf{M}=\textsf{M}_{\rm c}$,
	the only root of the Lopatinski\u{\i} determinant is a triple one,
	which leads to the following weaker estimate than \eqref{thm.1e}:
	\begin{align} \notag
	&\gamma \|V\|_{L^2_{\gamma}(\Omega)}^2+\|{V^{\mathrm{nc}}}|_{x_2=0}\|_{L^2_{\gamma}(\mathbb{R}^2)}^2+\|\psi\|_{H^1_{\gamma}(\mathbb{R}^2)}^2\\
	&\quad \lesssim
	\gamma^{-7}\|\mathbb{L}_{\pm}'V^{\pm}\|_{L^2(H_{\gamma}^3)}^2+\gamma^{-6}\|\mathbb{B}'({V^{\mathrm{nc}}}|_{x_2=0},\psi)\|_{H^3_{\gamma}(\mathbb{R}^2)}^2,\label{rem.e0}
	\end{align}
	for all $\gamma\geq 1$ and $(V,\psi)\in H^4_{\gamma}(\Omega)\times H^4_{\gamma}(\mathbb{R}^2)$.
	See Remarks {\rm \ref{rem3.2}}--{\rm \ref{rem3.4}} for more details.
	This latter case corresponds to a transition between a weakly stable zone and
	a violently unstable
	zone {(}{\it cf}.\;Coulombel--Secchi {\rm \cite{CS04bMR2099568}} for the non-relativistic case{)}.
\end{remark}

The rest of this section is devoted mainly to proving Theorem {\rm\ref{thm.1}}.

\subsection{Some Reductions}\label{sec3.1}

Before proving Theorem {\rm\ref{thm.1}}, we first make some reductions of problem \eqref{P1}.

\subsubsection{Reformulation of Theorem {\rm\ref{thm.1}}}
We first transform our problem \eqref{P1} into the one with diagonal boundary matrix.
For this purpose, we calculate the eigenvalues and corresponding eigenvectors of $A_2\big(\widebar{U}^{\pm}\big)$.
The eigenvalues of
\begin{align*}
A_2\big(\widebar{U}^{+}\big)=A_2\big(\widebar{U}^{-}\big)=
\begin{pmatrix}
0&0&\bar{c}^2\\
0&0&0\\
1&0&0
\end{pmatrix}
\end{align*}
are $\lambda_1=0,$ $\lambda_2=-\bar{c}$, and  $\lambda_3=\bar{c}$,
with the corresponding right eigenvectors:
\begin{align*}
r_1=(0,1,0)^{\mathsf{T}},\quad
r_2=(1,0,-\frac{1}{\bar{c}})^{\mathsf{T}},\quad
r_3=(1,0, \frac{1}{\bar{c}})^{\mathsf{T}}.
\end{align*}
Set $\widebar{R}:=(r_1\ r_2\ r_3)$. Then
$
\widebar{R}^{-1}A_2\big(\widebar{U}^{\pm}\big)\widebar{R}=\mathrm{diag}\,(0,-\bar{c},\bar{c}).
$
We thus perform the linear transformation $W^{\pm}:=\widebar{R}^{-1} V^{\pm}$ with
\begin{align*}
W_1^{\pm}:=V_2^{\pm},\quad
W_2^{\pm}:=\tfrac{1}{2}\left(V_1^{\pm}-\bar{c}V_3^{\pm}\right),\quad
W_3^{\pm}:=\tfrac{1}{2}\left(V_1^{\pm}+\bar{c}V_3^{\pm}\right).
\end{align*}
Let us multiply \eqref{P1.a} by $\widebar{S}\widebar{R}^{-1}$ with
$
\widebar{S}:=\mathrm{diag}\,(1,2/\bar{c}^{2},2/\bar{c}^{2}).
$
Then problem \eqref{P1} becomes equivalent to
\begin{align}
\label{P1a}
\left\{ \begin{aligned}
&\mathcal{L}W:=\mathcal{A}_0\partial_tW+\mathcal{A}_1\partial_1W+\mathcal{A}_2\partial_2W=f\quad &\mathrm{if}\ x_2>0,\\
&\mathcal{B}(W^{\mathrm{nc}},\psi):=\bar{b}\nabla\psi+\widebar{B} W^{\mathrm{nc}}=g\quad &\mathrm{if}\ x_2=0,
\end{aligned} \right.
\end{align}
with new $f$ and $g$,
where $W^{\mathrm{nc}}:=(W_2^+,W_3^+,W_2^-,W_3^-)^{\mathsf{T}}$ denotes
the ``noncharacteristic part'' of $W:=(W^+,W^-)^{\mathsf{T}}$.
The coefficient matrices $\mathcal{A}_j=\mathrm{diag}\,(\mathcal{A}_j^+,\mathcal{A}_j^-)$, $j=0,1,2$,
are block diagonal with
\setlength{\arraycolsep}{1pt}
\begin{align}\notag
\mathcal{A}_0^{\pm}:=\,&\widebar{S}\widebar{R}^{-1}A_0(\widebar{U}^{\pm})\widebar{R}\\
\label{A0.cal}=\,&\begin{pmatrix}
\widebar{\varGamma} (1-\epsilon^2\bar{v}^2)&0&0\\
\pm 2 \epsilon^2 \bar{v}&\dfrac{\widebar{\varGamma}(2-\epsilon^4\bar{c}^2\bar{v}^2)}{\bar{c}^2} &-\epsilon^4\widebar{\varGamma}\bar{v}^2\\
\pm 2 \epsilon^2 \bar{v} &-\epsilon^4\widebar{\varGamma}\bar{v}^2&\dfrac{\widebar{\varGamma}(2-\epsilon^4\bar{c}^2\bar{v}^2)}{\bar{c}^2}
\end{pmatrix},\\[2mm]
\notag \mathcal{A}_1^{\pm}:=\,&\widebar{S}\widebar{R}^{-1}A_1(\widebar{U}^{\pm})\widebar{R}\\
\label{A1.cal}=\,&\begin{pmatrix}
\pm\widebar{\varGamma}(1-\epsilon^2\bar{v}^2)\bar{v} &1-\epsilon^2\bar{v}^2&1-\epsilon^2\bar{v}^2\\
1+\epsilon^2\bar{v}^2&\pm\dfrac{\widebar{\varGamma}\bar{v}(2-\epsilon^4\bar{c}^2\bar{v}^2)}{\bar{c}^2}&\mp\epsilon^4\widebar{\varGamma}\bar{v}^3\\
1+\epsilon^2\bar{v}^2&\mp\epsilon^4\widebar{\varGamma}\bar{v}^3&\pm\dfrac{\widebar{\varGamma}\bar{v}(2-\epsilon^4\bar{c}^2\bar{v}^2)}{\bar{c}^2}
\end{pmatrix},
\end{align}
and
\begin{align}\label{A2.cal}
\mathcal{A}_2^{\pm}&:=\pm\widebar{S}\widebar{R}^{-1}A_2(\widebar{U}^{\pm})\widebar{R}
=\pm\mathrm{diag}(0,-\frac{2}{\bar{c}},\frac{2}{\bar{c}}).
\end{align}
We notice that \eqref{RE4} is a symmetrizable hyperbolic system
with the Friedrichs symmetrizer $S_2(U)$ defined in \eqref{Sym}.
Consequently, operator $\mathcal{L}$ is symmetrizable hyperbolic with the Friedrichs symmetrizer $S_3$ defined by
\begin{align}
S_3:=\mathrm{diag}\,(\widebar{R}^{\mathsf{T}}S_2(\widebar{U}^+)\widebar{R}\widebar{S}^{-1},\, \widebar{R}^{\mathsf{T}}S_2(\widebar{U}^-)\widebar{R}\widebar{S}^{-1}).
\label{S3}
\end{align}
Regarding the boundary coefficients, $\bar{b}$ is given in \eqref{b.M.bar},
and $\widebar{B}$ is defined by
\setlength{\arraycolsep}{3pt}
\begin{align}  \label{M.ubar}
\widebar{B}:= \begin{pmatrix}
(\widebar{\varGamma}\bar{c}\bar{h} )^{-1} &
-(\widebar{\varGamma}\bar{c}\bar{h} )^{-1} &
-(\widebar{\varGamma}\bar{c}\bar{h} )^{-1} &
(\widebar{\varGamma}\bar{c}\bar{h} )^{-1}\\[0.5mm]
(\widebar{\varGamma}\bar{c}\bar{h} )^{-1}&
-(\widebar{\varGamma}\bar{c}\bar{h})^{-1} & 0 &0\\[0.5mm]
1&1 &-1 &-1
\end{pmatrix}.
\end{align}

For $\gamma\geq 1$, we define
\begin{align*}
\mathcal{L}^{\gamma}:=\mathcal{L}+\gamma \mathcal{A}_0,\quad
\mathcal{B}^{\gamma}(W^{\mathrm{nc}},\psi):=
\bar{b}\begin{pmatrix}
\gamma\psi+\partial_t\psi \\
\partial_1\psi
\end{pmatrix}
+\widebar{B} W^{\mathrm{nc}}.
\end{align*}

It is easily shown that Theorem {\rm\ref{thm.1}} admits the following equivalent proposition.
\begin{proposition}\  \label{pro3.1}
	Assume that the stationary solution \eqref{RVS0} satisfies \eqref{H1} and \eqref{H2}.
	Then, for all $\gamma\geq 1$ and $(W,\psi)\in H^2(\Omega)\times H^2(\mathbb{R}^2)$,
	the following estimate holds{\rm :}
	\begin{align}  \notag
	\gamma\VERT W \VERT^2+\left\|W^{\mathrm{nc}}|_{x_2=0}\right\|^2+\|\psi\|_{1,\gamma}^2
	\lesssim \gamma^{-3}\VERT \mathcal{L}^{\gamma}W\VERT_{1,\gamma}^2
	+\gamma^{-2}\|\mathcal{B}^{\gamma}(W^{\mathrm{nc}}|_{x_2=0},\psi)\|_{1,\gamma}^2.
	\end{align}
\end{proposition}

\subsubsection{Partial homogenization}
In order to prove Proposition \ref{pro3.1}, we show that it suffices
to study the homogeneous case $\mathcal{L}^{\gamma}W\equiv 0$.
Given $(W,\psi)\in H^2(\Omega)\times H^2(\mathbb{R}^2)$, we set
\begin{align*}
f:=\mathcal{L}^{\gamma}W\in H^1(\Omega),\qquad
g:=\mathcal{B}^{\gamma}(W^{\mathrm{nc}}|_{x_2=0},\psi)\in H^1(\mathbb{R}^2),
\end{align*}
and consider the following auxiliary problem:
\begin{align} \label{aux.P}
\left\{ \begin{aligned}
&\mathcal{L}^{\gamma}W_1=f\ &\mathrm{if}\ x_2>0,\\
&B^{\mathrm{aux}}W_1^{\mathrm{nc}}=0\ &\mathrm{if}\ x_2=0,
\end{aligned} \right.
\end{align}
where
\begin{align*}
B^{\mathrm{aux}}:= \begin{pmatrix}
0 &1 &0 &0\\
0 &0 &1 &0
\end{pmatrix}.
\end{align*}
The boundary matrix for problem \eqref{aux.P} ({\it i.e.}\;$-\mathcal{A}_2$)
has two negative eigenvalues and is nonnegative on $\ker B^{\mathrm{aux}}=\{W_3^+=W_2^-=0\}$.
Thus, the boundary conditions in \eqref{aux.P} are maximally dissipative.
From Lax--Phillips \cite{LP60MR0118949}, there exists a unique solution $W_1\in L^2(\mathbb{R}_+;H^1(\mathbb{R}^2))$
to problem \eqref{aux.P} such that the trace of $W_1$ on $\{x_2=0\}$ is in $H^1(\mathbb{R}^2)$, and
\begin{align} \label{aux.e}
\gamma \VERT W_1\VERT^2\lesssim
\gamma^{-1}\VERT f\VERT^2,\qquad
\|{W_1^{\mathrm{nc}}}|_{x_2=0}\|_{1,\gamma}^2
\lesssim \gamma^{-1}\VERT f\VERT_{1,\gamma} ^2.
\end{align}
It is clear that $W_2:=W-W_1$ satisfies
\begin{subequations} \label{P1b}
	\begin{alignat}{3} \label{P1b.1}
	&\mathcal{L}^{\gamma}W_2=0\qquad &\mathrm{if}\ x_2>0,\\
	&\mathcal{B}^{\gamma}(W_2^{\mathrm{nc}},\psi)=\tilde{g}\qquad &\mathrm{if}\ x_2=0,
	\end{alignat}
\end{subequations}
where $\tilde{g}:=g-\widebar{B}W_1^{\mathrm{nc}}$.
By virtue of \eqref{aux.e}, we obtain
\begin{align} \label{P1b.e1}
\|\tilde{g}\|_{1,\gamma}^2
\lesssim \|g\|_{1,\gamma}^2+\gamma^{-1}\VERT f\VERT_{1,\gamma}^2.
\end{align}
Multiplying  \eqref{P1b.1} by the symmetrizer $S_3$ (\emph{cf.}\;\eqref{S3}), then taking the scalar
product of the resulting equations with $W_2$, and employing integration by parts yield
\begin{align} \label{P1b.e2}
\gamma\VERT W_2\VERT^2\lesssim
\|{W_2^{\mathrm{nc}}}|_{x_2=0}\|^2.
\end{align}
The next lemma follows directly from \eqref{aux.e} and \eqref{P1b.e1}--\eqref{P1b.e2}.

\begin{lemma}\  \label{lem.aux}
	If the solution of \eqref{P1b} satisfies the estimate{\rm :}
	\begin{align} \notag
	\|{W_2^{\mathrm{nc}}}|_{x_2=0}\|^2
	+\|\psi\|_{1,\gamma}^2\lesssim \gamma^{-2}\|\tilde{g}\|_{1,\gamma}^2,
	\end{align}
	then Proposition {\rm \ref{pro3.1}} holds.
\end{lemma}

\subsubsection{Eliminating the front}
We perform the Fourier transform of problem \eqref{P1b} in $(t,x_1)$,
with dual variables denoted by $(\delta,\eta)$.
Setting $\tau=\gamma+\mathrm{i}\delta$, we have
\begin{subequations}
	\label{P1c}
	\begin{alignat}{3}
	&\left(\tau\mathcal{A}_0+\mathrm{i}\eta\mathcal{A}_1 \right) \widehat{W}
	+\mathcal{A}_2\frac{\mathrm{d}\widehat{W}}{\mathrm{d}x_2}=0\qquad &\mathrm{if}\ x_2>0,\\ \label{P1c.2}
	&b(\tau,\eta)\widehat{\psi}+\widebar{B}\widehat{W^{\mathrm{nc}}}
	=\widehat{g}\qquad &\mathrm{if}\ x_2=0,
	\end{alignat}
\end{subequations}
where we write $g$ for $\tilde{g}$ and $W$ for $W_2$ for simplicity when no confusion arises.
The coefficient:
\begin{align*}
b(\tau,\eta):=\bar{b}\left(\tau,\, \mathrm{i}\eta\right)^{\mathsf{T}}=\left(2\mathrm{i}\bar{v}\eta,
\, \tau+\mathrm{i}\bar{v}\eta, \, 0\right)^{\mathsf{T}}
\end{align*}
is homogeneous of degree $1$ in $(\tau,\eta)$.
In order to take this homogeneity into account, we define the hemisphere:
\begin{align*}
\Xi_1:=\left\{(\tau,\eta)\in \mathbb{C}\times\mathbb{R}\, :\,
|\tau|^2+\eta^2=1,\RE \tau\geq 0 \right\},
\end{align*}
and the set of ``frequencies'':
\begin{align*}
\Xi:=\left\{(\tau,\eta)\in \mathbb{C}\times\mathbb{R}\, :\,
\RE \tau\geq 0, (\tau,\eta)\ne (0,0) \right\}=(0,\infty)\cdot\Xi_1.
\end{align*}
Notice that symbol $b(\tau,\eta)$ is elliptic, {\it i.e.}
it is always different from zero on $\Xi_1$.

We set $k:=\sqrt{|\tau|^2+\eta^2}$, and define
\begin{align*}
Q(\tau,\eta):=\frac{1}{k}
\begin{pmatrix}
0&0&k\\[0.5mm]
\tau+\mathrm{i}\bar{v}\eta&-2\mathrm{i}\bar{v}\eta&0\\[0.5mm]
-2\mathrm{i}\bar{v}\eta&\bar{\tau}-\mathrm{i}\bar{v}\eta& 0
\end{pmatrix} \qquad\mbox{for $(\tau,\eta)\in \Xi$},
\end{align*}
where $\bar{\tau}$ denotes the complex conjugation of $\tau$,
so that $Q\in\mathcal{C}^\infty(\Xi,GL_3(\mathbb{C}))$ is homogeneous of degree $0$ in $(\tau,\eta)$
and satisfies
\begin{align*}
Q(\tau,\eta)b(\tau,\eta)=(0,0,\theta(\tau,\eta))^{\mathsf{T}}\qquad \mathrm{with}\
\theta(\tau,\eta) = k^{-1} |b(\tau,\eta)|^2.
\end{align*}
Since $\bar{v}\ne 0$, and $\Xi_1$ is compact,
we obtain that $\min_{(\tau,\eta)\in \Xi_1}|\theta(\tau,\eta)| >0$.
Multiplying \eqref{P1c.2} by $Q(\tau,\eta)$ yields
\begin{align} \label{P1c.2a}
\begin{pmatrix}
0 \\
0 \\
\theta(\tau,\eta)
\end{pmatrix} \widehat{\psi} (\delta,\eta)+
\begin{pmatrix}
\beta(\tau,\eta) \\[0.5mm]
\ell(\tau,\eta)
\end{pmatrix} \widehat{W^{\mathrm{nc}}} (\delta,\eta,0) = Q(\tau,\eta)\widehat{g},
\end{align}
where $\beta$ is the $2\times4$ matrix given by the first two rows of $Q(\tau,\eta)\widebar{B}$,
and $\ell$ is the last row of $Q(\tau,\eta)\widebar{B}$.
Both $\beta$ and $\ell$ are $C^{\infty}$ and  homogeneous of degree $0$ on $\Xi$.
In view of \eqref{M.ubar}, symbol $\beta$ satisfies
\begin{align}
\label{beta.C}
\beta(\tau,\eta)=\begin{pmatrix}
1&1&-1&-1\\[0.5mm]
\dfrac{\tau-\mathrm{i}\bar{v}\eta}{\widebar{\varGamma}\bar{c}\bar{h}}
& \dfrac{-\tau+\mathrm{i}\bar{v}\eta}{\widebar{\varGamma}\bar{c}\bar{h}}
& \dfrac{-\tau-\mathrm{i}\bar{v}\eta}{\widebar{\varGamma}\bar{c}\bar{h}}
& \dfrac{\tau+\mathrm{i}\bar{v}\eta}{\widebar{\varGamma}\bar{c}\bar{h}}
\end{pmatrix}\qquad\,
\mathrm{on}\ \Xi_1.
\end{align}
The last component in \eqref{P1c.2a} reads
\begin{align*}
\theta(\tau,\eta)\widehat{\psi}+\ell(\tau,\eta) \widehat{W^{\mathrm{nc}}} (\delta,\eta,0) = Q_3(\tau,\eta)\widehat{g},
\end{align*}
where $Q_3(\tau,\eta)$ is the last row of $Q(\tau,\eta)$.
Hence, it is homogeneous of degree $0$.
Thanks to the homogeneity of $\theta$ and $\ell$, we obtain
\begin{align*}
k^2|\widehat{\psi}|^2\lesssim \big|\widehat{W^{\mathrm{nc}}}|_{x_2=0}\big|^2+|\widehat{g}|^2
\qquad\, \mathrm{in}\ \Xi,
\end{align*}
from which we employ Plancherel's theorem to deduce
\begin{align}\label{front.e}
\|\psi\|_{1,\gamma}^2\lesssim \left\|W^{\mathrm{nc}}\,\!|_{x_2=0}\right\|^2 + \gamma^{-2} \|g\|_{1,\gamma}^2.
\end{align}
After eliminating the front function $\psi$, we have
\begin{subequations}
	\label{P1d}
	\begin{alignat}{3} \label{P1d.1}
	&\left(\tau\mathcal{A}_0+\mathrm{i}\eta\mathcal{A}_1 \right) \widehat{W}+\mathcal{A}_2\frac{\mathrm{d}\widehat{W}}{\mathrm{d}x_2}=0
	\qquad &\mathrm{if}\ x_2>0,\\  \label{P1d.2}
	&\beta(\tau,\eta)\widehat{W^{\mathrm{nc}}}=\widehat{G}\qquad &\mathrm{if}\ x_2=0,
	\end{alignat}
\end{subequations}
where $\widehat{G}$ consists of the first two rows of $Q(\tau,\eta)\widehat{g}$.
From \eqref{A0.cal}--\eqref{A1.cal}, we have
\begin{align}
&\tau \mathcal{A}_0^{\pm}+\mathrm{i}\eta \mathcal{A}_1^{\pm}\notag \\
&=\begin{pmatrix} \label{B.cal}
a_{\pm}\widebar{\varGamma}(1-\epsilon^2\bar{v}^2) &
\mathrm{i}\eta (1-\epsilon^2\bar{v}^2)&
\mathrm{i}\eta (1-\epsilon^2\bar{v}^2)\\[1mm]
\mathrm{i}\eta(1+\epsilon^2\bar{v}^2) \pm 2 \epsilon^2 \bar{v}\tau &
\dfrac{\widebar{\varGamma}(2-\epsilon^4\bar{c}^2\bar{v}^2)}{\bar{c}^2}a_{\pm}&
-\epsilon^4\widebar{\varGamma}\bar{v}^2a_{\pm}\\
\mathrm{i}\eta(1+\epsilon^2\bar{v}^2) \pm 2 \epsilon^2 \bar{v}\tau&
-\epsilon^4\widebar{\varGamma}\bar{v}^2 a_{\pm}
&\dfrac{\widebar{\varGamma}(2-\epsilon^4\bar{c}^2\bar{v}^2)}{\bar{c}^2}a_{\pm}
\end{pmatrix},
\end{align}
where $a_{\pm}:=\tau\pm \mathrm{i}\bar{v}\eta.$
Recalling that formula \eqref{A2.cal} defines the boundary matrix $\mathcal{A}_2^{\pm}$,
we write the first and fourth equations of \eqref{P1d.1} as
\begin{align}\label{algeb}
a_{\pm}\widebar{\varGamma} \widehat{W_1^{\pm}}+
\mathrm{i}\eta  \widehat{W_2^{\pm}}+\mathrm{i}\eta \widehat{W_3^{\pm}}=0.
\end{align}
Then we utilize \eqref{algeb} to express $\widehat{W_1^{\pm}}$ in terms of $\widehat{W_2^{\pm}}$ and $\widehat{W_3^{\pm}}$,
and plug the resulting expressions into the other four equations of \eqref{P1d.1}.
As a consequence, we obtain a system of ordinary differential equations for $\widehat{W^{\mathrm{nc}}}$
in the following form:
\begin{align}
\label{P1e}
\left\{\begin{aligned}
&\frac{\mathrm{d}}{\mathrm{d}x_2}\widehat{W^{\mathrm{nc}}}=
\mathcal{A}(\tau,\eta)\widehat{W^{\mathrm{nc}}}\qquad &\mathrm{if}\ x_2>0,\\
&\beta(\tau,\eta)\widehat{W^{\mathrm{nc}}}=\widehat{G}\qquad &\mathrm{if}\ x_2=0.
\end{aligned}
\right.
\end{align}
Here matrix $\mathcal{A}(\tau,\eta)$ is given by
\begin{align}
\label{A.cal}
\mathcal{A}(\tau,\eta):=
\begin{pmatrix}
\,\mu_{+}\, & \, -m_{+}\, & 0 \, &\, 0 \, \\
\,m_{+}\, & \, -\mu_{+}\, & 0 \, &\, 0 \, \\
\, 0 \, &\, 0 \, & -\mu_{-}\;& \, m_{-}\,  \\
\, 0 \, &\, 0 \, & -m_{-}\, & \, \mu_{-}\,
\end{pmatrix},
\end{align}
where
\begin{align*}
\mu_{\pm}:=\frac{\widebar{\varGamma}a_{\pm}}{\bar{c}}-m_{\pm},\qquad
m_{\pm}:=\frac{\epsilon^2 \bar{c}\bar{v}^2\widebar{\varGamma}^2a_{\pm}^2
	+\mathrm{i}\eta \bar{c}(\mathrm{i}\eta (1+\epsilon^2\bar{v}^2)\pm 2\epsilon^2\bar{v}\tau)}{2\widebar{\varGamma} a_{\pm}}.
\end{align*}
Using the relation, $\widebar{\varGamma}^{-2}=1-\epsilon^2\bar{v}^2$, yields
\begin{align} \label{mu.m}
\mu_{\pm}=\frac{\widebar{\varGamma} (\tau\pm \mathrm{i}\bar{v}\eta)}{\bar{c}}
-\frac{\bar{c}\widebar{\varGamma}(\mathrm{i}\eta\pm \epsilon^2 \bar{v}\tau )^2}{2(\tau\pm \mathrm{i}\bar{v}\eta)},
\qquad
m_{\pm}=\frac{\bar{c}\widebar{\varGamma}(\mathrm{i}\eta\pm \epsilon^2 \bar{v}\tau )^2}{2(\tau\pm \mathrm{i}\bar{v}\eta)}.
\end{align}
The reader may recognize the form of the symbol in \eqref{A.cal} given also by \cite[Page\;957, (4.12)]{CS04MR2095445}.
The poles of symbol $\mathcal{A}(\tau,\eta)$ on $\Xi_1$ are exactly the points: $(\tau,\eta)\in\Xi_1$ with $\tau=\mp \mathrm{i}\bar{v}\eta$,
where the coefficient of $\widehat{W_1^{+}}$ or $\widehat{W_1^{-}}$ in \eqref{algeb} vanishes.

By virtue of \eqref{front.e} and Lemma~\ref{lem.aux}, we infer that, in order to prove Proposition~\ref{pro3.1},
it suffices to study problem \eqref{P1e}.
More precisely, we have the following lemma.

\begin{lemma}\
	\label{lem1a}
	If the solution of \eqref{P1e} satisfies the estimate:
	\begin{align} \label{P1e.e0}
	\left\|W^{\mathrm{nc}}\,\!|_{x_2=0}\right\|^2
	\lesssim \gamma^{-2}\|G\|_{1,\gamma}^2,
	\end{align}
	then Proposition {\rm \ref{pro3.1}} holds.
\end{lemma}

\subsection{Lopatinski\u{\i} Condition}\label{sec.Lop1}
In this subsection, we show that the Kreiss--Lopatinski\u{\i} condition (or briefly the Lopatinski\u{\i} condition)
holds only in the weak form under assumption \eqref{H2} by computing the Lopatinski\u{\i} determinant
associated to problem \eqref{P1e}.

We first calculate the stable subspace of the coefficient matrix $\mathcal{A}(\tau,\eta)$, that is,
the sum of eigenspaces of $\mathcal{A}(\tau,\eta)$ corresponding to the eigenvalues of negative real parts.

\begin{lemma}\  \label{lem.eig1}
	The following properties hold{\rm :}
	
	\begin{list}{}{\setlength{\parsep}{\parskip}
			\setlength{\itemsep}{0.1em}
			\setlength{\labelwidth}{2em}
			\setlength{\labelsep}{0.4em}
			\setlength{\leftmargin}{2.2em}
			\setlength{\topsep}{1mm}
		}
		\item[\emph{(a)}] If $(\tau,\eta)\in\Xi_1$ with $\RE \tau>0$, then
		the eigenvalues of $\mathcal{A}(\tau,\eta)$ are roots $\omega$ of
		\begin{subequations}\label{eig1a}
			\begin{alignat}{3}
			\notag \omega^2=\mu_{+}^2-m_{+}^2
			=\,&\frac{\widebar{\varGamma}^2}{\bar{c}^2}(\tau+\mathrm{i}\eta\bar{v})^2
			-\widebar{\varGamma}^2(\mathrm{i}\eta+\epsilon^2\bar{v}\tau)^2\\
		\label{eig1a.r}	=\,&\widebar{C}_0^2\big(\widebar{C}_1^2(\tau+\mathrm{i}\widebar{C}_2\eta)^2+\eta^2\big),\\
			\notag 	\omega^2=\mu_{-}^2-m_{-}^2
			=\,&\frac{\widebar{\varGamma}^2}{\bar{c}^2}(\tau- \mathrm{i}\eta\bar{v})^2
			-\widebar{\varGamma}^2(\mathrm{i}\eta- \epsilon^2 \bar{v}\tau )^2\\
			=\,&\widebar{C}_0^2\big(\widebar{C}_1^2(\tau-\mathrm{i}\widebar{C}_2\eta)^2+\eta^2\big),
			\label{eig1a.l}
			\end{alignat}
		\end{subequations}
		where $\widebar{C}_j$, $j=0,1,2$, are positive constants defined by
		\begin{align}\label{C.bar}
		\widebar{C}_0:=\frac{\widebar{\varGamma}(1-\epsilon^2\bar{v}^2)}{\sqrt{1-\epsilon^4 \bar{c}^2\bar{v}^2}},\quad\,
		\widebar{C}_1:=\frac{1-\epsilon^4 \bar{c}^2\bar{v}^2}{(1-\epsilon^2\bar{v}^2)\bar{c}},\quad\,
		\widebar{C}_2:=\frac{(1-\epsilon^2\bar{c}^2)\bar{v}}{1-\epsilon^4 \bar{c}^2\bar{v}^2}.
		\end{align}
		Moreover, \eqref{eig1a.r} {(}resp.\;\eqref{eig1a.l}{)} has a unique root $\omega_{+}$ {(}resp.\;$\omega_{-}${)}
		of negative real part. The other root of \eqref{eig1a.r} {(}resp. \eqref{eig1a.l}{)} is $-\omega_{+}$ {(}resp. $-\omega_{-}${)}.
		
		\item[\emph{(b)}] If $(\tau,\eta)\in\Xi_1$ with $\RE \tau>0$,
		then the stable subspace $\mathcal{E}^-(\tau,\eta)$ of $\mathcal{A}(\tau,\eta)$ has dimension two and is spanned by
		\begin{subequations}  \label{E.eig}
			\begin{alignat}{2} \label{E.r}
			&E_+(\tau,\eta):=\left((\tau+\mathrm{i}\bar{v}\eta)m_+,(\tau+\mathrm{i}\bar{v}\eta)(\mu_+-\omega_+),0,0\right)^{\mathsf{T}},\\
			&E_-(\tau,\eta):=\left(0,0,(\tau-\mathrm{i}\bar{v}\eta)(\mu_--\omega_-),(\tau-\mathrm{i}\bar{v}\eta)m_-\right)^{\mathsf{T}}.
			\label{E.l}
			\end{alignat}
		\end{subequations}
		
		\item[\emph{(c)}]
		Both $\omega_{+}$ and $\omega_{-}$ admit a continuous extension to any point $(\tau,\eta)\in\Xi_{1}$ with $\RE\tau=0$.
		In particular, if $\tau=\mathrm{i}\delta\in\mathrm{i}\mathbb{R}$, then
		\begin{align}
	&\notag	\omega_{\pm}(\tau,\eta)\\
				\label{eig1b}&{\small  =\left\{
		\begin{aligned}
		&-\widebar{C}_0\sqrt{\eta^2-\widebar{C}_1^2(\delta\pm\widebar{C}_2\eta)^2} &\textrm{if}\ \eta^2\geq \widebar{C}_1^2\big(\delta\pm\widebar{C}_2\eta\big)^2,\\[0.5mm]
		&-\mathrm{i}\,\mathrm{sgn}(\delta\pm\widebar{C}_2\eta)\widebar{C}_0\sqrt{\widebar{C}_1^2(\delta\pm\widebar{C}_2\eta)^2-\eta^2}&\textrm{if}\ \eta^2< \widebar{C}_1^2\big(\delta\pm\widebar{C}_2\eta\big)^2.
		\end{aligned}
		\right. }
		\end{align}
		
		\item[\emph{(d)}]
		Vectors $E_{\pm}(\tau,\eta)$ do not vanish at any point in $\Xi_1$.
		Both $E_+(\tau,\eta)$ and $E_-(\tau,\eta)$ can be extended continuously to any point $(\tau,\eta)\in\Xi_{1}$ with $\RE\tau=0$.
		These two vectors are linearly independent of the whole hemisphere $\Xi_1$.
		
		\item[\emph{(e)}]
		Matrix $\mathcal{A}(\tau,\eta)$ is diagonalizable as long as eigenvalues $\omega_{\pm}$ do not vanish,
		{\it i.e.}  when
		$\tau\ne \mathrm{i}(\mp \widebar{C}_2\pm \widebar{C}_1^{-1})\eta.$
		Apart from these points, $\mathcal{A}(\tau,\eta)$ has a $C^{\infty}$ basis of eigenvectors.
	\end{list}
\end{lemma}

\begin{proof}\
	The relations in \eqref{eig1a} and assertions (b)--(c) and (e) can be deduced
	from straightforward calculations and the implicit functions theorem.
	
	It follows from \eqref{H1} and \eqref{H1a} that $\widebar{C}_j$ is positive.
	We now show that root $\omega$ of \eqref{eig1a.r} is not purely imaginary when $\RE \tau>0$.
	If this were not true, there would exist $\sigma\in\mathbb{R}$ such that $\mathrm{i}\sigma$ would
	be a root of \eqref{eig1a.r}. Then we would have
	\begin{align*}
	\widebar{C}_0\widebar{C}_1(\tau+ \mathrm{i}\widebar{C}_2\eta)
	=\pm \mathrm{i}\,\sqrt{\sigma^2+\widebar{C}_0^2\eta^2}\in \mathrm{i}\mathbb{R},
	\end{align*}
	which would imply $\RE \tau=0$. This concludes assertion (a).
	
	It remains to prove assertion (d).
	We see from \eqref{mu.m} that, if $\tau+\mathrm{i}\bar{v}\eta=0$,
	then $(\tau+\mathrm{i}\bar{v}\eta)m_+=-\bar{c}\eta^2(1-\epsilon^2\bar{v}^2)/(2\widebar{\varGamma})\ne 0$.
	Hence, when $(\tau+\mathrm{i}\bar{v}\eta)m_+=0$,
	$$
	\tau\ne -\mathrm{i}\bar{v}\eta,\quad
	m_+=0, \quad \mu_+=\widebar{\varGamma}(\tau+\mathrm{i}\bar{v}\eta)/\bar{c}\ne 0,\quad \mu_+^2=\omega_+^2.
	$$
	Using the relations: $\RE\mu_+=\widebar{\varGamma}\RE\tau/\bar{c}\geq 0$ and $\RE \omega_+\leq 0$, we have
	\begin{align*}
	(\tau+\mathrm{i}\bar{v}\eta)(\mu_+-\omega_+)=2
	(\tau+\mathrm{i}\bar{v}\eta)\mu_+\ne 0.
	\end{align*}
	Therefore, $E_+(\tau,\eta) $ defined by \eqref{E.r} does not vanish.
	We can also show in a similar way that $E_-(\tau,\eta)$ does not vanish.
	Assertion (d) then follows.
 \qed\end{proof}

As in Majda--Osher \cite{MO75MR0410107}, we define the Lopatinski\u{\i} determinant associated
with problem \eqref{P1e} by
\begin{align}
\label{Lop1.def}
\Delta(\tau,\eta):=\det\left[\beta(\tau,\eta)\left(E_+(\tau,\eta)\ E_-(\tau,\eta)\right)\right],
\end{align}
where $\beta$ and $E_{\pm}$ are given in \eqref{beta.C} and \eqref{E.eig}, respectively.
We say that the Lopatinski\u{\i} condition holds if $\Delta(\tau,\eta)\ne 0$ for all $(\tau,\eta)\in\Xi$ with $\RE \tau>0$.
Furthermore, if $\Delta(\tau,\eta)\ne 0$ for all $(\tau,\eta)\in\Xi$, we say that the \emph{uniform} Lopatinski\u{\i} condition holds.
To deduce the energy estimate, we need to study the zeros of $\Delta(\tau,\eta)$.
For this, we have the following lemma.

\begin{lemma}\ \label{lem.Lop}
	Assume that \eqref{H1} and \eqref{H2} hold. Then, for any $(\tau,\eta)\in\Xi_1$,
	\begin{align} \label{root}
	\Delta(\tau,\eta)=0\;\;\; \textrm{if and only if}\;\;\; \tau\in\{0,\pm\mathrm{i}z_1\eta\},
	\end{align}
	where $z_1$ is some positive constant satisfying
	\begin{align} \label{root.e}
	0<z_1<\widebar{C}_2-\widebar{C}_1^{-1}<\widebar{C}_2<\bar{v}<\widebar{C}_2+\widebar{C}_1^{-1}.
	\end{align}
	Moreover, each of these roots is simple in the sense that, if $q\in\{0,-z_1,z_1\}$,
	then there exists a neighborhood $\mathscr{V}$ of $(\mathrm{i}q\eta,\eta)$ in $\Xi_1$
	and a $C^{\infty}$--function $h_q$ defined on $\mathscr{V}$ such that
	\begin{align}
	\label{factor}
	\Delta(\tau,\eta)=(\tau-\mathrm{i}q\eta)h_q(\tau,\eta),
	\quad\, h_q(\tau,\eta)\ne 0\qquad \textrm{for all }(\tau,\eta)\in\mathscr{V}.
	\end{align}
\end{lemma}

\begin{proof}\  We divide the proof into seven steps.
	
	\smallskip
	\noindent
	{\bf 1.}\,\, According to \eqref{beta.C} and \eqref{E.eig}, we have
	\begin{align}\notag
	&\beta(\tau,\eta)\left(E_+(\tau,\eta)\ E_-(\tau,\eta)\right)\\
	&={\small
	\begin{pmatrix}
	(\tau+\mathrm{i}\bar{v}\eta)(m_++\mu_+-\omega_+)&
	-(\tau-\mathrm{i}\bar{v}\eta)(m_-+\mu_--\omega_-)\\[0.5mm]
	\dfrac{\tau-\mathrm{i}\bar{v}\eta}{\widebar{\varGamma}\bar{c}\bar{h}}(\tau+\mathrm{i}\bar{v}\eta)(m_+-\mu_++\omega_+) &
	\dfrac{\tau+\mathrm{i}\bar{v}\eta}{\widebar{\varGamma}\bar{c}\bar{h}}(\tau-\mathrm{i}\bar{v}\eta)(m_--\mu_-+\omega_-)
	\end{pmatrix}. } \label{Lop1a}
	\end{align}
	By using \eqref{mu.m} and Lemma \ref{lem.eig1}\,(a), we have
	\begin{align}\label{LD.p0}
	m_{\pm}+\mu_{\pm}=\widebar{\varGamma}\frac{\tau\pm\mathrm{i}\bar{v}\eta}{\bar{c}},\quad\,
	m_{\pm}-\mu_{\pm}=-\frac{\bar{c}\omega_{\pm}^2}{\widebar{\varGamma}(\tau\pm\mathrm{i}\bar{v}\eta)}.
	\end{align}
	It then follows that
	\begin{align}
	\notag \Delta(\tau,\eta)=
	\,&\frac{1}{\bar{c}^2\bar{h}}\Big\{\tau+\mathrm{i}\bar{v}\eta-\frac{\bar{c}\omega_+}{\widebar{\varGamma}}\Big\}
	\Big\{\tau-\mathrm{i}\bar{v}\eta-\frac{\bar{c}\omega_-}{\widebar{\varGamma}}\Big\}\\
	&\times\left\{\omega_-(\tau+\mathrm{i}\bar{v}\eta)^2+\omega_+(\tau-\mathrm{i}\bar{v}\eta)^2\right\}.
	\label{Lop1b}
	\end{align}
	We will check the zeros of each factors in this expression separately.
	
	\vspace*{2mm}
	\noindent\emph{\bf 2.}\,\,  We show in this step that both
	$\widebar{\varGamma}(\tau+\mathrm{i}\bar{v}\eta)-\bar{c}\omega_+$
	and $\widebar{\varGamma}(\tau-\mathrm{i}\bar{v}\eta)-\bar{c}\omega_-$
	do not vanish at any point $(\tau,\eta)\in\Xi_1$.
	By contradiction, we assume without loss of generality that there exists a point $(\tau_0,\eta_0)\in\Xi_1$ such that
	\begin{align}
	\widebar{\varGamma}(\tau_0+\mathrm{i}\bar{v}\eta_0)=\bar{c}\omega_+(\tau_0,\eta_0).
	\label{LD.p1}
	\end{align}
	From \eqref{eig1a.r}, we have
	\begin{align*}
	\widebar{\varGamma}^2(\tau_0+\mathrm{i}\bar{v}\eta_0)^2-\bar{c}^2\omega_+(\tau_0,\eta_0)^2
	=\widebar{\varGamma}^2\bar{c}^2(\mathrm{i}\eta_0+\epsilon^2\bar{v}\tau_0)^2=0,
	\end{align*}
	which implies
	$\tau_0=\mathrm{i}\delta_0\in \mathrm{i}\mathbb{R}$ with $\eta_0=-\epsilon^{2}\bar{v}\delta_0$.
	Since $(\tau_0,\eta_0)\in\Xi_1$, we see that both $\eta_0$ and $\delta_0$ are nonzero real numbers.
	If $\eta_0^2\geq \widebar{C}_1^2(\delta_0+\widebar{C}_2\eta_0)^2$,
	then $\omega_+(\tau_0,\eta_0)\in\mathbb{R}$ due to \eqref{eig1b}.
	Then $\bar{c}\;\omega_+(\tau_0,\eta_0)\ne \widebar{\varGamma}(\tau_0+\mathrm{i}\bar{v}\eta_0)$, since
	$$
	\widebar{\varGamma}(\tau_0+\mathrm{i}\bar{v}\eta_0)=
	\mathrm{i} \widebar{\varGamma}(1-\epsilon^2\bar{v}^2)\delta_0\in \mathrm{i}\mathbb{R}\setminus\{0\}.
	$$
	According to \eqref{LD.p1},
	$\eta_0^2<\widebar{C}_1^2(\delta_0+\widebar{C}_2\eta_0)^2$
	so that
	$\widebar{C}_1^2(1-\epsilon^2\widebar{C}_2\bar{v})^2>\epsilon^4\bar{v}^2.$
	It then follows from \eqref{eig1b} that
	\begin{align*}
	\omega_+(\tau_0,\eta_0)
	&=-\mathrm{i}\,\mathrm{sgn}(\delta_0)\,\mathrm{sgn}(1-\epsilon^2\widebar{C}_2\bar{v})
	\widebar{C}_0\sqrt{\big(\widebar{C}_1^2(1-\epsilon^2\widebar{C}_2\bar{v})^2-\epsilon^4\bar{v}^2\big)\delta_0^2}\\
	&=-\mathrm{i}\,\delta_0\widebar{C}_0\sqrt{\widebar{C}_1^2
		(1-\epsilon^2\widebar{C}_2\bar{v})^2-\epsilon^4\bar{v}^2},
	\end{align*}
	where we have used that
	$1-\epsilon^2\widebar{C}_2\bar{v}=
	(1-\epsilon^2\bar{v}^2)/(1-\epsilon^4\bar{c}^2\bar{v}^2)>0$
	from \eqref{H1a}. Consequently, we have
	\begin{align*}
	&\widebar{\varGamma}(\tau_0+\mathrm{i}\bar{v}\eta_0)-\bar{c}\omega_+(\tau_0,\eta_0)\\
	&\quad =\mathrm{i}\delta_0
	\left\{\widebar{\varGamma}(1-\epsilon^2\bar{v}^2)+\bar{c}\,
	\widebar{C}_0\sqrt{\widebar{C}_1^2(1-\epsilon^2\widebar{C}_2\bar{v})^2-\epsilon^4\bar{v}^2}\right\}\ne 0.
	\end{align*}
	This contradicts \eqref{LD.p1}.
	
	\vspace*{2mm}
	\noindent{\bf 3.}\,\, From the above analysis,
	we know that  $\Delta(\tau,\eta)=0$ if and only if
	factor $\omega_-(\tau+\mathrm{i}\bar{v}\eta)^2+\omega_+(\tau-\mathrm{i}\bar{v}\eta)^2$ vanishes.
	We first prove that this factor does not vanish when $\eta=0$.
	
	If $\eta=0$, then we see from \eqref{eig1a} that
	$\omega_{\pm}^2=\bar{c}^{-2}\widebar{\varGamma}^2\tau^2(1-\epsilon^4\bar{c}^2\bar{v}^2)$.
	Using \eqref{H1a} and noting $\RE\tau\geq 0$, we find that
	$\omega_{\pm}=-\bar{c}^{-1}\widebar{\varGamma}\tau \sqrt{1-\epsilon^4\bar{c}^2\bar{v}^2}$,
	which yields
	\begin{align*}
	\omega_-(\tau+\mathrm{i}\bar{v}\eta)^2+\omega_+(\tau-\mathrm{i}\bar{v}\eta)^2=
	-2\bar{c}^{-1}\widebar{\varGamma}\tau^3(1-\epsilon^4\bar{c}^2\bar{v}^2)^{1/2}\ne 0.
	\end{align*}
	
	We thus assume that $\eta\ne 0$. Introducing  $z:={\tau}/{(\mathrm{i}\eta)}$,
	we find from \eqref{eig1a} that
	\begin{align}
	&\frac{\bar{c}^2\omega_-^2(\tau+\mathrm{i}\bar{v}\eta)^4}{\widebar{\varGamma}^2(\mathrm{i}\eta)^6}
	=(z+\bar{v})^4\left\{(z-\bar{v})^2-\bar{c}^2(1-\epsilon^2\bar{v}z)^2\right\}=:P_1(z), \label{P1.def}\\
	&\frac{\bar{c}^2\omega_+^2(\tau-\mathrm{i}\bar{v}\eta)^4}{\widebar{\varGamma}^2(\mathrm{i}\eta)^6}
	=(z-\bar{v})^4\left\{(z+\bar{v})^2-\bar{c}^2(1+\epsilon^2\bar{v}z)^2\right\}=:P_2(z). \label{P2.def}
	\end{align}
	Define
	\begin{align} \label{P.def}
	P(z):=P_1(z)-P_2(z).
	\end{align}
	Then $\Delta(\tau,\eta)=0$ holds only if $\omega_-^2(\tau+\mathrm{i}\bar{v}\eta)^4=\omega_+^2(\tau-\mathrm{i}\bar{v}\eta)^4$,
	which is equivalent to $P(z)=0$.
	A straightforward calculation yields
	\begin{align*}
	P(z)=-4z\bar{v}P_0(z),\qquad
	P_0(z):=E_1z^4+E_2z^2+E_3,
	\end{align*}
	where $E_1=2\epsilon^4\bar{c}^2\bar{v}^2-\epsilon^2\bar{c}^2-1$,
	$E_2=2\epsilon^4\bar{c}^2\bar{v}^4-6\epsilon^2\bar{c}^2\bar{v}^2+2\bar{v}^2+2\bar{c}^2,$ and
	\begin{gather} \label{E3}
	E_3=2\bar{c}^2\bar{v}^2-\epsilon^2\bar{c}^2\bar{v}^4-\bar{v}^4.
	\end{gather}
	It is trivial that $z=0$ is one zero of $P(z)$.
	Function $P_0(z)$ is a polynomial one of $z^2$ with the following zeros:
	\begin{align} \label{zeros1}
	-\frac{E_2\pm \sqrt{E_2^2-4E_1E_3}}{2E_1}.
	\end{align}
	By virtue of \eqref{H1} and \eqref{H1a}, we have
	\begin{align} \label{E12}
	\left\{\begin{aligned}
	&E_1=-(1-\epsilon^4\bar{c}^2\bar{v}^2)-\epsilon^2\bar{c}^2(1-\epsilon^2\bar{v}^2)<0,\\
	&E_2=2\bar{v}^2(1-\epsilon^2\bar{c}^2)+2\bar{c}^2(1-\epsilon^2\bar{v}^2)^2>0,\\
	&E_2^2-4E_1E_3\\
	&\quad =4\bar{c}^2\left( \epsilon \bar{v}-1\right)^2\left( \epsilon \bar{v}+1\right)^2
	\left(\epsilon^4\bar{c}^4\bar{v}^4-2\epsilon^2\bar{c}^2\bar{v}^2+4\bar{v}^2+\bar{c}^2\right)>0,
	\end{aligned}\right.
	\end{align}
	which yields that the zeros in \eqref{zeros1} are real and distinct.
	If \eqref{H2} holds, then $E_3<0$, which immediately implies that the zeros in \eqref{zeros1}
	are also positive.
	Let us denote these zeros by $z_1^2$ and $z_2^2$ with $0<z_1<z_2$ so that
	\begin{align}\label{zeros}
	z_1^2=\frac{E_2-\sqrt{E_2^2-4E_1E_3}}{-2E_1},\qquad\,
	z_2^2=\frac{E_2+\sqrt{E_2^2-4E_1E_3}}{-2E_1}.
	\end{align}
	Consequently, the Lopatinski\u{\i} determinant vanishes only if $z\in\{0,\pm z_1,\pm z_2\}$.
	
	\vspace*{2mm}
	\noindent{\bf 4.}\,\, In this step, we show that the Lopatinski\u{\i} determinant
	vanishes when $z=0$, {\it i.e.}  when $\tau=0$.
	We note  that $\bar{c}<\bar{v}$  by combining \eqref{H2} and \eqref{H1a}. Then
	\begin{align} \label{LD.p2}
	\widebar{C}_2-\widebar{C}_1^{-1}
	=\frac{(\bar{v}-\bar{c})(1+\epsilon^2\bar{c}\bar{v})}{1-\epsilon^4\bar{c}^2\bar{v}^2}>0.
	\end{align}
	It then follows directly from \eqref{eig1b} that
	\begin{align*}
	\omega_{\pm}(0,\eta)=
	-\mathrm{i}\,\mathrm{sgn}(\pm\widebar{C}_2\eta)\widebar{C}_0\sqrt{\widebar{C}_1^2\widebar{C}_2^2\eta^2-\eta^2}
	=\mp \mathrm{i}\,\eta \widebar{C}_0\sqrt{\widebar{C}_1^2\widebar{C}_2^2-1}.
	\end{align*}
	Then we infer
	\begin{align*}
	\left.\left\{\omega_-(\tau+\mathrm{i}\bar{v}\eta)^2+\omega_+(\tau-\mathrm{i}\bar{v}\eta)^2\right\}\right|_{\tau=0}
	=-\bar{v}^2\eta^2\left(\omega_{+}(0,\eta)+\omega_{-}(0,\eta)\right)=0,
	\end{align*}
	and hence $\Delta(0,\eta)=0$.
	
	\vspace*{2mm}
	\noindent{\bf 5.}\,\, We prove that $\omega_-(\tau+\mathrm{i}\bar{v}\eta)^2+\omega_+(\tau-\mathrm{i}\bar{v}\eta)^2\ne 0$
	when $z=\pm z_2$, {\it i.e.}   when $\tau=\pm\mathrm{i}z_2\eta$.
	To this end, we need to show that
	\begin{align} \label{LD.p3}
	z_2+\widebar{C}_2>z_2-\widebar{C}_2>\widebar{C}_1^{-1}.
	\end{align}
	The first inequality is trivial, so it suffices to prove the second one.
	From \eqref{C.bar} and \eqref{zeros}, we have
	\begin{align*}
	&z_2^2-\big(\widebar{C}_1^{-1}+\widebar{C}_2\big)^2
	=-\frac{E_2+\sqrt{E_2^2-4E_1E_3}}{2E_1}
	-\frac{\left((1-\epsilon^2\bar{v}^2)\bar{c}+(1-\epsilon^2\bar{c}^2)\bar{v}\right)^2}{(1-\epsilon^4 \bar{c}^2\bar{v}^2)^2}
	\\&\ =-\frac{(1-\epsilon^4 \bar{c}^2\bar{v}^2)^2\big(\sqrt{E_2^2-4E_1E_3}+ E_2\big)
		+2E_1\left((1-\epsilon^2\bar{v}^2)\bar{c}+(1-\epsilon^2\bar{c}^2)\bar{v}\right)^2}{2E_1(1-\epsilon^4 \bar{c}^2\bar{v}^2)^2}
	\\&\ =-\frac{2c(1-\epsilon^2 \bar{v}^2)(1-\epsilon^2 \bar{c}\bar{v})^2(R_2-L_2)}{2E_1(1-\epsilon^4 \bar{c}^2\bar{v}^2)^2},
	\end{align*}
	where $R_2:=(1+\epsilon^2\bar{c}\bar{v})^2\sqrt{\epsilon^4\bar{c}^2\bar{v}^4-2\epsilon^2\bar{c}^2\bar{v}^2+4\bar{v}^2+\bar{c}^2}$
	and
	\begin{align*}
	L_2:=\epsilon^6\bar{c}^3\bar{v}^4+2\epsilon^4\bar{c}^2\bar{v}^3-2\epsilon^4\bar{c}^3\bar{v}^2+4\epsilon^2\bar{c}\bar{v}^2+2\bar{v}+\epsilon^2\bar{c}^3.
	\end{align*}
	Then
	we obtain from \eqref{H1a} that
	\begin{align*}
	R_2^2-L_2^2=\bar{c}^2(\epsilon\bar{c}-1)(\epsilon\bar{c}+1)(\epsilon\bar{v}-1)^2
	(\epsilon\bar{v}+1)^2(2\epsilon^4\bar{c}^2\bar{v}^2-\epsilon^2\bar{c}^2-1)>0,
	\end{align*}
	which, combined with \eqref{H1a} and \eqref{E12}, implies that $z_2^2>\big(\widebar{C}_1^{-1}+\widebar{C}_2\big)^2$.
	Then \eqref{LD.p3} follows.
	
	In view of \eqref{LD.p3}, we see from \eqref{eig1b} that, for $\tau=\mathrm{i}z_2\eta$,
	\begin{align*}
	\omega_{\pm}(\tau,\eta)
	&=-\mathrm{i}\,\mathrm{sgn}(z_2\eta\pm \widebar{C}_2\eta) \widebar{C}_0\sqrt{\widebar{C}_1^2(z_2\pm \widebar{C}_2)^2\eta^2-\eta^2}
	\\&=-\mathrm{i}\,\eta\, \widebar{C}_0\sqrt{\widebar{C}_1^2(z_2\pm \widebar{C}_2)^2-1}.
	\end{align*}
	Therefore, we obtain that
	\begin{align*}
	&\omega_-(\tau+\mathrm{i}\bar{v}\eta)^2+\omega_+(\tau-\mathrm{i}\bar{v}\eta)^2
	=-\eta^2\left(\omega_+(z_2-\bar{v})^2+\omega_-(z_2+\bar{v})^2\right)\\
	& =\mathrm{i}\,\eta^3\Big\{
	\widebar{C}_0\sqrt{\widebar{C}_1^2(z_2+ \widebar{C}_2)^2-1}\;(z_2-\bar{v})^2
	+\widebar{C}_0\sqrt{\widebar{C}_1^2(z_2- \widebar{C}_2)^2-1}\;(z_2+\bar{v})^2\Big\},
	\end{align*}
	which is away from zero.
	Applying a similar argument and using \eqref{LD.p3} imply that the Lopatinski\u{\i} determinant $\Delta$
	does not vanish either for the case: $z=-z_2$.
	
	\vspace*{2mm}
	\noindent{\bf 6.}\,\,  Let us now show
	that $\omega_-(\tau+\mathrm{i}\bar{v}\eta)^2+\omega_+(\tau-\mathrm{i}\bar{v}\eta)^2=0$ if $z=\pm z_1$,
	{\it i.e.}  if $\tau=\pm \mathrm{i}z_1\eta$.
	For this purpose, we first prove
	\begin{align} \label{LD.p4}
	z_1+\widebar{C}_2>\widebar{C}_1^{-1},\qquad\,\,
	z_1-\widebar{C}_2<-\widebar{C}_1^{-1}.
	\end{align}
	The first inequality in \eqref{LD.p4} follows from \eqref{LD.p2}.
	For the second in \eqref{LD.p4}, it suffices to derive that $z_1^2<(\widebar{C}_2-\widebar{C}_1^{-1})^2$.
	From \eqref{C.bar} and \eqref{zeros}, we have
	\begin{align*}
	&\big(\widebar{C}_2-\widebar{C}_1^{-1}\big)^2-z_1^2
	=\frac{\left((1-\epsilon^2\bar{c}^2)\bar{v}-(1-\epsilon^2\bar{v}^2)\bar{c}\right)^2}{(1-\epsilon^4 \bar{c}^2\bar{v}^2)^2}
	+\frac{E_2-\sqrt{E_2^2-4E_1E_3}}{2E_1}\\
	&=\frac{(1-\epsilon^4 \bar{c}^2\bar{v}^2)^2\big(E_2-\sqrt{E_2^2-4E_1E_3}\big)
		+2E_1\left((1-\epsilon^2\bar{c}^2)\bar{v}-(1-\epsilon^2\bar{v}^2)\bar{c}\right)^2}{2E_1(1-\epsilon^4 \bar{c}^2\bar{v}^2)^2}\\
	&=\frac{2c(1-\epsilon^2 \bar{v}^2)(1+\epsilon^2 \bar{c}\bar{v})^2(R_4+L_4)}{-2E_1(1-\epsilon^4 \bar{c}^2\bar{v}^2)^2},
	\end{align*}
	where $R_4:=(1-\epsilon^2\bar{c}\bar{v})^2\sqrt{\epsilon^4\bar{c}^2\bar{v}^4-2\epsilon^2\bar{c}^2\bar{v}^2+4\bar{v}^2+\bar{c}^2}$ and
	\begin{align*}
	L_4:=\epsilon^6\bar{c}^3\bar{v}^4-2\epsilon^4\bar{c}^2\bar{v}^3-2\epsilon^4\bar{c}^3\bar{v}^2
	+4\epsilon^2\bar{c}\bar{v}^2-2\bar{v}+\epsilon^2\bar{c}^3.
	\end{align*}
	We compute that $R_4^2-L_4^2=R_2^2-L_2^2>0$.
	Hence, we deduce the second inequality in \eqref{LD.p4}.
	
	By virtue of \eqref{eig1b} and \eqref{LD.p4}, we derive that, for $\tau=\mathrm{i}z_1\eta$,
	\begin{align*}
	\omega_{+}(\tau,\eta)
	=-\mathrm{i}\,\eta\, \widebar{C}_0\sqrt{\widebar{C}_1^2(z_1+ \widebar{C}_2)^2-1},\ \
	\omega_{-}(\tau,\eta)
	=\mathrm{i}\,\eta\, \widebar{C}_0\sqrt{\widebar{C}_1^2(z_1- \widebar{C}_2)^2-1}.
	\end{align*}
	Since $z=z_1$ solves $P(z)=0$, if $\tau=\mathrm{i}z_1\eta$,
	it follows from the definition of $P(z)$ that
	$\omega_-^2(\tau+\mathrm{i}\bar{v}\eta)^4=\omega_+^2(\tau-\mathrm{i}\bar{v}\eta)^4$.
	Hence, the Lopatinski\u{\i} determinant vanishes for $z=z_1$ ({\it i.e.}  $\tau=\mathrm{i}z_1\eta$).
	The same argument can be applied to show that the Lopatinski\u{\i} determinant $\Delta(\tau,\eta)$
	also vanishes for $z=-z_1$, {\it i.e.}  for $\tau=-\mathrm{i}z_1\eta$.
	
	\vspace*{2mm}
	\noindent{\bf 7.}\,\, We obtain from \eqref{H1a} by a direct computation that
	$
	\widebar{C}_2<\bar{v}<\widebar{C}_2+\widebar{C}_1^{-1},
	$
	which, combined with \eqref{LD.p4}, yields \eqref{root.e}.
	
	It remains to show that the roots of the Lopatinski\u{\i} determinant are simple.
	By introducing $\Omega_{\pm}:={\omega_{\pm}}/{(\mathrm{i}\eta)}$,
	we find that, for $\eta\ne0$,
	\begin{align*}
	\frac{\omega_-(\tau+\mathrm{i}\bar{v}\eta)^2}{(\mathrm{i}\eta)^3}=\Omega_-(z+\bar{v})^2=:Q_1(z),\
	\frac{\omega_+(\tau-\mathrm{i}\bar{v}\eta)^2}{(\mathrm{i}\eta)^3}=\Omega_+(z-\bar{v})^2=:Q_2(z).
	\end{align*}
	It follows from \eqref{root.e} and Lemma \ref{lem.eig1}
	that $\omega_{\pm}(\tau,\eta)\ne 0$ and $\eta\ne 0$  when $(\tau,\eta)$ are
	near any root of the Lopatinski\u{\i} determinant.
	Hence, $\Omega_{\pm}$ are analytic functions of $z$ only and satisfy
	\begin{align*}
	\Omega_{\pm}^2=\bar{c}^{-2}\widebar{\varGamma}^2\left((z\pm\bar{v})^2-\bar{c}^2(1\pm\epsilon^2\bar{v}z)^2\right).
	\end{align*}
	Since $\Delta(\tau,\eta)=0$ if and only if $\omega_-(\tau+\mathrm{i}\bar{v}\eta)^2+\omega_+(\tau-\mathrm{i}\bar{v}\eta)^2=0$,
	it suffices to prove
	\begin{align*}
	\left.\frac{\mathrm{d}\left(Q_1+Q_2\right)}{\mathrm{d}z}\right|_{z=q}\ne 0\;\qquad\textrm{for all }q\in\{0,-z_1,z_1\}.
	\end{align*}
	Using \eqref{P1.def}--\eqref{P.def} and the fact that $Q_1(q)=-Q_2(q)\ne 0$ for $q\in\{0,-z_1,z_1\}$,
	we derive that, for $q\in\{0,-z_1,z_1\}$,
	\begin{align*}
	\left.\frac{\mathrm{d}\left(Q_1+Q_2\right)}{\mathrm{d}z}\right|_{z=q}
	&=\frac{1}{2Q_1(q)}\left.\frac{\mathrm{d}\left(Q_1^2-Q_2^2\right)}{\mathrm{d}z}\right|_{z=q}
	=\frac{\widebar{\varGamma}^2}{2\bar{c}^2Q_1(q)}\left.\frac{\mathrm{d}P}{\mathrm{d}z}\right|_{z=q} \\
	&=\frac{\widebar{\varGamma}^2}{\bar{c}^2Q_1(q)}\left\{-2\bar{v}P_0(q)-4\bar{v}q^2(2E_1q^2+E_2)\right\}\ne 0.
	\end{align*}
	Using the factorization property of holomorphic functions, we obtain
	\[
	Q_1(z)+Q_2(z)=(z-q)H_q(z)\qquad\,\textrm{for all }q\in\{0,-z_1,z_1\},
	\]
	where $H_q$ is holomorphic near $q$ and $H_q(q)\ne 0$.
	This yields that the Lopatinski\u{\i} determinant $\Delta$ has a factorization:
	\begin{align*}
	\Delta(\tau,\eta)=(\tau-\mathrm{i}q\eta)h_q(\tau,\eta)\qquad\,\textrm{for all }q\in\{0,-z_1,z_1\},
	\end{align*}
	where $h_q(\tau,\eta)$ is $C^{\infty}$ and does not vanish near $(\mathrm{i}q\eta,\eta)\in\Xi_1$.
	The proof is completed.
\qed\end{proof}

\begin{remark}\ \label{rem3.2}
	If $\textsf{M}=\textsf{M}_{\rm c}$,
	then both $E_3$ and $z_1$ defined by \eqref{E3} and \eqref{zeros} vanish.
	Employing a similar argument, we can show that the Lopatinski\u{\i} determinant $\Delta(\tau,\eta)$
	has only one triple root  $\tau=0$.
	On the other hand, if $\textsf{M}<\textsf{M}_{\rm c}$,
	then $E_3>0$.  In the latter case, the Lopatinski\u{\i} determinant $\Delta(\tau,\eta)$ vanishes
	if and only if $\tau/(\mathrm{i}\eta)\in\{0,\pm z_1\}$ with nonreal number $z_1$ given by \eqref{zeros}.
	Therefore, the relativistic vortex sheet \eqref{RVS0} is violently unstable,
	which means that the Lopatinski\u{\i} condition does not hold.
\end{remark}

\subsection{Proof of Theorem {\rm\ref{thm.1}}}\label{sec3.proof}

The following lemma relies heavily on the fact that each root of the Lopatinski\u{\i} determinant
is simple (see Lemma \ref{lem.Lop}).

\begin{lemma}\ \label{lem1.E1}
	For every point $(\tau_0,\eta_0)\in\Xi_1$, there exists a neighborhood $\mathscr{V}$ of $(\tau_0,\eta_0)$
	in $\Xi_1$ and a positive constant $c$ depending on $(\tau_0,\eta_0)$ such that
	\begin{align} \label{E1.e0}
	\left|\beta(\tau,\eta)(E_+(\tau,\eta) \ E_-(\tau,\eta))Z\right|\geq c \gamma |Z|
	\qquad\textrm{for all }(\tau,\eta)\in\mathscr{V},\,Z\in\mathbb{C}^2.
	\end{align}
\end{lemma}

\begin{proof}\
	The proof is divided into two steps.
	
	\vspace*{2mm}
	\noindent
	{\bf 1.}\,\, Let $(\tau_0,\eta_0)\in\Xi_1$ with $\Delta(\tau_0,\eta_0)\ne 0$.
	Since the Lopatinski\u{\i} determinant $\Delta(\tau,\eta)$ is continuous in $(\tau,\eta)$,
	then there exists a neighborhood $\mathscr{V}$ of $(\tau_0,\eta_0)$ in $\Xi_1$
	such that $\Delta(\tau,\eta)\ne 0$ for all $(\tau,\eta)\in\mathscr{V}$.
	It follows from definition \eqref{Lop1.def} that
	$\beta(\tau,\eta)(E_+ \ E_-)$ is invertible in $\mathscr{V}$.
	We combine this with the fact that $\gamma\leq 1$ to obtain \eqref{E1.e0}.
	
	\smallskip
	\noindent{\bf 2.}\,\, Let $(\tau_0,\eta_0)\in\Xi_1$ such that $\Delta(\tau_0,\eta_0)=0$.
	We see from Lemma \ref{lem.Lop}  that $\tau_0=\mathrm{i}q\eta_0$ for some $q\in\{0,-z_1,z_1\}$.
	Let us write \eqref{Lop1a} as
	\begin{align*}
	\beta (E_+\ E_-)=\begin{pmatrix}
	\zeta_1 &\zeta_2\\
	\zeta_3 &\zeta_4
	\end{pmatrix},
	\end{align*}
	where the upper left corner $\zeta_1$ is given by
	\[
	\zeta_1=(\tau-\mathrm{i}\bar{v}\eta)(m_++\mu_+-\omega_+)
	=\frac{\tau+\mathrm{i}\bar{v}\eta}{\bar{c}}\big(\widebar{\varGamma}(\tau+\mathrm{i}\bar{v}\eta)
	-\bar{c}\omega_+\big).
	\]
	From \eqref{root.e} and the proof of Lemma \ref{lem.Lop} (especially, Step\;2),
	we know that $\tau\ne-\mathrm{i}\bar{v}\eta$ and
	$\widebar{\varGamma}(\tau+\mathrm{i}\bar{v}\eta)\ne \bar{c}\;\omega_+$
	when $(\tau,\eta)$ is close to $(\tau_0,\eta_0)$.
	Hence, there exists a neighborhood $\mathscr{V}$ of $(\tau_0,\eta_0)$ in $\Xi_1$
	such that $\zeta_1(\tau,\eta)\ne 0$ for all $(\tau,\eta)\in\mathscr{V}$.
	Using the identity ({\it cf.}\;\cite[Page\;439]{C04MR2069632}):
	\begin{align} \label{Coulombel}
	\begin{pmatrix}
	1/\zeta_1 & 0\\[1.5mm]
	-\zeta_3/(\zeta_1 \zeta_5) &1/\zeta_5
	\end{pmatrix}
	\beta (E_+\ E_-)
	\begin{pmatrix}
	1 & -\zeta_2\\[1.5mm]
	0  &\zeta_1
	\end{pmatrix}
	=\begin{pmatrix}
	1& 0\\[1.5mm]
	0 &(\zeta_1\zeta_4-\zeta_2\zeta_3)/\zeta_5
	\end{pmatrix}
	\end{align}
	with $\zeta_5=1$, and noting
	$\Delta(\tau,\eta)=\det\,[\beta\, (E_+\ E_-)] =\zeta_1\zeta_4-\zeta_2\zeta_3$,
	we have
	\begin{align}\label{E1.p1}
	\left|\beta(\tau,\eta) (E_+(\tau,\eta) \ E_-(\tau,\eta) )Z\right|
	\geq c \min(1,|\Delta(\tau,\eta)|)|Z|
	\end{align}
	for all 	$(\tau,\eta)\in\mathscr{V},\,Z\in\mathbb{C}^2.$
	It thus remains to show that $|\Delta(\tau,\eta)|\geq c\gamma$
	for all $(\tau,\eta)\in\mathscr{V}$.
	Employ Lemma \ref{lem.Lop} and shrink $\mathscr{V}$ if necessary to find that
	factorization \eqref{factor} holds.
	Thus, we have
	\begin{align} \label{E1.p2}
	\partial_{\gamma}\Delta(\tau,\eta)=h_q(\tau,\eta)+(\tau-\mathrm{i}q\eta)\partial_\gamma h_q(\tau,\eta)
	\qquad\textrm{for all }(\tau,\eta)\in\mathscr{V}.
	\end{align}
	Let $(\mathrm{i}\delta,\eta)\in\mathscr{V}$ so that $(\mathrm{i}\delta,\eta)\in\mathscr{V}$ is
	close to $(\mathrm{i}q\eta_0,\eta_0)$.
	It follows from \eqref{root.e} that
	$\widebar{C}_1^2\big(\delta+\widebar{C}_2\eta\big)^2>\eta^2$,
	which, combined with \eqref{eig1b}, implies
	\begin{align} \label{E1.p3}
	\omega_{\pm}(\mathrm{i}\delta,\eta)
	\in \mathrm{i}\mathbb{R}\setminus\{0\}.
	\end{align}
	Then we obtain from \eqref{factor} and expression \eqref{Lop1b} that
	\[
	h_q(\mathrm{i}\delta,\eta)\ne 0,\qquad
	\mathrm{i}(\delta-q\eta)h_q(\mathrm{i}\delta,\eta)
	=\Delta(\mathrm{i}\delta,\eta)\in \mathrm{i}\mathbb{R},
	\]
	from which we have
	\begin{align}\label{E1.p4}
	h_q(\mathrm{i}\delta,\eta)\in\mathbb{R}\setminus\{0\}.
	\end{align}
	Since $\tau_0\ne \mathrm{i}(\pm \widebar{C}_2\pm \widebar{C}_1^{-1})\eta_0$,
	eigenvalues $\omega_{\pm}$ depend analytically on $(\tau,\eta)$ in a neighborhood of $(\tau_0,\eta_0)$
	by the implicit function theorem.
	We then use \eqref{eig1a} to obtain that, for $(\tau,\eta)$ near $(\tau_0,\eta_0)$,
	\begin{align} \label{E1.p5}
	\omega_+(\tau,\eta)\partial_{\gamma} \omega_+(\tau,\eta)
	= \widebar{C}_0^2\widebar{C}_1^2(\tau+\mathrm{i}\widebar{C}_2\eta).
	\end{align}
	From \eqref{E1.p3} and \eqref{E1.p5}, we infer that the derivative, $\partial_{\gamma} \omega_+(\mathrm{i}\delta,\eta)$,
	is real by shrinking $\mathscr{V}$ if necessary.
	Employ \eqref{Lop1b} to derive
	$
	\partial_{\gamma}\Delta(\mathrm{i}\delta,\eta)\in \mathbb{R}.
	$
	We then deduce from \eqref{E1.p2} and \eqref{E1.p4} that
	\begin{align}\label{E1.p6}
	\partial_\gamma  h_q(\mathrm{i}\delta,\eta)\in\mathrm{i}\mathbb{R}.
	\end{align}
	Using \eqref{factor} and the Taylor formula for $h_q$,
	we find that, for $(\tau,\eta)\in\mathscr{V}$,
	\begin{align*}
	&\Delta(\tau,\eta)
	=\big(\gamma+\mathrm{i}(\delta-q\eta)\big)\left(h_q(\mathrm{i}\delta,\eta)+\gamma\partial_\gamma h_q(\mathrm{i}\delta,\eta)+O(\gamma^2)\right)\\
	&=\mathrm{i}(\delta-q\eta)h_q(\mathrm{i}\delta,\eta)
	+\left\{h_q(\mathrm{i}\delta,\eta)+\mathrm{i}\partial_\gamma h_q(\mathrm{i}\delta,\eta)(\delta-q\eta)\right\}\gamma+O(\gamma^2)
	 \
	(\gamma\to 0),
	\end{align*}
	where we have used the Landau symbol $f=O(g)$ $(x\to x_0)$,
	which means that there exists a constant $C$ such that $|f(x)|\leq C|g(x)|$ for all $x$ sufficiently close to $x_0$.
	We can conclude from \eqref{E1.p4} and \eqref{E1.p6} that
	$$
	\RE \Delta(\tau,\eta)=
	\left\{h_q(\mathrm{i}\delta,\eta)+\mathrm{i}\partial_\gamma h_q(\mathrm{i}\delta,\eta)(\delta-q\eta)\right\}\gamma+O(\gamma^2)
	\ \ (\gamma\to 0).
	$$
	Shrinking $\mathscr{V}$ if necessary, we derive from \eqref{E1.p4} that
	\begin{align*}
	|\Delta(\tau,\eta)| \geq |\RE \Delta(\tau,\eta)|\geq c \gamma \qquad\textrm{for all }(\tau,\eta)\in\mathscr{V}.
	\end{align*}
	Plug this into \eqref{E1.p1} to complete the proof of this lemma.
\qed\end{proof}

\begin{remark}\ \label{rem3.3}
	In the case of  $\textsf{M}=\textsf{M}_{\rm c}$,
	we know from Remark {\rm \ref{rem3.2}} that the Lopatinski\u{\i} determinant $\Delta(\tau,\eta)$ has only one triple root $\tau=0$.
	In a similar way, we can find neighborhoods $\mathscr{V}_{\pm}$ of $(0,\pm1)$ in $\Xi_1$ and a positive constant $c$
	such that
	\begin{align} \label{rem.e1}
	\left|\beta(\tau,\eta)(E_+(\tau,\eta) \ E_-(\tau,\eta))Z\right|\geq c \gamma^3 |Z|
	\quad \textrm{for all }(\tau,\eta)\in\mathscr{V}_\pm,\,Z\in\mathbb{C}^2.
	\end{align}
\end{remark}

We now adopt the argument developed recently by Chen--Hu--Wang \cite{CHW17Adv} to avoid constructing the Kreiss' symmetrizers
in the derivation of energy estimates for the constant coefficient case.
To this end, we need the following lemma.

\begin{lemma}\ \label{lem1.E2}
	For each point $(\tau_0,\eta_0)\in\Xi_1$, there exist a neighborhood $\mathscr{V}$ of $(\tau_0,\eta_0)$ in $\Xi_1$
	and a continuous invertible matrix $T(\tau,\eta)$ defined on $\mathscr{V}$ such that
	\begin{align} \label{key1}
	T^{-1}\mathcal{A}T(\tau,\eta)=
	\begin{pmatrix}
	\omega_+&z_+&0&0\\
	0&-\omega_+&0&0\\
	0&0&\omega_-&z_-\\
	0&0&0&-\omega_-
	\end{pmatrix}
	\end{align}
	 for all $(\tau,\eta)\in\mathscr{V}\setminus\{\tau=\pm\mathrm{i}\,\bar{v}\eta\}$, where $z_{\pm}=z_{\pm}(\tau,\eta)$ are complex-valued functions
	defined on $\mathscr{V}\setminus\{\tau=\pm\mathrm{i}\,\bar{v}\eta\}$.
	Moreover, the first and third columns of $T(\tau,\eta)$ are $E_+(\tau,\eta)$ and $E_-(\tau,\eta)$, respectively.
\end{lemma}

\begin{proof}\
	We set $a_{\pm}(\tau,\eta):=\tau\pm\mathrm{i}\,\bar{v}\eta$
	and define the following vectors on a neighborhood $\mathscr{V}$ of $(\tau_0,\eta_0)$:
	\begin{align*}
	Y_+(\tau,\eta)&:=\left\{
	\begin{aligned}
	&(0,1,0,0)^{\mathsf{T}}\quad &&\mathrm{if}\ a_+m_+(\tau_0,\eta_0)\ne 0,\\
	&(1,0,0,0)^{\mathsf{T}}\quad &&\mathrm{if}\ a_+(\mu_+-\omega_+)(\tau_0,\eta_0)\ne 0,
	\end{aligned}\right.\\[1mm]
	Y_-(\tau,\eta)&:=\left\{
	\begin{aligned}
	&(0,0,1,0)^{\mathsf{T}}\quad &&\mathrm{if}\ a_-m_-(\tau_0,\eta_0)\ne 0,\\
	&(0,0,0,1)^{\mathsf{T}}\quad &&\mathrm{if}\ a_-(\mu_--\omega_-)(\tau_0,\eta_0)\ne 0.
	\end{aligned}\right.
	\end{align*}
	Recall that $E_{\pm}(\tau,\eta)$ defined by \eqref{E.eig} are continuous and never vanish on $\Xi_1$.
	Hence, one can define the following continuous and invertible matrix on $\mathscr{V}$:
	\[
	T(\tau,\eta):=(E_+(\tau,\eta)\ \ Y_+(\tau,\eta)\ \ E_-(\tau,\eta)\ \ Y_-(\tau,\eta)).
	\]
	When $\tau\ne\pm\mathrm{i}\,\bar{v}\eta$, by a direct computation and using \eqref{eig1a},
	we obtain \eqref{key1} with
	\begin{align*}
	z_{\pm}(\tau,\eta):=
	\left\{
	\begin{aligned}
	&-\frac{1}{a_{\pm}(\tau,\eta)}\quad &&\mathrm{if}\ a_{\pm}m_{\pm}(\tau_0,\eta_0)\ne0,\\[1mm]
	&-\frac{m_{\pm}}{a_{\pm}(\mu_{\pm}-\omega_{\pm})(\tau,\eta)}\quad &&\mathrm{if}\ a_{\pm}(\mu_{\pm}-\omega_{\pm})(\tau_0,\eta_0)\ne0,
	\end{aligned}
	\right.
	\end{align*}
	which are well-defined apart from the poles of $\mathcal{A}$, {\it i.e.}
	from $\tau=\pm\mathrm{i}\,\bar{v}\eta$.
\qed\end{proof}

\vspace*{3mm}
\noindent {\bf{Proof of Theorem {\rm \ref{thm.1}}}} \quad
	According to Lemma \ref{lem1a},
	it suffices to show estimate \eqref{P1e.e0} in order to prove Theorem {\rm\ref{thm.1}}.
	Using Lemmas \ref{lem1.E1}--\ref{lem1.E2}, for each point $(\tau_0,\eta_0)\in\Xi_1$,
	there exists a neighborhood $\mathscr{V}$ of $(\tau_0,\eta_0)$  in $\Xi_1$ and a continuous invertible matrix $T(\tau,\eta)$
	defined on $\mathscr{V}$ such that \eqref{E1.e0} and \eqref{key1} hold.
	Thanks to the compactness of hemisphere $\Xi_1$,
	there exists a  finite covering $\{\mathscr{V}_1,\ldots,\mathscr{V}_J\}$ of $\Xi_1$ by such neighborhoods
	with corresponding matrices $\{T_1(\tau,\eta),\ldots,T_J(\tau,\eta)\}$, and a smooth partition of
	unity $\{\chi_j(\tau,\eta)\}_{j=1}^J$ such that
	$$
	\chi_j\in C^{\infty}_c(\mathscr{V}_j), \qquad
	\sum_{j=1}^J\chi_j^2=1 \ \ \mathrm{on}\ \Xi_1.
	$$
	
	We now derive an energy estimate in $\Pi_j:=\{(\tau,\eta)\in\Xi: s\cdot(\tau,\eta)\in\mathscr{V}_j\textrm{ for some }s>0 \}$
	and then patch them together to obtain \eqref{P1e.e0}.
	We first extend $\chi_j$ and $T_j$ to the conic zone $\Pi_j$ as homogeneous mappings
	of degree $0$ with respect to $(\tau,\eta)$.
	Note that both $T_j(\tau,\eta)$ and its inverse are bounded on $\Pi_j$,
	and identity \eqref{key1} holds for all $(\tau,\eta)\in \Pi_j$ with $\tau\ne\pm\mathrm{i}\,\bar{v}\eta$.
	Define
	\begin{align*}
	\textsf{W}(\tau,\eta,x_2):=\chi_jT_j(\tau,\eta)^{-1}\widehat{W^{\mathrm{nc}}}(\tau,\eta,x_2)\qquad\textrm{for all }(\tau,\eta)\in\Pi_j.
	\end{align*}
	Assume that $(\tau,\eta)\in\Pi_j$ with $\RE\tau>0$. In light of \eqref{P1e}, we obtain that $\textsf{W}$ satisfies
	\begin{align*}
	\frac{\mathrm{d}\textsf{W}}{\mathrm{d}x_2}=T_j(\tau,\eta)^{-1}\mathcal{A}T_j(\tau,\eta)\textsf{W}.
	\end{align*}
	Since \eqref{key1} holds when $(\tau,\eta)\in \Pi_j$ with $\RE\tau>0$,
	the equations for $\textsf{W}_2$ and $\textsf{W}_4$ read
	\begin{align} \label{P1f}
	\frac{\mathrm{d}\textsf{W}_2}{\mathrm{d}x_2}=-\omega_+\textsf{W}_2,\qquad
	\frac{\mathrm{d}\textsf{W}_4}{\mathrm{d}x_2}=-\omega_-\textsf{W}_4.
	\end{align}
	Recall from Lemma \ref{lem.eig1}\,(a) that $\RE\omega_{\pm}(\tau,\eta)<0$ whenever $\RE\tau>0$.
	Integration by parts for \eqref{P1f} yields
	\[
	\|\textsf{W}_2(\tau,\eta,\cdot)\|_{L^2(\mathbb{R}_+)}=\|\textsf{W}_4(\tau,\eta,\cdot)\|_{L^2(\mathbb{R}_+)}=0,
	\]
	from which we immediately deduce
	\begin{align}
	\label{P1f.e1}
	\textsf{W}_2(\tau,\eta,x_2)=\textsf{W}_4(\tau,\eta,x_2)=0
	\end{align}
  for all $x_2\in\mathbb{R}_+$ and $(\tau,\eta)\in\Pi_j$ with $\RE\tau>0$,
	where we have used the continuity of $\textsf{W}_2$ and $\textsf{W}_4$.
	Using the boundary equations in \eqref{P1e} yields
	\begin{align} \label{P1f.e2}
	\chi_j\widehat{G}=\beta(\tau,\eta)T_j(\tau,\eta)\textsf{W}(\tau,\eta,0)=\beta(\tau,\eta)(E_+\  E_-)\begin{pmatrix}
	\textsf{W}_1(\tau,\eta,0)\\[0.5mm]
	\textsf{W}_3(\tau,\eta,0)
	\end{pmatrix}
	\end{align}
	for all $(\tau,\eta)\in\Pi_j$ with $\RE\tau>0$.
	By the homogeneity of $T_j$ and $\beta$, we obtain from \eqref{E1.e0} that
	\begin{align*}
	(|\tau|+|\eta|)|\beta(\tau,\eta)(E_+(\tau,\eta)\  E_-(\tau,\eta))Z|\geq c_j \gamma |Z|
	\quad\textrm{for all }(\tau,\eta)\in\Pi_j,\,Z\in\mathbb{C}^2.
	\end{align*}
	Combine this with \eqref{P1f.e2} to deduce
	\begin{align}\label{P1f.e3}
	\left|(\textsf{W}_1(\tau,\eta,0),\textsf{W}_3(\tau,\eta,0))\right|
	\leq \frac{|\tau|+|\eta|}{c_j\gamma}\big|\chi_j\widehat{G}(\tau,\eta)\big|
	\end{align}
	for all $(\tau,\eta)\in\Pi_j$  with $\RE\tau>0$.
	Combining \eqref{P1f.e1} and \eqref{P1f.e3} yields
	\begin{align*}
	\left|\textsf{W}(\tau,\eta,0)\right|\leq
	\frac{|\tau|+|\eta|}{c_j\gamma}\big|\chi_j\widehat{G}(\tau,\eta)\big|
	\qquad \mbox{for all $(\tau,\eta)\in\Pi_j$ with $\RE\tau>0$}.
	\end{align*}
	We then obtain from the definition of $\textsf{W}$ and boundedness of $T_j(\tau,\eta)$ that
	\begin{align*}
	\big|\chi_j\widehat{W^{\mathrm{nc}}}(\tau,\eta,0)\big|\leq
	\frac{|\tau|+|\eta|}{c_j\gamma}\big|\chi_j\widehat{G}(\tau,\eta)\big|
	\end{align*}
	for all $(\tau,\eta)\in\Pi_j$ with $\gamma=\RE\tau>0$  and new positive constants $c_j$.
	Adding the above estimates for all $j\in\{1,\ldots,J\}$ and integrating the resulting estimate over $\mathbb{R}^2$
	with respect to $(\delta,\eta)$, we can derive the desired estimate \eqref{P1e.e0} from the Plancherel theorem.
	This completes the proof of Theorem {\rm\ref{thm.1}}.
\qed

\begin{remark}\  \label{rem3.4}
	In the case of $\textsf{M}=\textsf{M}_{\rm c}$,
	we can derive the energy estimate \eqref{rem.e0} by using \eqref{rem.e1}
	and employing a completely similar argument as above.
\end{remark}

\section{Variable Coefficient Linearized Problem} \label{sec.4}

In this section, we derive the linearized problem of \eqref{RE0} around a basic
state $\big(\mathring{U}^{\pm},\mathring{\Phi}^{\pm}\big)$ that is a small perturbation of
$\big(\widebar{U}^{\pm},\widebar{\Phi}^{\pm}\big)$ given in \eqref{RVS0}.
More precisely, we assume that the perturbations:
$\mathring{V}^{\pm}:=\mathring{U}^{\pm}-\widebar{U}^{\pm}$ and $\mathring{\Psi}^{\pm}:=\mathring{\Phi}^{\pm}-\widebar{\Phi}^{\pm}$
satisfy
\begin{align}
&\mathrm{supp}\,\big(\mathring{V}^{\pm},\mathring{\Psi}^{\pm}\big)\subset \{-T\le t\le 2T,\ x_2\geq 0,\ |x|\leq R\},\label{bas.c1}\\ \label{bas.c2}
&\mathring{V}^{\pm}\in W^{2,\infty}(\Omega),\   \mathring{\Psi}^{\pm}\in W^{3,\infty}(\Omega),\
\big\|\mathring{V}^{\pm}\big\|_{W^{2,\infty}(\Omega)}+\big\|\mathring{\Psi}^{\pm}\big\|_{W^{3,\infty}(\Omega)}\leq K,
\end{align}
where $T$, $R$, and $K$ are positive constants.
Moreover, we assume that $\big(\mathring{U}^{\pm},\mathring{\Phi}^{\pm}\big)$ satisfies constraints \eqref{Phi.eq}
and the Rankine--Hugoniot conditions \eqref{RE0.b}:
\begin{subequations}\label{bas.eq}
	\begin{alignat}{2}
	\label{bas.eq.1}&\partial_t\mathring{\Phi}^{\pm}+\mathring{v}_1^{\pm}\partial_1\mathring{\Phi}^{\pm}-\mathring{v}_2^{\pm}=0&\qquad &\mathrm{if}\ x_2\geq 0,\\
	\label{bas.eq.2}&\pm\partial_2\mathring{\Phi}^{\pm}\geq \kappa_0>0&\qquad &\mathrm{if}\ x_2\geq 0,\\
	\label{bas.eq.3}&\mathring{\Phi}^{+}=\mathring{\Phi}^{-}=\mathring{\varphi}&\qquad &\mathrm{if}\ x_2= 0,\\
	\label{bas.eq.4}&\mathbb{B}\big(\mathring{U}^+,\mathring{U}^-,\mathring{\varphi}\big)=0&\qquad &\mathrm{if}\ x_2= 0,
	\end{alignat}
\end{subequations}
where $\kappa_0$ is a positive constant. We will use $\mathring{V}:=(\mathring{V}^+,\mathring{V}^-)^{\mathsf{T}}$
and $\mathring{\Psi}:=(\mathring{\Psi}^+,\mathring{\Psi}^-)^{\mathsf{T}}$ to avoid overloaded expressions.

\subsection{Linearized Problem}\label{sec4.a}
Let us consider the families, $U^{\pm}_{\theta}=\mathring{U}^{\pm}+\theta V^{\pm}$
and $\Phi^{\pm}_{\theta}=\mathring{\Phi}^{\pm}+\theta\Psi^{\pm}$, with  a small parameter $\theta$.
The linearized operators are given by
\begin{align*}
\left\{\begin{aligned}
&\mathbb{L}'\big(\mathring{U}^{\pm},\mathring{\Phi}^{\pm}\big)(V^{\pm},\Psi^{\pm})
:=\left.\frac{\mathrm{d}}{\mathrm{d}\theta}\mathbb{L}\big(U^{\pm}_{\theta}, \Phi^{\pm}_{\theta}\big)\right|_{\theta=0},\\
&\mathbb{B}'\big(\mathring{U}^{\pm},\mathring{\Phi}^{\pm}\big)(V,\psi)
:=\left.\frac{\mathrm{d}}{\mathrm{d}\theta}{\mathbb{B}(U^{+}_{\theta},U^{-}_{\theta},\varphi_{\theta})}\right|_{\theta=0},
\end{aligned}\right.
\end{align*}
where $V:=(V^+,V^-)^{\mathsf{T}}$, and $\varphi_{\theta}$ (resp.\;$\psi$) denotes the common trace
of $\Phi^{\pm}_{\theta}$ (resp.\;$\Psi^{\pm}$) on boundary  $\{x_2=0\}$.
A standard computation yields the following expression for $\mathbb{L}'$:
\begin{align} \label{L.prime}
\mathbb{L}'(U,\Phi)(V,\Psi)=L(U,\Phi)V+\mathcal{C}(U,\Phi)V -{{\frac{1}{\partial_2 \Phi}L(U,\Phi)\Psi\partial_2 U}},
\end{align}
where $\mathcal{C}(U,\Phi)$ is the zero-th order operator defined by
\begin{align}
{\mathcal{C}(U,\Phi)V:=\big(\p_{U_i}A_0(U) \partial_t U+\p_{U_i}A_1(U) \partial_1 U+\p_{U_i}\widetilde{A}_2(U,\Phi) \partial_2 U\big)V_i.}
\label{C.cal}
\end{align}
We notice that matrices $\mathcal{C}\big(\mathring{U}^{\pm},\mathring{\Phi}^{\pm}\big)$ are $C^{\infty}$--functions
of $\big(\mathring{V}^{\pm},\nabla \mathring{V}^{\pm},\nabla \mathring{\Psi}^{\pm}\big)$ {vanishing} 
at the origin.

We recall that the first component of $\mathbb{B}(U^{+}_{\theta},U^{-}_{\theta},\varphi_{\theta})$ is
$\left[v_1(U_{\theta})\right]\partial_1\varphi_{\theta}-\left[v_2(U_{\theta})\right]$.
Ignoring  indices ``$+$'' and ``$-$'' for the moment, it follows from \eqref{v.partial} and \eqref{bas.eq.1} that
\begin{align*}
&\left.\frac{\mathrm{d}}{\mathrm{d}\theta}\big(\mathbb{B}(U^{+}_{\theta},U^{-}_{\theta},\varphi_{\theta})\big)_1\right|_{\theta=0}
=\mathring{v}_1\partial_1\psi+\partial_1\mathring{\varphi}\nabla_{U} v_1(\mathring{U}) \cdot V-\nabla_{U} v_2(\mathring{U})\cdot V\\
&=\mathring{v}_1\partial_1\psi-\frac{\epsilon^2(\p_1 \mathring{\varphi}\mathring{v}_1-\mathring{v}_2)}{\mathring{N}\mathring{h}\mathring{\varGamma}^2}V_1+\frac{\p_1 \mathring\varphi(1-\epsilon^2\mathring{v}_1^2)+\epsilon^2\mathring{v}_1\mathring{v}_2}{\mathring{h}\mathring{\varGamma}}V_2\\
&\ \ \  - \frac{\p_1 \mathring\varphi\epsilon^2\mathring{v}_1\mathring{v}_2+(1-\epsilon^2\mathring{v}_2^2) }{\mathring{h}\mathring{\varGamma}}V_3 \\
&=\mathring{v}_1\partial_1\psi
+\frac{\epsilon^2\p_t \mathring\varphi  }{\mathring{N}\mathring{h}\mathring{\varGamma}^2}V_1
+\frac{\p_1 \mathring\varphi+\epsilon^2 \mathring{v}_1 \p_t\mathring\varphi }{\mathring{h}\mathring{\varGamma}}V_2
- \frac{1-\epsilon^2 \mathring{v}_2 \p_t\mathring\varphi  }{\mathring{h}\mathring{\varGamma}}V_3\quad \mathrm{on}\ \  \{x_2=0\},
\end{align*}
where $\mathring{v}_j:=v_j(\mathring{U})$, $\mathring{N}:=N(\mathring{U}_1)$, $\mathring{h}:=h(\mathring{U}_1)$,
and $\mathring{\varGamma}:=\varGamma(\mathring{U})$.
Performing a similar analysis to the other components of $\mathbb{B}(U^{+}_{\theta},U^{-}_{\theta},\varphi_{\theta})$
implies
\begin{align} \label{B'.bb}
\mathbb{B}'\big(\mathring{U}^{\pm},\mathring{\Phi}^{\pm}\big)(V,\psi):=\mathring{b}\nabla\psi+\mathring{B} V|_{x_2=0},
\end{align}
where  $\nabla\psi:=( \partial_t\psi,\partial_1\psi)^{\mathsf{T}}$.
Coefficients $\mathring{b}$ and $\mathring{B}$ are defined by
\setlength{\arraycolsep}{2pt}
\begin{align} \label{b.ring}
\mathring{b}(t,x_1)&:=
{\small  \begin{pmatrix}
0  &(\mathring{v}_1^+-\mathring{v}_1^-)|_{x_2=0}\\
1 &\mathring{v}_{1}^+|_{x_2=0}\\
0 &0
\end{pmatrix}},\\
\label{M.ring}
\mathring{B}(t,x_1)&:=
{\small  \left.\begin{pmatrix}
\dfrac{\epsilon^2\p_t \mathring\varphi}{\mathring{N}_+\mathring{h}_+\mathring{\varGamma}_+^2}
&\dfrac{\mathring{\varrho}_{+}}{\mathring{h}_+\mathring{\varGamma}_+}
& \dfrac{-\mathring{\varsigma}_{+}}{\mathring{h}_+\mathring{\varGamma}_+}
&\dfrac{-\epsilon^2\p_t \mathring\varphi}{\mathring{N}_-\mathring{h}_-\mathring{\varGamma}_-^2}
&\dfrac{-\mathring{\varrho}_{-}}{\mathring{h}_-\mathring{\varGamma}_-}
& \dfrac{\mathring{\varsigma}_{-}}{\mathring{h}_-\mathring{\varGamma}_-} \\
\dfrac{\epsilon^2\p_t \mathring\varphi}{\mathring{N}_+\mathring{h}_+\mathring{\varGamma}_+^2}
&\dfrac{\mathring{\varrho}_{+}}{\mathring{h}_+\mathring{\varGamma}_+}
&\dfrac{-\mathring{\varsigma}_{+}}{\mathring{h}_+\mathring{\varGamma}_+}
& 0&0  & 0\\[3mm]
1 &0 &  0  &-1 &0& 0
\end{pmatrix}\right|_{x_2=0}}.
\end{align}
In expression \eqref{M.ring}, we have set $\mathring{\varrho}_{\pm}:=\varrho\big(\mathring{U}^{\pm},\mathring{\Phi}^{\pm}\big)$
and $\mathring{\varsigma}_{\pm}:=\varsigma\big(\mathring{U}^{\pm},\mathring{\Phi}^{\pm}\big)$,
where
\begin{align} \label{m.n.ring}
\varrho(U,\Phi):=\partial_1\Phi+\epsilon^2 v_1\partial_t\Phi,\qquad
\varsigma(U,\Phi):=1-\epsilon^2 v_2\partial_t\Phi.
\end{align}
In particular, if $\mathring{\Psi}^{\pm}\equiv 0$, then $\mathring{\varrho}_{\pm}\equiv 0$ and $\mathring{\varsigma}_{\pm}\equiv 1$.
Moreover, $\mathring{b}$ is a $C^{\infty}$--function of $\mathring{V}|_{x_2=0}$,
and $\mathring{B}$ is a $C^{\infty}$--function of $(\mathring{V}|_{x_2=0},\nabla \mathring{\varphi})$.

We simplify expression \eqref{L.prime} as Alinhac \cite{A89MR976971} by employing the ``good unknown'':
\begin{align} \label{good}
\dot{V}^{\pm}:=V^{\pm}-\frac{\partial_2\mathring{U}^{\pm}}{\partial_2 \mathring{\Phi}^{\pm}}\Psi^{\pm}.
\end{align}
After some direct calculation, we find ({\it cf.}\;M{{\'e}}tivier \cite[Proposition\,1.3.1]{M01MR1842775}) that
\begin{align}
\notag &\mathbb{L}'(\mathring{U}^{\pm}, \mathring{\Phi}^{\pm})(V^{\pm},\Psi^{\pm})\\
\label{Alinhac}&\quad =L(\mathring{U}^{\pm}, \mathring{\Phi}^{\pm})\dot{V}^{\pm} +\mathcal{C}( \mathring{U}^{\pm},\mathring{\Phi}^{\pm})\dot{V}^{\pm}
+\frac{\Psi^{\pm}}{\partial_2\mathring{\Phi}^{\pm}}\partial_2\big(L(\mathring{U}^{\pm} ,\mathring{\Phi}^{\pm} )\mathring{U}^{\pm}\big).
\end{align}
In view of the nonlinear results obtained in \cite{A89MR976971,FM00MR1787068,CS08MR2423311}, we neglect the zero-th order
term in $\Psi^{\pm}$ and consider the following \emph{effective linear problem}:
\begin{subequations} \label{P2b}
	\begin{alignat}{3}   \label{P2b.1}
	&\mathbb{L}'_e\big(\mathring{U}^{\pm},\mathring{\Phi}^{\pm}\big)\dot{V}^{\pm}
	:=L\big(\mathring{U}^{\pm},\mathring{\Phi}^{\pm}\big)\dot{V}^{\pm}+\mathcal{C}( \mathring{U}^{\pm},\mathring{\Phi}^{\pm})\dot{V}^{\pm}=f^{\pm}\quad &\textrm{if } x_2>0,\\
	&\mathbb{B}'_e\big(\mathring{U}^{\pm},\mathring{\Phi}^{\pm}\big)(\dot{V},\psi)
	:=\mathring{b}\nabla\psi+b_{\sharp}\psi+\mathring{B} \dot{V} |_{x_2=0}=g\quad  &\textrm{if } x_2=0, \label{P2b.3}\\
	&\Psi^+=\Psi^-=\psi\quad  &\textrm{if } x_2=0, \label{P2b.2}
	\end{alignat}
\end{subequations}
where $\mathcal{C}( \mathring{U}^{\pm},\mathring{\Phi}^{\pm})$, $\mathring{b}$, and $\mathring{B}$ are defined
by \eqref{C.cal}, \eqref{b.ring}, and \eqref{M.ring} respectively,
$\dot{V}:=(\dot{V}^+,\dot{V}^-)^{\mathsf{T}}$,  and
\begin{align}\label{b.sharp}
b_{\sharp}(t,x_1):=\mathring{B}(t,x_1)
(\frac{\partial_2 \mathring{U}^+}{\partial_2 \mathring{\Phi}^+},
\frac{\partial_2 \mathring{U}^-}{\partial_2 \mathring{\Phi}^-})^{\mathsf{T}}|_{x_2=0}.
\end{align}
Note that $b_{\sharp}$ is a $C^{\infty}$--function of
$(\mathring{V}|_{x_2=0},\p_2\mathring{V}|_{x_2=0},\nabla \mathring{\varphi},\p_2\mathring{\Psi}|_{x_2=0})$
that {vanishes} at the origin.
By virtue of \eqref{bas.c2}, it follows that
$\mathcal{C}(\mathring{U}^{\pm}, \mathring{\Phi}^{\pm})\in W^{1,\infty}(\Omega)$,
and the coefficients of operators $L\big(\mathring{U}^{\pm},\mathring{\Phi}^{\pm}\big)$ are in $W^{2,\infty}(\Omega)$.
We observe that the trace of vector $\mathring{B} \dot{V} $ involved in boundary conditions \eqref{P2b.3} depends solely on the traces of
$\mathbb{P}^{+}(\mathring{\varphi})\dot{V}^+$ and
$\mathbb{P}^{-}(\mathring{\varphi})\dot{V}^-$ on $\{x_2=0\}$,
where $\mathbb{P}^{\pm}(\mathring{\varphi})$ are defined as
\begin{align} \label{P.bb}
\mathbb{P}^{\pm}(\mathring{\varphi})V
:=\left(V_1,\mathring{\varsigma}_{\pm}|_{x_2=0}V_3-\mathring{\varrho}_{\pm}|_{x_2=0}V_2\right)^{\mathsf{T}},
\end{align}
with $\mathring{\varrho}_{\pm}$ and $\mathring{\varsigma}_{\pm}$ defined by \eqref{m.n.ring}.
We will consider the dropped term in \eqref{Alinhac}
as an error term at each Nash--Moser iteration step in the subsequent nonlinear analysis.

\subsection{Reformulation}
It is more convenient to transform the linearized problem \eqref{P2b}
into a problem with a constant and diagonal boundary matrix.
This is possible because the boundary matrix for \eqref{P2b}
has constant rank on the whole closed half-space $\{x_2\geq 0\}$.

Let us calculate the eigenvalues and the corresponding eigenvectors
of the boundary matrix for \eqref{P2b}.
Using constraint \eqref{bas.eq.1} reduces the coefficient matrices
$\widetilde{A}_2\big(\mathring{U}^{\pm},\mathring{\Phi}^{\pm}\big)$ to:
\begin{align} \label{A2.tilde}
\widetilde{A}_2\big(\mathring{U}^{\pm},\mathring{\Phi}^{\pm}\big)=
\frac{1}{\partial_2\mathring{\Phi}^{\pm}}
\left(\begin{array}{ccc}
0&-\mathring{N}^{\pm}\mathring{c}_{\pm}^2\mathring{\varrho}_{\pm}& \mathring{N}^{\pm}\mathring{c}_{\pm}^2\mathring{\varsigma}_{\pm}\\
-\mathring{\varrho}_{\pm}/\mathring{N}^{\pm}&0&0\\
\mathring{\varsigma}_{\pm}/\mathring{N}^{\pm}&0&0
\end{array}\right).
\end{align}
After a direct calculation, we obtain that the eigenvalues are
\begin{align*}
\lambda_1=0,\quad
\lambda_{2}= -\frac{\mathring{c}_\pm\sqrt{\mathring{\varrho}_\pm^2+\mathring{\varsigma}_\pm^2}}{\partial_2\mathring{\Phi}^\pm},\quad
\lambda_{3}= \frac{\mathring{c}_\pm\sqrt{\mathring{\varrho}_\pm^2+\mathring{\varsigma}_\pm^2}}{\partial_2\mathring{\Phi}^\pm},
\end{align*}
with corresponding eigenvectors
\begin{align*}
r_1=\begin{pmatrix}
0\\
\mathring{\varsigma}_\pm \\
\mathring{\varrho}_\pm
\end{pmatrix}
\quad
r_2=\left(\begin{array}{c}
\sqrt{\mathring{\varrho}_\pm^2+\mathring{\varsigma}_\pm^2}\, \\
\mathring{\varrho}_\pm/(\mathring{N}^\pm\mathring{c}_\pm)\\
-\mathring{\varsigma}_\pm/(\mathring{N}^\pm\mathring{c}_\pm)
\end{array}\right),
\quad
r_3=\left(\begin{array}{c}
\sqrt{\mathring{\varrho}_\pm^2+\mathring{\varsigma}_\pm^2}\, \\
-\mathring{\varrho}_\pm/(\mathring{N}^\pm\mathring{c}_\pm)\\
\mathring{\varsigma}_\pm/(\mathring{N}^\pm\mathring{c}_\pm)
\end{array}\right).
\end{align*}
Define the matrices
\begin{align} \label{T}
R\big(\mathring{U}^\pm,\mathring{\Phi}^\pm\big):=
\begin{pmatrix}
0&\sqrt{\mathring{\varrho}_\pm^2+\mathring{\varsigma}_\pm^2}
&\sqrt{\mathring{\varrho}_\pm^2+\mathring{\varsigma}_\pm^2}\\
\mathring{\varsigma}_\pm &\mathring{\varrho}_\pm/(\mathring{N}^\pm\mathring{c}_\pm)
&-\mathring{\varrho}_\pm/(\mathring{N}^\pm\mathring{c}_\pm) \\
\mathring{\varrho}_\pm &-\mathring{\varsigma}_\pm/(\mathring{N}^\pm\mathring{c}_\pm)
&\mathring{\varsigma}_\pm/(\mathring{N}^\pm\mathring{c}_\pm)
\end{pmatrix},
\end{align}
and  $\widetilde{A}_0\big(\mathring{U}^\pm,\mathring{\Phi}^\pm\big):=\mathrm{diag}\,({1},\,\lambda_{2}^{-1},\,\lambda_3^{-1})$.
Then it follows that
\begin{align*}
\widetilde{A}_0R^{-1}\widetilde{A}_2R\big(\mathring{U}^{\pm},\mathring{\Phi}^{\pm}\big)=\bm{I}_2:=\mathrm{diag}\;(0,1,1).
\end{align*}
We thus perform the transformation:
\begin{align}\label{W.def}
W^{\pm}:=R^{-1}\big(\mathring{U}^{\pm},\mathring{\Phi}^{\pm}\big)\dot{V}^{\pm}.
\end{align}
Multiplying \eqref{P2b.1} by matrices $\widetilde{A}_0R^{-1}\big(\mathring{U}^{\pm},\mathring{\Phi}^{\pm}\big)$
yields the equivalent system of \eqref{P2b.1}:
\begin{align}\label{P2c}
\bm{A}_0^{\pm}\partial_t W^{\pm}+\bm{A}_1^{\pm}\partial_1 W^{\pm}+\bm{I}_2\partial_2 W^{\pm}+\bm{C}^{\pm} W^{\pm}=F^{\pm},
\end{align}
where $F^{\pm}:=\widetilde{A}_0R^{-1}\big(\mathring{U}^{\pm},\mathring{\Phi}^{\pm}\big)f^{\pm}$, and
\begin{align} \label{A.bm}
\bm{A}_0^{\pm}&:=\widetilde{A}_0R^{-1}A_0R\big(\mathring{U}^{\pm},\mathring{\Phi}^{\pm}\big),\quad
\bm{A}_1^{\pm}:=\widetilde{A}_0R^{-1}A_1R\big(\mathring{U}^{\pm},\mathring{\Phi}^{\pm}\big),\\
\notag\bm{C}^{\pm}&:=\widetilde{A}_0\left(R^{-1}A_0\partial_tR+R^{-1}A_1\partial_1R+R^{-1}{\widetilde{A}}_2\partial_2R+R^{-1}\mathcal{C}R\right)
\big(\mathring{U}^{\pm},\mathring{\Phi}^{\pm}\big).
\end{align}
Matrices $\bm{A}_0^{\pm}$ and $\bm{A}_1^{\pm}$ belong to $W^{2,\infty}(\Omega)$,
while matrices $\bm{C}^{\pm}$ are in $W^{1,\infty}(\Omega)$.
Moreover, $\bm{A}_0^{\pm}$ and $\bm{A}_1^{\pm}$ are $C^{\infty}$--functions of their arguments
$(\mathring{V}^{\pm},\nabla \mathring{\Psi}^{\pm})$,
and $\bm{C}^{\pm}$ are $C^{\infty}$--functions of their arguments
$(\mathring{V}^{\pm},\nabla\mathring{V}^{\pm},\nabla \mathring{\Psi}^{\pm},\nabla^2 \mathring{\Psi}^{\pm})$.
Under transformation \eqref{W.def}, the boundary conditions {\eqref{P2b.3}--\eqref{P2b.2}} become
\begin{subequations}\label{P2c.2}
	\begin{alignat}{3}
	&\mathcal{B}^{\gamma}(W,\psi):=
	\mathring{b}\nabla\psi+b_{\sharp}\psi+\bm{B} W=g\quad  &\mathrm{if}\ x_2=0, \label{P2c.2b}\\
	&\Psi^+=\Psi^-=\psi\quad  &\mathrm{if}\ x_2=0, \label{P2c.2a}
	\end{alignat}
\end{subequations}
where $\mathring{b}$ and $b_{\sharp}$ are given by \eqref{b.ring} and \eqref{b.sharp} respectively,
$W:=(W^+,W^-)^{\mathsf{T}}$,  and
\begin{align} \notag
&\bm{B}(t,x_1):=
\mathring{B}\left.\begin{pmatrix}
R(\mathring{U}^{+},\mathring{\Phi}^{+}) & 0\\
0&R(\mathring{U}^{-},\mathring{\Phi}^{-})
\end{pmatrix}\right|_{x_2=0}\\[1mm]
& =\left.\begin{pmatrix}
0& \, {m}_1^++{m}_2^+ & \, {m}_1^+-{m}_2^+& \, 0 & \, -{m}_1^--{m}_2^- & \, -{m}_1^-+{m}_2^-\\[1mm]
0& \, {m}_1^++{m}_2^+ & \, {m}_1^+-{m}_2^+ & \, 0 & \, 0  & \, 0\\[1mm]
0 & \, \sqrt{\mathring{\varrho}_{+}^2+\mathring{\varsigma}_{+}^2}
& \, \sqrt{\mathring{\varrho}_{+}^2+\mathring{\varsigma}_{+}^2}& \, 0
& \,-\sqrt{\mathring{\varrho}_{-}^2+\mathring{\varsigma}_{-}^2} & \, -\sqrt{\mathring{\varrho}_{-}^2+\mathring{\varsigma}_{-}^2}
\end{pmatrix}\right|_{x_2=0}. \label{M.bm}
\end{align}
In the last expression, for notational simplicity, we have introduced ${m}_j^{\pm}$ as
\begin{align} \label{m1m2.bm}
{m}_1^{\pm}:=\dfrac{\epsilon^2\p_t\mathring{\Phi}^{\pm}\sqrt{\mathring{\varrho}_{\pm}^2
		+\mathring{\varsigma}_{\pm}^2} }{\mathring{\varGamma}_{\pm}^2\mathring{h}_{\pm}\mathring{N}^{\pm}},\qquad
{m}_2^{\pm}:=\dfrac{\mathring{\varrho}_{\pm}^2
	+\mathring{\varsigma}_{\pm}^2 }{\mathring{\varGamma}_{\pm}\mathring{c}_{\pm}\mathring{h}_{\pm}\mathring{N}^{\pm}}.
\end{align}
It is clear that matrix $\bm{B}$ is a $C^{\infty}$--function of $(\mathring{V}|_{x_2=0},\nabla \mathring{\varphi})$.
According to \eqref{P.bb} and \eqref{W.def}, we have
\begin{align} \label{P2d.id1}
\mathbb{P}^{\pm}(\mathring{\varphi})\dot{V}^{\pm}\,\!|_{x_2=0}=\left.\begin{pmatrix}
\sqrt{\mathring{\varrho}_{\pm}^2+\mathring{\varsigma}_{\pm}^2}(W_2^{\pm}+W_3^{\pm})\\[1mm]
-\dfrac{\mathring{\varrho}_{\pm}^2+\mathring{\varsigma}_{\pm}^2}{\mathring{N}^{\pm}\mathring{c}_{\pm}}(W_2^{\pm}-W_3^{\pm})
\end{pmatrix}\right|_{x_2=0}.
\end{align}
{We find that the trace of vector $\bm{B}W$ involved in boundary conditions \eqref{P2c.2} depends only} on the traces of the noncharacteristic part of vector $W$, {\it i.e.}  sub-vector $W^{\mathrm{nc}}:=(W_2^{+},W_3^{+},W_2^{-},W_3^{-})^{\mathsf{T}}$.

\section{Basic Energy Estimate for the Linearized Problem}\label{sec.5}

In this section, we are going to prove the following theorem, which provides the basic energy estimate
for the effective linear problem \eqref{P2b}.

\begin{theorem}\ \label{thm.2}
	Assume that the stationary solution \eqref{RVS0} satisfies \eqref{H1} and \eqref{H2}.
	Assume further that the basic state $\big(\mathring{U}^{\pm},\mathring{\Phi}^{\pm}\big)$ satisfies \eqref{bas.c1}--\eqref{bas.eq}.
	Then there exist constants $K_0>0$ and $\gamma_0\geq 1$ such that,
	if $K\leq K_0$ and $\gamma\geq \gamma_0$,
	then, for all $(\dot{V},\psi)\in H^2_{\gamma}(\Omega)\times H^2_{\gamma}(\mathbb{R}^2)$, the following estimate holds{\rm :}
	\begin{align}
	&\gamma \|\dot{V}\|_{L^2_{\gamma}(\Omega)}^2+\|\mathbb{P}^{\pm}(\mathring{\varphi})\dot{V}^{\pm}\,\!|_{x_2=0}\|_{L^2_{\gamma}(\mathbb{R}^2)}^2
	+\|\psi\|_{H^1_{\gamma}(\mathbb{R}^2)}^2\notag \\
	&  \lesssim  \gamma^{-3}\big\|\mathbb{L}'_e\big(\mathring{U}^{\pm},\mathring{\Phi}^{\pm}\big)\dot{V}^{\pm}\big\|_{L^2(H_{\gamma}^1)}^2
	+\gamma^{-2}\big\|\mathbb{B}_e'\big(\mathring{U}^{\pm},\mathring{\Phi}^{\pm}\big)(\dot{V}|_{x_2=0},\psi)\big\|_{H^1_{\gamma}(\mathbb{R}^2)}^2.
	\label{thm.2e1}
	\end{align}
\end{theorem}

\begin{remark}\
	Since the Lopatinski\u{\i} determinant associated with problem \eqref{P2b} admits the roots on the  boundary of frequency space,
	the energy estimate \eqref{thm.2e1} has a loss of regularity of the solution with respect to the source terms.
	Furthermore, there is a loss of control on the traces of the solution in  \eqref{thm.2e1},
	which is mainly owing to the fact that \eqref{P2b} is a characteristic boundary problem.
\end{remark}

We notice that systems \eqref{P2b.1} are symmetrizable hyperbolic with the Friedrichs symmetrizers
$S_2\big(\mathring{U}^{\pm}\big)$ for operators $\mathbb{L}'_e\big(\mathring{U}^{\pm},\mathring{\Phi}^{\pm}\big)$,
where function $S_2(U)$ is defined in \eqref{Sym}.
By virtue of \eqref{bas.eq.1}, we compute
\begin{align*}
&S{_2}(\mathring{U}^{\pm})\widetilde{A}_2(\mathring{U}^{\pm},\mathring{\Phi}^{\pm})\\
&=\frac{1}{\partial_2\mathring{\Phi}^{\pm}}S{_2}(\mathring{U}^{\pm})\big(A_2(\mathring{U}^{\pm})-\partial_t\mathring{\Phi}^{\pm} A_0(\mathring{U}^{\pm})
-\partial_1\mathring{\Phi}^{\pm} A_1(\mathring{U}^{\pm})\big)\\
&=\frac{1}{\partial_2\mathring{\Phi}^{\pm}}
\begin{pmatrix}
0&-\mathring{N}^{\pm}\mathring{c}^2_{\pm}\mathring{\varrho}_{\pm}& \mathring{N}^{\pm}\mathring{c}^2_{\pm}\mathring{\varsigma}_{\pm}\\
-\mathring{N}^{\pm}\mathring{c}^2_{\pm}\mathring{\varrho}_{\pm}&0&0\\
\mathring{N}^{\pm}\mathring{c}^2_{\pm}\mathring{\varsigma}_{\pm}&0&0
\end{pmatrix},
\end{align*}
where $\mathring{\varrho}_{\pm}$ and $\mathring{\varsigma}_{\pm}$ are defined in \eqref{m.n.ring}.
Multiplying \eqref{P2b.1} by the Friedrichs symmetrizers $S{_2}(\mathring{U}^{\pm})$ and employing integration
by parts yield the following lemma:

\begin{lemma}\ \label{lem.2a}
	There exists a constant $\gamma_0\geq 1$ such that, for all $\gamma\geq \gamma_0$, the following estimate holds{\rm :}
	\begin{align*}
	\gamma\|\dot{V}^{\pm}\|_{L^2_\gamma(\Omega)}^2\lesssim
	\gamma^{-1}\big\|\mathbb{L}'_e\big(\mathring{U}^{\pm},\mathring{\Phi}^{\pm}\big)\dot{V}^{\pm}\big\|_{L^2_\gamma(\Omega)}^2
	+\|\mathbb{P}^{\pm}(\mathring{\varphi})\dot{V}^{\pm}\,\!|_{x_2=0}\|_{L^2_\gamma(\mathbb{R}^2)}^2.
	\end{align*}
\end{lemma}

To prove Theorem {\rm\ref{thm.2}}, it remains to deduce the desired energy estimate for the discontinuity front $\psi$
and the traces of $\mathbb{P}^{\pm}(\mathring{\varphi})\dot{V}^{\pm}$ on $\{x_2=0\}$
in terms of the source terms in the interior domain and on the boundary.

Introducing $\widetilde{W}^{\pm}:=\mathrm{e}^{-\gamma t}W^{\pm}$, system \eqref{P2c} equivalently reads
\begin{align}
\notag
&\mathcal{L}^{\gamma}_{\pm}\widetilde{W}^{\pm}\\
\label{P2d.1}&:=
\gamma \bm{A}_0^{\pm}\widetilde{W}^{\pm}+
\bm{A}_0^{\pm}\partial_t \widetilde{W}^{\pm}+\bm{A}_1^{\pm}\partial_1 \widetilde{W}^{\pm}
+\bm{I}_2\partial_2\widetilde{W}^{\pm}+\bm{C}^{\pm}\widetilde{W}^{\pm}
=\mathrm{e}^{-\gamma t}F^{\pm}.
\end{align}
We also introduce  $\widetilde{W}:=(\widetilde{W}^+,\widetilde{W}^-)^{\mathsf{T}}$,  $\widetilde{\Psi}^{\pm}:=\mathrm{e}^{-\gamma t}\Psi^{\pm}$,
and $\widetilde{\psi}:=\mathrm{e}^{-\gamma t}\psi$.
Then the boundary conditions \eqref{P2c.2} are equivalent to
\begin{subequations}\label{P2d.2}
	\begin{alignat}{3}
	&\mathcal{B}^{\gamma}(\widetilde{W},\widetilde{\psi}):=
	\gamma\bm{b}_0\widetilde{\psi}+\mathring{b}\nabla\widetilde\psi+b_{\sharp}\widetilde\psi+\bm{B} \widetilde{W} =\mathrm{e}^{-\gamma t}g\qquad  &\mathrm{if}\ x_2=0, \label{P2d.2b}\\
	& \widetilde\Psi^+=\widetilde\Psi^-=\widetilde\psi\quad  &\mathrm{if}\ x_2=0, \label{P2d.2a}
	\end{alignat}
\end{subequations}
where $\bm{b}_0:=(0,1,0)^{\mathsf{T}}$. In view of  \eqref{P2d.id1}, we obtain the estimate:
\begin{align}\label{P2d.e1}
\|\mathbb{P}^{\pm}(\mathring{\varphi})\dot{V}^{\pm}\,\!|_{x_2=0}\|_{L^2_\gamma(\mathbb{R}^2)}\lesssim \left\|W^{\mathrm{nc}}\,\!|_{x_2=0}\right\|_{L^2_\gamma(\mathbb{R}^2)}
\lesssim \big\|\widetilde{W}^{\mathrm{nc}}\,\!|_{x_2=0}\big\|,
\end{align}
where $\widetilde{W}^{\mathrm{nc}}:=\big(\widetilde{W}_2^{+},\widetilde{W}_3^{+},\widetilde{W}_2^{-},\widetilde{W}_3^{-}\big)^{\mathsf{T}}$.
By virtue of \eqref{P2d.e1} and Lemma \ref{lem.2a},
we obtain that Theorem {\rm\ref{thm.2}} admits the following equivalent proposition.

\begin{proposition}\  \label{pro2}
	Assume that the stationary solution \eqref{RVS0} satisfies \eqref{H1} and \eqref{H2}.
	Assume further that the basic state $\big(\mathring{U}^{\pm},\mathring{\Phi}^{\pm}\big)$
	satisfies \eqref{bas.c1}--\eqref{bas.eq}.
	Then there exist some constants $K_0>0$ and $\gamma_0\geq 1$ such that,
	if $K\leq K_0$ and $\gamma\geq \gamma_0$,
	then the following estimate holds: for all $(W,\psi)\in H^2(\Omega)\times H^2(\mathbb{R}^2)$,
	\begin{align} \label{P2d.e2}
	\left\|W^{\mathrm{nc}}\,\!|_{x_2=0}\right\|^2+\|\psi\|_{1,\gamma}^2
	\lesssim \gamma^{-3}\VERT \mathcal{L}^{\gamma}_{\pm}W^{\pm}\VERT_{1,\gamma}^2
	+\gamma^{-2}\|\mathcal{B}^{\gamma}(W^{\mathrm{nc}}\,\!|_{x_2=0},\psi)\|_{1,\gamma}^2,
	\end{align}
	where operators $\mathcal{L}^{\gamma}_{\pm}$ and $\mathcal{B}^{\gamma}$
	are given by \eqref{P2d.1} and \eqref{P2d.2b}, respectively.
\end{proposition}

In the rest of this section, we give the proof of Proposition \ref{pro2}.

\subsection{Paralinearization}
We now perform the paralinearization of the interior equations and the boundary conditions.

\subsubsection{Some results on paradifferential calculus}
For self-containedness, we list some definitions and results about paradifferential calculus
with a parameter that will be used in this paper.
See \cite[Appendix\;C]{B-GS07MR2284507} and the references cited therein for the rigorous proofs.

\begin{definition} For any $m\in\mathbb{R}$ and $k\in\mathbb{N}$, we define the following{\rm :}
	\begin{list}{}{\setlength{\parsep}{\parskip}
			\setlength{\itemsep}{0.1em}
			\setlength{\labelwidth}{2em}
			\setlength{\labelsep}{0.4em}
			\setlength{\leftmargin}{2.2em}
			\setlength{\topsep}{1mm}
		}
		\item[\rm{(i)}] A function $a(x,\xi,\gamma):\mathbb{R}^2\times\mathbb{R}^2\times[1,\infty)\to\mathbb{C}^{N\times N}$
		is called a paradifferential symbol of degree $m$ and regularity $k$ if $a$ is $C^{\infty}$ in $\xi$ and,
		for each $\alpha\in\mathbb{N}^2$, there exists a positive constant $C_{\alpha}$ such that
		\begin{align*}
		\|\partial_{\xi}^{\alpha}a(\cdot,\xi,\gamma)\|_{W^{k,\infty}(\mathbb{R}^2)}\leq C_{\alpha}\lambda^{m-|\alpha|,\gamma }(\xi)
		\qquad\textrm{for all }(\xi,\gamma)\in \mathbb{R}^2\times[1,\infty),
		\end{align*}
		where $\lambda^{s,\gamma }(\xi):=(\gamma^2+|\xi|^2)^{s/2}$ for $s\in\mathbb{R}$.
		
		\item[\rm{(ii)}] ${\mathrm{\Gamma}^m_k}$ denotes the set of paradifferential symbols of degree $m$ and regularity $k$.
		We denote by $\alpha_m$ a generic symbol in the class ${\mathrm{\Gamma}^m_1}$.
		
		\item[\rm{(iii)}] We say that a family of operators $\{P^{\gamma}\}_{\gamma\geq 1}$ is of order $\leq m$,
		if,  for every $s\in\mathbb{R}$ and $\gamma\geq 1$, there exists a constant $C(s,m)$ independent of $\gamma$ such that
		\begin{align*}
		\|P^{\gamma}u\|_{s,\gamma}\leq C(s,m)\|u\|_{s+m,\gamma}\qquad\textrm{for all } u\in H^{s+m}.
		\end{align*}
		We use $\mathcal{R}_m$ to denote a generic family of operators of order $\leq m$.
		
		\item[\rm{(iv)}]
		For $s\in\mathbb{R}$, operator $\Lambda^{s,\gamma}$ is defined in such a way that
		\begin{align*}
		\Lambda^{s,\gamma}u(x)
		:=\frac{1}{(2\pi)^2}\int_{\mathbb{R}^2}\mathrm{e}^{\mathrm{i}x\cdot\xi}\lambda^{s,\gamma}(\xi)\widehat{u}(\xi)\mathrm{d}\xi
		\end{align*}
		for all $u$ in the Schwartz class $\mathcal{S}$.
		
		\item[\rm{(v)}] To any symbol $a\in {\mathrm{\Gamma}_0^m}$,
		we associate the family of paradifferential operators $\{T^{\gamma}_a\}_{\gamma\geq 1}$ defined in such a way that
		\begin{align*}
		T_a^{\gamma} u(x):=
		\frac{1}{(2\pi)^2}\int_{\mathbb{R}^2}\int_{\mathbb{R}^2}\mathrm{e}^{\mathrm{i}x\cdot\xi}K^{\psi}(x-y,\xi,\gamma)a(y,\xi,\gamma)
		\widehat{u}(\xi)\,\mathrm{d}y\mathrm{d}\xi
		\end{align*}
		for all $u\in\mathcal{S}$.
		In the last expression, $K^{\psi}(\cdot,\xi,\gamma)$ is the inverse Fourier transform of $\psi(\cdot,\xi,\gamma)$
		with $\psi$ given by
		\begin{align*}
		\psi(x,\xi,\gamma):=\sum_{q\in\mathbb{N}}\chi(2^{2-1}x,0)\phi(2^{-q}\xi,2^{-q}\gamma),
		\end{align*}
		where $\phi(\xi,\gamma):=\chi(2^{-1}\xi, 2^{-1}\gamma)-\chi(\xi,\gamma)$,
		and $\chi$ is a $C^{\infty}$--function on $\mathbb{R}^3$ such that
		\begin{align*}
		\chi(z)\geq \chi(z')\ \  \mathrm{if}\ |z|\leq |z'|,\qquad
		\chi(z)=\begin{cases}
		1\qquad &\mathrm{if}\ |z|\leq \frac{1}{2},\\[0.5mm]
		0\qquad &\mathrm{if}\ |z|\geq 1.
		\end{cases}
		\end{align*}
	\end{list}
\end{definition}

\begin{lemma}\ \label{lem.para1}
The following statements hold:
	\begin{list}{}{\setlength{\parsep}{\parskip}
			\setlength{\itemsep}{0.1em}
			\setlength{\labelwidth}{2em}
			\setlength{\labelsep}{0.4em}
			\setlength{\leftmargin}{2.2em}
			\setlength{\topsep}{1mm}
		}
		\item[\emph{(i)}]
		If $a\in W^{1,\infty}(\mathbb{R}^2)$, $u\in L^2(\mathbb{R}^2)$, and $\gamma\geq 1$, then
		\begin{align*}
		\gamma \|au-T_a^{\gamma}u\|+ \|a\partial_{j}u-T_{\mathrm{i}\xi_{j}a}^{\gamma}u\|+\|au-T_a^{\gamma}u\|_{1,\gamma}
		\lesssim \|a\|_{W^{1,\infty}(\mathbb{R}^2)}\|u\|.
		\end{align*}
		
		\item[\emph{(ii)}]
		If $a\in W^{2,\infty}(\mathbb{R}^2)$, $u\in L^2(\mathbb{R}^2)$, and $\gamma\geq 1$, then
		\begin{gather*}
		\gamma \|au-T_a^{\gamma}u\|_{1,\gamma}+\|a\partial_{j}u-T_{\mathrm{i}\xi_{j}a}^{\gamma}u\|_{1,\gamma}
		\lesssim \|a\|_{W^{2,\infty}(\mathbb{R}^2)}\|u\|.
		\end{gather*}
		
		\item[\emph{(iii)}]
		If $a\in {\mathrm{\Gamma}_k^m}$, then $T_a^{\gamma}$ is of order $\leq m$.
		In particular, if  $a\in L^{\infty}(\mathbb{R}^2)$ is independent of $\xi$, then
		\[
		\|T_a^{\gamma}u\|_{s,\gamma}\lesssim \|a\|_{L^{\infty}(\mathbb{R}^2)}\|u\|_{s,\gamma}
		\qquad\textrm{for all }s\in\mathbb{R},\,u\in H^{s}(\mathbb{R}^2).
		\]
		\item[\emph{(iv)}]
		If $a\in{\mathrm{\Gamma}_1^{m}}$ and $b\in{\mathrm{\Gamma}_1^{m'}}$, then product $ab\in{\mathrm{\Gamma}_1^{m+m'}}$,
		family $\{T_a^{\gamma}T_b^{\gamma}-T_{ab}^{\gamma} \}_{\gamma\geq 1}$ is of order $\leq m+m'-1$,
		and family
		$\{(T_a^{\gamma})^{*}-T_{a^*}^{\gamma}\}_{\gamma\geq 1}$ is of order $\leq m-1$.
		
		\item[\emph{(v)}]
		If $a\in{\mathrm{\Gamma}_2^{m}}$ and $b\in{\mathrm{\Gamma}_2^{m'}}$, then
		$
		\{T_a^{\gamma}T_b^{\gamma}-T_{ab}^{\gamma}-T_{-\mathrm{i}\sum_{j}\partial_{\xi_j}a\partial_{x_j}b}^{\gamma}\}_{\gamma\geq 1}
		$
		is of order $\leq m+m'-2$.
		
		\item[\emph{(vi)}]
		G{\aa}rding's inequality{\rm :}
		If $a\in{\mathrm{\Gamma}_1^{2m}}$ is a square matrix symbol that satisfies
		\[
		\RE a(x,\xi,\gamma)\geq c(\gamma^2+|\xi|^2)^m I
		\qquad\textrm{for all }(x,\xi,\gamma)\in\mathbb{R}^4\times[1,\infty)
		\]
		for some constant $c$, then there exists $\gamma_0\geq 1$ such that
		\[
		\RE\left\langle T_a^{\gamma}u,u\right\rangle \geq \frac{c}{4}\|u\|^2_{m,\gamma}
		\qquad\textrm{for all $u\in H^m(\mathbb{R}^2)$ and $\gamma\geq \gamma_0$}.
		\]
		
		\item[\emph{(vii)}]
		Microlocalized G{\aa}rding's inequality{\rm :}
		Let $a\in{\mathrm{\Gamma}_1^{2m}}$ be a square matrix symbol and $\chi\in{\mathrm{\Gamma}_1^0}$.
		If there exist a scalar real symbol $\widetilde{\chi}\in{\mathrm{\Gamma}_1^0}$ and a constant $c>0$
		such that $\widetilde{\chi}\geq 0$, $\chi\widetilde{\chi}\equiv\chi$, and
		\[
		\widetilde{\chi}^2(x,\xi,\gamma)\RE a(x,\xi,\gamma)\geq c\widetilde{\chi}^2(x,\xi,\gamma)(\gamma^2+|\xi|^2)^m I
		\]
       for all $(x,\xi,\gamma)\in\mathbb{R}^4\times[1,\infty)$, then there exist  $\gamma_0\geq 1$ and $C>0$ such that
		\[
		\RE\left\langle T_a^{\gamma}T_{\chi}^{\gamma}u,T_{\chi}^{\gamma}u\right\rangle
		\geq \frac{c}{2}\|T_{\chi}^{\gamma}u\|^2_{m,\gamma}-C\|u\|_{m-1,\gamma}^2
		\]
		  for all $u\in H^m(\mathbb{R}^2)$, $\gamma\geq \gamma_0$.
	\end{list}
Here we have used the notation, $\RE B :=(B+B^*)/2$, for any complex square matrix $B$ with $B^*$
being its conjugate transpose.
\end{lemma}

The reader may find the detailed proof of Lemma \ref{lem.para1} (vii)
in M\'{e}tivier--Zumbrun \cite[Theorem\;B.18]{MZ05MR2130346}.

\subsubsection{Paralinearization of the interior equations}

In view of \eqref{bas.c2} and \eqref{bas.eq.2}, we have the following estimate
for the coefficients of $\mathcal{L}_{\pm}^{\gamma}$ given in \eqref{P2d.1}:
\begin{align*}
\|(\bm{A}_0^{\pm},\bm{A}_1^{\pm})\|_{W^{2,\infty}(\Omega)}+
\|\bm{C}^{\pm}\|_{W^{1,\infty}(\Omega)}\leq C(K,\kappa_0).
\end{align*}
It then follows from  Lemma \ref{lem.para1}\;(ii) that
\begin{align*}
\VERTbig \gamma \bm{A}_{0}^+W^+-T^{\gamma}_{\gamma \bm{A}_{0}^+} W^+\VERTbig_{1,\gamma}^2
&=\int_{0}^{\infty}\|\gamma\bm{A}_{0}^+W^+(\cdot,x_2)-T^{\gamma}_{\gamma\bm{A}_{0}^+} W^+(\cdot,x_2)\|_{1,\gamma}^2\mathrm{d}x_2\\
&\lesssim \int_{0}^{\infty} \|\bm{A}_{0}^+(\cdot,x_2)\|_{W^{2,\infty}(\mathbb{R}^2)}^2\|W^+(\cdot,x_2)\|^2\mathrm{d}x_2\\
&\leq C(K,\kappa_0)\VERT W^+ \VERT^2.
\end{align*}
Similarly, we derive from Lemma \ref{lem.para1}\;(i)--(ii) that
\begin{align*}
&\VERTbig \bm{A}_{0}^+\partial_tW^+-T^{\gamma}_{\mathrm{i}\delta \bm{A}_{0}^+} W^+\VERTbig_{1,\gamma}
\lesssim \|\bm{A}_{0}^+\|_{W^{2,\infty}(\Omega)}\VERT W^+ \VERT\leq C(K,\kappa_0)\VERT W^+ \VERT,\\
&\VERTbig  \bm{A}_{1}^+\partial_1W^+-T^{\gamma}_{\mathrm{i}\eta \bm{A}_{1}^+} W^+\VERTbig_{1,\gamma}
\lesssim \|\bm{A}_{1}^+\|_{W^{2,\infty}(\Omega)}\VERT W^+ \VERT
\leq C(K,\kappa_0)\VERT W^+ \VERT,\\
&\VERTbig  \bm{C}^+ W^+-T^{\gamma}_{  \bm{C}^+} W^+\VERTbig_{1,\gamma}
\lesssim \|\bm{C}^+\|_{W^{1,\infty}(\Omega)}\VERT W^+ \VERT
\leq C(K,\kappa_0)\VERT W^+ \VERT.
\end{align*}
Combining these estimates yields
\begin{align} \label{p.e1a}
\VERTbig  \mathcal{L}^{\gamma}_+W^+-\bm{I}_2\partial_2W^{+}-T^{\gamma}_{\tau \bm{A}_{0}^++\mathrm{i}\eta\bm{A}_1^{+}+\bm{C}^{+}} W^+\VERTbig_{1,\gamma}
\leq C(K,\kappa_0)\VERT W^+ \VERT,
\end{align}
where $\tau=\gamma+\mathrm{i}\delta$, and $\mathcal{L}^{\gamma}_+$ is the linearized operator defined by \eqref{P2d.1}.
We can also obtain the following estimate for the equations on $W^{-}$:
\begin{align}\label{p.e1b}
\VERTbig  \mathcal{L}^{\gamma}_-W^--\bm{I}_2\partial_2W^{-}-T^{\gamma}_{\tau \bm{A}_{0}^-+\mathrm{i}\eta\bm{A}_1^{-}+\bm{C}^{-}} W^-\VERTbig_{1,\gamma}
\leq C(K,\kappa_0)\VERT W^- \VERT.
\end{align}
The paralinearization for the interior equations is thus given as follows:
\begin{align} \label{p.p1a}
T^{\gamma}_{\tau \bm{A}_0^{\pm}+\mathrm{i}\eta\bm{A}_1^{\pm}+\bm{C}^{\pm}}W^{\pm}+\bm{I}_2\partial_2W^{\pm}=\widetilde{F}^{\pm}
\qquad \mathrm{if}\  x_2>0.
\end{align}
Note that the above paralinearized equations do not involve the discontinuity function $\varphi$.

\subsubsection{Paralinearization of the boundary conditions}
According to \eqref{P2d.2b}, we define
\begin{align*}
&\bm{b}_0:=(0,1,0)^{\mathsf{T}},\quad
\bm{b}_1(t,x_1):=\left((\mathring{v}_{1}^+-\mathring{v}_{1}^-)|_{x_2=0},\mathring{v}_{1}^+|_{x_2=0}, 0\right)^{\mathsf{T}},\\
&\bm{b}(t,x_1,\delta,\eta,\gamma):=\tau\bm{b}_0+\mathrm{i}\eta\bm{b}_1(t,x_1)
=
(\mathrm{i}\eta (\mathring{v}_{1}^+-\mathring{v}_{1}^-),
\tau +\mathrm{i}\mathring{v}_{1}^+\eta,
0)^{\mathsf{T}}|_{x_2=0}.
\end{align*}
Since $\bm{b}_0,\bm{b}_1\in W^{2,\infty}(\mathbb{R}^2)$,
we obtain from Lemma \ref{lem.para1} (iii) that
\begin{align} \notag
&\|\gamma\bm{b}_0\psi+\mathring{b}\nabla\psi-T^{\gamma}_{\bm{b}}\psi\|_{1,\gamma}\\
&\notag \quad=\|\gamma\bm{b}_0\psi+\bm{b}_0\partial_t\psi+\bm{b}_1\partial_1\psi-T^{\gamma}_{\bm{b}}\psi\|_{1,\gamma}\\
&\quad \lesssim
\|(\bm{b}_0,\bm{b}_1)\|_{W^{2,\infty}(\mathbb{R}^2)}\|\psi\|\leq C(K)\|\psi\|
\leq C(K)\gamma^{-1}\|\psi\|_{1,\gamma}.
\label{p.e2a}
\end{align}
It follows from \eqref{bas.c2}, \eqref{bas.eq.2}, and \eqref{b.sharp} that
$\|b_{\sharp}\|_{W^{1,\infty}(\mathbb{R}^2)}\leq C(K,\kappa_0)$.
Employing Lemma \ref{lem.para1} (ii)--(iii) yields
\begin{align} \notag
\|b_{\sharp}\psi\|_{1,\gamma}
&\leq
\|b_{\sharp}\psi-T^{\gamma}_{b_{\sharp}}\psi\|_{1,\gamma}+\|T^{\gamma}_{b_{\sharp}}\psi\|_{1,\gamma}\\
&\lesssim \|b_{\sharp}\|_{W^{1,\infty}(\mathbb{R}^2)}\|\psi\|+\|b_{\sharp}\|_{L^{\infty}(\mathbb{R}^2)}\|\psi\|_{1,\gamma}
\leq C(K,\kappa_0)\|\psi\|_{1,\gamma}.
\label{p.e2b}
\end{align}
In light of \eqref{M.bm}, we find that
$\|\bm{B}\|_{W^{2,\infty}(\mathbb{R}^2)}\leq C(K,\kappa_0)$,
and $\bm{B}$ acts only on the noncharacteristic part $W^{\mathrm{nc}}$ of vector $W$.
Hence, we deduce
\begin{align}\notag
&\|\bm{B}W|_{x_2=0}-T^{\gamma}_{\bm{B}}W|_{x_2=0}\|_{1,\gamma}\\
&\lesssim \gamma^{-1}\|\bm{B}\|_{W^{2,\infty}(\mathbb{R}^2)}
\left\|W^{\mathrm{nc}}\,\!|_{x_2=0}\right\|
\leq C(K,\kappa_0)\gamma^{-1}\left\|W^{\mathrm{nc}}\,\!|_{x_2=0}\right\|.
\label{p.e2c}
\end{align}
Combine \eqref{p.e2a}--\eqref{p.e2c} together
to find
\begin{align}
\notag &\|\mathcal{B}^{\gamma}(W|_{x_2=0},\psi)-T^{\gamma}_{\bm{b}}\psi-T^{\gamma}_{\bm{B}}W|_{x_2=0}\|_{1,\gamma}\\
&\quad \leq C(K,\kappa_0)\left(\|\psi\|_{1,\gamma}+\gamma^{-1}\left\|W^{\mathrm{nc}}\,\!|_{x_2=0}\right\|\right).
\label{p.e2}
\end{align}
The paralinearization of the boundary conditions \eqref{P2d.2b} is then given as follows:
\begin{align} \label{p.p1b}
T^{\gamma}_{\bm{b}}\psi+T^{\gamma}_{\bm{B}}W=G\qquad\mathrm{if}\ x_2=0.
\end{align}

\subsubsection{Eliminating the front}
We can eliminate front $\psi$ from the paralinearized boundary conditions \eqref{p.p1b}
as in the constant coefficient case.
For this purpose, we first  notice that symbol $\bm{b}$ is elliptic, which means that, for any $(t,x_1,\delta,\eta,\gamma)\in\mathbb{R}^4\times (0,\infty )$,
\begin{align}
|\bm{b}(t,x_1,\delta,\eta,\gamma)|^2\geq c(K)( \gamma^2+\delta^2+\eta^2).
\label{p.e3a}
\end{align}
To show this estimate, by observing that $\bm{b}$ is homogeneous of degree $1$ with respect
to $(\tau,\eta)$ and that $\Xi_1$ is compact,
we only need to prove
\begin{align*}
|\bm{b}(t,x_1,\delta,\eta,\gamma)|^2>0 \qquad \mathrm{on}\ \Xi_1.
\end{align*}
This estimate follows from the similar property for the constant coefficient case,
by taking the perturbation, $\mathring{V}$, small enough in $L^{\infty}(\Omega)$.

Using \eqref{p.e3a} and the G{\aa}rding inequality (Lemma \ref{lem.para1}\;(vi)), we have
\begin{align*}
\RE\big\langle  T^{\gamma}_{\bm{b}^{*}\bm{b}}\psi,\psi\big\rangle
\geq c(K) \|\psi\|_{1,\gamma}^2\qquad\textrm{for all }\gamma\geq \gamma_0,
\end{align*}
where $\gamma_0$ depends only on $K$.
Since $\bm{b}\in {\mathrm{\Gamma}_2^1}$, the operator:
\[
T^{\gamma}_{\bm{b}^{*}\bm{b}}-(T^{\gamma}_{\bm{b}})^{*}T^{\gamma}_{\bm{b}}
=T^{\gamma}_{\bm{b}^{*}\bm{b}}-T^{\gamma}_{\bm{b}^*}T^{\gamma}_{\bm{b}}+\left\{T^{\gamma}_{\bm{b}^*}-(T^{\gamma}_{\bm{b}})^{*}\right\}T^{\gamma}_{\bm{b}}
\]
is of order $\leq 1$. Then
\[
\|\psi\|_{1,\gamma}^2\leq C(K)\left(\|T^{\gamma}_{\bm{b}}\psi\|^2+\|\psi\|_{1,\gamma}\|\psi\|\right)
\leq C(K)\left(\|T^{\gamma}_{\bm{b}}\psi\|^2+\gamma^{-1}\|\psi\|_{1,\gamma}^2\right),
\]
from which we take $\gamma$ sufficiently large to derive
\[
\|\psi\|_{1,\gamma}\leq C(K)\|T^{\gamma}_{\bm{b}}\psi\|.
\]
Since the first and fourth columns of $\bm{B}\in W^{2,\infty}(\mathbb{R}^2)$ vanish,
we apply Lemma \ref{lem.para1} to obtain
\begin{align}\notag
\|\psi\|_{1,\gamma} &\leq C(K)\left(\|T^{\gamma}_{\bm{b}}\psi+T^{\gamma}_{\bm{B}}W|_{x_2=0}\|+\|{W^{\mathrm{nc}}}|_{x_2=0}\|\right)\\
&\leq C(K)\left(\gamma^{-1}\|T^{\gamma}_{\bm{b}}\psi+T^{\gamma}_{\bm{B}}W|_{x_2=0}\|_{1,\gamma}+\|{W^{\mathrm{nc}}}|_{x_2=0}\|\right).  \label{p.e3b}
\end{align}
Combine this estimate with \eqref{p.e2} and let $\gamma$ large enough to deduce
\begin{align} \label{p.e3c}
\|\psi\|_{1,\gamma}\leq C(K)\left(\gamma^{-1}\|\mathcal{B}^{\gamma}(W|_{x_2=0},\psi)\|_{1,\gamma}+\|{W^{\mathrm{nc}}}|_{x_2=0}\|\right).
\end{align}
This last estimate indicates that it only remains to deduce an estimate of ${W^{\mathrm{nc}}}|_{x_2=0}$ in terms of the source terms.

To eliminate $\psi$ in the boundary conditions \eqref{p.p1b}, we define the matrix:
\setlength{\arraycolsep}{2pt}
\begin{align*}
Q(t,x_1,\delta,\eta,\gamma):=\left.\begin{pmatrix}
0&0&\quad  1\\[1mm]
\tau+\mathrm{i}\eta \mathring{v}_1^+&\ \
-\mathrm{i}\eta (\mathring{v}_1^+-\mathring{v}_1^-)&\ \ 0
\end{pmatrix}\right|_{x_2=0} \quad \textrm{for all }(\tau,\eta)\in\Xi_1.
\end{align*}
Then we extend $Q$ as a homogeneous mapping of degree $0$ with respect to $(\tau,\eta)$ on $\Xi$.
It follows that $Q\in{\mathrm{\Gamma}_2^0}$ and $Q\bm{b}\equiv0$.
We define symbol $\bm\beta$ as
\[
\bm\beta(t,x_1,\delta,\eta,\gamma)
:=Q(t,x_1,\delta,\eta,\gamma)\bm{B}(t,x_1)\in{\mathrm{\Gamma}_2^0}
\]
for all $(t,x_1,\delta,\eta,\gamma)\in\mathbb{R}^4\times\mathbb{R}_+$.
After a direct calculation, we find that the first and fourth columns of $\bm\beta$ vanish,
so that we consider $\bm\beta$ as a matrix with only four columns and two rows.
More precisely, for all $(\tau,\eta)\in\Xi_1$, symbol $\bm\beta$ is given by
\setlength{\arraycolsep}{2pt}
\begin{align}\notag
&\bm\beta(t,x_1,\delta,\eta,\gamma)\\[1mm]
&{=\small  \left.\begin{pmatrix}
\sqrt{\mathring{\varrho}_+^2+\mathring{\varsigma}_+^2}&
\sqrt{\mathring{\varrho}_+^2+\mathring{\varsigma}_+^2}&
-\sqrt{\mathring{\varrho}_-^2+\mathring{\varsigma}_-^2}&
-\sqrt{\mathring{\varrho}_-^2+\mathring{\varsigma}_-^2}\\[1.5mm]
\mathring{a}_-({m}_1^++{m}_2^+)&
\mathring{a}_-({m}_1^+-{m}_2^+)&
-\mathring{a}_+({m}_1^-+{m}_2^-)&
-\mathring{a}_+({m}_1^--{m}_2^-)
\end{pmatrix}\right|_{x_2=0},}
\label{beta.bm}
\end{align}
where $\mathring{\varrho}_{\pm}$ and $\mathring{\varsigma}_{\pm}$ are given in \eqref{m.n.ring},
${m}_1^{\pm}$ and ${m}_2^{\pm}$ are given in \eqref{m1m2.bm}, and
\begin{align} \label{a.ring}
\mathring{a}_{\pm}:=\tau+\mathrm{i}\mathring{v}_1^{\pm}(t,x)\eta.
\end{align}
Since $\bm{B}\in {\mathrm{\Gamma}_2^0}$, $\bm{b}\in {\mathrm{\Gamma}_2^1}$, and $Q\bm{b}\equiv0$,  we find from \eqref{p.e2} and Lemma \ref{lem.para1}
that
\begin{align}
\notag
&\|T^{\gamma}_{\bm\beta}{W^{\mathrm{nc}}}|_{x_2=0}\|_{1,\gamma}\\ \notag
& =\|T^{\gamma}_{Q\bm{B}}W|_{x_2=0}-T^{\gamma}_{Q}T^{\gamma}_{\bm{B}}W|_{x_2=0}+T^{\gamma}_{Q}(T^{\gamma}_{\bm{B}}W|_{x_2=0}
+T^{\gamma}_{\bm{b}}\psi)-T^{\gamma}_{Q}T^{\gamma}_{\bm{b}}\psi\|_{1,\gamma}\\ \notag
&\lesssim \|{W^{\mathrm{nc}}}|_{x_2=0}\|
+\|T^{\gamma}_{\bm{B}}W|_{x_2=0}+T^{\gamma}_{\bm{b}}\psi\|_{1,\gamma}+\|T^{\gamma}_{Q}T^{\gamma}_{\bm{b}}\psi-T^{\gamma}_{Q\bm{b}}\psi\|_{1,\gamma}\\
&\lesssim \|{W^{\mathrm{nc}}}|_{x_2=0}\|
+\|\mathcal{B}^{\gamma}(W|_{x_2=0},\psi)\|_{1,\gamma}+\|\psi\|_{1,\gamma}.
\label{p.e3d}
\end{align}
In view of \eqref{p.p1a} and \eqref{p.p1b}, we obtain the following paralinearized problem with
reduced boundary conditions:
\begin{subequations}\label{p.p2}
	\begin{alignat}{3}\label{p.p2.a}
	&T^{\gamma}_{\tau \bm{A}_0^{+}+\mathrm{i}\eta\bm{A}_1^{+}+\bm{C}^{+}}W^{+}+\bm{I}_2\partial_2W^{+}=F^{+}&\qquad &\mathrm{if}\ x_2>0,\\ \label{p.p2.b}
	&T^{\gamma}_{\tau \bm{A}_0^{-}+\mathrm{i}\eta\bm{A}_1^{-}+\bm{C}^{-}}W^{-}+\bm{I}_2\partial_2W^{-}=F^{-}&\qquad&\mathrm{if}\  x_2>0,\\[0.5mm]
	\label{p.p2.c}
	&T^{\gamma}_{\bm\beta}W^{\mathrm{nc}}=G&\qquad&\mathrm{if}\  x_2=0.
	\end{alignat}
\end{subequations}
We can deduce the following proposition for problem \eqref{p.p2}
by using the error estimates \eqref{p.e1a}--\eqref{p.e1b}, \eqref{p.e2}, \eqref{p.e3b}--\eqref{p.e3c},
and \eqref{p.e3d} (see also \cite[Proposition\;5.3]{CS04MR2095445}).

\begin{proposition}\ \label{pro3}
	If there exist constants $K_0>0$ and $\gamma_0\geq 1$ such that solution $W$
	to the paralinearized problem \eqref{p.p2} satisfies
	\begin{align}\label{p.p2.e0}
	\|{W^{\mathrm{nc}}}|_{x_2=0}\|^2\lesssim \gamma^{-3}\VERT F^{\pm}\VERT_{1,\gamma}^2+\gamma^{-2}\|G\|_{1,\gamma}^2,
	\end{align}
	for $K\leq K_0$ and $\gamma\geq \gamma_0$, then Proposition {\rm \ref{pro2}} holds.
\end{proposition}

\subsection{A Reduced Problem}\label{sec.eq}
In order to derive the energy estimate \eqref{p.p2.e0},
we now derive a problem for the noncharacteristic variables $W^{\mathrm{nc}}$ from \eqref{p.p2}.
This is possible since the coefficient matrix $\bm{I}_2=\mathrm{diag}\,(0,1,1)$ has constant rank.
For convenience, we write
\begin{align} \label{B.bf}
\tau\bm{A}_0^{\pm}+\mathrm{i}\eta\bm{A}_1^{\pm}=:(b_{ij}^{\pm})\in {\mathrm{\Gamma}_2^1},
\end{align}
where $\bm{A}_k^{\pm}=({A}_{k,\pm}^{ij})$, $k=0,1$, are defined by \eqref{A.bm}.
In particular, we compute
\begin{align}\label{eff11}
{A}_{1,\pm}^{11}=\mathring{v}_1^{\pm}{A}_{0,\pm}^{11},\qquad {A}_{0,\pm}^{11}
=\mathring{\varGamma}_{\pm}\left\{1-\frac{\epsilon^2(\mathring{\varsigma}_{\pm}\mathring{v}_1^{\pm}
	+\mathring{\varrho}_{\pm}\mathring{v}_2^{\pm})^2}{\mathring{\varsigma}_{\pm}^2+\mathring{\varrho}_{\pm}^2}\right\}
=:\mathbb{F}_1^{\pm}\in\mathbb{R},
\end{align}
from which we obtain
\begin{align}\label{b11}
b_{11}^{\pm}:=\tau{A}_{0,\pm}^{11}+\mathrm{i}\eta{A}_{1,\pm}^{11}=\mathbb{F}_1^{\pm} (\tau+\mathrm{i}\mathring{v}_1^{\pm} {\eta}).
\end{align}
In view of \eqref{m.n.ring}, $\mathbb{F}_1^{\pm} =\widebar{\varGamma}(1-\epsilon^2\bar{v}^2)>0$
when $\big(\mathring{V}^{\pm},\mathring{\Psi}^{\pm}\big)=0$ ({\it cf.}\;\eqref{B.cal}).
We use the continuity of $\mathbb{F}_1^{\pm}$ and take $K$ in \eqref{bas.c2} small enough
to derive that $\mathbb{F}_1^{\pm}>0$ for all $(t,x)\in\widebar{\Omega}$. As a consequence, we have
\begin{align} \label{pole1}
b_{11}^{\pm}=0\qquad\, \textrm{if and only if}\qquad\,
\mathring{a}_{\pm}=\tau+\mathrm{i}\mathring{v}_1^{\pm}\eta=0.
\end{align}
To represent the characteristic variables $W_1^{\pm}$ in terms of $W^{\mathrm{nc}}$,
the singular points $(t,x,\tau,\eta)$ that are given by \eqref{pole1} should be excluded.
We thus introduce two $C^{\infty}$--functions $\chi_{+}$ and $\widetilde{\chi}_+$
defined on $\widebar{\Omega}\times\Xi$ such that
\begin{list}{}{\setlength{\parsep}{\parskip}
		\setlength{\itemsep}{0.1em}
		\setlength{\labelwidth}{2em}
		\setlength{\labelsep}{0.4em}
		\setlength{\leftmargin}{2.5em}
		\setlength{\topsep}{1mm}
	}
	\item[--] Both $\chi_{+}$ and $\widetilde{\chi}_+$ are homogeneous of degree zero with respect to $(\tau,\eta)\in\Xi$;
	\item[--] For all $(t,x,\tau,\eta)\in \widebar{\Omega}\times\Xi_1$,
	\begin{gather}\label{chi1a}
	0\leq \chi_{+}(t,x,\tau,\eta)\leq \widetilde{\chi}_+(t,x,\tau,\eta)\leq 1,\\
	\label{chi1b}
	\widetilde{\chi}_+\equiv 1\quad \mathrm{on}\; \mathrm{supp}\;\chi_+,\qquad
	\mathrm{supp}\;\widetilde{\chi}_+\subset\big\{\mathring{a}_{+}(t,x,\tau,\eta)\ne 0 \big\}.
	\end{gather}
\end{list}

Since $\tau \bm{A}_0^{+}+\mathrm{i}\eta\bm{A}_1^{+}\in {\mathrm{\Gamma}_2^1}$, $\bm{C}^+\in{\mathrm{\Gamma}_1^0}$,
and $\chi_+\in{\mathrm{\Gamma}_k^0}$ for all $k\in\mathbb{N}$, we find from Lemma \ref{lem.para1} (iv)--(v) that
\begin{align*}
T^{\gamma}_{\chi_+}T^{\gamma}_{\tau \bm{A}_0^{+}+\mathrm{i}\eta\bm{A}_1^{+}+\bm{C}^+}
=T^{\gamma}_{\tau \bm{A}_0^{+}+\mathrm{i}\eta\bm{A}_1^{+}}T^{\gamma}_{\chi_+}+
T^{\gamma}_{-\mathrm{i}\left\{\chi_+,\tau \bm{A}_0^{+}+\mathrm{i}\eta\bm{A}_1^{+}\right\}}+T^{\gamma}_{\bm{C}^+}T^{\gamma}_{\chi_+}+\mathcal{R}_{-1},
\end{align*}
where $\{a, b\}$ denotes the Poisson bracket of $a$ and $b$:
\begin{align}\label{Poisson}
\{a, b\}:=\frac{\partial a}{\partial \delta }\frac{\partial b}{\partial t} +\frac{\partial a}{\partial \eta }\frac{\partial b}{\partial x_1}
-\frac{\partial a}{\partial t }\frac{\partial b}{\partial \delta}-\frac{\partial a}{\partial x_1 }\frac{\partial b}{\partial \eta}.
\end{align}
Setting $$\bm{w}^{+}:=T^{\gamma}_{\chi_+}W^+$$ and applying operator $T^{\gamma}_{\chi_+}$ to \eqref{p.p2.a}, we obtain
\begin{align} \label{id.P.1}
T^{\gamma}_{\tau \bm{A}_0^{+}+\mathrm{i}\eta\bm{A}_1^{+}}\bm{w}^++T^{\gamma}_{\bm{C}^+}\bm{w}^++\bm{I}_2\partial_2\bm{w}^+=
T^{\gamma}_{r}W^++T^{\gamma}_{\chi_+}F^++\mathcal{R}_{-1}W^+,
\end{align}
where $r=\mathrm{i}\left\{\chi_+,\tau \bm{A}_0^{+}+\mathrm{i}\eta\bm{A}_1^{+}\right\}+\partial_2\chi_+\bm{I}_2.$
We will employ letter $r_+$ to denote a generic symbol that belongs to ${\mathrm{\Gamma}_1^0}$ and vanishes on
$\{\chi_+\equiv 1\}\cup\{\chi_+\equiv 0\}$.
Since $b_{11}^{+}\ne 0$ on $\mathrm{supp}\;\widetilde{\chi}_+$,
we infer that $\frac{\widetilde{\chi}_+}{b_{11}^+}\in{\mathrm{\Gamma}_2^{-1}}$ and
\[
T^{\gamma}_{\widetilde{\chi}_+/b_{11}^+}T^{\gamma}_{b_{1j}^+}\bm{w}_j^+=T^{\gamma}_{\widetilde{\chi}_+b_{1j}^+/b_{11}^+}\bm{w}_j^+
+T^{\gamma}_{\alpha_{-1}}\bm{w}_j^{+}+\mathcal{R}_{-2}W^+.
\]
Applying operator $T^{\gamma}_{\widetilde{\chi}_+/b_{11}^+}$ to the first equation in \eqref{id.P.1} yields
\begin{align} \notag
&T^{\gamma}_{\widetilde{\chi}_+}\bm{w}_1^++T^{\gamma}_{\widetilde{\chi}_+b_{12}^+/b_{11}^+}\bm{w}_2^++T^{\gamma}_{\widetilde{\chi}_+b_{13}^+/b_{11}^+}\bm{w}_3^+\\
&=\sum_{j=1}^{3}T^{\gamma}_{\alpha_{-1}}\bm{w}_j^{+}+T^{\gamma}_{r_+\alpha_{-1}}W^++T^{\gamma}_{\alpha_{-1}}T^{\gamma}_{\chi_+}F_1^++\mathcal{R}_{-2}W^+.
\label{id.P.2}
\end{align}
By virtue of \eqref{chi1b}, we have the identities: $\widetilde{\chi}_+\chi_+\equiv \chi_+$ and
\[\frac{\partial \widetilde{\chi}_+}{\partial \delta }\frac{\partial \chi_+}{\partial t} +\frac{\partial \widetilde{\chi}_+}{\partial x_1 }\frac{\partial \chi_+}{\partial \eta}\equiv 0,
\]
which imply
\[T^{\gamma}_{\widetilde{\chi}_+}\bm{w}_1^+=
T^{\gamma}_{\widetilde{\chi}_+}T^{\gamma}_{ \chi_+}W_1^+=
T^{\gamma}_{\widetilde{\chi}_+ \chi_+}W_1^++\mathcal{R}_{-2}W_1^{+}
=\bm{w}_1^++\mathcal{R}_{-2}W_1^{+}.
\]
Plug this identity into \eqref{id.P.2} to obtain
\begin{align}
\notag
\bm{w}_1^+=&-T^{\gamma}_{\widetilde{\chi}_+b_{12}^+/b_{11}^+}\bm{w}_2^+-T^{\gamma}_{\widetilde{\chi}_+b_{13}^+/b_{11}^+}\bm{w}_3^+\\
&
+\sum_{j=1}^{3}T^{\gamma}_{\alpha_{-1}}\bm{w}_j^{+} +T^{\gamma}_{r_+\alpha_{-1}}W^+
+T^{\gamma}_{\alpha_{-1}}T^{\gamma}_{\chi_+}F_1^++\mathcal{R}_{-2}W^+.
\label{omega1}
\end{align}
The second equation of \eqref{id.P.1} reads
\[
\sum_{j=1}^{3}T^{\gamma}_{b_{2j}^+}\bm{w}^+_j+\sum_{j=1}^{3}T^{\gamma}_{\alpha_{0}}\bm{w}_j^{+}+\partial_2\bm{w}_2^+=
T_{r_+}^{\gamma}W^++T^{\gamma}_{\chi_+}F_2^++\mathcal{R}_{-1}W^+.
\]
Since $b_{j1}^+\in{\mathrm{\Gamma}_2^1}$ and $\widetilde{\chi}_+b_{1j}^+/b_{11}^+\in{\mathrm{\Gamma}_2^{0}}$,
we then apply operator $T^{\gamma}_{b_{21}^+}$ to expression \eqref{omega1} and obtain
\begin{align*}
T^{\gamma}_{b_{21}^+}\bm{w}_1^+
=&-T^{\gamma}_{\widetilde{\chi}_+b_{21}^+b_{12}^+/b_{11}^+}\bm{w}_2^+-T^{\gamma}_{\widetilde{\chi}_+b_{21}^+b_{13}^+/b_{11}^+}\bm{w}_3^+\\
&+\sum_{j=1}^{3}T^{\gamma}_{\alpha_{0}}\bm{w}_j^{+}
+T^{\gamma}_{r_+}W^++\mathcal{R}_{0}T^{\gamma}_{\chi_+}F^++\mathcal{R}_{-1}W^+.
\end{align*}
Consequently, we have
\begin{align}
\notag \partial_2\bm{w}_2^+=\,&T^{\gamma}_{\mathbb{A}_{\widetilde{\chi}_+}^{11}}\bm{w}_2^+
+T^{\gamma}_{\mathbb{A}_{\widetilde{\chi}_+}^{12}}\bm{w}_3^+
\\ & +\sum_{j=1}^{3}T^{\gamma}_{\alpha_{0}}\bm{w}_j^{+}+T_{r_+}^{\gamma}W^++\mathcal{R}_0T^{\gamma}_{\chi_+}F^++\mathcal{R}_{-1}W^+,
 \label{id.P.3}
\end{align}
where
\begin{align} \notag
\mathbb{A}_{\widetilde{\chi}_+}^{11}=-b_{22}^++\widetilde{\chi}_+b_{21}^+b_{12}^+/b_{11}^+,\qquad
\mathbb{A}_{\widetilde{\chi}_+}^{12}=-b_{23}^++\widetilde{\chi}_+b_{21}^+b_{13}^+/b_{11}^+.
\end{align}
Note that $\bm{w}_1^+$ appears in a zero-th order term in \eqref{id.P.3}.
We thus apply $T^{\gamma}_{\alpha_0}$ to expression \eqref{omega1} and deduce
\[
T^{\gamma}_{\alpha_0}\bm{w}_1^+=T^{\gamma}_{\alpha_0}\bm{w}_2^++T^{\gamma}_{\alpha_0}\bm{w}_3^+
+T^{\gamma}_{r_+\alpha_{-1}}W^++\mathcal{R}_{-1}T^{\gamma}_{\chi_+}F^++\mathcal{R}_{-1}W^+,
\]
which, together with \eqref{id.P.3}, implies the following equation for $\bm{w}_2^+$:
\begin{align}
\notag \partial_2\bm{w}_2^+=\,&T^{\gamma}_{\mathbb{A}_{\widetilde{\chi}_+}^{11}}\bm{w}_2^+
+T^{\gamma}_{\mathbb{A}_{\widetilde{\chi}_+}^{12}}\bm{w}_3^+
\\ \label{id.P.4}
&+\sum_{j=2}^{3}T^{\gamma}_{\alpha_{0}}\bm{w}_j^{+}
+T_{r_+}^{\gamma}W^++\mathcal{R}_0T^{\gamma}_{\chi_+}F^++\mathcal{R}_{-1}W^+.
\end{align}
In this equation, the first and zero-th order terms in $\bm{w}_1^+$ have been eliminated.
Performing a similar computation to the third equation of \eqref{id.P.1},
we obtain the following equations for $\bm{w}_+^{\mathrm{nc}}:=(\bm{w}^+_2,\bm{w}^+_3)^{\mathsf{T}}$:
\begin{align} \label{id.P1a}
\partial_2 \bm{w}_+^{\mathrm{nc}}=
T^{\gamma}_{\mathbb{A}_{\widetilde{\chi}_+}}\bm{w}_+^{\mathrm{nc}}
+T^{\gamma}_{\mathbb{E}^+}\bm{w}_+^{\mathrm{nc}}+
T_{r_+}^{\gamma}W^++\mathcal{R}_0T^{\gamma}_{\chi_+}F^++\mathcal{R}_{-1}W^+,
\end{align}
where $\mathbb{E}^+\in{\mathrm{\Gamma}_1^0}$, symbol $r_+\in{\mathrm{\Gamma}_1^0}$ vanishes
on region $\{\chi_+\equiv1\}\cup\{\chi_+\equiv0\}$, and
\begin{align} \notag
\mathbb{A}_{\widetilde{\chi}_+}=\begin{pmatrix}
\mathbb{A}_{\widetilde{\chi}_+}^{11}&\mathbb{A}_{\widetilde{\chi}_+}^{12}\\[1.5mm]
\mathbb{A}_{\widetilde{\chi}_+}^{21}&\mathbb{A}_{\widetilde{\chi}_+}^{22}
\end{pmatrix}\in {\mathrm{\Gamma}_2^1},\qquad
\mathbb{A}_{\widetilde{\chi}_+}^{ij}=-b_{i+1,j+1}^++\widetilde{\chi}_+b_{i+1,1}^+b_{1,j+1}^+/b_{11}^+.
\end{align}

Let us define $\chi_-$ and $\widetilde{\chi}_-$ as $\chi_+$ and $\widetilde{\chi}_+$ by changing index ``$+$'' into ``$-$''.
We set $\bm{w}^-:=T^{\gamma}_{\chi_-}W^-$ and employ a similar analysis to find that
$\bm{w}_-^{\mathrm{nc}}:=(\bm{w}_2^-,\bm{w}_3^-)^{\mathsf{T}}$ satisfies
the same system as \eqref{id.P1a} with index ``$+$'' replaced by ``$-$''.
Applying the rule of symbolic calculus (Lemma \ref{lem.para1}(iv)) to \eqref{p.p2.c}
yields the boundary condition for
$\bm{w}^{\mathrm{nc}}:=(\bm{w}_2^+,\bm{w}_3^+,\bm{w}_2^-,\bm{w}_3^-)^{\mathsf{T}}$:
\begin{align*}
T^{\gamma}_{\bm\beta}\bm{w}^{\mathrm{nc}}|_{x_2=0}=
G+\mathcal{R}_{-1}W^{\mathrm{nc}}.
\end{align*}
We combine this last relation with the systems for $\bm{w}_{\pm}^{\mathrm{nc}}$
to obtain the reduced problem:
\begin{align} \label{id.P1}
\left\{\begin{aligned}
&\partial_2 \bm{w}^{\mathrm{nc}}
=T^{\gamma}_{\mathbb{A}_r}\bm{w}^{\mathrm{nc}}+T^{\gamma}_{\mathbb{E}}\bm{w}^{\mathrm{nc}}
+T_r^{\gamma}W+\mathcal{R}_0T^{\gamma}_{\chi}F+\mathcal{R}_{-1}W & \mathrm{if}\ x_2>0,\\
&T^{\gamma}_{\bm\beta}\bm{w}^{\mathrm{nc}}=
G+\mathcal{R}_{-1}W^{\mathrm{nc}} & \mathrm{if}\ x_2=0,
\end{aligned}\right.
\end{align}
where $\bm\beta$ is given by \eqref{beta.bm} for $(\tau,\eta)\in\Xi_1$.
The symbol matrix $\mathbb{A}_r\in{\mathrm{\Gamma}_2^1}$ is given by
\setlength{\arraycolsep}{2pt}
\begin{align}
\left\{\begin{aligned}
&\mathbb{A}_r=\begin{pmatrix}
\mathbb{A}_{\widetilde{\chi}_+}&0\\[0.5mm] 0&\mathbb{A}_{\widetilde{\chi}_-}
\end{pmatrix},\quad\,
\mathbb{A}_{\widetilde{\chi}_{\pm}}=\big(\mathbb{A}_{\widetilde{\chi}_{\pm}}^{ij}\big),
\\
&\textrm{with}\ \  \mathbb{A}_{\widetilde{\chi}_{\pm}}^{ij}=-b_{i+1,j+1}^{\pm}+\widetilde{\chi}_{\pm}b_{i+1,1}^{\pm}b_{1,j+1}^{\pm}/b_{11}^{\pm}.
\end{aligned}\right.
\label{A.bb.r}
\end{align}
Matrices $\mathbb{E}$ and $r$ both belong to ${\mathrm{\Gamma}_1^0}$ and
have the same block diagonal structure as $\mathbb{A}_r$.
Moreover, symbol $r$ vanishes on
region $\{\chi_{+}=\chi_{-}\equiv 1\}\cup \{\chi_{+}=\chi_{-}\equiv 0\}$.
\subsection{Microlocalization}  \label{sec.micro}
We now construct the degenerate Kreiss' symmetrizers that are microlocal ({\it i.e.}  local in the frequency space)
in order to derive our energy estimate.
The whole space $\widebar{\Omega}\times\Xi$ will be divided into three disjoint parts according to the poles
of the ``non-cutoff'' symbol $\mathbb{A}$ and the zeros of the associated Lopatinski\u{\i} determinant, where
\begin{align} \label{A.bb}
\mathbb{A}=\begin{pmatrix}
\mathbb{A}^+&0\\[0.5mm] 0&\mathbb{A}^-
\end{pmatrix},\quad\,
\mathbb{A}^{\pm}=\big(a^{\pm}_{ij}\big),\quad\,
a^{\pm}_{ij}=-b_{i+1,j+1}^{\pm}+b_{i+1,1}^{\pm}b_{1,j+1}^{\pm}/b_{11}^{\pm}.
\end{align}
Notice that $\mathbb{A}_{\widetilde{\chi}_{\pm}}=\mathbb{A}^{\pm}$ in region $\{\widetilde{\chi}_{\pm}\equiv1\}$.
In light of \eqref{pole1}, we obtain that the poles of $\mathbb{A}$ belong to the set $\Upsilon_p:=\Upsilon_p^+\cup\Upsilon_p^-$, with
\begin{align}\notag
 \Upsilon_p^{\pm}:=\big\{(t,x,\tau,\eta)\in\widebar{\Omega}\times \Xi:
\tau=-\mathrm{i}\eta\mathring{v}_1^{\pm}(t,x,\tau,\eta)  \big\}.
\end{align}

For the eigenvalues and the stable subspace of $\mathbb{A}(t,x,\tau,\eta)$,
we have the following lemma.

\begin{lemma}\ \label{lem.eig2}
	Assume that $\big(\mathring{V},\nabla \mathring{\Psi}\big)$ is sufficiently small in $W^{2,\infty}(\Omega)$.
	\begin{list}{}{\setlength{\parsep}{\parskip}
			\setlength{\itemsep}{0.1em}
			\setlength{\labelwidth}{2em}
			\setlength{\labelsep}{0.4em}
			\setlength{\leftmargin}{2.2em}
			\setlength{\topsep}{1mm}
		}
		\item[\emph{(a)}] If $(\tau,\eta)\in\Xi_1$ with $\RE \tau>0$,
		then the eigenvalues of $\mathbb{A}^{\pm}(t,x,\tau,\eta)$ are roots $\omega$ of
		\begin{align}
		\Big(\omega-\frac{a_{11}^{\pm}+a_{22}^{\pm}}{2}\Big)^2
		=\big(\mathring{C}_0^{\pm}\big)^2\left\{\big(\mathring{C}_1^{\pm}\big)^2\big(\tau\pm\mathrm{i}\mathring{C}_2^{\pm}\eta\big)^2+\eta^2\right\},
		\label{eig2a}
		\end{align}
		where $\mathring{C}_j^{\pm}, j=0,1,2$, are positive smooth functions of $(\mathring{V}^{\pm},\nabla\mathring{\Psi}^{\pm})$ such that
		$\mathring{C}_j^{\pm}=\widebar{C}_j$ when $\big(\mathring{V}^{\pm},\mathring{\Psi}^{\pm}\big)=0$, with $\widebar{C}_j$ given by \eqref{C.bar}.
		Moreover, $\mathbb{A}^{\pm}$ has a unique eigenvalue $\omega_{\pm}$ {(}resp.\;$\omega_{\pm}'${)}
		of negative {\rm(}resp.\;positive{)} real part.
		
		\item[\emph{(b)}] If $(\tau,\eta)\in\Xi_1$ with $\RE \tau>0$, then the stable subspace $\mathcal{E}^-(t,x,\tau,\eta)$ of $\mathbb{A}(t,x,\tau,\eta)$
		has dimension two and is spanned by
		\begin{align}\label{E.eig2}
		{\small \left\{\begin{aligned}
		E_+(t,x,\tau,\eta):=\left(-(\tau+\mathrm{i}\mathring{v}_1^+\eta)a^+_{12},(\tau+\mathrm{i}\mathring{v}_1^+\eta)(a_{11}^+-\omega_+),0,0\right)^{\mathsf{T}},\\[0.5mm]
		E_-(t,x,\tau,\eta):=\left(0,0,(\tau+\mathrm{i}\mathring{v}_1^-\eta)(a_{22}^--\omega_-),-(\tau+\mathrm{i}\mathring{v}_1^-\eta)a_{21}^-\right)^{\mathsf{T}}.
		\end{aligned}\right.}
		\end{align}
		
		\item[\emph{(c)}] Both $\omega_{+}$ and $\omega_{-}$ admit a continuous extension to any point $(\tau,\eta)\in\Xi_{1}$ with $\RE\tau=0$.
		If $(\tau,\eta)\in\Xi_{1}$ with $\tau=\mathrm{i}\delta\in\mathrm{i}\mathbb{R}$, then
		\begin{align}\notag
		&\widetilde{\omega}_{\pm}(t,x,\tau,\eta):=\omega_{\pm}(t,x,\tau,\eta)-\frac{a_{11}^{\pm}+a_{22}^{\pm}}{2}(t,x,\tau,\eta)\\[1mm]
		&={\small  \left\{ \begin{aligned}
		 &-\mathring{C}^{\pm}_0\sqrt{\eta^2-(\mathring{C}^{\pm}_1)^2(\delta\pm\mathring{C}^{\pm}_2\eta)^2}
		 \quad\ \, \mathrm{if}\ \eta^2\geq (\mathring{C}^{\pm}_1)^2(\delta\pm\mathring{C}^{\pm}_2\eta)^2,\\[0.5mm]
		 &-\mathrm{i}\,\mathrm{sgn}(\delta\pm\mathring{C}^{\pm}_2\eta)\mathring{C}^{\pm}_0\sqrt{(\mathring{C}^{\pm}_1)^2(\delta\pm\mathring{C}^{\pm}_2\eta)^2-\eta^2}
		\qquad\quad \textrm{elsewise}.
		\end{aligned} \right.}
		\label{eig2b}
		\end{align}
		
		\item[\emph{(d)}]
		Both $E_+(t,x,\tau,\eta)$ and $E_-(t,x,\tau,\eta)$ can be extended continuously to any point $(\tau,\eta)\in\Xi_{1}$ with $\RE\tau=0$.
		These two vectors are linearly independent on the whole hemisphere $\Xi_1$.
		
		\item[\emph{(e)}]
		If $(t,x,\tau,\eta)\notin \Upsilon_{nd}$, where $\Upsilon_{nd}$ is  given by
		\begin{align} \label{Upsilon.nd}
		\Upsilon_{nd}
		:=\left\{\tau\in\big\{\mathrm{i}(- \mathring{C}^{+}_2\pm (\mathring{C}^{+}_1)^{-1})\eta,\,\mathrm{i}(- \mathring{C}^{-}_2\pm (\mathring{C}^{-}_1)^{-1})\eta\big\}\right\},
		\end{align}
		then matrix $\mathbb{A}(t,x,\tau,\eta)$ is diagonalizable.
	\end{list}
\end{lemma}
\begin{proof}\
	We just need to deduce that relations \eqref{eig2a} hold and that $a_{11}^{\pm}+a_{22}^{\pm}$ are well-defined
	for any point $(\tau,\eta)\in\Xi_1$, since the other assertions can be proved similarly
	to the proof of Lemma \ref{lem.eig1}.
	
	By definition, we know that the eigenvalues of $\mathbb{A}^{\pm}$ are roots $\omega$ of
	\begin{align} \label{eig2a.1}
	\omega^2-(a_{11}^{\pm}+a_{22}^{\pm})\omega+a_{11}^{\pm}a_{22}^{\pm}-a_{12}^{\pm}a_{21}^{\pm}=0,
	\end{align}
	from which we have
	\begin{subequations}
		\begin{alignat}{1} \label{eig2a.2}
		&\widetilde{\omega}_+^2=\Big(\frac{a_{11}^{+}-a_{22}^{+}}{2}+a_{12}^{+}\Big)\Big(\frac{a_{11}^{+}-a_{22}^{+}}{2}-a_{12}^{+}\Big)+a_{12}^{+}(a_{12}^{+}+a_{21}^{+}),\\ \label{eig2a.3}
		&\widetilde{\omega}_-^2=\Big(\frac{a_{22}^{-}-a_{11}^{-}}{2}+a_{21}^{-}\Big)\Big(\frac{a_{22}^{-}-a_{11}^{-}}{2}-a_{21}^{-}\Big)+a_{21}^{-}(a_{12}^{-}+a_{21}^{-}).
		\end{alignat}
	\end{subequations}
	We now deduce the expressions for $a_{11}^{\pm}-a_{22}^{\pm}-2a_{12}^{\pm}$ and $a_{12}^{\pm}+a_{21}^{\pm}$.
	Recall that $a_{ij}^\pm$ and $b_{ij}^{\pm}$ are given by \eqref{A.bb} and \eqref{B.bf}, respectively.
	Entries ${A}^{11}_{0,\pm}$ and ${A}^{11}_{1,\pm}$ are given by \eqref{eff11}.
	For notational simplicity, we ignore indices ``$\pm$'' and ``$\mathring{\ \;}$'' in the following expressions.
	We calculate coefficients $\bm{A}^{\pm}_j$ defined in \eqref{A.bm} by using the computer algebra system ``Maxima''
	to obtain the relations:
	\begin{subequations}
		\begin{alignat}{1}\label{eff.id1}
		& {A}_1^{23}=v_1{A}_0^{23},\quad
		{A}_1^{32}=v_1{A}_0^{32},\quad {A}_1^{22}-{A}_1^{33}=v_1({A}_0^{22}-{A}_0^{33}),\\ \label{eff.id2}
		&{A}_1^{12}-{A}_1^{13}=v_1({A}_0^{12}-{A}_0^{13}),\quad
		{A}_1^{21}+{A}_1^{31}=v_1({A}_0^{21}+{A}_0^{31}),\\
		\notag
		&{A}_1^{21}({A}_0^{12}-{A}_0^{13})-{A}_{1}^{13}({A}_0^{21}+{A}_0^{31})\\
\label{eff.id3}&\quad 		=
		v_1\left\{{A}_0^{21}({A}_0^{12}-{A}_0^{13})-{A}_0^{13}({A}_0^{21}+{A}_0^{31})\right\},\\ \notag
		&{A}_1^{13}({A}_0^{21}+{A}_0^{31})+{A}_{1}^{31}({A}_0^{12}-{A}_0^{13})\\
		\label{eff.id4}&\quad =
		v_1\left\{{A}_0^{13}({A}_0^{21}+{A}_0^{31})+{A}_0^{31}({A}_0^{12}-{A}_0^{13})\right\}.
		\end{alignat}
	\end{subequations}
	Then it follows from \eqref{eff.id1} that
	\begin{align*}
	-b_{22}+b_{33}+2b_{23}=(-{A}_0^{22}+{A}_0^{33}+2{A}_0^{23}) (\tau+\mathrm{i}v_1\eta).
	\end{align*}
	By virtue of \eqref{eff.id2}--\eqref{eff.id3}, we obtain
	\begin{align*}
	b_{21}b_{12}-b_{31}b_{13}-2b_{21}b_{13}&=b_{21}(b_{12}-b_{13})-b_{13}(b_{21}+b_{31})\\
	&=\left\{b_{21}({A}_0^{12}-{A}_0^{13}) -b_{13} ({A}_0^{21}+{A}_0^{31})\right\} (\tau+\mathrm{i}v_1\eta)\\
	&=\{{A}_0^{21}({A}_0^{12}-{A}_0^{13})-{A}_0^{13}({A}_0^{21}+{A}_0^{31})\} (\tau+\mathrm{i}v_1\eta)^2.
	\end{align*}
	Then
	\begin{align} \label{eff.i1}
	a_{11}-a_{22}-2a_{12}=\mathbb{F}_2 (\tau+\mathrm{i}v_1\eta),
	\end{align}
	where
	\begin{align} \notag
	\mathbb{F}_2 &={\textrm{\small  $({A}_0^{11})^{-1}\left\{{A}_0^{11}(-{A}_0^{22}+{A}_0^{33}+2{A}_0^{23})
	+{A}_0^{21}({A}_0^{12}-{A}_0^{13})-{A}_0^{13}({A}_0^{21}+{A}_0^{31})  \right\}$}  }\\
	&=\frac{2\p_2\Phi
		\left\{\varGamma (\varsigma^2+\varrho^2)^{1/2}(\epsilon^2|v|^2-1)+\epsilon^2 c (\varsigma v_2-\varrho v_1)\right\}}{c\left( \epsilon^2 (\varrho v_2
		+\varsigma v_1)^2-\varsigma^2-\varrho^2 \right)}.
	\label{F2}
	\end{align}
	In particular, $\mathbb{F}_2^{\pm}=\pm 2\widebar{\varGamma}/\bar{c}\ne 0$
	when the perturbation $\big(\mathring{V}^{\pm},\mathring{\Psi}^{\pm}\big)$ vanishes.
	As a consequence, $\mathbb{F}_2^{\pm}$ never vanish by taking $K$ in \eqref{bas.c2} small enough.
	Using \eqref{eff.id1}--\eqref{eff.id2} and \eqref{eff.id4},
	we can deduce from a similar calculation  that
	\begin{align}\label{eff.i2}
	a_{12}+a_{21}=\mathbb{F}_3 (\tau+\mathrm{i}v_1\eta),
	\end{align}
	where
	\begin{align} \notag
	\mathbb{F}_3 &=({A}_0^{11})^{-1}\left\{{A}_0^{11}(-{A}_0^{23}-{A}_0^{32})
	+{A}_0^{21}{A}_0^{13}+{A}_0^{31}{A}_0^{12}  \right\} \\
	&=-\frac{2\p_2\Phi \epsilon^2\left( \varsigma v_2-\varrho v_1  \right)}{\epsilon^2 (\varrho v_2+\varsigma v_1)^2-\varsigma^2-\varrho^2}. \label{F3}
	\end{align}
	Relations \eqref{eig2a} follows by plugging \eqref{eff.i1} and \eqref{eff.i2} into \eqref{eig2a.2}--\eqref{eig2a.3}.
	
	We now show that $a_{11}^{\pm}+a_{22}^{\pm}$ are well-defined. Use \eqref{eff.id2} to derive
	\begin{align*}
	b_{21}b_{12}+b_{31}b_{13}
	&= b_{12}(b_{21}+b_{31})+b_{31}(b_{13}-b_{12})\\
	&= \left\{b_{12} ({A}_0^{21}+{A}_0^{31})+b_{31}({A}_0^{13}-{A}_0^{12})  \right\}(\tau+\mathrm{i}v_1\eta),
	\end{align*}
	which implies
	\begin{align}  \label{eff.i3}
	a_{11} +a_{22}=-b_{22}-b_{33}+\frac{b_{21}b_{12}+b_{31}b_{13} }{\mathbb{F}_1(\tau+\mathrm{i}{v}_1\eta)}
	=\mathbb{F}_4\tau+\mathrm{i}\eta\mathbb{F}_5,
	\end{align}
	where $\mathbb{F}_4$ and $\mathbb{F}_5$ are some smooth functions of $(\mathring{V},\nabla\mathring{\Psi})$
	that vanish when $\big(\mathring{V},\mathring{\Psi}\big)=0$
	({\it cf.}\;\eqref{A.cal} with $a_{11}^{\pm}=\pm \mu_{\pm}$ and $a_{22}^{\pm}=\mp\mu_{\pm}$).
	The proof of the lemma can be completed by using the fact that $\bm{A}_0^{\pm}$ and $\bm{A}_1^{\pm}$ are smooth
	with respect to $(\mathring{V}^{\pm},\nabla\mathring{\Psi}^{\pm})$.
\qed\end{proof}

As in the constant coefficient case, we define the Lopatinski\u{\i} determinant associated
with $\mathbb{A}$ and $\bm\beta$ as
\begin{align}\label{Lopa2a}
\Delta(t,x_1,\tau,\eta):=\det\left[\bm\beta(t,x_1,\tau,\eta)\left(E_+(t,x_1,0,\tau,\eta)\;\; E_-(t,x_1,0,\tau,\eta)\right)\right],
\end{align}
where $\bm\beta$ and $E_{\pm}$ are given by \eqref{beta.bm} and \eqref{E.eig2}, respectively.
For the zeros of $\Delta(t,x_1,\tau,\eta)$, we have the following lemma.

\begin{lemma}\ \label{lem.Lopa2}
	Assume that $\big(\mathring{V},\nabla\mathring{\Psi}\big)$ is sufficiently small in $W^{2,\infty}(\Omega)$.
	Then
	\begin{align*}
	\Delta(t,x_1,\tau,\eta)=0\qquad \textrm{if and only if}\qquad (t,x_1,\tau,\eta)\in\Upsilon_c,
	\end{align*}
	where $\Upsilon_c:=\Upsilon_{c}^{-1}\cup\Upsilon_{c}^0\cup \Upsilon_{c}^1$ is called the critical set with
		\begin{align}
	\notag  \Upsilon_{c}^q:=\left\{(t,x_1,\tau,\eta)\in\mathbb{R}^2\times\Xi: \tau=\mathrm{i}\eta\mathring{z}_q(t,x_1)\right\},
	\end{align}
	where $\mathring{z}_0$ and $\mathring{z}_{\pm1}$ are real-valued functions
	of $({\mathring{V}^{\pm}}|_{x_2=0},\nabla\mathring{\varphi})$ satisfying
	\[
	\Upsilon_c\cap \left((\Upsilon_{nd}\cup\Upsilon_p)\cap\{x_2=0\}\right)=\varnothing.
	\]
	Moreover, each of these roots is simple in the sense that, if $q\in\{0,\pm1\}$,
	then there exist a neighborhood $\mathscr{V}$ of $(\mathrm{i}\mathring{z}_{q}\eta,\eta)$ in $\Xi_1$
	and a $C^{\infty}$--function $h_q$ defined on $\mathbb{R}^2\times \mathscr{V}$ such that
	\begin{align}\label{factor.v}
	\Delta(t,x_1,\tau,\eta)=(\tau-\mathrm{i}\mathring{z}_{q}\eta)h_q(t,x_1,\tau,\eta),\quad h_q(t,x_1,\tau,\eta)\ne 0
	\end{align}
	for all $(\tau,\eta)\in\mathscr{V}$.
\end{lemma}

\begin{proof}\
	Thanks to \eqref{beta.bm} and \eqref{E.eig2}, we obtain that, for $(\tau,\eta)\in\Xi_1$,
	\setlength{\arraycolsep}{3pt}
	\begin{align}
	\bm\beta(t,x_1,\tau,\eta)(E_+(t,x_1,0,\tau,\eta)\;\; E_-(t,x_1,0,\tau,\eta))
	=\left.\begin{pmatrix}
	\mathring{\zeta}_1&\mathring{\zeta}_2\\   \mathring{\zeta}_3 &\mathring{\zeta}_4
	\end{pmatrix}\right|_{x_2=0},
	\label{Lopa2b}
	\end{align}
	where
	\begin{align*}
	&\mathring{\zeta}_1:=\mathring{a}_+\sqrt{\mathring{\varrho}_+^2+\mathring{\varsigma}_+^2}(a_{11}^+-a_{12}^+-\omega_+),\ \
	\mathring{\zeta}_2:=-\mathring{a}_-\sqrt{\mathring{\varrho}_-^2+\mathring{\varsigma}_-^2}(a_{22}^--a_{21}^--\omega_-),\\
	&\mathring{\zeta}_3:=\mathring{a}_+\mathring{a}_- \left\{-a_{12}^+({m}_1^++{m}_2^+)+(a_{11}^+-\omega_+)({m}_1^+-{m}_2^+)  \right\} ,   \\
	&\mathring{\zeta}_4:=\mathring{a}_+\mathring{a}_- \left\{-(a_{22}^--\omega_-)({m}_1^-+{m}_2^-)+a_{21}^-({m}_1^--{m}_2^-)  \right\}.
	\end{align*}
	Recall that $\mathring{a}_{\pm}$ and ${m}_j^{\pm}$ are defined by \eqref{a.ring} and \eqref{m1m2.bm}, respectively.
	
	Thanks to \eqref{eff.i1}, we have
	\begin{align}
	\notag \mathring{\zeta}_1&=\mathring{a}_+\sqrt{\mathring{\varrho}_+^2+\mathring{\varsigma}_+^2}\Big(\frac{a_{11}^+-a_{22}^+-2a_{12}^+  }{2}-\widetilde{\omega}_+\Big)
	\\ \label{zeta1.ring}&=\mathring{a}_+\sqrt{\mathring{\varrho}_+^2+\mathring{\varsigma}_+^2}\Big(\frac{\mathbb{F}_2^+ \mathring{a}_+}{2}-\widetilde{\omega}_+\Big).
	\end{align}
	Using \eqref{eff.i1} and \eqref{eff.i2} yields
	\begin{align}
	\mathring{\zeta}_2=\mathring{a}_-\sqrt{\mathring{\varrho}_-^2+\mathring{\varsigma}_-^2}\Big(\frac{\mathbb{F}_2^-+2\mathbb{F}_3^-}{2}\mathring{a}_-+\widetilde{\omega}_-\Big).
	\label{zeta2.ring}
	\end{align}
	It follows from  \eqref{eig2a.2}  that
	\begin{align*}
	\frac{a_{11}^+-a_{22}^++2a_{12}^+}{2}=\frac{2\widetilde{\omega}_+^2-2a_{12}^+(a_{12}^++a_{21}^+) }{a_{11}^+-a_{22}^+-2a_{12}^+}.
	\end{align*}
	By virtue of this last identity, we obtain
	\begin{align*}
	\mathring{\zeta}_3&=\mathring{a}_+\mathring{a}_- \Big\{\Big(\widetilde{\omega}_+-\frac{a_{11}^+-a_{22}^++2a_{12}^+}{2} \Big)({m}_2^+-{m}_1^+) -2a_{12}^+{m}_1^+ \Big\} \\
	&=\mathring{a}_+\mathring{a}_-  ({m}_2^+-{m}_1^+) \widetilde{\omega}_+\Big(1-\frac{2\widetilde{\omega}_+}{a_{11}^+-a_{22}^+-2a_{12}^+} \Big)   \\
	&\quad +\frac{2\mathring{a}_+\mathring{a}_- a_{12}^+}{a_{11}^+-a_{22}^+-2a_{12}^+}
	\left\{({m}_2^+-{m}_1^+)(a_{12}^++a_{21}^+)-{m}_1^+(a_{11}^+-a_{22}^+-2a_{12}^+)\right\}.
	\end{align*}
	Use \eqref{eff.i1} and \eqref{eff.i2} to deduce
	\begin{align*}
	&({m}_2^+-{m}_1^+)(a_{12}^++a_{21}^+)-{m}_1^+(a_{11}^+-a_{22}^+-2a_{12}^+)
	\\&\quad
	=\mathring{a}_+\left({m}_2^+\mathbb{F}_3^+-{m}_1^+(\mathbb{F}_2^++\mathbb{F}_3^+)  \right)=0,
	\end{align*}
	which, combined with \eqref{eff.i1}, yields
	\begin{align} \label{zeta3.ring}
	\mathring{\zeta}_3=\mathring{a}_+\mathring{a}_-  ({m}_2^+-{m}_1^+) \widetilde{\omega}_+\Big(1-\frac{2\widetilde{\omega}_+}{\mathring{a}_+ \mathbb{F}_2^+} \Big).
	\end{align}
	Similar to the derivation of \eqref{zeta3.ring}, we can infer from \eqref{eig2a.3}, \eqref{eff.i1}, and \eqref{eff.i2} that
	\begin{align} \label{zeta4.ring}
	\mathring{\zeta}_4=\mathring{a}_+\mathring{a}_-  ({m}_2^-+{m}_1^-) \widetilde{\omega}_-
	\Big(1+\frac{2\widetilde{\omega}_-}{\mathring{a}_-( \mathbb{F}_2^-+2\mathbb{F}_3^-)} \Big).
	\end{align}
	Therefore,  we find that
	$\Delta={\Delta_1\Delta_2 \Delta_3}|_{x_2=0}$,
	where
	\begin{align*}
	&\Delta_1:=\frac{\mathbb{F}_2^+ \mathring{a}_+}{2}-\widetilde{\omega}_+,\qquad
	\Delta_2:=\frac{\mathbb{F}_2^-+2\mathbb{F}_3^-}{2}\mathring{a}_-+\widetilde{\omega}_-,
	\\
	&\Delta_3:=\det\begin{pmatrix}
	\mathring{a}_+\sqrt{\mathring{\varrho}_+^2+\mathring{\varsigma}_+^2}&\mathring{a}_-\sqrt{\mathring{\varrho}_-^2+\mathring{\varsigma}_-^2}\\
	\dfrac{2({m}_2^+-{m}_1^+)}{\mathbb{F}_2^+}\mathring{a}_-  \widetilde{\omega}_+
	&\dfrac{2 ({m}_2^-+{m}_1^-)}{\mathbb{F}_2^-+2\mathbb{F}_3^-} \mathring{a}_+\widetilde{\omega}_-
	\end{pmatrix}.
	\end{align*}
	If $\big(\mathring{V}^{\pm},\mathring{\Psi}^{\pm}\big)=0$, then $\mathbb{F}_2^{\pm}=\pm 2\widebar{\varGamma}/\bar{c}$, $\mathbb{F}_3^{\pm}=0$,
	${m}_1^{\pm}=0$, and ${m}_2^{\pm}=1/(\widebar{\varGamma}\bar{c}\bar{h})$.
	Thus,
	\begin{align*}
	&\Delta_1|_{(\mathring{V},\mathring{\Psi})
		=0}=\bar{c}^{-1}\widebar{\varGamma}(\tau+\mathrm{i}\bar{v}\eta)-\omega_+,\quad
	\Delta_2|_{(\mathring{V},\mathring{\Psi})=0}=-\bar{c}^{-1}\widebar{\varGamma}(\tau-\mathrm{i}\bar{v}\eta)+\omega_-,\\
	&\Delta_3|_{(\mathring{V},\mathring{\Psi})=0}
	=-\frac{1}{\widebar{\varGamma}^2\bar{h}}\left\{\omega_-(\tau+\mathrm{i}\bar{v}\eta)^2+\omega_+(\tau-\mathrm{i}\bar{v}\eta)^2\right\}.
	\end{align*}
	Recalling the proof of Lemma \ref{lem.Lop} and using the continuity of $\Delta_{k}$ with respect to $(\mathring{V},\nabla\mathring{\Psi})$,
	we find that, if perturbation $\big(\mathring{V},\nabla\mathring{\Psi}\big)$ is suitably small in $W^{2,\infty}(\Omega)$,
	then $\Delta_1$ and $\Delta_2$ never vanish on $\mathbb{R}^2\times\Xi$, and $\Delta_3(t,x_1,\tau,\eta)\ne 0$ for $\eta=0$.
	Consequently, $\Delta(t,x_1,\tau,\eta)=0$ if and only if $\Delta_3(t,x_1,\tau,\eta)=0$ and $\eta\ne 0$.
	
	Let $\eta\ne 0$. Setting $z:=\tau/(\mathrm{i}\eta)$, we obtain that
	$
	\Delta_3/(\mathrm{i}\eta)^3=\mathring{Q}_1(z)+\mathring{Q}_2(z),
	$
	where
	\begin{align*}
	&\mathring{Q}_1(z):=
	\dfrac{2 ({m}_2^-+{m}_1^-)}{\mathbb{F}_2^-+2\mathbb{F}_3^-}\sqrt{\mathring{\varrho}_+^2
		+\mathring{\varsigma}_+^2} \frac{\mathring{a}_+^2\widetilde{\omega}_-}{(\mathrm{i}\eta)^3},\\
	&\mathring{Q}_2(z):=\dfrac{2({m}_1^+-{m}_2^+)}{\mathbb{F}_2^+}\sqrt{\mathring{\varrho}_-^2
		+\mathring{\varsigma}_-^2}
	\frac{\mathring{a}_-^2 \widetilde{\omega}_+}{(\mathrm{i}\eta)^3}.
	\end{align*}
	As in the proof of Lemma \ref{lem.Lop}, we define
	$$
	\mathring{P}(z):=\bar{c}^2\bar{h}^2\widebar{\varGamma}^2\big(\mathring{Q}_1(z)^2-\mathring{Q}_2(z)^2\big).
	$$
	When $\big(\mathring{V},\mathring{\Psi}\big)=0$,  $\mathring{P}(z)$ is exactly a  polynomial $P(z)$ of degree $5$, given by \eqref{P.def}.
	As a consequence, if $K$ in \eqref{bas.c2} is suitably small, then $\mathring{P}(z)$ is a polynomial function with degree $5$ or $6$,
	and there are functions $\mathring{z}_{k}$, $k\in\{0,\pm1,\pm 2,\pm3\}$, of $\big(\mathring{V},\mathring{\Psi}\big)$ such that
	\begin{align*}
	\mathring{P}(z)=(\mathring{z}_{3}z+1)\mathring{P}_1(z),\qquad
	\mathring{P}_1(z)=\mathring{z}_{-3}\prod_{k\in\{0,\pm 1,\pm 2\}}(z-\mathring{z}_k),
	\end{align*}
	where $\mathring{z}_3$ and $\mathring{P}_1(z)$ satisfy that
	\begin{align*}
	\mathring{z}_3=0, \quad \mathring{P}_1(z)=P(z) \qquad\,\, \mathrm{when}\;  \big(\mathring{V},\mathring{\Psi}\big)=0.
	\end{align*}
	
	Under condition \eqref{H2}, we can compute that
	the discriminants of $\frac{\mathrm{d}^jP(z)}{\mathrm{d}z^j}$, $j\in\{0,1,2,3\}$, are all positive.
	Since the discriminant for a polynomial is continuous with respect to the coefficients of the polynomial,
	we take $K$ suitably small to conclude that the discriminants
	of $\frac{\mathrm{d}^j\mathring{P}_1(z)}{\mathrm{d}z^j}$, $j\in\{0,1,2,3\}$, are all positive.
	Consequently, roots $\mathring{z}_k$, $k\in\{0,\pm1,\pm 2\}$, of $\mathring{P}_1(z)$ are real and distinct.
	
	Noting that the coefficients of $\mathring{P}(z)$ are all real,
	we obtain that $\mathring{z}_{\pm3}$ are both real.
	Since $\mathring{z}_k$, $k\in\{0,\pm1,\pm 2\}$, are all different,
	we infer that $\mathring{z}_k$, $k\in\{0,\pm1,\pm 2,\pm3\}$,
	can be expressed as continuous functions of the coefficients of $\mathring{P}(z)$.
	Choosing $K$ in \eqref{bas.c2} sufficiently small, we see that $\mathring{z}_{-3}$ is always nonzero,
	$\mathring{z}_3$ and
	$\mathring{z}_{0}$ are in a small neighborhood of $0$,
	and $\mathring{z}_k$, $k\in\{\pm1,\pm2\}$, are respectively
	in a small neighborhood of $\pm z_{|k|}$ with $z_{1}$ and $z_2$ given by \eqref{zeros}.
	
	We then use \eqref{eig2b} and employ an entirely similar argument
	as in the proof of Lemma \ref{lem.Lop} to conclude the result as expected.
\qed\end{proof}

In view of Lemma \ref{lem.Lopa2}, we can obtain the following result by using the continuity
of $\bm{A}_k^{\pm}$ and the fact that the perturbation, $(\mathring{V},\mathring{\Psi})$, has a compact
support (see \cite[Page\;423]{C04MR2069632} for the proof of Proposition \ref{pro.micro} (c)).

\begin{proposition}\ \label{pro.micro}
	Assume that  $(\mathring{V},\mathring{\Psi})$ satisfies \eqref{bas.c1}--\eqref{bas.c2} with
	$K$ being sufficiently small. Then we can find neighborhoods $\mathscr{V}_c^q$ of $\Upsilon_c^q, q\in\{0,\pm1\}$,
	in $\widebar{\Omega}\times \Xi$ such that
	\begin{list}{}{\setlength{\parsep}{\parskip}
			\setlength{\itemsep}{0.1em}
			\setlength{\labelwidth}{2em}
			\setlength{\labelsep}{0.4em}
			\setlength{\leftmargin}{2.2em}
			\setlength{\topsep}{1mm}
		}
		\item[\emph{(a)}] $\mathscr{V}_c^q\cap (\Upsilon_p\cup \Upsilon_{nd})=\varnothing$.
		\item[\emph{(b)}] Matrix $\mathbb{A}$ defined by \eqref{A.bb} is diagonalizable on $\mathscr{V}_c^q$.
		In particular, there exist matrices $Q_0^{\pm}\in{\mathrm{\Gamma}_2^0}$ such that
		\begin{align}
		Q_0^{\pm}(z)\mathbb{A}^{\pm}(z)Q_0^{\pm}(z)^{-1}
		=\mathrm{diag}\,\big(\omega_{\pm}(z),\omega_{\pm}'(z)\big)
		=:\mathbb{D}_1^{\pm}
		\end{align}
	 for all $z=(t,x,\tau,\eta)\in\mathscr{V}_c^q$, where  $\omega_{+}(z)\ne \omega_{+}'(z)$ and $\omega_{-}(z)\ne \omega_{-}'(z)$.
		\item[\emph{(c)}] Let $\hbar$ be $\IM\omega_+$ or $\IM\omega_-$. Then the solution of the system{\rm :}
		\begin{align} \label{bi.curve}
		\left\{\begin{aligned}
		&\frac{\mathrm{d}t}{\mathrm{d}x_2}=\frac{\partial \hbar}{\partial \delta}(t,x_1,x_2,\tau,\eta),\qquad\!\!
		\frac{\mathrm{d}x_1}{\mathrm{d}x_2}=\frac{\partial \hbar}{\partial \eta}(t,x_1,x_2,\tau,\eta),\\
		&\frac{\mathrm{d}\delta}{\mathrm{d}x_2}=-\frac{\partial \hbar}{\partial t}(t,x_1,x_2,\tau,\eta),\quad
		\frac{\mathrm{d}\eta}{\mathrm{d}x_2}=-\frac{\partial \hbar}{\partial x_1}(t,x_1,x_2,\tau,\eta),\\[1mm]
		&(t,x_1,\gamma+\mathrm{i}\delta,\eta)|_{x_2=0}\in \mathscr{V}_c^q\cap\{x_2=0\}
		\end{aligned}
		\right.
		\end{align}
		defines a curve $(t,x_1,\gamma+\mathrm{i}\delta,\eta)$ for all $x_2\geq 0$,
		which remains in $\mathscr{V}_c^q$ and is called the bicharacteristic curve.
	\end{list}
\end{proposition}

In order to absorb the error terms caused by microlocalization,
as in \cite{CS04MR2095445,C04MR2069632},
we will construct the weight functions that vanish on the bicharacteristic curves
originating from $\Upsilon_c$ and that are nonzero far from these curves.

We define the complex-valued functions: For all $z=(t,x_1,\tau,\eta)\in\mathbb{R}^2\times\Xi_1$ with $\tau=\gamma+\mathrm{i}\delta$,
\begin{align}\label{sigma}
\sigma_q(z):=-\mathrm{i}\gamma+\widetilde{\sigma}_q(z),\quad \widetilde{\sigma}_q(z):=\delta-\eta\mathring{z}_{q}(t,x_1),\qquad q\in\{0,\pm1\},
\end{align}
and extend $\sigma_q$ to $\mathbb{R}^2\times\Xi$ as a homogeneous mapping of degree 1 with respect to $(\tau,\eta)$.
Functions $\mathring{z}_0(t,x_1)$ and $\mathring{z}_{\pm1}(t,x_1)$ are given by Lemma \ref{lem.Lopa2}
and correspond to the points where the Lopatinski\u{\i} determinant vanishes.
Symbol $\sigma_q$ thus belongs to ${\mathrm{\Gamma}_2^1}$ such that
\begin{align*}
\Upsilon_{c}^q=\left\{z=(t,x_1,\tau,\eta)\in\mathbb{R}^2\times\Xi\; :\;\sigma_q(z)=0\right\}.
\end{align*}

In view of Proposition \ref{pro.micro}(c), we can construct solutions $\sigma_{\pm}^q$
of the linear transport equations:
\begin{align}\label{sigma.eq}
\left\{\begin{aligned}
&\partial_2\sigma_+^q+\{\sigma_+^q,\,\IM\omega_+\}=0&&\qquad \textrm{if } x_2>0,\\
&\partial_2\sigma_-^q+\{\sigma_-^q,\,\IM\omega_-\}=0&&\qquad \textrm{if }   x_2>0,\\
&\sigma_{+}^q=\sigma_{-}^q=\sigma_q&&\qquad \textrm{if }  x_2=0,
\end{aligned}\right.
\end{align}
where $\sigma_q$ is given in \eqref{sigma}, and $\{\cdot,\,\cdot\}$ is the Poisson bracket defined by \eqref{Poisson}.
Then we infer that $\sigma_+^q$ (resp.\;$\sigma_-^q$) is constant along each bicharacteristic curve
defined by \eqref{bi.curve} with $\hbar=\IM \omega_+$ (resp.\;$\hbar=\IM \omega_-$).
In particular, function $\sigma_+^q$ (resp.\;$\sigma_-^q$) vanishes \emph{only} on the bicharacteristic curves
originated from $\Upsilon_{c}^q$ with $\hbar=\IM \omega_+$ (resp.\;$\hbar=\IM \omega_-$).
By shrinking $\mathscr{V}_c^q$ if necessary, we may assume that $\sigma_{\pm}^q$ are defined
in the whole set $\mathscr{V}_c^q$.
We will see that functions $\sigma_{\pm}^q$ are appropriate to deal with the error terms
appearing in the energy estimates.

From the above analysis, the whole space $\widebar{\Omega}\times \Xi$ is naturally divided into three disjoint
subsets: $\Upsilon_p$, $\mathscr{V}_c$, and $\widebar{\Omega}\times \Xi\setminus(\Upsilon_p\cup \mathscr{V}_c)$,
where
$\mathscr{V}_c:=\mathscr{V}_c^{-1}\cup \mathscr{V}_c^0\cup \mathscr{V}_c^1$.
To derive our energy estimate \eqref{p.p2.e0}, we introduce smooth cut-off functions according to this division.
More precisely, we introduce nonnegative functions $\chi_p^{\pm}$ and $\chi_c^q$ (with values in $[0,1]$), $q\in\{0,\pm1\}$,
such that
\begin{list}{}{\setlength{\parsep}{\parskip}
		\setlength{\itemsep}{0.1em}
		\setlength{\labelwidth}{2em}
		\setlength{\labelsep}{0.4em}
		\setlength{\leftmargin}{1em}
		\setlength{\topsep}{1mm}
	}
	\item[--] $\chi_p^{\pm}$ and $\chi_c^q$ are $C^{\infty}$ and homogeneous of degree $0$ with respect
	to $(\tau,\eta)$ so that they belong to ${\mathrm{\Gamma}_k^0}$ for all integer $k$;
	\item[--] $\mathrm{supp}\,\chi_c^q\subset \mathscr{V}_c^q$ and $\chi_c^q\equiv 1$ in a neighborhood of
	the bicharacteristic curves originated from the critical set $\Upsilon_c^q$;
	\item[--] $\chi_p^{\pm}\equiv 1$ in a neighborhood of $\Upsilon_p^{\pm}$,  $\supp\chi_p^+\cap\supp\chi_p^-=\varnothing$,
	and $\mathrm{supp}\,\chi_p^{\pm}\cap \mathrm{supp}\,\chi_c^q=\varnothing$ for all $q\in\{0,\pm1\}$.
\end{list}
Since $\sigma_{+}^q$ and $\sigma_{-}^q$ vanish \emph{only} on the bicharacteristic curves originated from $\Upsilon_{c}^q$,
there exists a constant $c$ such that
\begin{align} \label{key_error}
|\sigma_{\pm}^q|\geq c>0\qquad\; \textrm{in  }  \{\chi_c^q< 1\}\cap \mathscr{V}_c^q.
\end{align}
We also define
\begin{align}\label{chi.p.c.u}
\chi_p:=\chi_p^++\chi_p^-,\qquad \chi_c:=\chi_c^{-1}+\chi_c^0+\chi_c^1,\qquad \chi_u:=1-\chi_p-\chi_c.
\end{align}
Then $\chi_u$ has support far from the poles and the bicharacteristic curves originated from $\Upsilon_c$.
We observe that the Lopatinski\u{\i} determinant does not vanish on $\mathrm{supp}\,\chi_u\cap\{x_2=0\}$.
This enables us to apply the standard Kreiss' symmetrizers to derive the energy estimate for $T_{\chi_u}^{\gamma}W^{\mathrm{nc}}$,
which will be shown in \S\,5.4.
After that, we will show how the traces of $T_{\chi_p}^{\gamma}W^{\mathrm{nc}}$ and $T_{\chi_c}^{\gamma}W^{\mathrm{nc}}$
can be estimated.
At the end of this section, we will complete the proof of Theorem {\rm\ref{thm.2}} by using a weighted energy estimate with
the weight functions $\sigma_{\pm}^q$ given by \eqref{sigma.eq}.
In particular, we will prove that the microlocalization error terms can be absorbed by such a weighted estimate.

\subsection{Estimate at Good Frequencies}
In this subsection, we show how the solutions of problem \eqref{p.p2} can be estimated
for the frequencies that are far from both the poles $\Upsilon_p$ and the critical set $\Upsilon_c$.
We define
\[
W_u^{\mathrm{nc}}
:=(T^{\gamma}_{\chi_u}W^{+}_2,T^{\gamma}_{\chi_u}W^{+}_3,T^{\gamma}_{\chi_u}W^{-}_2,T^{\gamma}_{\chi_u}W^{-}_3)^{\mathsf{T}}
\]
and introduce a smooth cut-off function $\widetilde{\chi}_{u}$ with values in $[0,1]$ such that
\[
\widetilde{\chi}_{u}\equiv1 \,\,\,\mathrm{on}\ \ \mathrm{supp}\,\chi_u,\quad
\mathrm{supp}\,\widetilde{\chi}_u\cap\Upsilon_p=\varnothing,\quad
(\mathrm{supp}\,\widetilde{\chi}_u\cap\{x_2=0\})\cap\Upsilon_c=\varnothing,
\]
where $\chi_u$ is given by \eqref{chi.p.c.u}.
Employing the same analysis as in \S\,\ref{sec.eq},
we derive that $W_u^{\mathrm{nc}}$ satisfies
\begin{align} \label{u.P}
\left\{\begin{aligned}
&\partial_2 W_u^{\mathrm{nc}}=T^{\gamma}_{\mathbb{A}_u}W_u^{\mathrm{nc}}+T^{\gamma}_{\mathbb{E}}W_u^{\mathrm{nc}}
+T_r^{\gamma}W+\mathcal{R}_0T^{\gamma}_{\chi_u}F+\mathcal{R}_{-1}W, \!\!\!&  \ x_2>0,\\
&T^{\gamma}_{\bm\beta}W_u^{\mathrm{nc}}=G+\mathcal{R}_{-1}W^{\mathrm{nc}},&  \ x_2=0,
\end{aligned}\right.
\end{align}
where $\bm\beta\in{\mathrm{\Gamma}_2^0}$ is given by \eqref{beta.bm} for $(\tau,\eta)\in\Xi_1$.
The symbol matrices $\mathbb{A}_u$ is defined as $\mathbb{A}_r$ in \eqref{A.bb.r}
with $\widetilde{\chi}_{\pm}$ replaced by $\widetilde{\chi}_u$.
Both $\mathbb{E}$ and $r$ have the same block diagonal structure as $\mathbb{A}_u$ and belong to ${\mathrm{\Gamma}_1^0}$.
We note that $\mathbb{A}_u\equiv\mathbb{A}$ on region $\{\widetilde{\chi}_u\equiv1\}$,
and $r$ is identically zero on region $\{\chi_u\equiv1\}\cup \{\chi_u\equiv0\}$.

In view of Lemma \ref{lem.Lopa2}, we find that the Lopatinski\u{\i} determinant never vanishes
on $\mathrm{supp}\,\chi_u\cap\{x_2=0\}$.
Note that the perturbation, $\big(\mathring{V},\mathring{\Psi}\big)$, is assumed in \eqref{bas.c1}
to have a compact support. In the following lemma, we construct  the Kreiss' symmetrizers that
are microlocalized at all frequencies in the compact set $\mathbb{K}$, where
\[
\mathbb{K}:= \mathrm{supp}\,\chi_u\cap\{-T\le t\le 2T, \,x_2\geq 0,\,|x|\leq R,(\tau,\eta)\in\Xi_1 \}.
\]

\begin{lemma}\ \label{Kreiss2}
	Assume that \eqref{bas.c1}--\eqref{bas.c2} hold for a sufficiently small positive constant $K$.
	Then, for each $z_0\in \mathbb{K}$, there exist a neighborhood $\mathscr{V}$ of $z_0$ in $\mathbb{K}$
	and $C^{\infty}$--mappings $r(z)$ and $T(z)$ defined on $\mathscr{V}$ such that
	\begin{list}{}{\setlength{\parsep}{\parskip}
			\setlength{\itemsep}{0.1em}
			\setlength{\labelwidth}{2em}
			\setlength{\labelsep}{0.4em}
			\setlength{\leftmargin}{2.2em}
			\setlength{\topsep}{1mm}
		}
		\item[\rm (a)] Matrix $r(z)$ is Hermitian, and $T(z)$ is invertible for all $z\in\mathscr{V}${\rm ;}
		\item[\rm (b)] There exists $c>0$ so that
		\begin{align}\label{Kreiss2a}
		\RE\left(r(z)T(z)\mathbb{A}(z)T(z)^{-1}\right)\geq c \gamma I
		\quad\textrm{for all }z\in \mathscr{V}\textrm{ with }\gamma=\RE\tau{\rm ;}
		\end{align}
		\item[\rm (c)] If $z_0\in\mathbb{K}\cap\{x_2=0\}$, then there exists a positive constant $C$ so that
		\begin{align} \label{Kreiss2b}
		r(z)+C\big(\bm\beta(z) T(z)^{-1}\big)^{*}\bm\beta(z) T(z)^{-1}\geq I
		\end{align}
		for all $z\in \mathscr{V}\cap\{x_2=0\},$
	\end{list}
	where $\mathbb{A}$ and $\bm{\beta}$ are given by \eqref{A.bb} and \eqref{beta.bm}, respectively.
\end{lemma}

To prove Lemma \ref{Kreiss2}, we first establish the following result.

\begin{lemma}\ \label{lem.eig3}
	Let $z_0=(t_0,x_0,\tau_0,\eta_0)\in\widebar{\Omega}\times\Xi_1$ so that
	$\RE\tau_0=0$ and $\tau_0\neq \mathrm{i}\big(- \mathring{C}^{+}_2\pm (\mathring{C}^{+}_1)^{-1}\big)\eta_0$.
	Assume that $K$ given in \eqref{bas.c2} is sufficiently small.
	Then there exists a neighborhood $\mathscr{V}$ of $z_0$ in $\widebar{\Omega}\times\Xi_1$ such that
	\begin{align}\label{eig3.e2}
	\RE\omega_{+}(z)\lesssim -\gamma,\quad \RE\omega_{+}'(z)\gtrsim \gamma
	\end{align}
	for all $z=(t,x,\tau,\eta)\in\mathscr{V}\textrm{ with }\gamma=\RE\tau.$
	A similar result holds for $\omega_-$ and $\omega_{-}'$ near $z_0=(t_0,x_0,\tau_0,\eta_0)\in\widebar{\Omega}\times\Xi_1$
	so that  $\RE\tau_0=0$ and $\tau_0\neq \mathrm{i}\big(- \mathring{C}^{-}_2\pm (\mathring{C}^{-}_1)^{-1}\big)\eta_0$.
\end{lemma}

\begin{proof}\
	This proof is divided into two steps.
	
	\smallskip
	{\bf 1.}  If $\RE\omega_{+}(z_0)<0$, then one uses the identities:
	\begin{align}\label{C2.p0}
	\widetilde{\omega}_+=\omega_+-\frac{a_{11}^++a_{22}^+}{2},\qquad
	-\widetilde{\omega}_+=\omega_+'-\frac{a_{11}^++a_{22}^+}{2},
	\end{align}
	and \eqref{eff.i3} to infer that $\RE\omega_{+}'(z_0)>0$ for sufficiently small $K$.
	Since $\gamma=\RE\tau\leq 1$ for $(\tau,\eta)\in\Xi_1$,
	estimates \eqref{eig3.e2} follow directly from the continuity
	of $\omega_+(z)$ and $\omega_+'(z)$ with respect to $z$.
	
	\smallskip
	{\bf 2.} Assume that $\RE\omega_{+}(z_0)=0$.
	It follows from $\RE\tau_0=0$ and \eqref{eff.i3} that $\RE\widetilde{\omega}_{+}(z_0)=0$.
	Thanks to \eqref{eig2b},
	$\eta_0^2<(\mathring{C}_1^+)^2(\delta_0+\mathring{C}_2^+\eta_0)^2$ for $\delta_0=\IM\tau_0$
	so that $\delta_0\ne -\mathring{C}_2^+\eta_0$.
	For all $(\mathrm{i}\delta,\eta)$ with $(\mathring{C}_1^+)^2\big(\delta+\mathring{C}_2^+\eta\big)^2>\eta^2$,
	we apply \eqref{eig2b}  again to derive that
	\begin{align}  \label{C2.p1}
	\delta\ne -\mathring{C}_2^+\eta,  \quad \widetilde{\omega}_+(t,x,\mathrm{i}\delta,\eta)
	=\textrm{\small $ -\mathrm{i}\,\mathrm{sgn}(\delta+\mathring{C}^+_2\eta)
	\mathring{C}_0^+\sqrt{(\mathring{C}_1^+)^2(\delta+\mathring{C}_2^+\eta)^2-\eta^2}.$}
	\end{align}
	Since $\tau_0\ne \mathrm{i}(-\mathring{C}_2^+\pm (\mathring{C}_1^+)^{-1})\eta_0$,
	$\widetilde{\omega}_+$ depends analytically on $(\tau,\eta)$ by applying the implicit functions
	theorem to \eqref{eig2a}.
	In particular, we obtain that, for $z$ near $z_0$,
	\begin{align} \label{C2.p2}
	\widetilde{\omega}_+(t,x,\tau,\eta)\partial_{\gamma}\widetilde{\omega}_+(t,x,\tau,\eta)
	= (\mathring{C}_0^+\mathring{C}_1^+)^2(\tau+\mathrm{i}\mathring{C}_2^+\eta).
	\end{align}
	From \eqref{C2.p1}--\eqref{C2.p2},
	$\partial_{\gamma} \widetilde{\omega}_+(t,x,\mathrm{i}\delta,\eta)$ is real and negative
	for $(t,x,\mathrm{i}\delta,\eta)$ in a suitable neighborhood $\mathscr{V}$ of $z_0$.
	Using the Taylor expansion yields that, for all $(\tau,\eta)\in\mathscr{V}$,
	\[
	\widetilde{\omega}_+(\tau,\eta)
	=\widetilde{\omega}_+(\mathrm{i}\delta,\eta)+\partial_{\gamma} \widetilde{\omega}_+(\mathrm{i}\delta,\eta)\gamma
	+O(\gamma^2)\quad\, (\gamma\to 0).
	\]
	Then we deduce that $\RE\widetilde{\omega}_+\lesssim-\gamma$, up to shrinking $\mathscr{V}$.
	In view of \eqref{eff.i3} and \eqref{C2.p0}, estimates \eqref{eig3.e2} follow by taking $K$ small enough.
\qed\end{proof}

\vspace*{3mm}
\noindent {\bf Proof of Lemma {\rm \ref{Kreiss2}}}\quad
	The proof  is divided into two cases.
	
	\smallskip
	\emph{Case 1.} Let $z_0\in\mathbb{K}\setminus \Upsilon_{\rm nd}$ with $\Upsilon_{\rm nd}$ given in \eqref{Upsilon.nd}.
	In light of Lemmas \ref{lem.Lopa2} and \ref{lem.eig3},
	we can find a neighborhood $\mathscr{V}$ of $z_0$ in $\mathbb{K}$ such that
	\begin{align}\label{C1.p1}
	\RE\omega_{\pm}(z)\lesssim -\gamma,\quad \RE\omega_{\pm}'(z)\gtrsim \gamma\qquad\textrm{for all }z\in\mathscr{V},
	\end{align}
	and
	\begin{align} \label{C1.p2}
	\Delta(z)\ne 0\qquad\textrm{for all }z\in\mathscr{V}\cap\{x_2=0\}.
	\end{align}
	According to Lemma \ref{lem.eig2}, matrix $\mathbb{A}$ is diagonalizable in $\mathscr{V}$.
	Indeed, a smooth basis of the eigenvectors is given by
	\begin{alignat}{2} \label{Y.r}
	&E_+(z),\quad
	&Y_+(z)&:=\left((\tau+\mathrm{i}\mathring{v}_1^+\eta)(-a_{22}^++\omega'_+),
	(\tau+\mathrm{i}\mathring{v}_1^+\eta)a_{21}^+,0,0\right)^{\mathsf{T}},\\[0.5mm]
	&E_-(z),\quad
	&Y_-(z)&:=\left(0,0,(\tau+\mathrm{i}\mathring{v}_1^-\eta)a_{12}^-,
	(\tau+\mathrm{i}\mathring{v}_1^-\eta)(-a_{11}^-+\omega'_-)\right)^{\mathsf{T}},
	\label{Y.l}
	\end{alignat}
	where $E_{\pm}$ are given by \eqref{E.eig2}.
	Notice that $E_{\pm}$ and $Y_{\pm}$ are linearly independent in $\mathscr{V}$.
	We can thus define the smooth and invertible matrix $T(z):=\left(E_+\ E_-\ Y_+\ Y_-\right)^{-1}$ in $\mathscr{V}$
	so that
	\[
	T(z)\mathbb{A}(z)T(z)^{-1}=\mathrm{diag}\,(\omega_+,\,\omega_-,\,\omega'_+,\,\omega'_-)
	\qquad\textrm{for all }z\in\mathscr{V}.
	\]
	Construct the symmetrizer, $r(z)$, as
	\begin{align} \label{symm1}
	r(z):=\mathrm{diag}\,(-1,-1,K',K')\qquad\textrm{for all }z\in\mathscr{V},
	\end{align}
	where $K'\geq 1$ is a constant to be chosen.
	Then we can obtain \eqref{Kreiss2a} from \eqref{C1.p1} directly.
	
	Thanks to \eqref{C1.p2}, we have
	\begin{align} \label{C1.p3}
	\left|\bm\beta(z)(E_+(z) \ E_-(z))Z^-\right|\gtrsim |Z^-|
	\ \ \textrm{for all }z\in\mathscr{V}\cap\{x_2=0\},\,Z^-\in\mathbb{C}^2.
	\end{align}
	This implies
	\[
	|Z^-|^2\leq C_0\big( |Z^+|^2+|\bm\beta(z) T(z)^{-1}Z|^2\big)
	\]
	for all $Z=(Z^-,Z^+)^{\mathsf{T}}\in\mathbb{C}^4$ with $Z^{\pm}\in\mathbb{C}^2$,
	where $C_0$ is some positive constant independent of $z\in\mathscr{V}$.
	Then we have
	\begin{align*}
	&\big\langle\big(r(z)+2C_0\big(\bm\beta(z) T(z)^{-1}\big)^{*}\bm\beta(z) T(z)^{-1}\big)Z,Z\big\rangle_{\mathbb{C}^4}
	\\[0.5mm]
	&\quad =-|Z^-|^2+K|Z^+|^2+2C_0\left|\bm\beta(z) T(z)^{-1}Z\right|^2
	\\&\quad
	\geq |Z^-|^2+(K'-2C_0)|Z^+|^2\geq |Z|^2
	\end{align*}
	by choosing $K'\geq 2C_0+1$, which implies \eqref{Kreiss2b}.
	
	\vspace{2mm}
	\emph{Case 2.} Let $z_0\in\mathbb{K}\cap\Upsilon_{\rm nd}$.
	Then symbol $\mathbb{A}$ is not diagonalizable at $z_0$.
	We consider without loss of generality
	that $z_0=(t_0,x_0,\tau_0,\eta_0)\in\mathbb{K}$ satisfies
	$\tau_0=-\mathrm{i}(\mathring{C}^+_2\pm(\mathring{C}^+_1)^{-1})\eta_0$.
	The case, $\tau_0=-\mathrm{i}(\mathring{C}^-_2\pm(\mathring{C}^-_1)^{-1})\eta_0$,
	can be dealt with in an entirely similar way.
	Using \eqref{root.e} and the continuity of $\mathring{C}_j^{\pm}$
	in $(\mathring{V}^{\pm},\nabla\mathring{\Psi}^{\pm})$,
	we take $K$ sufficiently small to find
	\begin{align*}
	\tau_0+\mathrm{i}\mathring{v}_1^+\eta_0\ne 0,\qquad
	\tau_0\ne-\mathrm{i}(\mathring{C}^-_2\pm(\mathring{C}^-_1)^{-1})\eta_0,
	\end{align*}
	which, combined with \eqref{eig2a} and \eqref{eff.i1}, imply
	\begin{align} \notag
	\left(a_{11}^+-a_{22}^+-2a_{12}^+\right)(z_0)=\mathbb{F}_2^+(z_0)a_+(z_0)\ne 0,\quad \omega_-(z_0)\ne \omega_{-}'(z_0).
	\end{align}
	Since $\widetilde{\omega}_+(z_0)=0$,
	we use \eqref{eig2a.2}, \eqref{eff.i1}, and \eqref{eff.i2} to obtain
	\begin{align} \label{C4.p1}
	\frac{(a_{11}^+-a_{22}^+)(z_0)}{2}
	=-\Big(1+\frac{2\mathbb{F}_3^+(z_0)}{\mathbb{F}_2^+(z_0)}\Big)a_{12}^+(z_0),
	\end{align}
	and hence
	\[
	(a_{11}^+-\omega_+)(z_0)=\frac{a_{11}^+-a_{22}^+}{2}(z_0)-\widetilde{\omega}_+(z_0)
	=-\Big(1+\frac{2\mathbb{F}_3^+(z_0)}{\mathbb{F}_2^+(z_0)}\Big)a_{12}^{+}(z_0)\ne 0.
	\]
	In view of \eqref{E.eig2} and \eqref{Y.r}--\eqref{Y.l}, we find that  the vectors:
	\begin{align*}
	\widetilde{E}_+:=(-a_{12}^{+},-(\mathbb{F}_2^++2\mathbb{F}_3^+)a_{12}^{+}/\mathbb{F}_2^+,0,0)^{\mathsf{T}},
	\  \widetilde{Y}_+:=(\mathrm{i},0,0,0)^{\mathsf{T}},\
	E_-(z),\   Y_-(z)
	\end{align*}
	form a smooth basis in neighborhood $\mathscr{V}$.
	Define $T(z):=\!\!\big(\widetilde{E}_+\; \widetilde{Y}_+\; E_-\; Y_-\big)^{-1}$.  Then
	\begin{align*}
	T(z)\mathbb{A}(z)T(z)^{-1}=\begin{pmatrix}
	a_r\;&\,0\;&0\\
	0\;&\,\omega_-\;&0\\
	0\;&\,0\;&\omega'_-
	\end{pmatrix}\qquad\textrm{for all }z\in\mathscr{V},
	\end{align*}
	where $a_r$ is the $2\times 2$ matrix with $(i,j)$--entry $a_r^{ij}(z)$:
	\begin{alignat*}{2}
	a_r^{11}(z)&=\frac{(2\mathbb{F}_3^++\mathbb{F}_2^+)a_{22}^++\mathbb{F}_2^+ a_{21}^+  }{2\mathbb{F}_3^++\mathbb{F}_2^+},&\qquad
	a_r^{12}(z)&=\frac{-\mathbb{F}_2^+ a_{21}^+  \mathrm{i}}{(2\mathbb{F}_3^++\mathbb{F}_2^+)a_{12}^+},\\
	a_r^{22}(z)&=\frac{(2\mathbb{F}_3^++\mathbb{F}_2^+)a_{11}^+-\mathbb{F}_2^+ a_{21}^+  }{2\mathbb{F}_3^++\mathbb{F}_2^+},&\qquad
	a_r^{21}(z)&=\frac{  a_{12}^+ \tilde{a}_r^{21}(z) }{(2\mathbb{F}_3^++\mathbb{F}_2^+)\mathbb{F}_2^+ \mathrm{i}}
	\end{alignat*}
	with
	\begin{align*}
	\tilde{a}_r^{21}(z):=\mathbb{F}_2^+ (2\mathbb{F}_3^++\mathbb{F}_2^+)(a_{22}^+-a_{11}^+-2a_{12}^+) -4(\mathbb{F}_3^+)^2 a_{12}^++(\mathbb{F}_2^+)^2(a_{12}^++a_{21}^+).
	\end{align*}
	By virtue of \eqref{eff.i1}, \eqref{eff.i2}, and \eqref{C4.p1}, we derive
	\begin{align*}
	2\mathbb{F}_2^+ a_{21}^+(z_0)&=-2\mathbb{F}_2^+a_{12}^+(z_0)+2\mathbb{F}_3^+(a_{11}^+-a_{22}^+-2a_{12}^+)(z_0)\\
	& =(2\mathbb{F}_3^++\mathbb{F}_2^+)(a_{11}^+-a_{22}^+)(z_0), \\
	\tilde{a}_r^{21}(z_0)&=4\mathbb{F}_3^+ (\mathbb{F}_3^++\mathbb{F}_2^+)a_{12}^+(z_0)+(\mathbb{F}_2^+)^2(a_{12}^++a_{21}^+)(z_0)\\
	& =4\mathbb{F}_3^+ (\mathbb{F}_3^++\mathbb{F}_2^+)a_{12}^+(z_0)+\mathbb{F}_2^+\mathbb{F}_3^+(a_{11}^+-a_{22}^+-2a_{12}^+)(z_0)\\
	& =0,
	\end{align*}
	which implies
	\begin{align} \label{C4.p2}
	a_r^{11}(z_0)-a_r^{22}(z_0)=a_r^{21}(z_0)=0.
	\end{align}
	
	We now look for a symmetrizer $r$ with the form:
	\[
	r(z)=\begin{pmatrix}
	s(z) & 0 & 0\\
	0 & -1 & 0\\
	0 & 0 & K'
	\end{pmatrix}, \]
	where $K'\geq 1$ is some real constant, and $s$ is some $2\times 2$ Hermitian matrix,
	depending smoothly on $z$.
	Both $K'$ and $s$ are to be fixed such that \eqref{Kreiss2b} holds
	for $z\in\mathscr{V}\cap\{x_2=0\}$ when $z_0\in \{x_2=0\}$
	and  \eqref{Kreiss2a} holds for all $z\in\mathscr{V}$.
	
	We recall that $\Delta(z)\ne 0$ in $\mathscr{V}\cap\{x_2=0\}$ so that \eqref{C1.p3} holds.
	Noting that the first and third columns of $T(z_0)^{-1}$ are $E_+(z_0)$ and $E_-(z_0)$,
	we can find a positive constant $C_0$ such that, if $z_0\in\mathbb{K}\cap \{x_2=0\}$, then
	\[
	|Z_1|^2+|Z_3|^2\leq C_0\big( |Z_2|^2+|Z_4|^2+|\bm{\beta}(z_0) T(z_0)^{-1}Z|^2\big)
	\]
for all $Z=(Z_1,Z_2,Z_3,Z_4)^{\mathsf{T}}\in\mathbb{C}^4$.
	Assume that the Hermitian matrix $s$ satisfies
	\begin{align}\label{s_0}
	s(z_0)=\begin{pmatrix}
	0 & e_1\\   e_1 & e_2
	\end{pmatrix}=:E,
	\end{align}
	where $e_1$ and $e_2$ are some real constants to be fixed. Then we have
	\begin{align*}
	&\big\langle\big(r(z_0)+C'C_0\big(\bm{\beta}(z_0) T(z_0)^{-1}\big)^{*}\bm{\beta}(z_0) T(z_0)^{-1}\big)Z,Z\big\rangle_{\mathbb{C}^4}
	\\& \ =2e_1\RE\langle Z_1,Z_2\rangle_{\mathbb{C}}+e_2|Z_2|^2
	-|Z_3|^2+K'|Z_4|^2+C'C_0\left|\bm{\beta}(z_0) T(z_0)^{-1}Z\right|^2\\
	&\   \geq (C'-\max\{|e_1|,1\} ) \left(|Z_1|^2 + |Z_3|^2\right) \\
	&\quad \ + (e_2-|e_1|-C'C_0)|Z_2|^2+ (K'-C'C_0)|Z_4|^2.
	\end{align*}
	We choose $C'=\max\{|e_1|,1\}+2,$ $e_2=|e_1|+C'C_0+2$, and $K'=C'C_0+2$ to obtain
	\[
	r(z_0) + C'C_0 \big(\bm{\beta}(z_0) T(z_0)^{-1}\big)^{*}\bm{\beta}(z_0) T(z_0)^{-1} \geq 2I.
	\]
	Using the continuity and shrinking $\mathscr{V}$ if necessary,
	we derive estimate \eqref{Kreiss2b} for $C=C'C_0$.
	
	It remains to choose a suitable Hermitian matrix $s(\tau,\eta)$ and $e_1\in\mathbb{R}$
	such that both  \eqref{Kreiss2a} and \eqref{s_0} hold.
	Since $\tau_0\ne -\mathrm{i}(\mathring{C}^-_2\pm(\mathring{C}^-_1)^{-1})\eta_0$,
	we find that $\RE \omega_-(z)\lesssim -\gamma $ and $\RE \omega_-'(z)\gtrsim \gamma$
	for all $z\in\mathscr{V}$ from Lemma \ref{lem.eig3}.
	Consequently, it suffices to find $e_1\in\mathbb{R}$ and a Hermitian matrix $s(z)$
	satisfying \eqref{s_0} and
	\begin{align} \label{C4.e1}
	\RE\left(s(z)a_r(z)\right)\gtrsim \gamma I\qquad \textrm{for all }z\in\mathscr{V}.
	\end{align}
	To this end, we let
	\begin{align*}
	s(z)=E+F(z)+\gamma G(z)
	\end{align*}
	for some smooth $2\times 2$ Hermitian matrices $F$ and $G$ satisfying $F(z_0)=0$,
	where $E$ is defined by \eqref{s_0}.
	In light of Taylor's formula, we may write
	\[
	a_r(z)=a_r(t,x,\gamma+\mathrm{i}\delta,\eta)
	= a_r(t,x,\mathrm{i}\delta,\eta) + \gamma \partial_\gamma a_r(t,x,\mathrm{i}\delta,\eta)
	+  \gamma^2 N_1(z)
	\]
	for a suitable continuous function $N_1$.
	Noting from \eqref{B.bf} and \eqref{A.bb} that $a^{ij}_r(t,x,\mathrm{i}\delta,\eta)$ are purely imaginary,
	we may choose
	\[
	F(z):=\mathrm{diag}\,(f(z),0)
	\]
	with
	\begin{align*}
	f(z)=\frac{e_1(a_r^{11}-a_r^{22})(t,x,\mathrm{i}\delta,\eta)
		+e_2a_r^{21}(t,x,\mathrm{i}\delta,\eta)}{a_r^{12}(t,x,\mathrm{i}\delta,\eta)},
	\end{align*}
	so that matrix $(E+F(z))a_r(t,x,\mathrm{i}\delta,\eta)$ is symmetric and purely imaginary for all $z\in\mathscr{V}$.
	It follows from \eqref{C4.p2} that $F(z_0)=0$.
	Therefore, we have
	\begin{align*}
	&\RE\left(s(z)a_r(z)\right)\\
	&=\RE\left\{\gamma(E+F(z))\partial_\gamma a_r(t,x,\mathrm{i}\delta,\eta)+\gamma G(z)a_r(t,x,\mathrm{i}\delta,\eta)+\gamma^2N_2(z)\right\}\\
	&=\gamma\RE\left\{E\partial_\gamma a_r(t,x,\mathrm{i}\delta,\eta)
	+G(z)a_r(t,x,\mathrm{i}\delta,\eta)+N_3(z)\right\}
	\end{align*}
	for some continuous functions $N_2$ and $N_3$ satisfying $N_3(z_0)=0$,
	where we have used $F(z_0)=0$.
	According to \eqref{mu.m}, we see that,
	for $\tilde{\tau}=-\mathrm{i}(\widebar{C}_2\pm \widebar{C}_1^{-1})\eta_0$,
	\begin{align*}
	&\partial_{\gamma}a_{r}^{21}(t_0,x_0,\tilde{\tau},\eta_0)_{|(\mathring{V},\mathring{\Psi})=0}
	=\partial_{\gamma}\left\{2\mathrm{i}\,m_+(m_+-\mu_+)\right\}(\tilde{\tau},\eta_0)\\
	&=\frac{2\mathrm{i}\widebar{\varGamma}^2(\mathrm{i}\eta_0+ \epsilon^2 \bar{v}\tilde{\tau})}{(\tilde{\tau}+ \mathrm{i}\bar{v}\eta_0)^3}\\
	&\ \ \ \,\times \Big\{2\epsilon^2\bar{c}^2\bar{v}(\tilde{\tau}+ \mathrm{i}\bar{v}\eta_0)(\mathrm{i}\eta_0
	+ \epsilon^2\bar{v}\tilde{\tau})^2-\bar{c}^2(\mathrm{i}\eta_0+ \epsilon^2\bar{v}\tilde{\tau})^3
	-\epsilon^2\bar{v}(\tilde{\tau}+ \mathrm{i}\bar{v}\eta_0)^3
	\Big\}
	\\&=2 \widebar{\varGamma}^2(1-\epsilon^2\bar{v}^2)
	\frac{\eta_0^2(\mathrm{i}\eta_0+ \epsilon^2 \bar{v}\tilde{\tau})}{\tilde{\tau}
		+ \mathrm{i}\bar{v}\eta_0}\in\mathbb{R}\setminus\{0\} ,
	\end{align*}
	where we have used
	$\bar{c}^2(\mathrm{i}\eta_0+ \epsilon^2 \bar{v}\tilde{\tau})^2=(\tilde{\tau}+ \mathrm{i}\bar{v}\eta_0)^2$
	and condition \eqref{H1}.
	Then $\partial_{\gamma}a_{r}^{21}(z_0)$ is always non-zero by choosing $K$ sufficiently small
	and using the continuity of $\bm{A}_j^{\pm}$ and $\mathring{C}_j^{\pm}$.
	In order to obtain \eqref{Kreiss2a}, we choose
	\begin{align*}
	e_1:=\left(\partial_{\gamma}a_{r}^{21}(z_0)\right)^{-1},\quad
	G(z):=\begin{pmatrix}
	0 & \mathrm{i}g\\[0.5mm] -\mathrm{i}g &0
	\end{pmatrix}
	\end{align*}
	for some positive constant $g$.
	This choice of $e_1$ and $G$ yields
	\begin{align*}
	\RE\left\{E\partial_\gamma a_r(z_0)+G(z)a_r(z_0)\right\}
	=\begin{pmatrix}
	1 &\star\\[0.5mm] \star & \star
	\end{pmatrix}
	+\begin{pmatrix}
	0& \mathrm{i}g a_r^{22}(z_0)\\[0.5mm] -\mathrm{i}ga_r^{11}(z_0)  &  -\mathrm{i}ga_r^{12}(z_0)
	\end{pmatrix},
	\end{align*}
	where the entries with $\star$ are the coefficients that depend only
	on $z_0$, $e_1$, and $e_2$ (which have been fixed earlier).
	Notice that, if $(\mathring{V},\mathring{\Psi})=0$, then
	\begin{align*}
	a_r^{11}(z_0)=- a_r^{22}(z_0)=(\mu_+-m_+)(\tilde{\tau},\eta_0)=0,\qquad
	a_r^{12}(z_0)=\mathrm{i}
	\end{align*}
	for $\tilde{\tau}=-\mathrm{i}(\widebar{C}_2\pm \widebar{C}_1^{-1})\eta_0$.
	Then we can take $K$ sufficiently small,
	$g$ suitably large, and shrink $\mathscr{V}$ to conclude \eqref{Kreiss2a}.
	This completes the proof.
\qed

Thanks to Lemma \ref{Kreiss2}, one can deduce the following lemma by using a partition of unity.
We refer to \cite[Theorem\;9.1]{B-GS07MR2284507} and \cite[\S\,4.7.3]{WYuan15MR3327369} for
a detailed derivation of the following ``global'' symmetrizer $\bm{S}$.

\begin{lemma}\ \label{Kreiss1}
	Assume that \eqref{bas.c1}--\eqref{bas.c2} hold for a sufficiently small positive constant $K$.
	Then there exists a mapping
	\[\bm{S}\,:\, \widebar{\Omega}\times (\mathbb{R}^2\times\mathbb{R}_+\setminus\{0\})\to \mathcal{M}_4(\mathbb{C})
	\]
	that satisfies the following properties{\rm :}
	\begin{list}{}{\setlength{\parsep}{\parskip}
			\setlength{\itemsep}{0.1em}
			\setlength{\labelwidth}{2em}
			\setlength{\labelsep}{0.4em}
			\setlength{\leftmargin}{2.2em}
			\setlength{\topsep}{1mm}
		}
		\item[\emph{(a)}] For all $z\in\widebar{\Omega}\times (\mathbb{R}^2\times\mathbb{R}_+\setminus\{0\})$,
		matrix $\bm{S}(z)$ is Hermitian and $\bm{S}\in {\mathrm{\Gamma}_2^{2}}${\rm ;}
		\item[\emph{(b)}] For all $z=(t,x_1,\delta,\eta,\gamma)\in\partial\Omega\times (\mathbb{R}^2\times\mathbb{R}_+\setminus\{0\})$,
		\begin{align}  \notag  
		\widetilde{\chi}_u(z)^2\bm{S}(z)+C\widetilde{\chi}_u(z)^2\lambda^{2,\gamma}\bm{\beta}(z)^*\bm{\beta}(z)\geq c\widetilde{\chi}_u(z)^2\lambda^{2,\gamma}I,
		\end{align}
		where $\lambda^{m,\gamma}:=(\gamma^2+\delta^2+\eta^2)^{m/2}${\rm ;}
		\item[\emph{(c)}] There exists a finite set of matrix-valued mappings $V_j$, $H_j$, and $E_j$ such that
		\begin{align} \notag
		\RE\left(\bm{S}(z)\mathbb{A}_{u}(z)\right)=\sum_j V_j(z)^*\begin{pmatrix}
		\gamma H_j(z)&0\\ 0& E_j(z)
		\end{pmatrix}V_j(z),
		\end{align}
		where $V_j$ and $E_j$ belong to ${\mathrm{\Gamma}_2^1}$, $H_j\in{\mathrm{\Gamma}_2^0}$,
		and the following estimates hold{\rm :}
		\begin{align} \notag
		\sum_jV_j(z)^*V_j(z)\geq c\lambda^{2,\gamma}\widetilde{\chi}_u(z)^2I,\qquad
		H_j(z)\geq cI,\qquad E_j(z)\geq c\lambda^{1,\gamma}I.
		\end{align}
	\end{list}
\end{lemma}

With Lemma \ref{Kreiss1} in hand, we can choose $\bm{S}$ as a symmetrizer for problem \eqref{u.P}
to show the energy estimate as in \cite[\S\,3.5]{C04MR2069632},
for which we only give the result here for brevity.
{We just recall that the components $T_{\chi_u}^{\gamma}W_1^\pm$ are given in terms of $T_{\chi_u}^{\gamma}W_{2,3}^\pm$ by relations similar to \eqref{omega1}.}
The estimate for $T_{\chi_u}^{\gamma}W$ reads:
\begin{align}\notag
&\gamma\VERT T_{\chi_u}^{\gamma}W\VERT_{1,\gamma}^2+\|T_{\chi_u}^{\gamma}{W^{\mathrm{nc}}}|_{x_2=0}\|_{1,\gamma}^2\\
&\,\, \lesssim \|G\|_{1,\gamma}^2+\|{W^{\mathrm{nc}}}|_{x_2=0}\|^2
+\gamma^{-1}\left(\VERT F\VERT_{1,\gamma}^2+\VERT W \VERT^2+\VERT T_r^{\gamma}W\VERT_{1,\gamma}^2\right),
\label{est.u}
\end{align}
where symbol $r\in{\mathrm{\Gamma}_1^0}$ vanishes on region $\{\chi_u\equiv1\}\cup\{\chi_u\equiv0\}$.

\subsection{Estimate near the Poles}
This subsection is devoted to deriving the energy estimate near poles $\Upsilon_p=\Upsilon_p^+\cup \Upsilon_p^-$.
Matrix $\mathbb{A}$ is not defined at points in $\Upsilon_p$,
while the stable subspace $\mathcal{E}^-$ of $\mathbb{A}$ admits a continuous extension
at these points, due to Lemma \ref{lem.eig2}.
We show the estimate near $\Upsilon_p^+$ without loss of generality.
For this purpose, we define two cut-off functions $\widetilde{\chi}$ and $\widetilde{\chi}_1$
with values in $[0,1]$ that are both $C^{\infty}$ and homogeneous of degree $0$ with respect to $(\tau,\eta)$
and satisfy that
\begin{align}
\widetilde{\chi}\equiv 1 \ \ \mathrm{on}\ \mathrm{supp}\,\chi_p^+,\ \
\widetilde{\chi}_1\equiv 1 \ \ \mathrm{on}\ \supp\widetilde{\chi},\ \
\mathrm{supp}\,\widetilde{\chi}_1\cap\mathrm{supp}\,\chi_{\rm c}=\varnothing, \label{chi.p}
\end{align}
where $\chi_p^+$ and $\chi_{\rm c}$ are introduced at the end of \S\,\ref{sec.micro}.
As in \cite{CS04MR2095445}, we go back to the original problem \eqref{p.p2} and set
\begin{align}\label{pole.W}
W_{\rm p}^{+}:=T_{\chi_{\rm p}^+}^{\gamma}W^{+}, \quad   W_{\rm p}^{-}:=T_{\chi_{\rm p}^+}^{\gamma}W^{-}.
\end{align}
Then we employ the argument in the derivation of \eqref{id.P.1} to obtain
\begin{align} \label{pole1.eq}
T^{\gamma}_{\tau \bm{A}_0^{+}+\mathrm{i}\eta\bm{A}_1^{+}}W_p^++T^{\gamma}_{\bm{C}^+}W_p^++\bm{I}_2\partial_2W_p^+=
T^{\gamma}_{r}W^++T^{\gamma}_{\chi_p^+}F^++\mathcal{R}_{-1}W^+,
\end{align}
where $r=\mathrm{i}\left\{\chi_p^+,\tau \bm{A}_0^{+}+\mathrm{i}\eta\bm{A}_1^{+}\right\}+\partial_2\chi_p^+\bm{I}_2\in{\mathrm{\Gamma}_1^0}$.
The equation for $W_p^-$ is the same as \eqref{pole1.eq} with index $``+''$ replaced by $``-''$.

Let us introduce symbols $R_+$ and $L^+_1$, both belonging to ${\mathrm{\Gamma}_2^0}$, such that, for $(\tau,\eta)\in\Xi_1$,
\setlength{\arraycolsep}{5pt}
\begin{align*}
R_+:=\begin{pmatrix}
1&k_1&k_2\\ 0&-\mathring{a}_+a_{12}^+&0\\
0&\mathring{a}_+(a_{11}^+-\omega_{+})&1
\end{pmatrix},\qquad
L^+_1:=\begin{pmatrix}
1&0&0\\ 0&1&0\\ l_1&l_2&l_3
\end{pmatrix}.
\end{align*}
Recalling from \eqref{B.bf} and \eqref{b11} that $(b_{ij}^{+})=\tau\bm{A}_0^{+}+\mathrm{i}\eta\bm{A}_1^{+}$ and
$b_{11}^+=\mathbb{F}_{1}^+\mathring{a}_+$,  we choose
\begin{align*}
&k_1=(\mathbb{F}_1^{+})^{-1}\left\{b_{12}^+a_{12}^+-b_{13}^+(a_{11}^+-\omega_+) \right\},\quad
k_2=-{b_{23}^+}/{b_{21}^+},\\
&l_1=-(\mathbb{F}_1^{+})^{-1}\left\{b_{21}^+(a_{11}^+-\omega_+)+b_{31}^+a_{12}^+ \right\},\quad
l_2=\mathring{a}_+(a_{11}^+-\omega_+),\quad
l_3=\mathring{a}_+a_{12}^+,
\end{align*}
so that $L_1^+\bm{I}_2R^+=\mathrm{diag}\,(0,\,-\mathring{a}_+a_{12}^+,\,\mathring{a}_+a_{12}^+)$ and
\begin{align*}
&L_1^+(\tau\bm{A}_0^{+}+\mathrm{i}\eta\bm{A}_1^{+})R^+
=\begin{pmatrix}
b_{11}^+&0&k_2 b_{11}^++b_{13}^+\\
b_{21}^+&d_1& 0\\
0&d_3&d_2
\end{pmatrix},
\end{align*}
where
\begin{align*}
d_1&=b_{21}^+k_1-b_{22}^+\mathring{a}_+a_{12}^++b_{23}^+\mathring{a}_+(a_{11}^+-\omega_+),\quad
d_2=l_j(k_2b_{j1}^++b_{j3}^+)=l_j b_{j3}^+, \\
d_3&=l_j\left(k_1b_{j1}^+-b_{j2}^+\mathring{a}_+a_{12}^++b_{j3}^+\mathring{a}_+(a_{11}^+-\omega_+)\right)\\
&=l_j\left(-b_{j2}^+\mathring{a}_+a_{12}^++b_{j3}^+\mathring{a}_+(a_{11}^+-\omega_+)\right).
\end{align*}
We have used the relation: $l_jb_{j1}^+=0$.
From \eqref{b11} and definition \eqref{A.bb} of $\mathbb{A}$, we have
\begin{align*}
d_1&=(\mathbb{F}_1^{+})^{-1}\left\{(b_{21}^+b_{12}^+-b_{22}^+\mathbb{F}_1^{+}\mathring{a}_+)a_{12}^+
+(-b_{21}^+b_{13}^++b_{23}^+\mathbb{F}_1^{+}\mathring{a}_+)(a_{11}^+-\omega_{+})\right\}\\
&=(\mathbb{F}_1^{+})^{-1}\left\{b_{11}^+a_{11}^+a_{12}^+-b_{11}^+a_{12}^+(a_{11}^+-\omega_{+})\right\}=\mathring{a}_+a_{12}^+\omega_{+},\\[1mm]
d_2&=(\mathbb{F}_1^{+})^{-1}\left\{(a_{11}^+-\omega_{+})(-b_{21}^+b_{13}^++\mathbb{F}_1^{+}\mathring{a}_+b_{23}^+)
+a_{12}^+(-b_{31}^+b_{13}^++\mathbb{F}_1^{+}\mathring{a}_+b_{33}^+)  \right\}\\
&=-(\mathbb{F}_1^{+})^{-1}\left\{(a_{11}^+-\omega_{+})b_{11}^+a_{12}^++a_{12}^+b_{11}^+a_{22}^+\right\}
=\mathring{a}_+a_{12}^+(\omega_{+}-a_{11}^+-a_{22}^+),
\end{align*}
and
\begin{align*}
\mathbb{F}_1^{+}d_3=\;&(a_{11}^+-\omega_{+})^2(-b_{21}^+\mathring{a}_+b_{13}^++\mathbb{F}_1^{+}\mathring{a}_+^2b_{23}^+)
+(a_{12}^+)^2\mathring{a}_+(b_{12}^+b_{31}^+-\mathbb{F}_1^{+}\mathring{a}_+b_{32}^+)\\
&+(a_{11}^+-\omega_{+})a_{12}^+\mathring{a}_+(b_{12}^+b_{21}^+-b_{13}^+b_{31}^+
-\mathbb{F}_1^{+}\mathring{a}_+b_{22}^++\mathbb{F}_1^{+}\mathring{a}_+b_{33}^+)\\
=\;&-\mathring{a}_+b_{11}^+a_{12}^+\left\{(a_{11}^+-\omega_{+})^2-a_{12}^+a_{21}^+-(a_{11}^+-\omega_{+})(a_{11}^+-a_{22}^+)\right\}=0.
\end{align*}
We have used \eqref{eig2a.1}  for deriving the last identity.
We notice from \eqref{A.cal}--\eqref{mu.m} that
\begin{align*}
\mathring{a}_+a_{12}^+|_{(\mathring{V},\mathring{\Psi})=0}=-(\tau+\mathrm{i}\bar{v}\eta)m_+
={\widebar{\varGamma}\bar{c}}(\mathrm{i}\eta+\epsilon^2\bar{v}\tau)^2/2
\end{align*}
does not vanish on $\mathrm{supp}\,\widetilde{\chi}_1$ by shrinking $\mathrm{supp}\,\widetilde{\chi}_1$ if necessary.
Since $\mathring{a}_+a_{12}^+$ is smooth in $(\tau,\eta)$ and $(\mathring{V},\nabla\mathring{\Psi})$,
matrix $L_2^+:=\mathrm{diag}\,(1,\,-(\mathring{a}_+a_{12}^+)^{-1},\,(\mathring{a}_+a_{12}^+)^{-1})$
is a smooth and invertible mapping on $\mathrm{supp}\,\widetilde{\chi}_1$.
We then derive
\begin{align} \label{pole.id1a}
&L_+\bm{I}_2R_+=\bm{I}_2,\qquad
L_+:=L_2^+L_1^+,\\[1mm]
\notag
&L_+(\tau\bm{A}_0^{+}+\mathrm{i}\eta\bm{A}_1^{+})R_+\\
&\quad \label{pole.id1b} =\begin{pmatrix}
b_{11}^+&0&k_2 b_{11}^++b_{13}^+\\
-b_{21}^+/(\mathring{a}_+a_{12}^+)&-\omega_{+}& 0\\
0&0&\omega_{+}-a_{11}^+-a_{22}^+
\end{pmatrix}=:A_+^{d}.
\end{align}

We also introduce symbols $R_-$ and $L^-_1$ that belong to ${\mathrm{\Gamma}_2^0}$ and satisfy that, for $(\tau,\eta)\in\Xi_1$,
\setlength{\arraycolsep}{5pt}
\begin{align*}
R_-&:=\begin{pmatrix}
1&k_1^-&k_2^-\\ 0&\mathring{a}_-(a_{22}^--\omega_{-})&\mathring{a}_-a_{12}^-\\
0&-\mathring{a}_-a_{21}^-&\mathring{a}_- (-a_{11}^-+\omega_{-}')
\end{pmatrix},\\[1mm]
L_1^-&:=\begin{pmatrix}
1&0&0\\
l_1^-&\mathring{a}_-(a_{11}^--\omega_{-}')&\mathring{a}_-a_{12}^-\\
l_2^-&\mathring{a}_-a_{21}^-&\mathring{a}_-(a_{22}^--\omega_{-})
\end{pmatrix},
\end{align*}
with
\begin{alignat*}{3}
k_1^-&=\;&(\mathbb{F}_1^{-})^{-1}\left\{-(a_{22}^--\omega_{-})b_{12}^-+a_{21}^-b_{13}^- \right\}, \\
k_2^-&=\;&(\mathbb{F}_1^{-})^{-1}\left\{-a_{12}^-b_{12}^-+(a_{11}^--\omega_{-}')b_{13}^-\right\},\\[0.5mm]
l_1^-&=\;&(\mathbb{F}_1^{-})^{-1}\left\{-(a_{11}^--\omega_{-}')b_{21}^- -a_{12}^-b_{31}^- \right\}, \\
l_2^-&=\;&(\mathbb{F}_1^{-})^{-1}\left\{-a_{21}^-b_{21}^- - (a_{22}^--\omega_{-})b_{31}^- \right\},
\end{alignat*}
so that $L_1^-\bm{I}_2R_-=\mathrm{diag}\,(0,\,d_4,\,-d_4)$ and
\begin{align*}
L_1^-(\tau\bm{A}_0^{-}+\mathrm{i}\eta\bm{A}_1^{-})R_-=\begin{pmatrix}
b_{11}^-&0&0\\ 0& d_5&d_6\\  0& d_7&d_8
\end{pmatrix},
\end{align*}
where  $d_4=\mathring{a}_-^2\left\{(a_{11}^--\omega_-')(a_{22}^--\omega_-)-a_{12}^-a_{21}^-\right\}$ and
\begin{align*}
d_5=\;&\mathring{a}_-(a_{11}^--\omega_-')\left\{b_{21}^-k_1^-+b_{22}^-\mathring{a}_-(a_{22}^--\omega_{-})-b_{23}^-\mathring{a}_-a_{21}^- \right\}\\
&+\mathring{a}_-a_{12}^-\left\{b_{31}^-k_1^- +b_{32}^-\mathring{a}_-(a_{22}^--\omega_{-})-b_{33}^-\mathring{a}_-a_{21}^-   \right\}.
\end{align*}
We omit the expressions for $d_j$, $j\in\{6,7,8\}$,
since they are quite similar to that of $d_5$.
By virtue of the identity:
$a_{22}^--\omega_{-}=-a_{11}^-+\omega_{-}'$ and \eqref{eig2a.1}, we deduce
\begin{align}
(a_{11}^--\omega_-')(a_{22}^--\omega_-)-a_{12}^-a_{21}^-
&=-(a_{22}^--\omega_-)^2-a_{12}^-a_{21}^- \notag\\
&=(\omega_--a_{22}^-)(-2\omega_{-}+a_{11}^-+a_{22}^-),
\label{pole.p1}
\end{align}
which yields $d_4=\mathring{a}_-^2(\omega_{-}-a_{22}^-)(-2\omega_{-}+a_{11}^-+a_{22}^-)$.
Using  \eqref{B.bf}, \eqref{b11}, and the identity: $a_{22}^--\omega_{-}=-a_{11}^-+\omega_{-}'$, we compute
\begin{align*}
d_5=\mathring{a}_-^2\left\{a_{11}^-(a_{22}^--\omega_-)^2-2a_{12}^-a_{21}^-(a_{22}^--\omega_-)+a_{12}^-a_{21}^-a_{22}^-\right\},
\end{align*}
which, combined with \eqref{eig2a.1} and \eqref{pole.p1}, implies
\begin{align*}
d_5=\mathring{a}_-^2\omega_{-}(\omega_{-}-a_{22}^-)(2\omega_{-}-a_{11}^--a_{22}^-)=-\omega_{-}d_4.
\end{align*}
Performing the similar calculations to $d_j$ for $j=6,7,8$,
we can discover that $d_6=d_7=0$ and $d_8=\omega_-' d_4$ so that
$$L_1^-(\tau\bm{A}_0^{-}+\mathrm{i}\eta\bm{A}_1^{-})R_-=\mathrm{diag}\,(0,\,-\omega_{-}d_4,\,\omega_{-}'d_4).$$
Note that $d_4$ does not vanish in neighborhood $\mathrm{supp}\,\widetilde{\chi}_1$ of $\Upsilon_p^+$
up to shrinking $\mathrm{supp}\,\widetilde{\chi}_1$.
Setting $L_2^-:=\mathrm{diag}\,(1,\,d_4^{-1},\,-d_4^{-1})$ and $L_-:=L_2^-L_1^-$, we obtain
\begin{align} \label{pole.id2}
\left\{
\begin{aligned}
& L_-\bm{I}_2R_-=\bm{I}_2,\\
&L_-(\tau\bm{A}_0^{-}+\mathrm{i}\eta\bm{A}_1^{-})R_-=\mathrm{diag}\,(b_{11}^-,\,-\omega_{-},\,-\omega_{-}')=:A_-^d.
\end{aligned}
\right.
\end{align}

Let us define
\begin{align*}
Z^+:=T^{\gamma}_{\widetilde{\chi}R_+^{-1}}W_p^+,
\qquad
Z^-:=T^{\gamma}_{\widetilde{\chi}R_-^{-1}}W_p^-.
\end{align*}
Applying operator $T^{\gamma}_{\widetilde{\chi}L_+}$ to \eqref{pole1.eq} and using  \eqref{pole.id1a} yield
\begin{align*}
&T^{\gamma}_{\widetilde{\chi}L_+(\tau \bm{A}_0^{+}+\mathrm{i}\eta\bm{A}_1^{+})}W_p^+
+T^{\gamma}_{-\mathrm{i}\sum_{j=0}^{1}\partial_{\xi_j}(\widetilde{\chi}L_+)\partial_{x_j}(\tau \bm{A}_0^{+}
	+\mathrm{i}\eta\bm{A}_1^{+})}W_p^++T^{\gamma}_{\widetilde{\chi}L_+\bm{C}^+}W_p^+\\
&\quad =-T^{\gamma}_{\widetilde{\chi}L_+\bm{I}_2}\partial_2W_p^++T^{\gamma}_{\widetilde{\chi}L_+}T^{\gamma}_{r}W^+
+T^{\gamma}_{\widetilde{\chi}L_+}T^{\gamma}_{\chi_p^+}F^++\mathcal{R}_{-1}W^+\\
&\quad =-\bm{I}_2\partial_2Z^++T^{\gamma}_{\bm{I}_2\partial_2(\widetilde{\chi}R_+^{-1})}W_p^+
+T^{\gamma}_{\widetilde{\chi}L_+}T^{\gamma}_{r}W^++T^{\gamma}_{\widetilde{\chi}L_+}T^{\gamma}_{\chi_p^+}F^++\mathcal{R}_{-1}W^+,
\end{align*}
where $x_0:=t$, $\xi_0:=\delta$, and $\xi_1:=\eta$ to avoid overloaded equations.
On the other hand, it follows from \eqref{pole.W} and \eqref{pole.id1b} that
\begin{align*}
T^{\gamma}_{A^d_+}Z^+=T^{\gamma}_{\widetilde{\chi}L_+(\tau \bm{A}_0^{+}+\mathrm{i}\eta\bm{A}_1^{+})}W_p^+
+T^{\gamma}_{-\mathrm{i}\sum_{j=0}^{1}\partial_{\xi_j}A^d_+\partial_{x_j}(\widetilde{\chi}R_+^{-1})}W_p^++\mathcal{R}_{-1}W^+.
\end{align*}
Then we have
\begin{align} \label{pole2.eq}
\bm{I}_2\partial_2Z^++T^{\gamma}_{\widetilde{A}^d_+}Z^++T^{\gamma}_{\mathbb{D}_+^p}Z^+
=T^{\gamma}_{r}W^++\mathcal{R}_{0}T^{\gamma}_{\chi_p^+}F^++\mathcal{R}_{-1}W^+
\end{align}
for new $r\in{\mathrm{\Gamma}_1^0}$ vanishing on $\{\chi_p^+\equiv1 \}\cup\{\chi_p^+\equiv0 \}$,
where $\widetilde{A}^d_+$ is an extension of $A^d_+$ to the whole set $\widebar{\Omega}\times\Xi$,
and $\mathbb{D}_+^p\in{\mathrm{\Gamma}_1^0}$ is given by
\begin{align*}
\mathbb{D}_+^p:=\;&\widetilde{\chi}L_+\bm{C}^+R_+-\bm{I}_2\partial_2(\widetilde{\chi}R_+^{-1})R_+\\
&+\mathrm{i}\sum_{j=0,1}\big\{\partial_{\xi_j}\widetilde{A}^d_+\partial_{x_j}(\widetilde{\chi}R_+^{-1})
-\partial_{\xi_j}(\widetilde{\chi}L_+)\partial_{x_j}(\tau \bm{A}_0^{+}+\mathrm{i}\eta\bm{A}_1^{+}) \big\}  R_+.
\end{align*}
Similarly, we have
\begin{align} \label{pole2.eq2}
\bm{I}_2\partial_2Z^-+T^{\gamma}_{\widetilde{A}^d_-}Z^-+T^{\gamma}_{\mathbb{D}_-^p}Z^-
=T^{\gamma}_{r}W^-+\mathcal{R}_{0}T^{\gamma}_{\chi_p^+}F^-+\mathcal{R}_{-1}W^-,
\end{align}
where $r\in{\mathrm{\Gamma}_1^0}$ vanishes on $\{\chi_p^+\equiv1 \}\cup\{\chi_p^+\equiv0 \}$,
$\widetilde{A}^d_-$ is an extension to the whole set $\widebar{\Omega}\times\Xi$ of $A^d_-$,
and $\mathbb{D}_-^p\in{\mathrm{\Gamma}_1^0}$.
According  to the definitions of $R_{\pm}$,  we have
\begin{align} \label{pole.id3}
Z^{\mathrm{nc}}=T^{\gamma}_{\widetilde{\chi}\widetilde{R}^{-1}}T^{\gamma}_{\chi_p^+}W^{\mathrm{nc}},
\end{align}
where $Z^{\mathrm{nc}}:=(Z_2^+,Z_3^+,Z_2^-,Z_3^-)^{\mathsf{T}}$ and
\begin{align*}
\widetilde{R}:=\begin{pmatrix}
-\mathring{a}_+a_{12}^+&0&0&0\\
\mathring{a}_+(a_{11}^+-\omega_{+})&1&0&0\\
0&0&\mathring{a}_-(a_{22}^--\omega_{-})&\mathring{a}_-a_{12}^-\\
0&0&-\mathring{a}_-a_{21}^-&\mathring{a}_- (-a_{11}^-+\omega_{-}')
\end{pmatrix}.
\end{align*}
Note from \eqref{E.eig2} that the first and third columns of $\widetilde{R}$ are
$E_+$ and $E_-$.
By virtue of \eqref{p.p2.c}, we obtain the following boundary conditions
in terms of $Z^{\mathrm{nc}}$:
\begin{align} \label{pole.bdy}
T^{\gamma}_{\bm{\beta}(E_+\,  E_-)}\begin{pmatrix}
Z_{2}^+\\ Z_2^-
\end{pmatrix}+\mathcal{R}_0\begin{pmatrix}
Z_{3}^+\\  Z_3^-
\end{pmatrix}=\mathcal{R}_0G+\mathcal{R}_{-1}W^{\mathrm{nc}}\qquad \mathrm{if}\ x_2=0.
\end{align}

For problem \eqref{pole2.eq}--\eqref{pole2.eq2} and \eqref{pole.bdy},
we obtain the following energy estimate:

\begin{lemma}\ \label{lem.Zp}
	There exists constants $K_0\leq 1$ and $\gamma_0\geq 1$ such that,
	if $\gamma\geq \gamma_0$ and $K\leq K_0$ for $K$ given in \eqref{bas.c2}, then
	\begin{align}\notag
	&\gamma\VERT Z^{\pm}\VERT_{1,\gamma}^2+\|(Z_2^{\pm},Z_3^{\pm})|_{x_2=0}\|_{1,\gamma}^2\\
	&\,\,\, \lesssim \|G\|_{1,\gamma}^2+\|{W^{\mathrm{nc}}}|_{x_2=0}\|^2
	+\gamma^{-1}\big(\VERT T^{\gamma}_{\chi_p^+}F\VERT_{1,\gamma}^2+\VERT W\VERT^2+\VERT T^{\gamma}_{r}W\VERT_{1,\gamma}^2\big),
	\label{pole.e.1}
	\end{align}
	where symbol $r\in{\mathrm{\Gamma}_1^0}$ vanishes in region $\{\chi_p^+\equiv1\}\cup\{\chi_p^+\equiv0\}$.
\end{lemma}

\begin{proof}\  We divide the proof into five steps.
	
	\smallskip
	\noindent
	{\bf 1.}\; {\it Estimate for $Z_3^+$}. \quad
	According to the form of $A_+^d$ given by \eqref{pole.id1b},
	the third equation in \eqref{pole2.eq} for $Z_3^+$ reads
	\begin{align} \label{pole2.eq3}
	\partial_2Z_3^+=T^{\gamma}_{-\omega_{+}+a_{11}^++a_{22}^+}Z_3^++T^{\gamma}_{\alpha_{0}}Z^+
	+T^{\gamma}_r W^++\mathcal{R}_0T^{\gamma}_{\chi_p^+}F^+_3+\mathcal{R}_{-1}W^+.
	\end{align}
	Take the scalar product in $L^2(\Omega)$ of \eqref{pole2.eq3} with $\Lambda^{2,\gamma}Z_3^+$
	to obtain
	\begin{align} \label{pole.e1}
	\|Z^+_{3}|_{x_2=0}\|_{1,\gamma}^2+2\RE \big\llangle \Lambda^{1,\gamma}Z_3^+,\Lambda^{1,\gamma}T^{\gamma}_{-\omega_{+}+a_{11}^++a_{22}^+}Z_3^+\big\rrangle
	=\sum_{j=1}^{4}\mathcal{H}_j,
	\end{align}
	where each term $\mathcal{H}_j$ in the decomposition is defined in the following:
	\begin{align*}
	\mathcal{H}_1&:=-2\RE\left\llangle \Lambda^{1,\gamma}Z_3^+,\Lambda^{1,\gamma}T^{\gamma}_{\alpha_{0}}Z^+\right\rrangle
	\lesssim \VERT Z^+\VERT_{1,\gamma}^2,\\[1mm]
	\mathcal{H}_2&:=-2\RE\left\llangle \Lambda^{1,\gamma}Z_3^+,\Lambda^{1,\gamma}T^{\gamma}_r W^+\right\rrangle
	\lesssim \varepsilon\gamma\VERT Z_3^+\VERT_{1,\gamma}^2+\frac{1}{\varepsilon\gamma}\VERT T^{\gamma}_r W^+\VERT_{1,\gamma}^2,\\
	\mathcal{H}_3&:=-2\RE\big\llangle \Lambda^{1,\gamma}Z_3^+,\Lambda^{1,\gamma}\mathcal{R}_0T^{\gamma}_{\chi_p^+}F^+_3\big\rrangle
	\lesssim \varepsilon\gamma\VERT Z_3^+\VERT_{1,\gamma}^2+\frac{1}{\varepsilon\gamma}\VERT T^{\gamma}_{\chi_p^+}F^+_3\VERT_{1,\gamma}^2,\\
	\mathcal{H}_4&:=-2\RE\left\llangle \Lambda^{1,\gamma}Z_3^+,\Lambda^{1,\gamma}\mathcal{R}_{-1}W^+\right\rrangle
	\lesssim \varepsilon\gamma\VERT Z_3^+\VERT_{1,\gamma}^2+\frac{1}{\varepsilon\gamma}\VERT W^+ \VERT^2.
	\end{align*}
	For the second term on the left-hand side of \eqref{pole.e1}, we employ Lemma~\ref{lem.para1} (iv) to deduce
	\begin{align*}
&	\RE\big\llangle  \Lambda^{1,\gamma}Z_3^+,\Lambda^{1,\gamma}T^{\gamma}_{-\omega_{+}+a_{11}^++a_{22}^+}Z_3^+\big\rrangle\\
	&\quad \geq \RE\big\llangle \Lambda^{1,\gamma}Z_3^+,T^{\gamma}_{-\omega_{+}+a_{11}^+
		+a_{22}^+}\Lambda^{1,\gamma}Z_3^+\big\rrangle-C\VERT Z_3^+\VERT_{1,\gamma}^2.
	\end{align*}
	Thanks to \eqref{eff.i3},  $\RE(a_{11}^++a_{22}^+)=\mathbb{F}_4^+\gamma$,
	where $\mathbb{F}_4^+$ is a smooth function of $(\mathring{V},\nabla\mathring{\Psi})$ that vanishes at the origin.
	We then employ Lemma~\ref{lem.eig3} and take $K$ in \eqref{bas.c2} sufficiently small
	to obtain that $\RE (-\omega_{+}+a_{11}^++a_{22}^+)\gtrsim \gamma$.
	Apply G{\aa}rding's inequality (Lemma~\ref{lem.para1}\;(vi)) to obtain
	\begin{align*}
	\RE\big\llangle \Lambda^{1,\gamma}Z_3^+,T^{\gamma}_{-\omega_{+}+a_{11}^++a_{22}^+}\Lambda^{1,\gamma}Z_3^+\big\rrangle\gtrsim \gamma\VERT \Lambda^{1,\gamma}Z_3^+\VERT^2\gtrsim \gamma\VERT Z_3^+\VERT_{1,\gamma}^2,
	\end{align*}
	from which we have
	\begin{align} \notag
	\RE\big\llangle \Lambda^{1,\gamma}Z_3^+,\Lambda^{1,\gamma}T^{\gamma}_{-\omega_{+}+a_{11}^+
		+a_{22}^+}Z_3^+\big\rrangle\gtrsim (\gamma-C)\VERT Z_3^+\VERT_{1,\gamma}^2.
	\end{align}
	Choosing $\varepsilon$ small and $\gamma$ large, we derive from \eqref{pole.e1} that
	\begin{align} \notag
	&\gamma\VERT Z_3^+\VERT_{1,\gamma}^2+\|Z^+_{3}|_{x_2=0}\|_{1,\gamma}^2\\
	&\,\,\, \lesssim \VERT Z^+\VERT_{1,\gamma}^2+ \gamma^{-1}\big(\VERT T^{\gamma}_r W^+\VERT_{1,\gamma}^2
	+\VERT T^{\gamma}_{\chi_p^+}F^+_3\VERT_{1,\gamma}^2+\VERT W^+ \VERT^2 \big). \label{pole.e2}
	\end{align}
	
	\vspace{1mm}
	\noindent
	{\bf 2.}\;{\it Estimate for $Z_1^+$}.\quad
	The equation for $Z_1^+$ in \eqref{pole2.eq} is as follows:
	\begin{align} \label{pole2.eq4}
	T^{\gamma}_{b_{11}^+}Z_1^++T^{\gamma}_{k_2b_{11}^++b_{13}^+}Z_3^+=T^{\gamma}_{\alpha_{0}}Z^+
	+T^{\gamma}_r W^++\mathcal{R}_0T^{\gamma}_{\chi_p^+}F^+_1+\mathcal{R}_{-1}W^+.
	\end{align}
	Recall from \eqref{b11} that $\RE b_{11}^+=\mathbb{F}_1^+\gamma$ and $\RE (k_2b_{11}^++b_{13}^+)=\gamma\alpha_0$.
	Similar to Step\;1, we take the scalar product in $L^2(\Omega)$ of \eqref{pole2.eq4} with $\Lambda^{2,\gamma}Z_1^+$
	and use \eqref{pole.e2} to obtain
	\begin{align}
	\notag &\gamma\VERT Z_1^+\VERT_{1,\gamma}^2\\
	& \lesssim \VERT Z^+\VERT_{1,\gamma}^2+\gamma \VERT Z_3^+\VERT_{1,\gamma}^2+ \gamma^{-1}\big(\VERT T^{\gamma}_r W^+\VERT_{1,\gamma}^2
	+\VERT T^{\gamma}_{\chi_p^+}F^+_1\VERT_{1,\gamma}^2+\VERT W^+ \VERT^2 \big) \notag
	\\ \label{pole.e3}
	& \lesssim \VERT Z^+\VERT_{1,\gamma}^2+ \gamma^{-1}\big(\VERT T^{\gamma}_r W^+\VERT_{1,\gamma}^2
	+\VERT T^{\gamma}_{\chi_p^+}F^+\VERT_{1,\gamma}^2+\VERT W^+ \VERT^2 \big)
	\end{align}
	for $\gamma$ sufficiently large.
	
	\vspace{2mm}
	\noindent
	{\bf 3.}\; {\it Estimate for $Z_2^+$}.\quad   The equation for $Z_2^+$ in \eqref{pole2.eq} reads
	\begin{align} \notag   
	\partial_2Z_2^+=\,&T^{\gamma}_{(\mathring{a}_+a_{12}^+)^{-1}b_{21}^+}Z_1^++T^{\gamma}_{\omega_{+}}Z_2^+
	+T^{\gamma}_{\alpha_{0}}Z^+\\
	&\notag +T^{\gamma}_r W^++\mathcal{R}_0T^{\gamma}_{\chi_p^+}F^+_2+\mathcal{R}_{-1}W^+.
	\end{align}
	We note that $\RE \omega_+\lesssim -\gamma$ and $\RE ((\mathring{a}_+a_{12}^+)^{-1}b_{21}^+)=\gamma\alpha_0$.
	Employing a similar analysis as Step 1 and using \eqref{pole.e3} yield
	\begin{align} \notag
	&\gamma\VERT Z_2^+\VERT_{1,\gamma}^2 -C \|Z^+_{2}|_{x_2=0}\|_{1,\gamma}^2\\ \notag
	&\,\, \lesssim \VERT Z^+\VERT_{1,\gamma}^2+\gamma\VERT Z_1^+\VERT_{1,\gamma}^2
	+ \gamma^{-1}\big(\VERT T^{\gamma}_r W^+\VERT_{1,\gamma}^2+\VERT T^{\gamma}_{\chi_p^+}F^+_2\VERT_{1,\gamma}^2+\VERT W^+ \VERT^2 \big)\\
	&\,\, \lesssim \VERT Z^+\VERT_{1,\gamma}^2
	+ \gamma^{-1}\big(\VERT T^{\gamma}_r W^+\VERT_{1,\gamma}^2+\VERT T^{\gamma}_{\chi_p^+}F^+\VERT_{1,\gamma}^2+\VERT W^+ \VERT^2 \big). \label{pole.e4}
	\end{align}
	
	\vspace{1mm}
	\noindent
	{\bf 4.}\;\,  Combine estimates \eqref{pole.e2} and \eqref{pole.e3}--\eqref{pole.e4},
	and take $\gamma$ suitably large to find
	\begin{align}\notag
	&\gamma\VERT Z^+\VERT_{1,\gamma}^2+\|Z^+_{3}|_{x_2=0}\|_{1,\gamma}^2\\
	&\,\,\, \lesssim \|Z^+_{2}|_{x_2=0}\|_{1,\gamma}^2+ \gamma^{-1}\big(\VERT T^{\gamma}_r W^+\VERT_{1,\gamma}^2
	+\VERT T^{\gamma}_{\chi_p^+}F^+\VERT_{1,\gamma}^2+\VERT W^+ \VERT^2 \big). \label{pole.e5}
	\end{align}
	The derivation for the estimate of $Z^-$ is entirely similar so that
	\begin{align}\notag
	&\gamma\VERT Z^{\pm}\VERT_{1,\gamma}^2+\|Z^{\pm}_{3}|_{x_2=0}\|_{1,\gamma}^2\\
	&\,\,\, \lesssim \|Z^{\pm}_{2}|_{x_2=0}\|_{1,\gamma}^2+ \gamma^{-1}\big(\VERT T^{\gamma}_r W\VERT_{1,\gamma}^2
	+\VERT T^{\gamma}_{\chi_p^+}F\VERT_{1,\gamma}^2+\VERT W \VERT^2 \big). \label{pole.e6}
	\end{align}
	
	\vspace{1mm}
	\noindent
	{\bf 5.}\;{\it Estimate on the boundary}.\ \
	It remains to make an estimate for $ \|Z^{\pm}_{2}|_{x_2=0}\|_{1,\gamma}$.
	Using the boundary conditions \eqref{pole.bdy}, we have
	\begin{align} \label{pole.bdy1}
	\|T^{\gamma}_{\widetilde{\bm{\beta}}}Z_{2}|_{x_2=0}\|_{1,\gamma}^2
	\lesssim \|Z^{\pm}_{3}|_{x_2=0}\|_{1,\gamma}^2+\|G\|_{1,\gamma}^2+\|{W^{\mathrm{nc}}}|_{x_2=0}\|^2,
	\end{align}
	where $\widetilde{\bm{\beta}}:=\bm{\beta}(E_+\ E_-)$ and $Z_2:=(Z_2^+,Z_2^-)^{\mathsf{T}}$.
	Setting $V^{\pm}:=T^{\gamma}_{\widetilde{R}^{-1}}T^{\gamma}_{\chi_p^+}W_{\pm}^{\mathrm{nc}}$,
	we see from \eqref{pole.id3} that
	\begin{align} \label{Z2.V1}
	Z_2^{\pm}=T^{\gamma}_{\widetilde{\chi}}V_1^{\pm}+\mathcal{R}_{-1}W^{\mathrm{nc}}.
	\end{align}
	Since $\widetilde{\bm{\beta}}\in{\mathrm{\Gamma}_2^0}$,
	we apply the rule of symbolic calculus (Lemma \ref{lem.para1} (iv)) to find that
	\begin{align*}
	T^{\gamma}_{\widetilde{\bm{\beta}}}\mathcal{R}_{-1}=\mathcal{R}_{-1},\quad
	(\Lambda^{1,\gamma}T^{\gamma}_{\widetilde{\bm{\beta}}})^*\Lambda^{1,\gamma}T^{\gamma}_{\widetilde{\bm{\beta}}}
	- T^{\gamma}_{\lambda^{2,\gamma}\widetilde{\bm{\beta}}^*\widetilde{\bm{\beta}}}=\mathcal{R}_{-1}.
	\end{align*}
	Thus, we have
	\begin{align*}
&	\|T^{\gamma}_{\widetilde{\bm{\beta}}}Z_{2}|_{x_2=0}\|_{1,\gamma}^2\\
&\quad 	\gtrsim  \RE \big\langle  (\Lambda^{1,\gamma}T^{\gamma}_{\widetilde{\bm{\beta}}})^*\Lambda^{1,\gamma}
	T^{\gamma}_{\widetilde{\bm{\beta}}}T^{\gamma}_{\widetilde{\chi}}V_{1}|_{x_2=0},T^{\gamma}_{\widetilde{\chi}}V_{1}|_{x_2=0}\big\rangle
	-C\|{W^{\mathrm{nc}}}|_{x_2=0}\|^2
	\\
	&\quad \gtrsim \RE \big\langle  T^{\gamma}_{\lambda^{2,\gamma}\widetilde{\bm{\beta}}^*\widetilde{\bm{\beta}}}
	T^{\gamma}_{\widetilde{\chi}}V_{1}|_{x_2=0},T^{\gamma}_{\widetilde{\chi}}V_{1}|_{x_2=0}\big\rangle\\
	&\qquad\, -C\|T^{\gamma}_{\widetilde{\chi}}V_{1}|_{x_2=0}\|_{1,\gamma}\|T^{\gamma}_{\widetilde{\chi}}V_{1}|_{x_2=0}\|-C\|{W^{\mathrm{nc}}}|_{x_2=0}\|^2
	\end{align*}
	for $V_1:=(V_1^+,V_1^-)^{\mathsf{T}}$, which, combined with \eqref{pole.bdy1}, implies
	\begin{align}\notag
	&\RE \big\langle  T^{\gamma}_{\lambda^{2,\gamma}\widetilde{\bm{\beta}}^*\widetilde{\bm{\beta}}}
	T^{\gamma}_{\widetilde{\chi}}V_{1}|_{x_2=0},T^{\gamma}_{\widetilde{\chi}}V_{1}|_{x_2=0}\big\rangle\\[1mm]
	&\,\,\, \lesssim \|Z^{\pm}_{3}|_{x_2=0}\|_{1,\gamma}^2+\|G\|_{1,\gamma}^2+\gamma^{-1}\|Z^{\pm}_{2}|_{x_2=0}\|_{1,\gamma}^2
	+\|{W^{\mathrm{nc}}}|_{x_2=0}\|^2. \label{V1.e1}
	\end{align}
	Recall from Lemma \ref{lem.Lopa2} that the Lopatinski\u{\i} determinant $\Delta$ does not vanish on $\supp\widetilde{\chi}_1$,
	owing to \eqref{chi.p}.
	It then follows from definition \eqref{Lopa2a} of $\Delta$ that
	\begin{align*}
	\widetilde{\chi}_1^2\RE(\lambda^{2,\gamma}\widetilde{\bm{\beta}}^*\widetilde{\bm{\beta}})\gtrsim \widetilde{\chi}_1^2 \lambda^{2,\gamma} I.
	\end{align*}
	Then we can employ the localized G{\aa}rding's inequality (Lemma \ref{lem.para1} (vii)) and utilize \eqref{Z2.V1}  to derive
	\begin{align}
	&\RE \big\langle  \notag  T^{\gamma}_{\lambda^{2,\gamma}\widetilde{\bm{\beta}}^*\widetilde{\bm{\beta}}}
	T^{\gamma}_{\widetilde{\chi}}V_{1}|_{x_2=0},T^{\gamma}_{\widetilde{\chi}}V_{1}|_{x_2=0}\big\rangle \\[1mm]
	&\,\,\, \gtrsim
	\|T^{\gamma}_{\widetilde{\chi}}V_{1}|_{x_2=0}\|_{1,\gamma}^2-C\|V_{1}|_{x_2=0}\|^2\gtrsim \|Z_{2}|_{x_2=0}\|_{1,\gamma}^2-C\|{W^{\mathrm{nc}}}|_{x_2=0}\|^2.
	\label{V1.e2}
	\end{align}
	Combine \eqref{V1.e1} with \eqref{V1.e2} and take $\gamma$ small to infer that
	\begin{align}\label{pole.bdy2}
	\|Z^{\pm}_{2}|_{x_2=0}\|_{1,\gamma}^2 \lesssim
	\|Z^{\pm}_{3}|_{x_2=0}\|_{1,\gamma}^2+\|G\|_{1,\gamma}^2
	+\|{W^{\mathrm{nc}}}|_{x_2=0}\|^2.
	\end{align}
	We combine \eqref{pole.bdy2} with \eqref{pole.e2} to eliminate the first term on the right hand side of \eqref{pole.bdy2},
	and then use \eqref{pole.e6} to conclude estimate \eqref{pole.e.1}.
	This completes the proof.
\qed\end{proof}

Recall that $\chi_p=\chi_p^++\chi_p^-$ and $\supp \chi_p^+\cap \supp\chi_p^-=\varnothing$.
Shrinking the support of $\chi_p$ if necessary, we obtain the following result from Lemma~\ref{lem.Zp}.
\begin{proposition}\  \label{pro.pole}
	There exist constants $K_0\leq 1$ and $\gamma_0\geq 1$ such that, if $\gamma\geq \gamma_0$ and $K\leq K_0$ for $K$ given in \eqref{bas.c2},
	then
	\begin{align}\notag
	&\gamma\VERT T^{\gamma}_{\chi_p}W\VERT_{1,\gamma}^2+\|T^{\gamma}_{\chi_p}{W^{\mathrm{nc}}}|_{x_2=0}\|_{1,\gamma}^2\\[1mm]
	& \lesssim \|G\|_{1,\gamma}^2+\|{W^{\mathrm{nc}}}|_{x_2=0}\|^2
	+\gamma^{-1}\big(\VERT T^{\gamma}_{\chi_p}F\VERT_{1,\gamma}^2+\VERT W\VERT^2+\VERT T^{\gamma}_{r}W\VERT_{1,\gamma}^2 \big),
	\label{pole.e}
	\end{align}
	where symbol $r\in{\mathrm{\Gamma}_1^0}$ vanishes in region $\{\chi_p\equiv1\}\cup\{\chi_p\equiv0\}$.
\end{proposition}

\subsection{Estimate near Bad Frequencies}
We now show the energy estimate near the points
in $\Upsilon_c=\cup_{q\in\{0,\pm1\}}\Upsilon_c^q$, {\it i.e.}  near the zeros of the Lopatinski\u{\i} determinant.
We consider the case near set $\Upsilon_c^0$, without loss of generality.
To this end, we introduce two smooth cut-off functions $\chi_1$ and $\chi_2$ with values in $[0,1]$ such that
\begin{list}{}{\setlength{\parsep}{\parskip}
		\setlength{\itemsep}{0.1em}
		\setlength{\labelwidth}{2em}
		\setlength{\labelsep}{0.4em}
		\setlength{\leftmargin}{2.2em}
		\setlength{\topsep}{1mm}
	}
	\item[--] $\chi_1\equiv 1$ on the support of $\chi_c^0$, $\chi_2\equiv 1$ on the support of $\chi_c^0$, and $\supp\chi_2\subset \mathscr{V}_c^0$;
	\item[--] $\chi_1$ and $\chi_2$ are both $C^{\infty}$ and homogeneous of degree $0$ with respect to $(\tau,\eta)$,
\end{list}
where $\chi_c^0$ is given at the end of \S\,\ref{sec.micro}. Defining
\begin{align} \notag  
{\bm{w}}^{\pm}:=T^{\gamma}_{\chi_c^0}W^{\pm},\quad \bm{w}_{\pm}^{\mathrm{nc}}:=(\bm{w}^{\pm}_2,\bm{w}^{\pm}_3)^{\mathsf{T}},
\end{align}
we perform similar calculations as we have done in \S\,\ref{sec.eq} to obtain the following system:
\begin{align} \label{bad1.eq}
\partial_2 \bm{w}_{\pm}^{\mathrm{nc}}=
T^{\gamma}_{\mathbb{A}_{\chi_2}^{\pm}}\bm{w}_{\pm}^{\mathrm{nc}}
+T^{\gamma}_{\mathbb{E}^{\pm}}\bm{w}_{\pm}^{\mathrm{nc}}+
T_{r}^{\gamma}W^{\pm}+\mathcal{R}_0T^{\gamma}_{\chi_c^0}F^{\pm}+\mathcal{R}_{-1}W^{\pm},
\end{align}
where $\mathbb{E}^{\pm}\in{\mathrm{\Gamma}_1^0}$, $\mathbb{A}_{\chi_2}^{\pm}\in{\mathrm{\Gamma}_2^1}$ is given in \eqref{A.bb.r}
with $\widetilde{\chi}_{\pm}$ replaced by $\chi_2$,
and $r\in{\mathrm{\Gamma}_1^0}$ vanishes in region $\{\chi_c^0\equiv1 \}\cup\{\chi_c^0\equiv0 \}$.

Since matrix $\mathbb{A}_{\chi_2}^{\pm}\equiv \mathbb{A}_{\pm}$ in region $ \{\chi_2\equiv1 \}$,
we obtain from Proposition~\ref{pro.micro} that
\begin{align} \label{Q0.id}
Q_0^{\pm}\mathbb{A}_{\chi_2}^{\pm}=\mathbb{D}_1^{\pm}Q_0^{\pm}\qquad
\mathrm{in}\ \{\chi_2\equiv 1\}.
\end{align}
More precisely, we have
\begin{align} \label{Q0}
\begin{pmatrix}
Q_0^{+} &0\\ 0& Q_0^-
\end{pmatrix}^{-1}=(E_{+}\ Y_{+}\ E_{-}\ Y_{-}),
\end{align}
where $E_{\pm}$, $Y_{+}$, and $Y_-$ are defined by \eqref{E.eig2} and \eqref{Y.r}--\eqref{Y.l},
respectively.
Then the following lemma can be proved as in \cite[Page\;425]{C04MR2069632} by using \eqref{Q0.id}.

\begin{lemma}\  \label{lem.bad1}
	There exist symbols $Q_{-1}^{\pm}\in{\mathrm{\Gamma}_1^{-1}}$ and diagonal symbols $\mathbb{D}_0^{\pm}\in{\mathrm{\Gamma}_1^0}$,
	which are defined in region $\{\chi_2\equiv 1\}$, such that
	\begin{align*}
	&(Q_0^{\pm}+Q_{-1}^{\pm})(\mathbb{A}_{\chi_2}^{{\pm}}+\mathbb{E}^{\pm})
	-(\mathbb{D}_1^{\pm}+\mathbb{D}_0^{\pm})(Q_0^{\pm}+Q_{-1}^{\pm})+\partial_{2} Q_0^{\pm}\\
	&\,\,\, -\mathrm{i}(\partial_{\delta}Q_0^{\pm}\partial_t \mathbb{A}_{\chi_2}^{\pm}+\partial_{\eta}Q_0^{\pm}\partial_{x_1} \mathbb{A}_{\chi_2}^{\pm}-\partial_{\delta}\mathbb{D}_1^{\pm}
	\partial_t Q_0^{\pm}-\partial_{\eta}\mathbb{D}_1^{\pm}\partial_{x_1} Q_0^{\pm})\in {\mathrm{\Gamma}_1^{-1}}.
	\end{align*}
\end{lemma}

We now prove the estimates for
\begin{align}\label{bad2}
Z^{\pm}:=T^{\gamma}_{\chi_1(Q_0^{\pm}+Q_{-1}^{\pm})}\bm{w}_{\pm}^{\mathrm{nc}},
\end{align}
which will be shown to satisfy the paradifferential equations
with diagonal principle symbols.

In fact, using  Lemmas~\ref{lem.para1} and \ref{lem.bad1},
we see from \eqref{bad1.eq} that
\begin{align} \label{bad2.eq}
\partial_2 Z^+=T^{\gamma}_{\widetilde{\mathbb{D}}_1^+}Z^++T^{\gamma}_{\widetilde{\mathbb{D}}_0^+}Z^+
+T^{\gamma}_{r}W^++\mathcal{R}_0T^{\gamma}_{\chi_c^0}F^{+}+\mathcal{R}_{-1}W^{+},
\end{align}
where $\widetilde{\mathbb{D}}_1^+$ (resp.\;$\widetilde{\mathbb{D}}_0^+$) is an extension of $\mathbb{D}_1^+$ (resp.\;$\mathbb{D}_0^+$)
to the whole set $\widebar{\Omega}\times \Xi$.
Thanks to Lemma \ref{lem.eig3}, these extension can be chosen such that
\setlength{\arraycolsep}{2pt}
\begin{align} \label{D1}
\widetilde{\mathbb{D}}_1^+=\begin{pmatrix}
\omega_+&0\\ 0& \omega_+'
\end{pmatrix}=\begin{pmatrix}
\gamma e_++\mathrm{i}\hbar_+&0\\
0& \gamma e_+'+\mathrm{i}\hbar_+'
\end{pmatrix},\ \
\widetilde{\mathbb{D}}_0^+=\mathrm{diag}\,(d_+,\,d_+'),
\end{align}
\setlength{\arraycolsep}{4pt}
where $e_+,e_+'\in{\mathrm{\Gamma}_2^0}$ and $\hbar_+,\hbar_+'\in{\mathrm{\Gamma}_2^1}$ are real-valued symbols,
and $d_+,d_+'\in{\mathrm{\Gamma}_1^0}$ such that
\begin{align*}
e_+\lesssim -1,\quad e_+'\gtrsim 1.
\end{align*}

We obtain the following result for functions $Z^{\pm}$ that are given in \eqref{bad2}:

\begin{lemma}\ \label{lem.bad2}
	There exist constants $K_0\leq 1$ and $\gamma_0\geq 1$ such that, if $\gamma\geq \gamma_0$ and $K\leq K_0$ for $K$ given in \eqref{bas.c2},
	then
	\begin{align}  \notag
	&\gamma^3\VERT Z_1\VERT ^2+\gamma\VERT Z_2\VERT_{1,\gamma}^2+\|Z_{2}|_{x_2=0}\|_{1,\gamma}^2
	+\gamma^2 \|Z_{1}|_{x_2=0}\|^2+\|T^{\gamma}_{\widetilde{\sigma}_0}Z_{1}|_{x_2=0}\|^2\\  \label{bad2.e0}
	&\,\,\, \lesssim \gamma^{-1}\big(\VERT T^{\gamma}_r W\VERT_{1,\gamma}^2
	+\VERT T^{\gamma}_{\chi_c^0}F\VERT_{1,\gamma}^2+\VERT W \VERT^2 \big)+\|G\|_{1,\gamma}^2+\|{W^{\mathrm{nc}}}|_{x_2=0}\|^2,
	\end{align}
	where $Z_j:=(Z_j^+,Z_j^-)^{\mathsf{T}}, j=1,2$,  $\widetilde{\sigma}_0$ is the scalar real symbol given by \eqref{sigma},
	and  $r\in{\mathrm{\Gamma}_1^0}$ vanishes in region $\{\chi_c^0\equiv1\}\cup\{\chi_c^0\equiv0\}$.
\end{lemma}

\begin{proof}\  The proof is divided into two steps.
	
	\smallskip
	{\bf 1.}\;{\it Estimate in domain $\Omega$}. \quad    The first equation in \eqref{bad2.eq} reads
	\begin{align}\label{bad2.eq1}
	\partial_2 Z^+_1=T^{\gamma}_{\omega_+}Z^+_1+T^{\gamma}_{d_+}Z^+_1+T^{\gamma}_{r}W^+
	+\mathcal{R}_0T^{\gamma}_{\chi_c^0}F^{+}+\mathcal{R}_{-1}W^{+}.
	\end{align}
	Recalling that $\RE\omega_{+}=\gamma e_+\lesssim -\gamma$, we choose the identity as a symmetrizer
	and obtain the following $L^2$ estimate:
	\begin{align}
	&\gamma\VERT Z_1^+\VERT ^2\lesssim \|Z^+_{1}|_{x_2=0}\|^2
	+ \gamma^{-1}\VERT(T^{\gamma}_r W^+,\mathcal{R}_0T^{\gamma}_{\chi_c^0}F^+,\mathcal{R}_{-1}W^+)\VERT^2\notag \\
	&\lesssim \|Z^+_{1}|_{x_2=0}\|^2
	+ \gamma^{-3}\big(\VERT T^{\gamma}_r W^+\VERT_{1,\gamma}^2+\VERT T^{\gamma}_{\chi_c^0}F^+\VERT_{1,\gamma}^2+\VERT W^+ \VERT^2 \big)
	\label{bad2.e1}
	\end{align}
	for sufficiently large $\gamma$.
	
	The second equation in \eqref{bad2.eq} reads
	\begin{align*}
	\partial_2 Z^+_2=T^{\gamma}_{\omega_+'}Z^+_2+T^{\gamma}_{d_+'}Z^+_2+T^{\gamma}_{r}W^+
	+\mathcal{R}_0T^{\gamma}_{\chi_c^0}F^{+}+\mathcal{R}_{-1}W^{+}.
	\end{align*}
	Recalling that $\RE\omega_{+}'=\gamma e_+'\gtrsim \gamma$,
	we perform a similar calculation as Step\;1 in the proof of Lemma \ref{lem.Zp} to deduce
	\begin{align}
\notag &	\gamma\VERT  Z_2^+\VERT_{1,\gamma} ^2+\|Z^+_{2}|_{x_2=0}\|_{1,\gamma}^2\\
	&\quad \lesssim  \gamma^{-1}\big(\VERT T^{\gamma}_r W^+\VERT_{1,\gamma}^2+\VERT T^{\gamma}_{\chi_c^0}F^+\VERT_{1,\gamma}^2
	+\VERT W^+ \VERT^2 \big)
	\label{bad2.e2}
	\end{align}
	for sufficiently large $\gamma$.
	
	A similar analysis enables us to deduce the energy estimates for $Z_1^-$ and $Z_2^-$
	as \eqref{bad2.e1} and \eqref{bad2.e2}. The combination of all these estimates is
	\begin{align}  \notag
	&\gamma^3\VERT Z_1\VERT^2+\gamma\VERT Z_2\VERT_{1,\gamma}^2+\|Z_{2}|_{x_2=0}\|_{1,\gamma}^2\\  \label{bad2.e3}
	&\,\,\, \lesssim\gamma^2 \|Z_{1}|_{x_2=0}\|^2+\gamma^{-1}\big(\VERT T^{\gamma}_r W\VERT_{1,\gamma}^2
	+\VERT T^{\gamma}_{\chi_c^0}F\VERT_{1,\gamma}^2+\VERT W \VERT^2 \big).
	\end{align}
	
	\vspace{1mm}
	\noindent{\bf 2.}\; {\it Estimate for the boundary terms}.\quad
	We now estimate the traces of the incoming modes $Z_1$ in terms of the outgoing models $Z_2$
	and the source term $G$.
	Using the boundary condition \eqref{p.p2.c} yields
	\begin{align*}
	T^{\gamma}_{\bm{\beta}}\bm{w}^{\mathrm{nc}}=G+\mathcal{R}_{-1}W^{\mathrm{nc}}\qquad \mathrm{if}\ x_2=0.
	\end{align*}
	From the proof of Lemma \ref{lem.Lopa2},
	we find that $\mathring{\zeta}_1\ne 0$, and
	\begin{align*}
	(\mathring{\zeta}_1 \mathring{\zeta}_4-\mathring{\zeta}_2\mathring{\zeta}_3)|_{x_2=0}=\Delta
	=(\tau-\mathrm{i}\mathring{z}_0\eta) h_0(t,x_1,\tau,\eta),\quad h_0(t,x_1,\tau,\eta)\ne 0
	\end{align*}
	in a neighborhood of $(\mathrm{i}\mathring{z}_{0}\eta,\eta)\in\Xi_1$.
	According to identity \eqref{Coulombel}, we define
	the following invertible matrices in a suitably small neighborhood of $\Upsilon_{c}^0$:
	\begin{align*}
	P_1=\begin{pmatrix}
	1/\mathring{\zeta}_1 & 0\\[0.5mm]
	-\mathring\zeta_3/(\mathring\zeta_1 \zeta_5) &1/\zeta_5
	\end{pmatrix}, \qquad
	P_2=\begin{pmatrix}
	1 & -\mathring\zeta_2 \\[0.5mm]  0  &\mathring\zeta_1
	\end{pmatrix},
	\end{align*}
	with $\zeta_5:=h_0(t,x_1,\tau,\eta)$ such that $P_1$ and $P_2$ belong to ${\mathrm{\Gamma}_2^0}$.
	Shrinking $\mathscr{V}_c^0$ if necessary, we have
	\begin{align}\label{beta.in}
	\bm{\beta}_{\mathrm{in}}:=P_1 \bm{\beta}(E_+\ E_-)P_2
	=\begin{pmatrix}
	1&0\\ 0&\lambda^{-1,\gamma}(\gamma+\mathrm{i}\widetilde{\sigma}_0)
	\end{pmatrix}
	\qquad \mathrm{in}\ \mathscr{V}_c^0,
	\end{align}
	where $\widetilde{\sigma}_0=\delta-\mathring{z}_0\eta$ is the scalar real symbol in ${\mathrm{\Gamma}_2^1}$.
	We recall from \eqref{sigma} that $\sigma_0=-\mathrm{i}\gamma+\widetilde{\sigma}_0$.
	
	We then fix the four cut-off functions $\chi_{c_1}$, $\chi_{c_2}$, $\chi_{c_3}$,
	and $\chi_{c_4}$ such that
	\begin{list}{}{\setlength{\parsep}{\parskip}
			\setlength{\itemsep}{0.1em}
			\setlength{\labelwidth}{2em}
			\setlength{\labelsep}{0.4em}
			\setlength{\leftmargin}{2.2em}
			\setlength{\topsep}{1mm}
		}
		\item[--] $\chi_{c_1}\equiv 1$ in a neighborhood of $\mathrm{supp}\,\chi_1\cap\{x_2=0\}$;
		\item[--] $\chi_{c_j}\equiv 1$ in a neighborhood of $\mathrm{supp}\,\chi_{c_{j-1}}$ for $j=2,3,4$;
		\item[--] $\mathrm{supp}\,\chi_{c_{4}}\subset \mathscr{V}_c^0\cap\{x_2=0\}$.
	\end{list}
	As in \cite[\S\,3.4.3]{C04MR2069632}, the following estimate can be obtained
	by using the localized G{\aa}rding's inequality:
	\begin{align} \notag
	&\|T^{\gamma}_{\chi_{c_2}\lambda^{1,\gamma}\bm{\beta}_{\mathrm{in}}}T^{\gamma}_{\chi_{c_1}}T^{\gamma}_{\chi_{c_4}P_2^{-1} }Z_{1}|_{x_2=0}\|
\\\label{bad2.e4} & \quad 	\lesssim
	\|G\|_{1,\gamma}+\|Z_{2}|_{x_2=0}\|_{1,\gamma}+\|{W^{\mathrm{nc}}}|_{x_2=0}\|.
	\end{align}
	Now we utilize the special structure of $\bm{\beta}_{\mathrm{in}}$ to derive a lower bound
	for the term on the left-hand side of \eqref{bad2.e4}.
	Setting
	\begin{align}\label{bad3}
	(\upsilon_1,\upsilon_2)^{\mathsf{T}}:=T^{\gamma}_{\chi_{c_4}P_2^{-1} }Z_{1}|_{x_2=0},
	\end{align}
	we obtain from \eqref{beta.in} that
	\begin{align} \notag
	&\|T^{\gamma}_{\chi_{c_2}\lambda^{1,\gamma}\bm{\beta}_{\mathrm{in}}}T^{\gamma}_{\chi_{c_1}}T^{\gamma}_{\chi_{c_4}P_2^{-1} }Z_{1}|_{x_2=0}\|^2\\
	&\quad =\|T^{\gamma}_{\chi_{c_2}\lambda^{1,\gamma}}T^{\gamma}_{\chi_{c_1}}\upsilon_1\|^2
	+\|T^{\gamma}_{\chi_{c_2}(\gamma+\mathrm{i}\widetilde{\sigma}_0)}T^{\gamma}_{\chi_{c_1}}\upsilon_2\|^2. \label{bad3.e1}
	\end{align}
	Use Lemma \ref{lem.para1}\;(iv) and apply the localized G{\aa}rding's inequality (Lemma \ref{lem.para1}\;(vii))
	to obtain
	\begin{align*}
	&\|T^{\gamma}_{\chi_{c_2}\lambda^{1,\gamma}}T^{\gamma}_{\chi_{c_1}}\upsilon_1\|^2
	=\big\langle\big(T^{\gamma}_{\chi_{c_2}\lambda^{1,\gamma}}\big)^*T^{\gamma}_{\chi_{c_2}
		\lambda^{1,\gamma}}T^{\gamma}_{\chi_{c_1}}\upsilon_1, T^{\gamma}_{\chi_{c_1}}\upsilon_1 \big\rangle\\
	&\quad \geq \RE \big\langle T^{\gamma}_{\chi_{c_2}^2\lambda^{2,\gamma}}T^{\gamma}_{\chi_{c_1}}\upsilon_1,
	T^{\gamma}_{\chi_{c_1}}\upsilon_1 \big\rangle
	-C\big\|T^{\gamma}_{\chi_{c_1}}\upsilon_1\big\|\big\|T^{\gamma}_{\chi_{c_1}}\upsilon_1\big\|_{1,\gamma}\\
	&\quad \geq
	c\|T^{\gamma}_{\chi_{c_1}}\upsilon_1\|_{1,\gamma}^2-C\|\upsilon_1\|^2-C\|T^{\gamma}_{\chi_{c_1}}\upsilon_1\|^2\\
	&\quad \gtrsim
	\|\upsilon_1\|_{1,\gamma}^2-C\|Z_{1}|_{x_2=0}\|^2
	\gtrsim \gamma^2\|\upsilon_1\|^2+\|T^{\gamma}_{\widetilde{\sigma}_0}\upsilon_1\|^2-C\|Z_{1}|_{x_2=0}\|^2
	\end{align*}
	for large enough $\gamma$.
	Similarly, we obtain that, for sufficiently large $\gamma$,
	\begin{align*}
	\|T^{\gamma}_{\chi_{c_2}(\gamma+\mathrm{i}\widetilde{\sigma}_0)}T^{\gamma}_{\chi_{c_1}}\upsilon_2\|^2\gtrsim \gamma^2\|\upsilon_2\|^2+\|T^{\gamma}_{\widetilde{\sigma}_0}\upsilon_2\|^2-C\|Z_{1}|_{x_2=0}\|^2.
	\end{align*}
	Plug the above two estimates into \eqref{bad3.e1} to infer
	\begin{align}
	\notag&\|T^{\gamma}_{\chi_{c_2}\lambda^{1,\gamma}\bm{\beta}_{\mathrm{in}}}T^{\gamma}_{\chi_{c_1}}
	T^{\gamma}_{\chi_{c_4}P_2^{-1} }Z_{1}|_{x_2=0}\|^2 \\
	&\quad  \label{bad3.e2} \gtrsim
	\gamma^2\|(\upsilon_1,\upsilon_2)\|^2+\|T^{\gamma}_{\widetilde{\sigma}_0}(\upsilon_1,\upsilon_2)\|^2
	-C\|Z_{1}|_{x_2=0}\|^2.
	\end{align}
	Since $\chi_{c_3}\chi_1\equiv\chi_1$, we see from  \eqref{bad2} and \eqref{bad3} that
	\begin{align*}
	T^{\gamma}_{\chi_{c_3}}T^{\gamma}_{\widetilde{\sigma}_0}Z_1
	=T^{\gamma}_{\widetilde{\sigma}_0}T^{\gamma}_{\chi_{c_3}}Z_1+\mathcal{R}_0Z_1
	=T^{\gamma}_{\widetilde{\sigma}_0}Z_1+\mathcal{R}_0Z_1,
	\end{align*}
	so that
	\begin{align*}
	T^{\gamma}_{\widetilde{\sigma}_0}(\upsilon_1,\upsilon_2)
	&=T^{\gamma}_{\chi_{c_4}P_2^{-1}}T^{\gamma}_{\widetilde{\sigma}_0}Z_{1}|_{x_2=0}+\mathcal{R}_0Z_{1}|_{x_2=0}\\
	&=T^{\gamma}_{\chi_{c_4}P_2^{-1}}T^{\gamma}_{\chi_{c_3}}
	T^{\gamma}_{\widetilde{\sigma}_0}Z_{1}|_{x_2=0}+\mathcal{R}_0Z_{1}|_{x_2=0}.
	\end{align*}
	Thanks to the ellipticity of $(P_2^{-1})^*P_2^{-1}$ on the support of $\chi_{c_4}$,
	we apply the localized G{\aa}rding's inequality to obtain
	\begin{align*}
	&\|T^{\gamma}_{\widetilde{\sigma}_0}(\upsilon_1,\upsilon_2)\|^2\\
	 &\geq\big \langle \big(T^{\gamma}_{\chi_{c_{4}}P_2^{-1}}\big)^* T^{\gamma}_{\chi_{c_{4}}P_2^{-1}}T^{\gamma}_{\chi_{c_3}}T^{\gamma}_{\widetilde{\sigma}_0}
	Z_{1}|_{x_2=0},T^{\gamma}_{\chi_{c_3}}T^{\gamma}_{\widetilde{\sigma}_0}Z_{1}|_{x_2=0}\big \rangle
	-C\|Z_{1}|_{x_2=0}\|^2\\
	&\gtrsim\|T^{\gamma}_{\chi_{c_3}}T^{\gamma}_{\widetilde{\sigma}_0}Z_{1}|_{x_2=0}\|^2
	-C\|T^{\gamma}_{\widetilde{\sigma}_0}Z_{1}|_{x_2=0}\|_{-1,\gamma}^2
	\\ &\quad\  -C\|T^{\gamma}_{\chi_{c_3}}T^{\gamma}_{\widetilde{\sigma}_0}Z_{1}|_{x_2=0}\|_{-1,\gamma}^2
	-C\|Z_{1}|_{x_2=0}\|^2
	\end{align*}
	for large enough $\gamma$. Then we take $\gamma$ sufficiently large to deduce
	\begin{align} \label{bad3.e3}
	\|T^{\gamma}_{\widetilde{\sigma}_0}(\upsilon_1,\upsilon_2)\|^2
	\gtrsim \|T^{\gamma}_{\widetilde{\sigma}_0}Z_{1}|_{x_2=0}\|^2-C\|Z_{1}|_{x_2=0}\|^2.
	\end{align}
	Similarly, we have
	\begin{align} \notag
	\|(\upsilon_1,\upsilon_2)\|^2
	&\gtrsim \|Z_{1}|_{x_2=0}\|^2-C\|Z_{1}|_{x_2=0}\|_{-1,\gamma}^2\\
	&\gtrsim \|Z_{1}|_{x_2=0}\|^2
	-\frac{C}{\gamma^2}\|Z_{1}|_{x_2=0}\|^2. \label{bad3.e4}
	\end{align}
	
	Combining estimates \eqref{bad2.e3}, \eqref{bad2.e4}, and \eqref{bad3.e2}--\eqref{bad3.e4},
	we take $\gamma$ large enough to derive  \eqref{bad2.e0} and conclude the proof.
\qed\end{proof}

Recall that vectors $Z^{\pm}$ are defined by \eqref{bad2}
and that matrices $Q_0^{\pm}\in{\mathrm{\Gamma}_2^0}$ are invertible in a neighborhood of
the support of $\chi_1$ and $Q_{-1}^{\pm}\in{\mathrm{\Gamma}_1^{-1}}$.
It then follows from  Lemma \ref{lem.bad2} that
\begin{align}  \notag
&\gamma^3\VERT T^{\gamma}_{\chi_c^0}W^{\mathrm{nc}}\VERT ^2
+\gamma^2 \|T^{\gamma}_{\chi_c^0}{W^{\mathrm{nc}}}|_{x_2=0}\|^2+\|T^{\gamma}_{\widetilde{\sigma}_0}
T^{\gamma}_{\chi_c^0}{W^{\mathrm{nc}}}|_{x_2=0}\|^2\\[1mm]  \label{bad.e5}
&\,\,\, \lesssim \gamma^{-1}\big(\VERT T^{\gamma}_r W\VERT_{1,\gamma}^2
+\VERT T^{\gamma}_{\chi_c^0}F\VERT_{1,\gamma}^2+\VERT W \VERT^2 \big)
+\|G\|_{1,\gamma}^2+\|{W^{\mathrm{nc}}}|_{x_2=0}\|^2.
\end{align}
Noting that components $T^{\gamma}_{\chi_c^0}W_1^{\pm}$ are given in terms
of $T^{\gamma}_{\chi_c^0}W_2^{\pm}$ and $T^{\gamma}_{\chi_c^0}W_3^{\pm}$
by relation \eqref{omega1},
we can deduce an $L^2$ estimate for $T^{\gamma}_{\chi_c^0}W_1^{\pm}$,
that is, we can add the terms, $\gamma^3\VERT T^{\gamma}_{\chi_c^0}W_1^{\pm}\VERT^2$,
on the left-hand side of \eqref{bad.e5}.

The following proposition then follows by combining the estimates for the three
cases $\Upsilon_c^q$ with $q\in\{0,\pm 1\}$.

\begin{proposition}\ \label{pro.bad}
	There exist constants $K_0\leq 1$ and $\gamma_0\geq 1$ such that,
	if $\gamma\geq \gamma_0$ and $K\leq K_0$ for $K$ given in \eqref{bas.c2}, then
	\begin{align}  \notag
	&\gamma^3\VERT T^{\gamma}_{\chi_c}W \VERT^2+\gamma^2 \|T^{\gamma}_{\chi_c}{W^{\mathrm{nc}}}|_{x_2=0}\|^2+\sum_{q\in\{0,\pm1\}}
	\|T^{\gamma}_{\widetilde{\sigma}_q}T^{\gamma}_{\chi_c}{W^{\mathrm{nc}}}|_{x_2=0}\|^2\\  \label{bad.e0}
	& \lesssim \gamma^{-1}\left(\VERT T^{\gamma}_r W\VERT_{1,\gamma}^2
	+\VERT T^{\gamma}_{\chi_c}F\VERT_{1,\gamma}^2+\VERT W \VERT^2 \right)+\|G\|_{1,\gamma}^2+\|{W^{\mathrm{nc}}}|_{x_2=0}\|^2,
	\end{align}
	where $\widetilde{\sigma}_q\in{\mathrm{\Gamma}_2^1}$ is given in \eqref{sigma} and $r\in{\mathrm{\Gamma}_1^0}$ vanishes in
	region $\{\chi_c\equiv1\}\cup\{\chi_c\equiv0\}$.
\end{proposition}

\subsection{Proof of Theorem {\rm\ref{thm.2}}}\label{sec.thm1}
We now patch the microlocalized energy estimates \eqref{est.u}, \eqref{pole.e}, and \eqref{bad.e0}
together to deduce estimate \eqref{p.p2.e0}.
Since $\chi_u+\chi_p+\chi_c\equiv1$,
\begin{align}
\notag &\gamma^3\VERT W \VERT^2+\gamma^2\|{W^{\mathrm{nc}}}|_{x_2=0}\|^2\\
&\lesssim \gamma\VERT (T_{\chi_u}^{\gamma}W,T_{\chi_p}^{\gamma}W)\VERT_{1,\gamma}^2+\gamma^3\VERT T^{\gamma}_{\chi_c}W \VERT^2\notag \\
&\quad +\|(T_{\chi_u}^{\gamma}W^{\mathrm{nc}},T_{\chi_p}^{\gamma}W^{\mathrm{nc}})|_{x_2=0}\|_{1,\gamma}^2
+\gamma^2 \|T^{\gamma}_{\chi_c}{W^{\mathrm{nc}}}|_{x_2=0}\|^2.
\label{error.e3}
\end{align}
Adding estimates \eqref{est.u}, \eqref{pole.e}, and \eqref{bad.e0},
we use  \eqref{error.e3} and take $\gamma$ large enough to deduce
\begin{align} \notag
&\gamma\VERT (T_{\chi_u}^{\gamma}W,T_{\chi_p}^{\gamma}W)\VERT_{1,\gamma}^2
+\|(T_{\chi_u}^{\gamma}W^{\mathrm{nc}},T_{\chi_p}^{\gamma}W^{\mathrm{nc}})|_{x_2=0}\|_{1,\gamma}^2\\
&\notag  +\gamma^3\VERT T^{\gamma}_{\chi_c}W \VERT^2  +\gamma^2 \|T^{\gamma}_{\chi_c}{W^{\mathrm{nc}}}|_{x_2=0}\|^2\\
& \quad \qquad \lesssim  \|G\|_{1,\gamma}^2+\gamma^{-1}\VERT F\VERT_{1,\gamma}^2+\gamma^{-1}\VERT T^{\gamma}_r W\VERT_{1,\gamma}^2.
\label{error.e0}
\end{align}
In order to absorb the \emph{microlocalization error} term $\VERT T^{\gamma}_r W\VERT_{1,\gamma}$,
we decompose symbol $r$ in terms of $\chi_u$, $\chi_p$, and $\sigma_q$ ($q\in\{0,\pm1 \}$).
Notice that symbol $r\in{\mathrm{\Gamma}_1^0}$ vanishes in the region:
$$
\{\chi_c\equiv1 \}\cap\{\chi_p\equiv 0\}\cap \{\chi_u\equiv0 \}\equiv\{\chi_c\equiv1 \}.
$$
In region $\{\chi_c\leq 1/2\}$,  $\chi_u+\chi_p\geq \frac{1}{2}$, so that we can write
\begin{align*}
r=\alpha_u \chi_u+\alpha_p\chi_p,
\end{align*}
where matrices $\alpha_u$ and $\alpha_p$ belong to ${\mathrm{\Gamma}_1^0}$ and
have the same block diagonal structure as $\mathbb{A}$.
In region $\{\frac{1}{2}\leq \chi_c<1 \}$, we can utilize \eqref{key_error} to write
\setlength{\arraycolsep}{3pt}
\begin{align*}
r=\sum_{q\in\{0,\pm1\}}\alpha_c^q
\begin{pmatrix}
\sigma_{+}^q I_{3}&0 \\ 0 & \sigma_{-}^q I_{3}
\end{pmatrix}
\chi_c^q,
\end{align*}
where $\alpha_c^q\in{\mathrm{\Gamma}_1^{-1}}$ has the same block diagonal structure
as $\mathbb{A}$, and $\sigma_{\pm}^q$ are solutions to \eqref{sigma.eq}.
Thus we obtain
\begin{align}
\VERT T^{\gamma}_r W\VERT_{1,\gamma}
\lesssim
\VERT (T_{\chi_u}^{\gamma}W,T_{\chi_p}^{\gamma}W)\VERT_{1,\gamma}+
\sum_{q\in\{0,\pm1\}}\VERT T^{\gamma}_{\sigma_{\pm}^q}T^{\gamma}_{\chi_c^q}W \VERT.
\label{error.e4}
\end{align}
We now make the estimate for the last term in \eqref{error.e4} in the following lemma:

\begin{lemma}\ \label{lem.error}
	There exist constants $K_0\leq 1$ and $\gamma_0\geq 1$ such that,
	if $\gamma\geq \gamma_0$ and $K\leq K_0$ for $K$ given in \eqref{bas.c2},
	then, for $q\in\{0,\pm1\}$,
	\begin{align}
	\notag \gamma\VERT T^{\gamma}_{\sigma_{\pm}^q}T^{\gamma}_{\chi_c^q}W\VERT^2
	\lesssim\;&\gamma\VERT W \VERT^2
	+ \gamma^{-1}\big(\VERT T^{\gamma}_r W\VERT_{1,\gamma}^2+\VERT T^{\gamma}_{\chi_c^q}F\VERT_{1,\gamma}^2 \big)\\
	\label{error.e1}&+\|G\|_{1,\gamma}^2
	+\gamma^2\|{W^{\mathrm{nc}}}|_{x_2=0}\|^2,
	\end{align}
	where $r\in{\mathrm{\Gamma}_1^0}$ vanishes in region $\{\chi_c^q\equiv1\}$.
\end{lemma}

\begin{proof}\
	Let us show an estimate for $T^{\gamma}_{\sigma_{+}^0}Z_1^{+}$ with $Z_1^{+}$ defined by \eqref{bad2}.
	Recall from \eqref{sigma.eq} that symbol $\sigma_{+}^0$ satisfies the transport equation:
	\begin{align}  \label{error.eq1}
	\left\{\begin{aligned}
	&\partial_2\sigma_+^0+\{\sigma_+^0,\,\IM\omega_+\}=0&&\quad \textrm{if } x_2>0,\\
	&\sigma_{+}^0=-\mathrm{i}\gamma+\widetilde{\sigma}_0&&\quad \textrm{if } x_2=0.
	\end{aligned}\right.
	\end{align}
	Setting $S:=(T^{\gamma}_{\sigma_{+}^0})^* T^{\gamma}_{\sigma_{+}^0}$,
	we take the scalar product in $L^2(\Omega)$ of \eqref{bad2.eq1} with $SZ_1^+$ and
	apply integration by parts to derive
	\begin{align} \label{error.id2}
	\big\|\big(T^{\gamma}_{\sigma_{+}^0}Z_{1}^+\big)|_{x_2=0}\big\|^2+\sum_{j=1}^6 \mathcal{I}_j=0,
	\end{align}
	where each term $\mathcal{I}_j$ in the decomposition will be defined and estimated below.
	First,
	noting that $\sigma_+^0\in{\mathrm{\Gamma}_2^1}$ and
	$T^{\gamma}_{\sigma_{+}^0} T^{\gamma}_{d_+}= T^{\gamma}_{d_+}T^{\gamma}_{\sigma_{+}^0}+\mathcal{R}_0$,
	we obtain
	\begin{align} \notag
	\mathcal{I}_1:=\;&2\RE \llangle T^{\gamma}_{\sigma_{+}^0} T^{\gamma}_{d_+} Z_1^+,T^{\gamma}_{\sigma_{+}^0}Z_1^+\rrangle \\
	\label{I1}
	\lesssim\;& \VERT T^{\gamma}_{\sigma_{+}^0}Z_1^+\VERT ^2+\VERT Z_1^+\VERT \VERT T^{\gamma}_{\sigma_{+}^0}Z_1^+\VERT
	\lesssim \VERT T^{\gamma}_{\sigma_{+}^0}Z_1^+\VERT^2+ \VERT Z_1^+\VERT ^2.
	\end{align}
	Moreover, we have
	\begin{align}
\notag
	\mathcal{I}_2:=\;&2\RE \llangle  T^{\gamma}_{\sigma_{+}^0}T^{\gamma}_{r}W^+,   T^{\gamma}_{\sigma_{+}^0}Z_1^+ \rrangle\\
	\label{I2}	\lesssim\;& {(\varepsilon\gamma)^{-1}}\VERT T^{\gamma}_r W
	\VERT_{1,\gamma}^2+\varepsilon\gamma \VERT T^{\gamma}_{\sigma_{+}^0}Z_1^+\VERT^2, \\   \notag
	\mathcal{I}_3:=\;&2\RE \llangle T^{\gamma}_{\sigma_{+}^0}
	\mathcal{R}_0T^{\gamma}_{\chi_c^0}F^{+}, T^{\gamma}_{\sigma_{+}^0}Z_1^+ \rrangle\\
\label{I3}	 \lesssim\;& {(\varepsilon\gamma)^{-1}}\VERT T^{\gamma}_{\chi_c^0}F^{+}\VERT_{1,\gamma}^2+\varepsilon\gamma \VERT  T^{\gamma}_{\sigma_{+}^0}Z_1^+ \VERT^2, \\   \notag
	\mathcal{I}_4:=\;&2\RE \llangle T^{\gamma}_{\sigma_{+}^0}\mathcal{R}_{-1}W^{+},T^{\gamma}_{\sigma_{+}^0}Z_1^+ \rrangle\\
\label{I4}	 \lesssim\;& {(\varepsilon\gamma)^{-1}}\VERT W^+\VERT^2+\varepsilon\gamma \VERT T^{\gamma}_{\sigma_{+}^0}Z_1^+ \VERT^2.
	\end{align}
	For the terms:
	\begin{align*}
	\mathcal{I}_5:=\RE\llangle (\partial_2 S)Z_1^+,Z_1^+ \rrangle, \qquad \,\,
	\mathcal{I}_6:=2\RE\llangle S T^{\gamma}_{\omega_+}Z^+_1, Z^+_1 \rrangle,
	\end{align*}
	we use the identity:
	$\partial_2 S=(T^{\gamma}_{\partial_2\sigma_{+}^0})^*T^{\gamma}_{\sigma_{+}^0}
	+(T^{\gamma}_{\sigma_{+}^0})^* T^{\gamma}_{\partial_2\sigma_{+}^0}$
	to obtain
	\begin{align} \label{error.e2}
	\mathcal{I}_5+\mathcal{I}_6
	=2\RE \llangle T^{\gamma}_{\partial_2 \sigma_{+}^0} Z_1^+
	+T^{\gamma}_{\sigma_{+}^0}T^{\gamma}_{\omega_+}Z^+_1,T^{\gamma}_{\sigma_{+}^0} Z_1^+ \rrangle.
	\end{align}
	We write $\omega_{+}=\gamma e_++\mathrm{i}\hbar_+$ with real-valued symbols $e_+\in{\mathrm{\Gamma}_2^0}$
	and $\hbar_+\in{\mathrm{\Gamma}_2^1}$ as in \eqref{D1}.
	Employing the rule of symbolic calculus (Lemma \ref{lem.para1}), we have
	\begin{align*}
	T^{\gamma}_{\sigma_{+}^0}T^{\gamma}_{\omega_+}
	=\,&\gamma T^{\gamma}_{e_+}T^{\gamma}_{\sigma_{+}^0}+\gamma T^{\gamma}_{-\mathrm{i}\{\sigma_{+}^0, e_+\}}\\
	&+\gamma \mathcal{R}_{-1} +T^{\gamma}_{\mathrm{i}\hbar_+} T^{\gamma}_{\sigma_{+}^0}+T^{\gamma}_{\{\sigma_{+}^0, \hbar_+\}}+\mathcal{R}_0.
	\end{align*}
	It follows from \eqref{error.eq1} and \eqref{error.e2} that
	\begin{align*}
	\mathcal{I}_5+\mathcal{I}_6\lesssim \;&
	\RE \big\llangle (\gamma T^{\gamma}_{e_+}T^{\gamma}_{\sigma_{+}^0}
	+\gamma T^{\gamma}_{-\mathrm{i}\{\sigma_{+}^0, e_+\}}+T^{\gamma}_{\mathrm{i}\hbar_+}T^{\gamma}_{\sigma_{+}^0})
	Z^+_1,T^{\gamma}_{\sigma_{+}^0} Z_1^+ \big\rrangle\\
	&+\VERT Z_1^+\VERT \VERT T^{\gamma}_{\sigma_{+}^0} Z_1^+ \VERT.
	\end{align*}
	Since $\sigma_{+}^0\in{\mathrm{\Gamma}_2^1}$ and $\mathrm{i}\hbar_+\in\mathrm{i}\mathbb{R}$,
	operators $T^{\gamma}_{-\mathrm{i}\{\sigma_{+}^0, e_+\}}$ and $\RE T^{\gamma}_{\mathrm{i}\hbar_+}$
	are both of order $\leq 0$.
	It then follows that
	\begin{align*}
	&\RE \big\llangle (\gamma T^{\gamma}_{-\mathrm{i}\{\sigma_{+}^0, e_+\}}
	+T^{\gamma}_{\mathrm{i}\hbar_+}T^{\gamma}_{\sigma_{+}^0}) Z^+_1,T^{\gamma}_{\sigma_{+}^0} Z_1^+ \big\rrangle\\[1mm]
	&\,\,\, \lesssim \gamma \VERT Z_1^+\VERT\VERT T^{\gamma}_{\sigma_{+}^0} Z_1^+\VERT+\VERT T^{\gamma}_{\sigma_{+}^0} Z_1^+ \VERT^2
	\lesssim {\varepsilon^{-1}\gamma }\VERT Z_1^+\VERT^2+\varepsilon \gamma \VERT T^{\gamma}_{\sigma_{+}^0} Z_1^+ \VERT^2,
	\end{align*}
	which implies
	\begin{align}  \label{I5}
	\mathcal{I}_5+\mathcal{I}_6
	\lesssim \RE \big\llangle \gamma T^{\gamma}_{e_+}T^{\gamma}_{\sigma_{+}^0} Z^+_1,T^{\gamma}_{\sigma_{+}^0} Z_1^+ \big\rrangle
	+ {\varepsilon^{-1}\gamma }\VERT Z_1^+\VERT^2+\varepsilon \gamma \VERT T^{\gamma}_{\sigma_{+}^0} Z_1^+ \VERT^2.
	\end{align}
	Since $e_+\leq -c<0$, we apply G{\aa}rding's inequality to deduce
	\begin{align}  \label{I5a}
	-\RE \big\llangle \gamma T^{\gamma}_{e_+}T^{\gamma}_{\sigma_{+}^0} Z^+_1,T^{\gamma}_{\sigma_{+}^0} Z_1^+ \big\rrangle \gtrsim
	\gamma \VERT T^{\gamma}_{\sigma_{+}^0} Z_1^+ \VERT^2
	\end{align}
	for sufficiently large $\gamma$.
	Plugging estimates \eqref{I1}--\eqref{I4} and \eqref{I5}--\eqref{I5a}
	into \eqref{error.id2}, we take $\varepsilon$ small enough to deduce
	\begin{align}
	\gamma\VERT T^{\gamma}_{\sigma_{+}^0} Z_1^+ \VERT^2\lesssim\,&
	\|(T^{\gamma}_{\sigma_{+}^0}Z_{1}^+)|_{x_2=0}\|^2\notag \\
	&\label{error.e5} +\gamma^{-1}\VERT T^{\gamma}_r W\VERT_{1,\gamma}^2
	+\gamma^{-1}\VERT T^{\gamma}_{\chi_c^0}F^{+}\VERT_{1,\gamma}^2+\gamma\VERT W^+\VERT^2.
	\end{align}
	For the first term on the right-hand side, we use the fact
	that $\sigma_{+}^0|_{x_2=0}=-\mathrm{i}\gamma+\widetilde{\sigma}_0$ to obtain
	\begin{align} \label{error.e6}
	\|(T^{\gamma}_{\sigma_{+}^0}Z_{1}^+)|_{x_2=0}\|
	\lesssim \gamma \|Z_{1}^+|_{x_2=0}\|+\|T^{\gamma}_{\widetilde{\sigma}_0}Z_{1}^+|_{x_2=0}\|.
	\end{align}
	We plug \eqref{error.e6} into \eqref{error.e5} and use \eqref{bad2.e0} to find
	\begin{align} \notag
	\gamma\VERT T^{\gamma}_{\sigma_{+}^0} Z_1^+ \VERT^2
	\lesssim \,& \gamma^{-1}\big(\VERT T^{\gamma}_r W\VERT_{1,\gamma}^2+\VERT T^{\gamma}_{\chi_c^0}F\VERT_{1,\gamma}^2\big)
\\ &	+\gamma\VERT W^+\VERT^2+\|G\|_{1,\gamma}^2+\gamma^2\left\|W^{\mathrm{nc}}\,\!|_{x_2=0}\right\|^2.
	\end{align}
	Recall that $Z_1^+$ is defined by \eqref{bad2} with $Q_0\in{\mathrm{\Gamma}_2^0}$ being invertible.
	We then use  \eqref{bad2.e2} and \eqref{omega1} to conclude \eqref{error.e1}.
	The proof is completed.
\qed\end{proof}
Combining estimates \eqref{error.e0}--\eqref{error.e4} together,
using \eqref{error.e3}, and taking $\gamma$ suitably large, we obtain \eqref{p.p2.e0}.
In view of Proposition \ref{pro3}, estimate \eqref{thm.2e1} also holds.
This completes the proof of  Theorem {\rm\ref{thm.2}}.

\section{Well-posedness for the Linearized Problem}\label{sec.6}
In this section, we establish a well-posedness result for the linearized problem \eqref{P2b}
in the usual Sobolev space $H^s$ with $s$ large enough.
The essential point is to deduce a tame estimate in $H^s$.
For a hyperbolic problem with a characteristic boundary,
there is a loss of derivatives in \emph{a priori} energy estimates.
To overcome this difficulty, it is natural to introduce Sobolev spaces with conormal regularity,
where two tangential derivatives count as one normal derivative (see Secchi \cite{S96MR1405665}
and the references therein).
However, for our problem \eqref{P2b}, we can manage to compensate the loss of derivatives
and deduce \emph{a priori} estimates in the usual Sobolev spaces.
This is achieved by employing the idea in \cite{CS08MR2423311} and estimating the missing
derivatives through the equation of the linearized vorticity.

The main result in this section is stated as follows:

\begin{theorem}\ \label{thm.L}
	Let $T>0$ and $s\in [3,\tilde{\alpha}]\cap \mathbb{N}$ with any integer $ \tilde{\alpha}\geq 3$.
	Assume that the background state \eqref{RVS0} satisfies \eqref{H1} and \eqref{H2},
	and that perturbations $(\mathring{V}^{\pm},\mathring{\Psi}^{\pm})$
	belong to $H^{s+3}_{\gamma}(\Omega_T)$ for all $\gamma\geq 1$ and satisfy \eqref{bas.c1}--\eqref{bas.eq},
	and
	\begin{align} \label{H6.1}
	\|(\mathring{V}^{\pm},\nabla\mathring{\Psi}^{\pm})\|_{H^5_{\gamma}(\Omega_T)}
	+\|(\mathring{V}^{\pm},\p_2\mathring{V}^{\pm},\nabla\mathring{\Psi}^{\pm})|_{x_2=0}\|_{H^4_{\gamma}(\omega_T)}\leq K.
	\end{align}
	Assume further that the source terms $(f,g)\in H^{s+1}(\Omega_T)\times H^{s+1}(\omega_T)$ vanish in  the past.
	Then there exists a positive constant $K_0$, which is independent of $s$ and  $T$, and there exist two constants $C>0$ and $\gamma\geq 1$,
	which depend solely on $K_0$, such that, if $K\leq K_0$, then problem \eqref{P2b}
	admits a unique solution $(\dot{V}^{\pm},\psi)\in H^{s}(\Omega_T)\times H^{s+1}(\omega_T)$ that vanishes in the past
	and obeys the following tame estimate{\rm :}
	\begin{align} \notag
	&\|\dot{V}\|_{H^{s}_{\gamma}(\Omega_T)}+\|\mathbb{P}^{\pm}(\mathring{\varphi})\dot{V}^{\pm}\,\!|_{x_2=0}
	\|_{H^{s}_{\gamma}(\omega_T)}+\|\psi\|_{H^{s+1}_{\gamma}(\omega_T)}\\
\notag 	&\leq C\big\{ \|f\|_{H^{s+1}_{\gamma}(\Omega_T)}
	+\|g\|_{H^{s+1}_{\gamma}(\omega_T)}\\
	&\qquad\ \, +	 \big(\|f\|_{H^4_{\gamma}(\Omega_T)}	+\|g\|_{H^4_{\gamma}(\omega_T)}\big)\|(\mathring{V}^{\pm},\mathring{\Psi}^{\pm})\|_{H^{s+3}_{\gamma}(\Omega_T)} \big\}.
	\label{E6.1}
	\end{align}
\end{theorem}
We consider the case where the source terms $f$ and $g$ vanish in the past,
which corresponds to the case with zero initial data.
The case of general initial data is postponed to the nonlinear analysis which involves the construction of
a so-called approximate solution.
Before estimating the higher order derivatives of solutions,
we first prove that the linearized problem \eqref{P2b} is well-posed in $L^2$.

\subsection{First Well-Posedness Result}

In this subsection, we apply the well-posedness result in $L^2$ of Coulombel \cite{C05MR2138641} to the effective linear problem \eqref{P2b}.
We recall that system \eqref{P2b.1} is symmetrizable hyperbolic and observe that the coefficients of the linearized operators
satisfy the regularity assumptions of \cite{C05MR2138641}.
We also recall that problem \eqref{P2b} satisfies the energy estimate \eqref{thm.2e1},
which exhibits a loss of one tangential derivative.
According to the result in \cite{C05MR2138641},
we only need to find a dual problem that obeys an appropriate energy estimate.

Let us define a dual problem for \eqref{P2b}.
We introduce the following matrices:
\setlength{\arraycolsep}{1pt}
\begin{align}  \label{M1.ring}
{\small \mathring{B}_1:= \left.\begin{pmatrix}
0  & 0 &  0&-\mathring{\ell}_1^- &0 &  0 \\
\mathring{\ell}_1^+& 0 &   0 & \mathring{\ell}_1^-&0  & 0\\
0 &\mathring{\ell}_2^+& \mathring{\ell}_3^+  &0 &-\mathring{\ell}_2^-& -\mathring{\ell}_3^-
\end{pmatrix}\right|_{x_2=0},\
{\mathring{D}}_1:= \left.\begin{pmatrix}
0  & 0 &  0&0&0 &  0 \\
0 & 0 &   0 & 0&0  & 0\\
0 &\mathring{\ell}_2^+& \mathring{\ell}_3^+ &0 &\mathring{\ell}_2^-& \mathring{\ell}_3^-
\end{pmatrix}\right|_{x_2=0} },
\end{align}
and \setlength{\arraycolsep}{3pt}
\begin{align}
\notag        
\mathring{D}&:={\small \left.\begin{pmatrix}
0  & \mathring{\varGamma}_+^{-1} \mathring{\varrho}_{+ } &   -\mathring{\varGamma}_+^{-1} \mathring{\varsigma}_{+ }
&0  & \mathring{\varGamma}_-^{-1} \mathring{\varrho}_{- } &  -\mathring{\varGamma}_-^{-1} \mathring{\varsigma}_{- } \\
0 & \mathring{\varGamma}_+^{-1} \mathring{\varrho}_{+ } &   -\mathring{\varGamma}_+^{-1} \mathring{\varsigma}_{+ } & 0&0  & 0\\
1 &0 &  0  &1 &0& 0
\end{pmatrix}\right|_{x_2=0}},
\end{align}
where
\begin{align*}
\mathring{\ell}_1^{\pm}:=-\frac{\mathring{N}^{\pm}\mathring{c}_{\pm}^2\mathring{\varGamma}_{\pm}}{\partial_{2}\mathring{\Phi}^{\pm}},\quad
\mathring{\ell}_2^{\pm}:=-\frac{\mathring{\varrho}_{\pm}}{2\mathring{N}^{\pm}\partial_{2}\mathring{\Phi}^{\pm}},\quad
\mathring{\ell}_3^{\pm}:=\frac{\mathring{\varsigma}_{\pm}}{2\mathring{N}^{\pm}\partial_{2}\mathring{\Phi}^{\pm}}.
\end{align*}
Thanks to \eqref{A2.tilde}, we compute that these matrices satisfy the relation:
\begin{align} \label{dual.id1}
\mathring{B}_1^{\mathsf{T}}\mathring{B}+\mathring{D}_1^{\mathsf{T}}\mathring{D}
=\mathrm{diag}\,\big(\widetilde{A}_2(\mathring{U}^{+},\mathring{\Phi}^{+}),
\,\widetilde{A}_2(\mathring{U}^{-},\mathring{\Phi}^{-})\big)
\big|_{x_2=0},
\end{align}
where $\mathring{B}$ is defined by \eqref{M.ring}.
Moreover, we infer from \eqref{bas.c2} that all matrices $\mathring{B}$,  $\mathring{B}_1$, $\mathring{D}$,
and $\mathring{D}_1$ belong to $W^{2,\infty}(\mathbb{R}^2)$.
Following \cite[\S\,3.2]{M01MR1842775}, we define a dual problem for \eqref{P2b} as:
\begin{align}\label{dual}
\left\{\begin{aligned}
&\mathbb{L}_e'\big(\mathring{U}^{\pm},\mathring{\Phi}^{\pm}\big)^*U^{\pm}=\widetilde{f}_{\pm},\quad &x_2>0,\\
&\mathring{D}_1 U=0,\quad &x_2=0,\\
&\dive(\mathring{b}^{\mathsf{T}}\mathring{B}_1 U)-b_{\sharp}^{\mathsf{T}} \mathring{B}_1 U=0,\quad &x_2=0,
\end{aligned}\right.
\end{align}
where $\mathring{b}$, $b_{\sharp}$,  $\mathring{B}_1$, and  $\mathring{D}_1$ are defined
in \eqref{b.ring}, \eqref{b.sharp}, and \eqref{M1.ring},
$\dive$ denotes the divergence operator in $\mathbb{R}^2$ with respect to $(t,x_1)$,
and the dual operators $\mathbb{L}_e'\big(\mathring{U}^{\pm},\mathring{\Phi}^{\pm}\big)^*$
are the formal adjoints of $\mathbb{L}_e'\big(\mathring{U}^{\pm},\mathring{\Phi}^{\pm}\big)$.
More precisely, we have
\begin{align*}
\mathbb{L}_e'(V,\Psi)^*U=\;&- A_0(V)^{\mathsf{T}} \partial_t U-A_1(V)^{\mathsf{T}} \partial_1 U- \widetilde{A}_2(V,\Psi)^{\mathsf{T}} \partial_2 U\\
&+\mathcal{C}(V,\Psi)^*U-\big(\partial_tA_0(V)^{\mathsf{T}}+\partial_1 A_1(V)^{\mathsf{T}}+\partial_2\widetilde{A}_2(V,\Psi)^{\mathsf{T}} \big)U,
\end{align*}
where $\mathcal{C}(V,\Psi)^*$, the adjoint of $\mathcal{C}(V,\Psi)$, is a zero-th order operator.
We refer to \cite[\S\,3.2]{M01MR1842775} for the derivation of the dual problem
by using  integration by parts and identity \eqref{dual.id1}.

Since the first two rows of matrix $\mathring{D}_1$ given in \eqref{M1.ring} are zero,
we see that the number of the boundary conditions in \eqref{dual} is exactly two.
This is compatible with the number of incoming characteristics,
that is, the number of negative eigenvalues of the boundary matrix for \eqref{dual}.
In fact, the boundary matrix of operator $\mathbb{L}_e'(V,\Psi)^*$ in the
half-space $\Omega$ is $\widetilde{A}_2(V,\Psi)^{\mathsf{T}}$.
Then we infer from  \eqref{A2.tilde} that problem \eqref{dual}
has two incoming  characteristics and two outgoing  characteristics.

We can define and analyze the Lopatinski\u{\i} determinant associated with
the boundary conditions in \eqref{dual} as we have done in \S\,\ref{sec.5}.
Then we have the following result, which is an analogue of Lemma \ref{lem.Lopa2}
by changing $\gamma$ into $-\gamma$.
\begin{lemma}\ \label{lem.dual}
	Assume that \eqref{bas.c1}--\eqref{bas.c2} hold for a sufficiently small $K>0$.
	Then the dual problem \eqref{dual} satisfies the backward Lopatinski\u{\i} condition.
	Moreover, the roots of the associated Lopatinski\u{\i} determinant are simple and
	coincide with the roots of the  Lopatinski\u{\i} determinant \eqref{Lopa2a}
	for the original problem \eqref{P2b}.
\end{lemma}

One can reproduce the same analysis as we have done in \S\,\ref{sec.5}
to show that the dual problem satisfies an {\it a priori} estimate that
is similar to \eqref{thm.2e1}.
The linearized problem \eqref{P2b} thus satisfies all the assumptions ({\it i.e.}  symmetrizability, regularity,
and weak stability) listed in \cite{C05MR2138641}.
We therefore obtain the following well-posedness result.

\begin{theorem}\ \label{thm.3}
	Let $T>0$ be any fixed constant.
	Assume that the background state \eqref{RVS0} satisfies  \eqref{H1} and \eqref{H2}.
	Assume further that the basic state $\big(\mathring{V}^{\pm},\mathring{\Psi}^{\pm}\big)$
	satisfies \eqref{bas.c1}--\eqref{bas.eq}.
	Then there exist positive constants $K_0>0$ and $\gamma_0\geq 1$,
	independent of $T$, such that, if $K\leq K_0$,
	then, for the source terms $f^{\pm}\in L^2(\mathbb{R}_+;H^1(\omega_T))$ and $g\in H^1(\omega_T)$ that vanish for $t<0$,
	the problem{\rm :}
	\begin{align*}
	\left\{\begin{aligned}
	&\mathbb{L}'_e\big(\mathring{U}^{\pm},\mathring{\Phi}^{\pm}\big)\dot{V}^{\pm}=f^{\pm} \quad &\mbox{for $t<T,\ x_2>0$},\\
	&\mathbb{B}'_e\big(\mathring{U}^{\pm},\mathring{\Phi}^{\pm}\big)(\dot{V},\psi)=g\qquad  &\mbox{for $t<T,\ x_2=0$},
	\end{aligned}\right.
	\end{align*}
	has a unique solution $(\dot{V}^+,\dot{V}^-,\psi)\in L^2(\Omega_T)\times L^2(\Omega_T)\times H^1(\omega_T)$
	that vanishes for $t<0$ and satisfies $\mathbb{P}^{\pm}(\mathring{\varphi})\dot{V}^{\pm}\,\!|_{x_2=0}\in L^2(\omega_T)$.
	Moreover, the following estimate holds for all $\gamma\geq \gamma_0$ and for all $t\in[0,T]${\rm :}
	\begin{align}
	&\gamma\|\dot{V}\|_{L^2_{\gamma}(\Omega_t)}^2
	+\|\mathbb{P}^{\pm}(\mathring{\varphi})\dot{V}^{\pm}\,\!|_{x_2=0}\|_{L^2_{\gamma}(\omega_t)}^2+\|\psi\|_{H^1_{\gamma}(\omega_t)}^2
	\notag \\
	&\,\, \lesssim \gamma^{-3}\|f^{\pm}\|_{L^2(H_{\gamma}^1(\omega_t))}^2+\gamma^{-2}\|g\|_{H^1_{\gamma}(\omega_t)}^2.   \label{thm.3e}
	\end{align}
\end{theorem}

Theorem {\rm\ref{thm.3}} shows the well-posedness of problem \eqref{P2b} in $L^2$ when the source terms $(f,g)$
belong to $L^2(H^1)\times H^1$.
We now turn to the energy estimates for the higher-order derivatives of solutions.

\subsection{{\it A Priori} Tame Estimates}
To obtain the estimates for the higher-order derivatives of solutions to \eqref{P2b},
it is convenient to deal with the reformulated problem \eqref{P2c} and \eqref{P2c.2}
for the new unknowns $W$.
Until the end of this section, we always assume that $\gamma\geq \gamma_0$ and $K\leq K_0$,
where $\gamma_0$ and $K_0$ are given by Theorem {\rm\ref{thm.3}}.
Then estimate \eqref{thm.3e} can be rewritten as
\begin{align}
& \sqrt{\gamma}\|W\|_{L^2_{\gamma}(\Omega_T)}+\|W^{\mathrm{nc}}\,\!|_{x_2=0}\|_{L^2_{\gamma}(\omega_T)}
+\|\psi\|_{H^1_{\gamma}(\omega_T)}\notag\\
&\,\, \lesssim \gamma^{-{3}/{2}}\|F^{\pm}\|_{L^2(H^1_{\gamma}(\omega_T))}
+\gamma^{-1}\|g\|_{H^1_{\gamma}(\omega_T)}. \label{thm.3e2}
\end{align}

We first derive the estimate of the tangential derivatives.
Let $k\in[1,s]$ be a fixed integer.
Applying the tangential derivative $\p^{\alpha}=\p_t^{\alpha_0}\p_1^{\alpha_1}$ with $|\alpha|=k$
to system \eqref{P2c} yields the equations for $\p^{\alpha}W^{\pm}$ that involve the linear terms
of the derivatives, $\p^{\alpha-\beta}\p_t W^{\pm}$ and $\p^{\alpha-\beta}\p_1 W^{\pm}$, with $|\beta|=1$.
These terms cannot be treated simply as source terms, owing to the loss of derivatives in the
energy estimate \eqref{thm.3e2}.
To overcome this difficulty, we adopt the idea of \cite{CS08MR2423311} and
deal with a boundary value problem for all the tangential derivatives of order equal to $k$, {\it i.e.}
for $W^{(k)}:=\{ \p_t^{\alpha_0}\p_1^{\alpha_1}W^{\pm},\, \alpha_{0}+\alpha_{1}=k\}$.
Such a problem satisfies the same regularity and stability properties as the original problem \eqref{P2c} and \eqref{P2c.2}.
Repeating the derivation in \S\,\ref{sec.5}, we find that $W^{(k)}$ obeys an energy estimate similar to \eqref{thm.3e2}
with new source terms $\mathcal{F}^{(k)}$ and $\mathcal{G}^{(k)}$.
Then we can employ the Gagliardo--Nirenberg and Moser-type inequalities ({\it cf}.\;\cite[Theorems 8--10]{CS08MR2423311})
to derive the following estimate for tangential derivatives (see \cite[Proposition 1]{CS08MR2423311} for the detailed proof).

\begin{lemma}[Estimate of tangential derivatives]\label{lem.tame1}
	Assume that the hypotheses of Theorem {\rm \ref{thm.L}} hold.
	Then there exist constants $C_{s}>0$ and $\gamma_{s}\geq 1$,  independent of $T$, such that,
	for all $\gamma\geq \gamma_{s}$ and for all $(W,\psi)\in H^{s+2}_{\gamma}(\Omega_T)\times H^{s+2}_{\gamma}(\omega_T)$
	that are solutions of problem \eqref{P2c} and \eqref{P2c.2}, the following estimate holds{\rm :}
	\begin{align}
	&\notag\sqrt{\gamma}\|W\|_{L^2(H^{s}_{\gamma}(\omega_T))}+\|W^{\mathrm{nc}}\,\!|_{x_2=0}\|_{H^{s}_{\gamma}(\omega_T)}
	+\|\psi\|_{H^{s+1}_{\gamma}(\omega_T)}\\
	& \leq C_{s}\big\{\gamma^{-1}\|g\|_{H^{s+1}_{\gamma}(\omega_T)}
	+\gamma^{-{3}/{2}}\big\|F^{\pm}\big\|_{L^2(H^{s+1}_{\gamma}(\omega_T))}
	\notag \\
	&\qquad\ \ \,  +\gamma^{-1} \big(\|W^{\mathrm{nc}}\,\!|_{x_2=0}\|_{L^{\infty}(\omega_T)}
	+\|\psi\|_{W^{1,\infty}(\omega_T)}\big)\big\|\big(\mathring{V},\p_2\mathring{V},\nabla\mathring{\Psi}\big)|_{x_2=0}\big\|_{{H^{s+1}_{\gamma}(\omega_T)}}\notag \\
	&\qquad\ \ \, +\gamma^{-{3}/{2}}\|W\|_{W^{1,\infty}(\Omega_T)}\big\|\big(\mathring{V},\nabla\mathring{\Psi}\big)\big\|_{H^{s+2}_{\gamma}(\Omega_T)}
	\big\}.
	\label{tame1}
	\end{align}
\end{lemma}

We recall that the boundary matrix for our problem \eqref{P2c} and \eqref{P2c.2} is not invertible.
Thus, there is no hope to estimate all the normal derivatives of $W$ directly from \eqref{P2c}
by employing the standard argument for noncharacteristic boundary problems as in \cite{RM74MR0340832}.
Nevertheless, for our problem \eqref{P2b}, we can obtain the estimate of the missing normal
derivatives through the equation of the ``linearized vorticity''.

In view of the original equations \eqref{hw.eq}, taking into account the change of variables and linearization,
we define  the ``linearized vorticity'' as:
\begin{align} \label{xi.def}
\dot{\xi}^{\pm}:=\Big(\p_1-\frac{\p_1\mathring{\Phi}^{\pm}}{\p_2\mathring{\Phi}^{\pm}}\p_2\Big)\dot{V}_3^{\pm}
-\frac{1}{\p_2\mathring{\Phi}^{\pm}}\p_2\dot{V}_2^{\pm},
\end{align}
where  $\dot{V}^{\pm}_2$ and $\dot{V}^{\pm}_3$ are the second and third components of the good unknown \eqref{good}, respectively.
We notice that multiplying \eqref{RE3} by $S_1(U)$ leads to system \eqref{RE4},
where $S_1(U)$ is defined by \eqref{S1}.
Thus, we multiply \eqref{P2b.1} with matrices $S_1(\mathring{U}^{\pm})^{-1}$ to obtain
\begin{align}\label{vor1}
\big(B_0(\mathring{U}^{\pm})\p_t  +B_1(\mathring{U}^{\pm})\p_1
+\widetilde{B}_2(\mathring{U}^{\pm},\mathring{\Phi}^{\pm})\p_2\big) \dot{V}^{\pm}
+ \widetilde{\mathcal{C}}(\mathring{U}^{\pm},\mathring{\Phi}^{\pm}) \dot{V}^{\pm}=\tilde{f}^{\pm},
\end{align}
where $\widetilde{\mathcal{C}}(\mathring{U}^{\pm},\mathring{\Phi}^{\pm}):=S_1(\mathring{U}^{\pm})^{-1}\mathcal{C}(\mathring{U}^{\pm},\mathring{\Phi}^{\pm})$,
$\tilde{f}^{\pm}:=S_1(\mathring{U}^{\pm})^{-1}f^{\pm}$, and matrices $B_j$ are defined by \eqref{B0}--\eqref{B12}.
In light of \eqref{bas.eq.1}, we have
\begin{align*}
\widetilde{B}_2(\mathring{U},\mathring{\Phi})
:=\,&\frac{1}{\p_2\mathring{\Phi}}\big(B_2(\mathring{U})-\p_t\mathring{\Phi} B_0(\mathring{U})-\p_2\mathring{\Phi} B_1(\mathring{U})\big)\\
=\,&\frac{1}{\p_2\mathring{\Phi}}\begin{pmatrix}
\widetilde{B}_2^{11}&\widetilde{B}_2^{12}&\widetilde{B}_2^{13}\\
-\mathring{N}^{-1}\p_1\mathring{\Phi}&0&0\\ \mathring{N}^{-1}&0&0
\end{pmatrix},
\end{align*}
where the explicit form of $\widetilde{B}_2^{1j}$ is of no interest.
Using the second and third components of \eqref{vor1} yields the following equations for $\dot{\xi}^{\pm}$:
\begin{align}
\notag &(\p_t+\mathring{v}_1^{\pm}\p_1)\dot{\xi}^{\pm}\\
&\quad \label{xi.eq} =\p_1 \mathcal{F}^{\pm}_2
-\frac{1}{\p_2 \mathring{\Phi}^{\pm}}\left(\p_1 \mathring{\Phi}^{\pm} \p_2\mathcal{F}_2^{\pm}+\p_2\mathcal{F}_1^{\pm}\right)
+\mathring{\Lambda}^{\pm}_1\cdot\p_1\dot{V}^{\pm}+\mathring{\Lambda}^{\pm}_2\cdot\p_2\dot{V}^{\pm},
\end{align}
where vectors $\mathring{\Lambda}^{\pm}_1$ and $\mathring{\Lambda}^{\pm}_2$ are $C^{\infty}$--functions
of $(\mathring{V}^{\pm},\nabla\mathring{V}^{\pm},\nabla\mathring{\Psi}^{\pm}, \nabla^2\mathring{\Psi}^{\pm})$
that vanish at the origin, and the source terms $\mathcal{F}_1^{\pm}$ and $\mathcal{F}_2^{\pm}$ are given by
\begin{align} \notag
\mathcal{F}_1^{\pm}:=\mathring{\varGamma}_{\pm}^{-1}\big(\tilde{f}^{\pm}-\widetilde{C}(\mathring{U}^{\pm},\mathring{\Phi}^{\pm})\dot{V}^{\pm}\big)_2,
\qquad \mathcal{F}_2^{\pm}:=\mathring{\varGamma}_{\pm}^{-1}\big(\tilde{f}^{\pm}
-\widetilde{C}(\mathring{U}^{\pm},\mathring{\Phi}^{\pm})\dot{V}^{\pm}\big)_3.
\end{align}
Employing a standard energy estimate to the transport equations \eqref{xi.eq},
we can apply the Gagliardo--Nirenberg and Moser-type inequalities to derive the following estimate of $\dot{\xi}^{\pm}$:

\begin{lemma}[Estimate of vorticity]\label{lem.tame2}
	Assume that the hypotheses of Theorem {\rm \ref{thm.L}} hold.
	Then there exist constants $C_{s}>0$ and $\gamma_{s}\geq 1$, independent of $T$, such that,
	for all $\gamma\geq \gamma_{s}$ and for all $(W,\psi)\in H^{s+2}_{\gamma}(\Omega_T)\times H^{s+2}_{\gamma}(\omega_T)$
	that are solutions of problem \eqref{P2c} and \eqref{P2c.2},
	functions $\dot{\xi}^{\pm}$ defined by \eqref{xi.def}
	satisfy the following estimate:
	\begin{align}
	\notag &\gamma\big\|\dot{\xi}^{\pm}\big\|_{H^{s-1}_{\gamma}(\Omega_T)}\\
	&\notag
	\leq  C_{s}\big\{ \big\|f^{\pm}\big\|_{H^{s}_{\gamma}(\Omega_T)}
	+\big\|f^{\pm}\big\|_{L^{\infty}(\Omega_T)}\big\|\big(\mathring{V},\nabla\mathring{\Psi}\big)\big\|_{H^{s}_{\gamma}(\Omega_T)}+\big\|\dot{V}^{\pm}\big\|_{H^{s}_{\gamma}(\Omega_T)}
	\\ &\qquad 	\ \ \, +\big\|\dot{V}^{\pm}\big\|_{W^{1,\infty}(\Omega_T)}\big(\big\|\mathring{V}\big\|_{H^{s+1}_{\gamma}(\Omega_T)}
	+\big\|\nabla\mathring{\Psi}\big\|_{H^{s}_{\gamma}(\Omega_T)}\big)\big\}.
	\label{tame2}
	\end{align}
\end{lemma}

We are going to make the estimate for all the normal derivatives by means of estimates \eqref{tame1} and \eqref{tame2}
for the tangential derivatives and the linearized vorticity.
To this end, we need to express the normal derivatives $\p_2W^{\pm}$ in terms of the tangential derivatives
$\p_t W^{\pm}$, $\p_1 W^{\pm}$, and vorticity $\dot{\xi}^{\pm}$.
Since $\bm{I}_2=\mathrm{diag}\,(0,1,1)$, the normal derivatives $\p_2 W_2^{\pm}$ and $\p_2 W_3^{\pm}$
are directly given by \eqref{P2c} as:
\begin{align} \label{normal1}
\bm{I}_2\partial_2 W^{\pm}
=\bm{I}_2(F^{\pm}-\bm{A}_0^{\pm}\partial_t W^{\pm}
-\bm{A}_1^{\pm}\partial_1 W^{\pm}-\bm{C}^{\pm} W^{\pm}).
\end{align}
The ``missing'' normal derivatives $\p_2 W^{\pm}_1$ can be expressed by $\dot{\xi}^{\pm}$
and equations \eqref{P2c}.
From transformation \eqref{W.def} and definition \eqref{T}, we have
\begin{align*}
&\p_2 \dot{V}_2=\mathring{\varsigma}\p_2 W_1+\frac{\mathring{\varrho}}{\mathring{N}\mathring{c}}\p_2(W_2-W_3)
+\p_2\mathring{\varsigma} W_1+\p_2\left(\frac{\mathring{\varrho}}{\mathring{N}\mathring{c}}\right)(W_2-W_3),\\
&\p_2 \dot{V}_3=\mathring{\varrho}\p_2 W_1-\frac{\mathring{\varsigma}}{\mathring{N}\mathring{c}}\p_2(W_2-W_3)
+\p_2\mathring{\varrho} W_1-\p_2\left(\frac{\mathring{\varsigma}}{\mathring{N}\mathring{c}}\right)(W_2-W_3),
\end{align*}
where we have omitted indices ``$\pm$''.
By definition \eqref{xi.def}, we obtain
\begin{align} \notag
&\big(\p_1\mathring{\Phi}\mathring{\varrho} +\mathring{\varsigma}\big) \p_2 W_1\\
&\quad =\p_2 \mathring{\Phi} \big(\p_1 \dot{V}_3-\dot{\xi}\big)
+\frac{\p_1\mathring{\Phi}\mathring{\varsigma}-\mathring{\varrho}}{\mathring{N}\mathring{c}}\p_2(W_2-W_3)
+ C(\mathring{U},\mathring{\Phi})W,
\label{normal2}
\end{align}
where $C(\mathring{U},\mathring{\Phi})$ is a $C^{\infty}$--function of
$(\mathring{V}, \nabla\mathring{V},\nabla\mathring{\Psi},\nabla^2\mathring{\Psi})$ that
vanishes at the origin. According to \eqref{m.n.ring}, we see
that $\p_1\mathring{\Phi}\mathring{\varrho} +\mathring{\varsigma}\gtrsim 1$ by taking $K_0>0$ sufficiently small.
Then we find from \eqref{normal1}--\eqref{normal2} that
\begin{align} \notag
\p_2 W^{\pm}=\,&\widetilde{\bm{A}}_3^{\pm}F^{\pm}+\widetilde{\bm{A}}_0^{\pm}\partial_t W^{\pm}
\\ \label{normal3} &
+\widetilde{\bm{A}}_1^{\pm}\partial_1 W^{\pm}+\widetilde{\bm{C}}^{\pm} W^{\pm}-\frac{\p_2 \mathring{\Phi}^{\pm}}{\p_1\mathring{\Phi}^{\pm}\mathring{\varrho}^{\pm}
	+\mathring{\varsigma}^{\pm}}\begin{pmatrix}
\dot{\xi}^{\pm}\\0\\0
\end{pmatrix},
\end{align}
where $\widetilde{\bm{A}}_{0,1,3}^{\pm}$ are $C^{\infty}$--functions of $(\mathring{V},\nabla\mathring{\Psi})$,
and $\widetilde{\bm{C}}^{\pm}$ are $C^{\infty}$--functions of $(\mathring{V}, \nabla\mathring{V},\nabla\mathring{\Psi},\nabla^2\mathring{\Psi})$
that vanish at the origin.
Although the linearized problem \eqref{P2c} and \eqref{P2c.2} is characteristic,
we manage to express all the normal derivatives $\p_2 W^{\pm}$ by \eqref{normal3}
as a linear combination of the tangential derivatives, vorticity, zero-th order terms, and source terms.
Then we can prove the following result similar to \cite[Proposition 3]{CS08MR2423311},
so we omit its proof.

\begin{lemma}[Estimate of normal derivatives] \label{lem.tame3}
	Assume that the hypotheses of Theorem {\rm \ref{thm.L}} hold.
	Then there exist constants $C_{s}>0$ and $\gamma_{s}\geq 1$,
	which are independent of $T$, such that,
	for all $\gamma\geq \gamma_{s}$ and solutions $(W,\psi)\in H^{s+2}_{\gamma}(\Omega_T)\times H^{s+2}_{\gamma}(\omega_T)$
	of problem \eqref{P2c} and \eqref{P2c.2},
	the following estimate holds for all integer $k\in[1,m]${\rm :}
	\begin{align}
	\notag \big\|\p_2^kW^{\pm}\big\|_{L^2(H^{s-k}_{\gamma}(\omega_T))}
	\leq  C_{s}\big\{& \big\|\big(F^{\pm},W^{\pm},\dot{\xi}^{\pm}\big)\big\|_{H^{s-1}_{\gamma}(\Omega_T)}
	+\big\|W^{\pm}\big\|_{L^2(H^{s}_{\gamma}(\omega_T))}
\\
	& +\big\|\dot{\xi}^{\pm}\big\|_{L^{\infty}(\Omega_T)}\big\|\big(\mathring{V},\nabla\mathring{\Psi}\big) \big\|_{H^{s-1}_{\gamma}(\Omega_T)}\notag \\
	&  +\big\|\big(F^{\pm},W^{\pm}\big)\big\|_{L^{\infty}(\Omega_T)}\big\|\big(\mathring{V},\nabla\mathring{\Psi}\big) \big\|_{H^{s}_{\gamma}(\Omega_T)}\big\}.
	\label{tame3}
	\end{align}
\end{lemma}
In light of definition \eqref{norm.def}, we see that, for all $s\in\mathbb{N}$ and $\theta \in H^{s}_{\gamma}(\Omega_T)$,
\begin{align*}
\|\theta\|_{H^{s}_{\gamma}(\Omega_T)}=\sum_{k=0}^{s}\|\p_2^k\theta \|_{L^2(H^{s-k}_{\gamma}(\omega_T))},
\qquad \gamma \|\theta\|_{H^{s-1}_{\gamma}(\Omega_T)}\leq \|\theta\|_{H^{s}_{\gamma}(\Omega_T)}.
\end{align*}
By virtue of these identities, we combine Lemmas \ref{lem.tame1}--\ref{lem.tame3}
and employ the Moser-type inequalities to obtain the following \emph{a priori} estimates on the $H^{s}_{\gamma}$--norm
of solution $\dot{V}^{\pm}$ to the linearized problem \eqref{P2b}.

\begin{proposition}\ \label{pro.tame}
	Assume that the hypotheses of Theorem {\rm \ref{thm.L}} hold.
	Then there exists a constant $K_0>0$ {(}independent of $s$ and $T${)}
	and constants $C_{s}>0$ and $\gamma_{s}\geq 1$ {(}depending on $s$, but independent of $T${)} such that,
	if $K\leq K_0$, then, for all $\gamma\geq \gamma_{s}$ and
	solutions $(\dot{V},\psi)\in H^{s+2}_{\gamma}(\Omega_T)\times H^{s+2}_{\gamma}(\omega_T)$ of problem \eqref{P2b},
	the following estimate holds{\rm :}
	\begin{align}
	\notag &\sqrt{\gamma}\big\|\dot{V}^{\pm}\big\|_{H^{s}_{\gamma}(\Omega_T)}
	+\big\|\mathbb{P}^{\pm}(\mathring{\varphi})\dot{V}^{\pm}\,\!|_{x_2=0}\big\|_{H^{s}_{\gamma}(\omega_T)}
	+\|\psi\|_{H^{s+1}_{\gamma}(\omega_T)}\\
	\notag &\leq C_{s}\big\{\gamma^{-{1}/{2}}\left\|f^{\pm}\right\|_{H^{s}_{\gamma}(\Omega_T)}
	+\gamma^{-{3}/{2}}\left\|f^{\pm}\right\|_{L^2(H^{s+1}_{\gamma}(\omega_T))}+ \gamma^{-1}\|g\|_{H^{s+1}_{\gamma}(\omega_T)}\\
	\notag&\ \quad  +\gamma^{-1}\big(\big\|\mathbb{P}^{\pm}(\mathring{\varphi})\dot{V}^{\pm}\,\!\big\|_{L^{\infty}(\omega_T)}
	+\|\psi\|_{W^{1,\infty}(\omega_T)}\big)\|(\mathring{V},\p_2\mathring{V},\nabla\mathring{\Psi})\|_{{H^{s+1}_{\gamma}(\omega_T)}} \\
	&\ \quad  +\big(\gamma^{-{1}/{2}}\left\|f^{\pm}\right\|_{L^{\infty}(\Omega_T)} +\gamma^{-{3}/{2}}\big\|\dot{V}^{\pm}\big\|_{W^{1,\infty}(\Omega_T)}\big)\big\|\big(\mathring{V},\nabla\mathring{\Psi}\big)\big\|_{H^{s+2}_{\gamma}(\Omega_T)}\big\}.
	\label{tame4}
	\end{align}
\end{proposition}

\subsection{Proof of Theorem {\rm\ref{thm.L}}}
Theorem {\rm\ref{thm.3}} shows that the linearized problem \eqref{P2b} is well-posed for sources
terms $(f^{\pm},g)\in L^2(H^1(\omega_T))\times H^1(\omega_T)$ that vanish in the past.
Following \cite{RM74MR0340832,CP82MR678605}, we can use Proposition \ref{pro.tame}
to covert Theorem {\rm\ref{thm.3}} into a well-posedness result of \eqref{P2b} in $H^{s}$.
More precisely, we can prove that, under the assumptions of Theorem {\rm\ref{thm.L}},
if $(f^{\pm},g)\in H^{s+1}(\Omega_T)\times H^{s+1}(\omega_T)$ vanish in the past,
then there exists a unique solution $(\dot{V}^{\pm},\psi)\in H^{s}(\Omega_T)\times H^{s+1}(\omega_T)$
that vanishes in the past and satisfies \eqref{tame4} for all $\gamma\geq \gamma_{s}$.

It remains to prove the tame estimate \eqref{E6.1}.
To this end, we first fix the value of $\gamma$ such that $\gamma$ is greater than the maximum
of $\gamma_{3},\ldots,\gamma_{\tilde{\alpha}}$.
Using \eqref{tame4} with $s=3$ and \eqref{H6.1}, we have
\begin{align}
\notag &\big\|\dot{V}^{\pm}\big\|_{H^3_{\gamma}(\Omega_T)}
+\big\|\mathbb{P}^{\pm}(\mathring{\varphi})\dot{V}^{\pm}\,\!|_{x_2=0}\big\|_{H^3_{\gamma}(\omega_T)}+\|\psi\|_{H^{4}_{\gamma}(\omega_T)}\notag\\
&\lesssim
K \big(\left\|f^{\pm}\right\|_{L^{\infty}(\Omega_T)}+\big\|\dot{V}^{\pm}\big\|_{W^{1,\infty}(\Omega_T)}
+\big\|\mathbb{P}^{\pm}(\mathring{\varphi})\dot{V}^{\pm}\,\!|_{x_2=0}\big\|_{L^{\infty}(\omega_T)}+\|\psi\|_{W^{1,\infty}(\omega_T)}\big)\notag\\
&\quad\, +\left\|f^{\pm}\right\|_{H^{4}_{\gamma}(\Omega_T)}+\|g\|_{H^{4}_{\gamma}(\omega_T)}.
\label{tame.p1}
\end{align}
Note that $T>0$ and $\gamma$ have been fixed.
Thanks to the classical Sobolev inequalities that
$\|\theta\|_{L^{\infty}(\Omega_T)}\lesssim \|\theta\|_{H^2(\Omega_T)}$
and $\|\theta\|_{L^{\infty}(\omega_T)}\lesssim \|\theta\|_{H^2(\omega_T)}$,
we utilize \eqref{tame.p1} and take $K>0$ sufficiently small to obtain
that $\left\|f^{\pm}\right\|_{L^{\infty}(\Omega_T)}\lesssim  \left\|f^{\pm}\right\|_{H^{4}_{\gamma}(\Omega_T)}$ and
\begin{align*}
&\big\|\dot{V}^{\pm}\big\|_{W^{1,\infty}(\Omega_T)}
+\big\|\mathbb{P}^{\pm}(\mathring{\varphi})\dot{V}^{\pm}\,\!|_{x_2=0}\big\|_{L^{\infty}(\omega_T)}+\|\psi\|_{W^{1,\infty}(\omega_T)}
\\ &\quad \lesssim  \big\|\dot{V}^{\pm}\big\|_{H^3_{\gamma}(\Omega_T)}
+\big\|\mathbb{P}^{\pm}(\mathring{\varphi})\dot{V}^{\pm}\,\!|_{x_2=0}\big\|_{H^3_{\gamma}(\omega_T)}+\|\psi\|_{H^{4}_{\gamma}(\omega_T)}\\
&\quad \lesssim \left\|f^{\pm}\right\|_{H^{4}_{\gamma}(\Omega_T)}+\|g\|_{H^{4}_{\gamma}(\omega_T)}.
\end{align*}
Plugging these estimates into \eqref{tame4} yields \eqref{E6.1}, which completes the proof of Theorem {\rm\ref{thm.L}}.

\section{Construction of the Approximate Solution}\label{secCA}

In this section, we introduce the ``approximate'' solution $(U^a,\Phi^a)$ in order to reduce the original
problem \eqref{RE0} and \eqref{Phi.eq} into a nonlinear problem with zero initial data.
We naturally expect to solve this reformulated problem in the space of functions vanishing in the past,
so that Theorem {\rm\ref{thm.L}}, which is the well-posedness result in the same function space for the linearized problem,
can be applied.
We need to impose the necessary compatibility conditions on the initial data $(U_0^{\pm},\varphi_0)$
for the existence of smooth approximate solutions $(U^a,\Phi^a)$,
which are solutions of problem \eqref{RE0} and \eqref{Phi.eq}
in the sense of Taylor's series at time $t=0$.

Let $s\geq 3$ be an integer.
Assume that $\tilde{U}_0^{\pm}:=U_0^{\pm}-\widebar{U}^{\pm}\in H^{s+1/2}(\mathbb{R}_+^2)$
and  $\varphi_0\in H^{s+1}(\mathbb{R})$.
We also assume without loss of generality that $(\tilde{U}_0^{\pm},\varphi_0)$ has a compact support:
\begin{align} \label{CA1}
\supp \tilde{U}_0^{\pm}\subset \{x_2\geq 0,\, x_1^2+x_2^2\leq 1\},\qquad \supp \varphi_0\subset [-1,1].
\end{align}
We extend $\varphi_0$ from $\mathbb{R}$ to $\mathbb{R}_+^2$ by constructing
$\tilde{\Phi}_0^+=\tilde{\Phi}_0^-\in H^{s+3/2}(\mathbb{R}_+^2)$, which satisfies
\begin{align} \label{CA2}
\tilde{\Phi}_0^{\pm}\,\!|_{x_2=0}
=\varphi_0,\quad \supp \tilde{\Phi}_0^{\pm}\subset \{x_2\geq 0,\, x_1^2+x_2^2\leq 2\},
\end{align}
and the estimate:
\begin{align} \label{CA3}
\big\|\tilde{\Phi}_0^{\pm}\big\|_{H^{s+3/2}(\mathbb{R}_+^2)}
\leq C\|\varphi_0\|_{H^{s+1}(\mathbb{R})}.
\end{align}
By virtue of \eqref{CA3} and the Sobolev embedding theorem,
we infer that, for $\varphi_0$ small enough in $H^{s+1}(\mathbb{R})$,
the following estimates hold for $\Phi_0^{\pm}:=\tilde{\Phi}_0^{\pm}+\widebar{\Phi}_0^{\pm}$:
\begin{align} \label{CA4}
\pm \p_2\Phi_0^{\pm}\geq \frac{7}{8}\qquad\,\, \textrm{for all }x\in\mathbb{R}_+^2.
\end{align}
For problem \eqref{Phi.eq}, we prescribe the initial data:
\begin{align} \label{Phi.initial}
\Phi^{\pm}|_{t=0}=\Phi_0^{\pm}.
\end{align}

Let us denote the perturbation by
$(\tilde{U}^{\pm},\tilde{\Phi}^{\pm}):=(U^{\pm}-\widebar{U}^{\pm},\Phi^{\pm}-\widebar{\Phi}^{\pm})$,
and  the traces of the $\ell$-th order time derivatives on $\{t=0\}$  by
$$
\tilde{U}^{\pm}_{\ell}:=\p_t^{\ell}\tilde{U}^{\pm}|_{t=0},\quad
\tilde{\Phi}^{\pm}_{\ell}:=\p_t^{\ell}\tilde{\Phi}^{\pm}|_{t=0},\qquad\,\, \ell\in\mathbb{N}.
$$

To introduce the compatibility conditions, we need to determine traces $\tilde{U}^{\pm}_{\ell}$ and $\tilde{\Phi}^{\pm}_{\ell}$
in terms of the initial data $\tilde{U}^{\pm}_0$ and $\tilde{\Phi}^{\pm}_0$ through equations \eqref{RE0.a} and \eqref{Phi.eq.a}.
For this purpose, we set $\mathcal{W}^{\pm}:=(\tilde{U}^{\pm},\nabla_x\tilde{U}^{\pm},\nabla_x\tilde{\Phi}^{\pm})\in\mathbb{R}^{11}$,
and rewrite \eqref{RE0.a} and \eqref{Phi.eq.a} as
\begin{align}\label{tilde.U.Phi}
\p_t \tilde{U}^{\pm}=\mathbf{F}_1(\mathcal{W}^{\pm}),\qquad \p_t \tilde{\Phi}^{\pm}=\mathbf{F}_2(\mathcal{W}^{\pm}),
\end{align}
where $\mathbf{F}_1$ and $\mathbf{F}_2$ are suitable $C^{\infty}$--functions that vanish at the origin.
After applying operator $\p^\ell_t$ to \eqref{tilde.U.Phi}, we take the traces at time $t=0$.
One can employ the generalized Fa\`a di Bruno's formula ({\it cf}.\;\cite[Theorem 2.1]{M00MR1781515})
to derive
\begin{align} \label{tilde.U.0}
&\tilde{U}^{\pm}_{\ell+1}
=\sum_{\alpha_{i}\in\mathbb{N}^{11},|\alpha_1|+\cdots+\ell |\alpha_{\ell}|=\ell}
D^{\alpha_1+\cdots+\alpha_\ell}\mathbf{F}_1(\mathcal{W}^{\pm}_{0})\prod_{i=1}^\ell\frac{\ell!}{\alpha_{i}!}
\left(\frac{\mathcal{W}_{i}^{\pm}}{i!}\right)^{\alpha_{i}},\\ \label{tilde.Phi.0}
&\tilde{\Phi}^{\pm}_{\ell+1}
=\sum_{\alpha_{i}\in\mathbb{N}^{11},|\alpha_1|+\cdots+\ell |\alpha_{\ell}|=\ell}
D^{\alpha_1+\cdots+\alpha_\ell}\mathbf{F}_2(\mathcal{W}^{\pm}_{0})\prod_{i=1}^\ell\frac{\ell!}{\alpha_{i}!}
\left(\frac{\mathcal{W}_{i}^{\pm}}{i!}\right)^{\alpha_{i}},
\end{align}
where $\mathcal{W}_{i}^{\pm}$ denotes trace $(\tilde{U}_{i}^{\pm},\nabla_x\tilde{U}_{i}^{\pm},\nabla_x\tilde{\Phi}_{i}^{\pm})$
at $t=0$.
From \eqref{tilde.U.0}--\eqref{tilde.Phi.0}, one can determine
$(\tilde{U}^{\pm}_{\ell},\tilde{\Phi}^{\pm}_{\ell})_{\ell\geq 0}$ inductively as functions of
the initial data $\tilde{U}^{\pm}_{0}$ and $\tilde{\Phi}^{\pm}_{0}$.
Furthermore, we have the following lemma (see \cite[Lemma 4.2.1]{M01MR1842775} for the proof):

\begin{lemma}\  \label{lem.Metivier}
	Assume that \eqref{CA1}--\eqref{CA4} hold.
	Then the equations \eqref{tilde.U.0}--\eqref{tilde.Phi.0} determine
	$\tilde{U}^{\pm}_{\ell}\in H^{s+1/2-\ell}(\mathbb{R}_+^2)$ for $\ell=1,\ldots,s$,
	and $\tilde{\Phi}^{\pm}_{\ell}\in H^{s+3/2-\ell}(\mathbb{R}_+^2)$ for $\ell=1,\ldots,s+1$,
	such that
	\begin{align*}
	\supp \tilde{U}_{\ell}^{\pm}\subset \{x_2\geq 0,\, x_1^2+x_2^2\leq 1\}, \qquad
	\supp \tilde{\Phi}_{\ell}^{\pm}\subset \{x_2\geq 0,\, x_1^2+x_2^2\leq 2\}.
	\end{align*}
In addition,
	\begin{align} \notag
&	\sum_{\ell=0}^{s}\big\|\tilde{U}^{\pm}_{\ell}\big\|_{H^{s+1/2-\ell}(\mathbb{R}_+^2)}
	+\sum_{\ell=0}^{s+1}\big\|\tilde{\Phi}^{\pm}_{\ell}\big\|_{H^{s+3/2-\ell}(\mathbb{R}_+^2)}\\
&\notag \qquad 	\leq C\big(\big\|\tilde{U}^{\pm}_0\big\|_{H^{s+1/2}(\mathbb{R}_+^2)}+\|\varphi_0\|_{H^{s+1}(\mathbb{R})} \big),
	\end{align}
	where constant $C>0$ depends only on $s$
	and $\|(\tilde{U}^{\pm}_{0},\tilde{\Phi}^{\pm}_{0})\|_{W^{1,\infty}(\mathbb{R}_+^2)}$.
\end{lemma}

To construct a \emph{smooth} approximate solution, one has to impose certain assumptions on
traces $\tilde{U}^{\pm}_{\ell}$ and $\tilde{\Phi}^{\pm}_{\ell}$.
We are now ready to introduce the following terminology.

\begin{definition}[Compatibility conditions]\quad
	Let $s\geq 3$ be an integer.
	Let $\tilde{U}^{\pm}_0:=U_0^{\pm}-\widebar{U}_0^{\pm}\in H^{s+1/2}(\mathbb{R}_+^2)$ and $\varphi_0\in H^{s+1}(\mathbb{R})$
	satisfy \eqref{CA1}.
	The initial data $U_0^{\pm}$ and $\varphi_0$ are said to be compatible up to order $s$
	if there exist functions $\tilde{\Phi}_0^{\pm}\in H^{s+3/2}(\mathbb{R}_+^2)$
	satisfying \eqref{CA2}--\eqref{CA4}
	such that
	functions $\tilde{U}^{\pm}_1,\ldots,\tilde{U}^{\pm}_{s},\tilde{\Phi}^{\pm}_1,\ldots,\tilde{\Phi}^{\pm}_{s+1}$
	determined by \eqref{tilde.U.0}--\eqref{tilde.Phi.0} satisfy:
	\begin{subequations} \label{compa1}
		\begin{alignat}{2}
		&\p_2^{j}\big(\tilde{\Phi}^{+}_{\ell}-\tilde{\Phi}^{-}_{\ell}\big)|_{x_2=0}=0
		\qquad\,\, && \textrm{for }j,\ell\in\mathbb{N}\textrm{ with } j+\ell< s+1,\\
		&\p_2^{j}\big(\tilde{p}^{+}_{\ell}-\tilde{p}^{-}_{\ell}\big)|_{x_2=0}=0
		\qquad\,\, && \textrm{for }j,\ell\in\mathbb{N} \textrm{ with } j+\ell< s,
		\end{alignat}
	\end{subequations}
	and
	\begin{subequations}\label{compa2}
		\begin{alignat}{2}
		&\int_{\mathbb{R}_+^2}\big|\p_2^{s+1-\ell}\big(\tilde{\Phi}^{+}_{\ell}-\tilde{\Phi}^{-}_{\ell}\big)\big|^2\d x_1\frac{\d x_2}{x_2}
		<\infty\qquad\,\, && \textrm{for }\ell=0,\ldots,  s+1,\\
		&\int_{\mathbb{R}_+^2}\big|\p_2^{s-\ell}\big(\tilde{p}^{+}_{\ell}-\tilde{p}^{-}_{\ell}\big)\big|^2\d x_1\frac{\d x_2}{x_2}
		<\infty\qquad\,\, && \textrm{for }\ell=0,\ldots, s.
		\end{alignat}
	\end{subequations}
\end{definition}

It follows from Lemma \ref{lem.Metivier}
that $\tilde{p}^{\pm}_{0},\ldots,\tilde{p}^{\pm}_{s-2},\tilde{\Phi}^{\pm}_{0},\ldots,
\tilde{\Phi}^{\pm}_{s-1}\in  H^{5/2}(\mathbb{R}_+^2)\subset W^{1,\infty}(\mathbb{R}_+^2)$.
Then we can take the $j$-th order derivatives of the traces in \eqref{compa1}.
In what follows, we employ $\varepsilon_0(\cdot)$ to denote a function that tends to $0$ when its argument tends to $0$.
Relations \eqref{compa1}--\eqref{compa2}  enable us to utilize the lifting result
in \cite[Theorem 2.3]{LM72MR0350178} to construct the approximate solution in the following lemma.
We refer to \cite[Lemma 3]{CS08MR2423311} for the  proof.

\begin{lemma}\  \label{lem.app}
	Let $s\geq 3$ be an integer.
	Assume that $\tilde{U}^{\pm}_0:=U_0^{\pm}-\widebar{U}_0^{\pm}\in H^{s+1/2}(\mathbb{R}_+^2)$
	and $\varphi_0\in H^{s+1}(\mathbb{R})$ satisfy \eqref{CA1},
	and that the initial data $U_0^{\pm}$ and $\varphi_0$ are compatible up to order $s$.
	If $\tilde{U}^{\pm}_0$ and $\varphi_0$ are sufficiently small,
	then there exist functions $U^{a\pm}$, $\Phi^{a\pm}$, and $\varphi^a$ such that
	$\tilde{U}^{a\pm}:=U^{a\pm}-\widebar{U}^{\pm}\in H^{s+1}(\Omega)$,
	$\tilde{\Phi}^{a\pm}:=\Phi^{a\pm}-\widebar{\Phi}^{\pm}\in H^{s+2}(\Omega)$,
	$\varphi^a\in H^{s+3/2}(\p\Omega)$, and
	\begin{subequations} \label{app}
		\begin{alignat}{2}
		\label{app.1}&\p_t\Phi^{a\pm}+v_1^{a\pm}\p_1\Phi^{a\pm}-v_2^{a\pm}=0\qquad &&\textrm{in }\Omega,\\
		\label{app.2}&\p_t^j\mathbb{L}(U^{a\pm},\Phi^{a\pm})|_{t=0}=0\qquad &&\textrm{for }j=0,\ldots,s-1,\\
		\label{app.3}&\Phi^{a+}=\Phi^{a-}=\varphi^a\qquad &&\textrm{on }\p\Omega,\\
		\label{app.4}&\mathbb{B}(U^{a+},U^{a-},\varphi^a)=0\qquad &&\textrm{on }\p\Omega.
		\end{alignat}
	\end{subequations}
	Furthermore, we have
	\begin{align}
	\label{app2}&\pm \p_2\Phi^{a\pm}\geq \frac{3}{4}\quad \textrm{for all }(t,x)\in\Omega,\\
	\label{app4}&\supp \big(\tilde{U}^{a\pm},\tilde{\Phi}^{a\pm} \big)\subset \left\{t\in[-T,T],\,x_2\geq 0,\,x_1^2+x_2^2\leq 3 \right\},
	\end{align}
	and
	\begin{align}
	&\big\|\tilde{U}^{a\pm}\big\|_{H^{s+1}(\Omega)}
	+\big\|\tilde{\Phi}^{a\pm}\big\|_{H^{s+2}(\Omega)}+\|\varphi^a\|_{H^{s+3/2}(\p\Omega)}
	\notag \\
	&\quad \leq\varepsilon_0\big(\big\|\tilde{U}^{\pm}_0\big\|_{H^{s+1/2}(\mathbb{R}_+^2)}
	+\|\varphi_0\|_{H^{s+1}(\mathbb{R})}\big). \label{app3}
	\end{align}
\end{lemma}

Let us denote $U^a:=(U^{a+},U^{a-})^{\mathsf{T}}$ and  $\Phi^a:=(\Phi^{a+},\Phi^{a-})^{\mathsf{T}}$.
Vector $(U^a,\Phi^a)$ in Lemma \ref{lem.app} is called the approximate solution to problem \eqref{RE0} and \eqref{Phi.eq}.
Relations \eqref{app.3} and \eqref{app4} immediately imply that $\varphi^a$ is supported within $\{-T\le t\le T,\,x_1^2\leq 3\}$.
Since $s\geq 3$, it follows from \eqref{app3} and the Sobolev embedding theorem that
\begin{align} \notag
\big\|\tilde{U}^{a\pm}\big\|_{W^{2,\infty}(\Omega)}+\big\|\tilde{\Phi}^{a\pm}\big\|_{W^{3,\infty}(\Omega)}
\leq\varepsilon_0\big(\big\|\tilde{U}^{\pm}_0\big\|_{H^{s+1/2}(\mathbb{R}_+^2)}+\|\varphi_0\|_{H^{s+1}(\mathbb{R})}\big).
\end{align}

We are going to reformulate the original problem into that with zero initial data
by using the approximate solution $(U^{a},\Phi^{a})$. Let us introduce
\begin{align}\label{f.a}
f^{a}:=\left\{\begin{aligned}
& -\mathbb{L}(U^{a},\Phi^{a}) \qquad &\textrm{if }t>0,\\
& 0 \qquad &\textrm{if }t<0.
\end{aligned}\right.
\end{align}
Since $(\tilde{U}^{a\pm},\nabla\tilde{\Phi}^{a\pm})\in  H^{s+1}(\Omega)$,
taking into account \eqref{app.2} and \eqref{app4}, we obtain that $f^{a}\in H^{s}(\Omega)$ and
\begin{align} \notag
\supp f^{a}\subset \left\{0\le t\le T,\,x_2\geq 0,\,x_1^2+x_2^2\leq 3 \right\}.
\end{align}
Using the Moser-type inequalities and the fact that $f^{a}$ vanishes in the past,
we obtain from \eqref{app3} the estimate:
\begin{align}\label{app5}
\|f^{a}\|_{ H^{s}(\Omega)}\leq \varepsilon_0\big(\big\|\tilde{U}^{\pm}_0\big\|_{H^{s+1/2}(\mathbb{R}_+^2)}
+\|\varphi_0\|_{H^{s+1}(\mathbb{R})}\big).
\end{align}

Let $(U^a,\Phi^a)$ be the approximate solution defined in Lemma \ref{lem.app}.
By virtue of \eqref{app} and \eqref{f.a}, we see that $(U,\Phi)=(U^a,\Phi^a)+(V,\Psi)$ is a solution of
the original problem  \eqref{RE0} and \eqref{Phi.eq} on $[0,T]\times \mathbb{R}_+^2$,
if $V=(V^+,V^-)^{\mathsf{T}}$ and $\Psi=(\Psi^+,\Psi^-)^{\mathsf{T}}$ solve the following problem:
\begin{align} \label{P.new}
\left\{\begin{aligned}
&\mathcal{L}(V,\Psi):=\mathbb{L}(U^a+V,\Phi^a+\Psi)-\mathbb{L}(U^a,\Phi^a)=f^a\quad&&\textrm{in }\Omega_T,\\
&\mathcal{E}(V,\Psi):=\p_t\Psi+(v_1^a+v_1)\p_1\Psi+v_1\p_1\Phi^a-v_2=0\quad&&\textrm{in }\Omega_T,\\
&\mathcal{B}(V,\psi):=\mathbb{B}(U^a+V,\varphi^a+\psi)=0\quad&&\textrm{on }\omega_T,\\
&\Psi^+=\Psi^-=\psi\quad&&\textrm{on }\omega_T,\\
&(V,\Psi)=0,\quad&&\textrm{for }t< 0.
\end{aligned}\right.
\end{align}
The initial data \eqref{RE0.c} and \eqref{Phi.initial} have been absorbed into the interior equations.
From \eqref{app.1} and \eqref{app.4}, we observe that $(V,\Psi)=0$ satisfies \eqref{P.new} for $t<0$.
Therefore, the original nonlinear problem on $[0,T]\times \mathbb{R}_+^2$ is now reformulated
as a problem on $\Omega_T$ whose solutions vanish in the past.

\section{Nash--Moser Iteration}\label{sec.NM}

In this section, we prove the existence of solutions to problem \eqref{P.new}
by a suitable iteration scheme of Nash--Moser type ({\it cf.}\;H\"ormander \cite{H76MR0602181}).
First, we introduce the smoothing operators $S_{\theta}$ and describe the iterative scheme
for problem \eqref{P.new}.

\subsection{The Iterative Scheme}
We first state the following result from \cite[Proposition 4]{CS08MR2423311}.

\begin{proposition}\ \label{pro.smooth}
	Let $T>0$ and $\gamma\geq 1$, and let $m\geq 4$ be an integer.
	Then there exists a family $\{S_{\theta}\}_{\theta\geq 1}$ of smoothing operators{\rm :}
	\begin{align*}
	S_{\theta}:\ \mathcal{F}_{\gamma}^3(\Omega_T)\times\mathcal{F}_{\gamma}^3(\Omega_T)
	\longrightarrow \bigcap_{\beta\geq 3}\mathcal{F}_{\gamma}^\beta(\Omega_T)\times\mathcal{F}_{\gamma}^\beta(\Omega_T),
	\end{align*}
	where $\mathcal{F}_{\gamma}^\beta(\Omega_T):=\big\{u\in H^{\beta}_{\gamma}(\Omega_T):u=0\textrm{ for }t<0\big\}$
	is a closed subspace of $H^{\beta}_{\gamma}(\Omega_T)$ such that
	\begin{subequations}\label{smooth.p1}
		\begin{alignat}{2}
		\label{smooth.p1a}&\|S_{\theta} u\|_{H^{\beta}_{\gamma}(\Omega_T)}\leq C\theta^{(\beta-\alpha)_+}\|u\|_{H^{\alpha}_{\gamma}(\Omega_T)}
		&&\quad\textrm{for all }\alpha,\beta\in[1,m],\\[1.5mm]
		\label{smooth.p1b}&\|S_{\theta} u-u\|_{H^{\beta}_{\gamma}(\Omega_T)}\leq C\theta^{\beta-\alpha}\|u\|_{H^{\alpha}_{\gamma}(\Omega_T)}
		&&\quad\textrm{for all }1\leq \beta\leq \alpha\leq m,\\
		\label{smooth.p1c}&\left\|\frac{\d}{\d \theta}S_{\theta} u\right\|_{H^{\beta}_{\gamma}(\Omega_T)}
		\leq C\theta^{\beta-\alpha-1}\|u\|_{H^{\alpha}_{\gamma}(\Omega_T)}&&\quad\textrm{for all }\alpha,\beta\in[1,m],
		\end{alignat}
	\end{subequations}
	and
	\begin{align}
&\notag	\|(S_{\theta}u-S_{\theta}v)|_{x_2=0}\|_{H^{\beta}_{\gamma}(\omega_T)}\\
&\label{smooth.p2}\quad 	\leq C\theta^{(\beta+1-\alpha)_+}\|(u-v)|_{x_2=0}\|_{H^{\alpha}_{\gamma}(\omega_T)} \quad \, \textrm{for all }\alpha,\beta\in[1,m],
	\end{align}
	where $\alpha,\beta\in\mathbb{N}$, $(\beta-\alpha)_+:=\max\{0,\beta-\alpha\}$, and $C>0$ is a constant depending only on $m$.
	In particular, if $u=v$ on $\omega_T$, then $S_{\theta}u=S_{\theta}v$ on $\omega_T$.
	Furthermore, there exists another family of smoothing operators {(}still denoted by $S_{\theta}${)} acting on the functions
	defined on the boundary $\omega_T$ and satisfying the properties in \eqref{smooth.p1} with
	norms $\|\cdot\|_{H^{\alpha}_{\gamma}(\omega_T)}$.
\end{proposition}

The proof of \eqref{smooth.p2} is based on the following lifting
operator (see \cite[Chapter 5]{FM00MR1787068} and \cite{CS08MR2423311}).

\begin{lemma}\  \label{lem.smooth2}
	Let $T>0$ and $\gamma\geq 1$, and let $m\geq 1$ be an integer.
	Then there exists an operator $\mathcal{R}_T$, which is continuous from $\mathcal{F}_{\gamma}^s(\omega_T)$
	to $\mathcal{F}_{\gamma}^{s+1/2}(\Omega_T)$ for all $s\in [1,m]$,
	such that, if $s\geq 1$ and $u\in  \mathcal{F}_{\gamma}^s(\omega_T)$,
	then $(\mathcal{R}_T u)|_{x_2=0}=u$.
\end{lemma}

Now, following \cite{CS08MR2423311}, we describe the iteration scheme for problem \eqref{P.new}.

\vspace*{2mm}
\noindent{\bf Assumption\;(A-1)}: \emph{$ (V_0,\Psi_0, \psi_0)=0$ and, for $k=0,\ldots,{n}$,
	$(V_k,\Psi_k,\psi_k)$ are already given  and satisfy}
\begin{align}\label{NM.H1}
(V_k,\Psi_k,\psi_k)|_{t<0}=0,\quad \Psi_k^{+}|_{x_2=0}=\Psi_k^{-}|_{x_2=0}=\psi_k.
\end{align}

Let us consider
\begin{align}\label{NM0}
V_{{n}+1}=V_{{n}}+\delta V_{{n}},\quad \Psi_{{n}+1}=\Psi_{{n}}+\delta \Psi_{{n}},
\quad\,\, \psi_{{n}+1}=\psi_{{n}}+\delta \psi_{{n}},
\end{align}
where these differences $\delta V_{{n}}$, $\delta \Psi_{{n}}$, and $\delta \psi_{{n}}$ will be specified below.

First we are going to find $(\delta \dot{V}_{{n}},\delta\psi_{{n}})$ by solving the effective linear problem:
\begin{align} \label{effective.NM}
\left\{\begin{aligned}
&\mathbb{L}_e'(U^a+V_{{n}+1/2},\Phi^a+\Psi_{{n}+1/2})\delta \dot{V}_{{n}}=f_{{n}}
\qquad &&\textrm{in }\Omega_T,\\
& \mathbb{B}_e'(U^a+V_{{n}+1/2},\Phi^a+\Psi_{{n}+1/2})(\delta \dot{V}_{{n}},\delta\psi_{{n}})=g_{{n}}
\qquad &&\textrm{on }\omega_T,\\
& (\delta \dot{V}_{{n}},\delta\psi_{{n}})=0\qquad &&\textrm{for }t<0,
\end{aligned}\right.
\end{align}
where operators  $\mathbb{L}_e'$ and $\mathbb{B}_e'$ are defined in \eqref{P2b.1}--\eqref{P2b.3},
\begin{align} \label{good.NM}
\delta \dot{V}_{{n}}:=\delta V_{{n}}-\frac{\p_2 (U^a+V_{{n}+1/2})}{\p_2 (\Phi^a+\Psi_{{n}+1/2})}\delta\Psi_{{n}}
\end{align}
is the ``good unknown'' ({\it cf.}\;\eqref{good}), and $(V_{{n}+1/2},\Psi_{{n}+1/2})$ is a smooth modified state
such that $(U^a+V_{{n}+1/2},\Phi^a+\Psi_{{n}+1/2})$ satisfies
constraints \eqref{bas.c1}--\eqref{bas.eq}.
The source term $(f_{{n}},g_{{n}})$ will be defined through the accumulated error terms
at Step ${n}$ later on.

We define the modified state as
\begin{align} \label{modified}
\left\{\begin{aligned}
&\Psi_{{n}+1/2}^{\pm}:=S_{\theta_{{n}}}\Psi_{{n}}^{\pm},\quad v_1\big(V_{{n}+1/2}^{\pm}\big)
:=S_{\theta_{{n}}}v_1\big(V_{{n}}^{\pm}\big),\\
&p_{{n}+1/2}^{\pm}:=S_{\theta_{{n}}}p_{{n}}^{\pm}\mp \tfrac{1}{2}\mathcal{R}_T\big(S_{\theta_{{n}}}p_{{n}}^{+}|_{x_2=0}
-S_{\theta_{{n}}}p_{{n}}^{-}|_{x_2=0}  \big),\\
&v_2\big(V_{{n}+1/2}^{\pm}\big):=\p_t\Psi_{{n}+1/2}^{\pm}+\big(v_1^{a\pm}+v_1\big(V_{{n}+1/2}^{\pm}\big) \big)\p_1\Psi_{{n}+1/2}^{\pm}\\
&\qquad \qquad \qquad +v_1\big(V_{{n}+1/2}^{\pm}\big)\p_1\Phi^{a\pm},
\end{aligned}\right.
\end{align}
where $S_{\theta_{{n}}}$ are the smoothing operators defined in Proposition \ref{pro.smooth} with
sequence $\{\theta_{{n}}\}$ given by
\begin{align} \label{theta}
\theta_0\geq 1,\qquad \theta_{{n}}=\sqrt{\theta^2_0+{n}},
\end{align}
and $\mathcal{R}_T$ is the lifting operator given in Lemma \ref{lem.smooth2}.
Thanks to \eqref{NM.H1}, we have
\begin{align} \label{modified.1}
\left\{\begin{aligned}
&\Psi_{{n}+1/2}^{+}|_{x_2=0}=\Psi_{{n}+1/2}^{-}|_{x_2=0}=:\psi_{{n}+1/2},\\
&p_{{n}+1/2}^{+}|_{x_2=0}=p_{{n}+1/2}^{-}|_{x_2=0},\\
&\mathcal{E}(V_{{n}+1/2},\Psi_{{n}+1/2})=0,\\
&\big(V_{{n}+1/2},\Psi_{{n}+1/2},\psi_{{n}+1/2} \big)|_{t<0}=0.
\end{aligned}\right.
\end{align}
It then follows from \eqref{app} that $(U^a+V_{{n}+1/2},\Phi^a+\Psi_{{n}+1/2})$
satisfies \eqref{bas.eq.1} and \eqref{bas.eq.3}--\eqref{bas.eq.4}.
The additional constraint \eqref{bas.eq.2} will be obtained by taking the initial data small enough.

The error terms at Step ${n}$ are defined from the following decompositions:
\begin{align}
\notag&\mathcal{L}(V_{{n}+1},\Psi_{{n}+1})-\mathcal{L}(V_{{n}},\Psi_{{n}})\\
\notag&\quad = \mathbb{L}'(U^a+V_{{n}},\Phi^a+\Psi_{{n}})(\delta V_{{n}},\delta\Psi_{{n}})+e_{{n}}'\\
\notag&\quad = \mathbb{L}'(U^a+S_{\theta_{{n}}}V_{{n}},\Phi^a+S_{\theta_{{n}}}\Psi_{{n}})(\delta V_{{n}},\delta\Psi_{{n}})+e_{{n}}'+e_{{n}}''\\
\notag&\quad = \mathbb{L}'(U^a+V_{{n}+1/2},\Phi^a+\Psi_{{n}+1/2})(\delta V_{{n}},\delta\Psi_{{n}})+e_{{n}}'+e_{{n}}''+e_{{n}}'''\\
\label{decom1}&\quad = \mathbb{L}_e'(U^a+V_{{n}+1/2},\Phi^a+\Psi_{{n}+1/2})\delta \dot{V}_{{n}}+e_{{n}}'+e_{{n}}''+e_{{n}}'''+D_{{n}+1/2} \delta\Psi_{{n}}
\end{align}
and
\begin{align}
\notag&\mathcal{B}(V_{{n}+1}|_{x_2=0},\psi_{{n}+1})-\mathcal{B}(V_{{n}}|_{x_2=0},\psi_{{n}})\\
\notag&\quad = \mathbb{B}'(U^a+V_{{n}},\Phi^a+\Psi_{{n}})(\delta V_{{n}}|_{x_2=0},\delta\psi_{{n}})+\tilde{e}_{{n}}'\\
\notag&\quad = \mathbb{B}'(U^a+S_{\theta_{{n}}}V_{{n}},\Phi^a+S_{\theta_{{n}}}\Psi_{{n}})(\delta V_{{n}}|_{x_2=0},\delta\psi_{{n}})+\tilde{e}_{{n}}'+\tilde{e}_{{n}}''\\
\label{decom2}&\quad =\mathbb{B}_e'(U^a+V_{{n}+1/2},\Phi^a+\Psi_{{n}+1/2})(\delta \dot{V}_{{n}}|_{x_2=0},\delta\psi_{{n}})+\tilde{e}_{{n}}'+\tilde{e}_{{n}}''+\tilde{e}_{{n}}''',
\end{align}
where we have set
\begin{align}\label{error.D}
D_{{n}+1/2}:=\frac{1}{\p_2(\Phi^a+\Psi_{{n}+1/2})}\p_2\mathbb{L}(U^a+V_{{n}+1/2},\Phi^a+\Psi_{{n}+1/2}),
\end{align}
and have used \eqref{Alinhac} to obtain the last identity in \eqref{decom1}.
Let us set
\begin{align} \label{e.e.tilde}
e_{{n}}:=e_{{n}}'+e_{{n}}''+e_{{n}}'''+D_{{n}+1/2} \delta\Psi_{{n}},\qquad
\tilde{e}_{{n}}:=\tilde{e}_{{n}}'+\tilde{e}_{{n}}''+\tilde{e}_{{n}}'''.
\end{align}

\vspace*{1mm}
\noindent{\bf Assumption\;(A-2)}: \emph{$f_0:=S_{\theta_0}f^a$, $(e_0,\tilde{e}_0,g_0):=0$, and $(f_k,g_k,e_k,\tilde{e}_k)$
	are already given and vanish in the past for $k=0,\ldots,{n}-1$.}

\vspace*{2mm}
We compute the accumulated error terms at Step ${n}$,  $n\geq 1$, by
\begin{align}  \label{E.E.tilde}
E_{{n}}:=\sum_{k=0}^{{n}-1}e_{k},\quad \widetilde{E}_{{n}}:=\sum_{k=0}^{{n}-1}\tilde{e}_{k}.
\end{align}
Then we compute $f_{{n}}$ and $g_{{n}}$ for ${n}\geq 1$  from the equations:
\begin{align} \label{source}
\sum_{k=0}^{{n}} f_k+S_{\theta_{{n}}}E_{{n}}=S_{\theta_{{n}}}f^a,
\qquad \sum_{k=0}^{{n}}g_k+S_{\theta_{{n}}}\widetilde{E}_{{n}}=0.
\end{align}

Under assumptions {(A-1)}--{(A-2)},
$(V_{{n}+1/2},\Psi_{{n}+1/2})$ and $(f_{{n}},g_{{n}})$ have been specified from \eqref{modified} and \eqref{source}.
Then we can obtain $(\delta \dot{V}_{{n}},\delta\psi_{{n}})$ as the solution of
the linear problem  \eqref{effective.NM} by applying Theorem {\rm\ref{thm.L}}.

Next we need to construct $\delta\Psi_{{n}}=(\delta\Psi_{{n}}^+,\delta\Psi_{{n}}^-)^{\mathsf{T}}$
satisfying $\delta\Psi_{{n}}^{\pm}|_{x_2=0}=\delta\psi_{{n}}$.
We use the boundary conditions in \eqref{effective.NM}
({\it cf.}\;\eqref{b.ring}--\eqref{M.ring}, \eqref{b.sharp}, and \eqref{good.NM})
to derive that $\delta\psi_{{n}}$ satisfies
\begin{align}
\notag & \bigg(\dfrac{\epsilon^2\p_t \Phi_{{n}+1/2}^+}{N^+_{{n}+1/2}h^+_{{n}+1/2}(\varGamma_{{n}+1/2}^+)^2},
\;  \dfrac{\varrho^+_{{n}+1/2} }{h^+_{{n}+1/2}\varGamma_{{n}+1/2}^+},\;
\dfrac{-\varsigma^+_{{n}+1/2}}{h^+_{{n}+1/2}\varGamma_{{n}+1/2}^+}\bigg)\bigg|_{x_2=0}\\
\notag&\quad  \times\bigg( \delta \dot{V}_{{n}}^++\delta\psi_{{n}} \frac{\p_2 U^+_{{n}+1/2} }{\p_2\Phi^+_{{n}+1/2}} \bigg)\bigg|_{x_2=0}
 +\p_t(\delta\psi_{{n}})+v_{1}\big(U_{{n}+1/2}^{+}\big)\big|_{x_2=0}\p_1(\delta\psi_{{n}})\\
\notag& =g_{{n},2},\\[2.5mm]
\notag &\bigg(\dfrac{\epsilon^2\p_t \Phi_{{n}+1/2}^-}{N^-_{{n}+1/2}h^-_{{n}+1/2}(\varGamma_{{n}+1/2}^-)^2},
\;  \dfrac{\varrho^-_{{n}+1/2} }{h^-_{{n}+1/2}\varGamma_{{n}+1/2}^-},\;
\dfrac{-\varsigma^-_{{n}+1/2}}{h^-_{{n}+1/2}\varGamma_{{n}+1/2}^-}\bigg)\bigg|_{x_2=0}\\
\notag&\quad
\times \bigg(\delta \dot{V}_{{n}}^- +\delta\psi_{{n}} \frac{\p_2 U^-_{{n}+1/2} }{\p_2\Phi^-_{{n}+1/2}} \bigg)\bigg|_{x_2=0}
+\p_t(\delta\psi_{{n}})+v_{1}\big(U_{{n}+1/2}^{-}\big)\big|_{x_2=0}\p_1(\delta\psi_{{n}})\\
\notag&=g_{{n},2}-g_{{n},1},
\end{align}
where we have set $U_{{n}+1/2}^{\pm}:=U^{a\pm}+V^{\pm}_{{n}+1/2}$, $\Phi_{{n}+1/2}^{\pm}:=\Phi^{a\pm}+\Psi^{\pm}_{{n}+1/2}$, and
\begin{align*}
(N^{\pm}_{{n}+1/2},h^{\pm}_{{n}+1/2},\varGamma^{\pm}_{{n}+1/2},\varrho^{\pm}_{{n}+1/2},\varsigma^{\pm}_{{n}+1/2})
:=(N,h,\varGamma,\varrho,\varsigma)(U_{{n}+1/2}^{\pm},\Phi_{{n}+1/2}^{\pm}),
\end{align*}
with $(N, h,\varGamma,\varrho,\varsigma)$ defined in \eqref{N.c.def}, \eqref{Gamma.v}, and \eqref{m.n.ring}.
Then we define $\delta\Psi_{{n}}^+$ and $\delta\Psi_{{n}}^-$ as solutions to the following equations:
\begin{align}
\notag & \bigg(\dfrac{\epsilon^2\p_t \Phi_{{n}+1/2}^+  }{N^+_{{n}+1/2}h^+_{{n}+1/2}(\varGamma_{{n}+1/2}^+)^2},
\;  \dfrac{\varrho^+_{{n}+1/2} }{h^+_{{n}+1/2}\varGamma_{{n}+1/2}^+},\;   \dfrac{-\varsigma^+_{{n}+1/2}}{h^+_{{n}+1/2}\varGamma_{{n}+1/2}^+}\bigg)\\
&\notag \quad \times
 \bigg(\delta\Psi_{{n}}^+ \frac{\p_2 U^+_{{n}+1/2} }{\p_2\Phi^+_{{n}+1/2}}+ \delta \dot{V}_{{n}}^+\bigg)
+\p_t(\delta\Psi_{{n}}^+)+v_{1}\big(U_{{n}+1/2}^{+}\big)\p_1(\delta\Psi_{{n}}^+)\\
&\label{delta.Psi1}=\mathcal{R}_Tg_{{n},2}+G_{{n}}^+,\\
\notag& \bigg(\dfrac{\epsilon^2\p_t \Phi_{{n}+1/2}^-  }{N^-_{{n}+1/2}h^-_{{n}+1/2}(\varGamma_{{n}+1/2}^-)^2},
\;  \dfrac{\varrho^-_{{n}+1/2} }{h^-_{{n}+1/2}\varGamma_{{n}+1/2}^-},\;   \dfrac{-\varsigma^-_{{n}+1/2}}{h^-_{{n}+1/2}\varGamma_{{n}+1/2}^-}\bigg)
\\
&\notag \quad \times\bigg(\delta\Psi_{{n}}^- \frac{\p_2 U^-_{{n}+1/2} }{\p_2\Phi^-_{{n}+1/2}}  +\delta \dot{V}_{{n}}^-\bigg)
 +\p_t(\delta\Psi_{{n}}^-)+v_{1}\big(U_{{n}+1/2}^{-}\big)\p_1(\delta\Psi_{{n}}^-)\\
 &\label{delta.Psi2}=\mathcal{R}_T(g_{{n},2}-g_{{n},1})+G_{{n}}^-,
\end{align}
where the source terms $G_{{n}}^{\pm}$ will be chosen by using a decomposition for operator $\mathcal{E}$.

We define the error terms: $\hat{e}_{{n}}'$, $\hat{e}_{{n}}''$, and $\hat{e}_{{n}}'''$ by
\begin{align}
&\mathcal{E}(V_{{n}+1},\Psi_{{n}+1})-\mathcal{E}(V_{{n}},\Psi_{{n}})
= \mathcal{E}'(V_{{n}},\Psi_{{n}})(\delta V_{{n}},\delta\Psi_{{n}})+\hat{e}_{{n}}' \notag  \\
&\quad = \mathcal{E}'(S_{\theta_{{n}}}V_{{n}},S_{\theta_{{n}}}\Psi_{{n}})(\delta V_{{n}},\delta\Psi_{{n}})
+\hat{e}_{{n}}'+\hat{e}_{{n}}''\notag  \\
&\quad = \mathcal{E}'(V_{{n}+1/2},\Psi_{{n}+1/2})(\delta V_{{n}},\delta\Psi_{{n}})+\hat{e}_{{n}}'
+\hat{e}_{{n}}''+\hat{e}_{{n}}''', \label{decom3}
\end{align}
and denote
\begin{align} \label{e.hat}
\hat{e}_{{n}}:=\hat{e}_{{n}}'+\hat{e}_{{n}}''+\hat{e}_{{n}}''',\qquad \hat{E}_{{n}}:=\sum_{k=0}^{{n}-1}\hat{e}_{k}.
\end{align}
From \eqref{app.1}, we have
\begin{align*}
\mathcal{E}(V,\Psi)=\p_t(U^a+V)+v_1(U^a+V)\p_1 (\Phi^a+\Psi)-v_2(U^a+V).
\end{align*}
Similar to the derivation of \eqref{B'.bb} and \eqref{P2b.3},
we deduce that $$\mathcal{E}'(V_{{n}+1/2}^{\pm},\Psi_{{n}+1/2}^{\pm})(\delta V_{{n}}^{\pm},\delta\Psi_{{n}}^{\pm})$$
are equal to the left-hand sides of  \eqref{delta.Psi1}--\eqref{delta.Psi2}, respectively.
Then it follows from \eqref{delta.Psi1}--\eqref{decom3}
that
\begin{align}\notag 
\mathcal{E}(V_{{n}+1},\Psi_{{n}+1})-\mathcal{E}(V_{{n}},\Psi_{{n}}) =\begin{pmatrix}
\mathcal{R}_Tg_{{n},2}+G_{{n}}^++\hat{e}_{{n}}^+\\ \mathcal{R}_T(g_{{n},2}-g_{{n},1})+G_{{n}}^-+\hat{e}_{{n}}^-
\end{pmatrix}.
\end{align}
Summing these relations and using $\mathcal{E}(V_0,\Psi_0)=0$ yield
\begin{align*}
\mathcal{E}(V_{{n}+1}^-,\Psi_{{n}+1}^-)=\mathcal{R}_T\Big(\sum_{k=0}^{{n}}(g_{k,2}-g_{k,1})\Big)+\sum_{k=0}^{{n}}G_{k}^-+\hat{E}_{{n}+1}^-.
\end{align*}
On the other hand, we obtain from \eqref{effective.NM} and \eqref{decom2} that
\begin{align} \label{decom2.b}
g_{{n}}=\mathcal{B}(V_{{n}+1}|_{x_2=0},\psi_{{n}+1})-\mathcal{B}(V_{{n}}|_{x_2=0},\psi_{{n}})-\tilde{e}_{{n}}.
\end{align}
In view of \eqref{P.new} and \eqref{B.def},  one obtains the relations:
\begin{align}
&\notag \big(\mathcal{B}(V_{{n}+1}|_{x_2=0},\psi_{{n}+1})\big)_2\\
&\quad =\mathcal{E}(V_{{n}+1}^+|_{x_2=0},\psi_{{n}+1})\notag  \\
& \quad =\mathcal{E}(V_{{n}+1}^-|_{x_2=0},\psi_{{n}+1})+\big(\mathcal{B}(V_{{n}+1}|_{x_2=0},\psi_{{n}+1})\big)_1. \label{B.E.relation}
\end{align}
Summing \eqref{decom2.b} and using $\mathcal{B}(V_{0}|_{x_2=0},\psi_{0})=0$, we have
\begin{align}\notag
&\mathcal{E}(V_{{n}+1}^-,\Psi_{{n}+1}^-)\\
& \label{decom3.c1} \quad =\mathcal{R}_T\Big(\mathcal{E}\big(V_{{n}+1}^-|_{x_2=0},\psi_{{n}+1}\big)-\widetilde{E}_{{n}+1,2}+\widetilde{E}_{{n}+1,1}\Big)+\sum_{k=0}^{{n}}G_{k}^-+\hat{E}_{{n}+1}^-.
\end{align}
Similarly, we can also obtain
\begin{align} \notag
&\mathcal{E}(V_{{n}+1}^+,\Psi_{{n}+1}^+)\\ \label{decom3.c2}
&\quad =\mathcal{R}_T\Big(\mathcal{E}\big(V_{{n}+1}^+|_{x_2=0},\psi_{{n}+1}\big)
-\widetilde{E}_{{n}+1,2}\Big)+\sum_{k=0}^{{n}}G_{k}^++\hat{E}_{{n}+1}^+.
\end{align}

\vspace*{1mm}
\noindent{\bf Assumption\;(A-3)}: \emph{$(G_0^+,G_0^-,\hat{e}_0)=0$, and $(G_k^+,G_k^-,\hat{e}_k)$ are already given
	and vanish in the past for $k=0,\ldots,{n}-1$.}

\vspace*{2mm}
Under assumptions {(A-1)}--{(A-3)}, taking into account  \eqref{decom3.c1}--\eqref{decom3.c2}
and the property of $\mathcal{R}_T$,
we compute the source terms $G_{{n}}^{\pm}$ from
\begin{subequations} \label{source2}
	\begin{alignat}{1}
	&S_{\theta_{{n}}}\big(\hat{E}_{{n}}^+-\mathcal{R}_T\widetilde{E}_{{n},2}\big)+\sum_{k=0}^{{n}}G_{k}^+=0,\\
	&S_{\theta_{{n}}}\big(\hat{E}_{{n}}^--\mathcal{R}_T\widetilde{E}_{{n},2}+\mathcal{R}_T\widetilde{E}_{{n},1}\big)
	+\sum_{k=0}^{{n}}G_{k}^-=0.
	\end{alignat}
\end{subequations}

From assumption {(A-3)} and the properties of $S_{\theta}$, it is clear that $G_{{n}}^{\pm}$ vanish in the past.
As in \cite{FM00MR1787068}, one can also check that the trace of $G_{{n}}^{\pm}$ on $\omega_T$ vanishes.
Hence, we can find $\delta\Psi_{{n}}^{\pm}$, vanishing in the past and satisfying $\delta\Psi_{{n}}^{\pm}|_{x_2=0}=\delta\psi_{{n}}$,
as the unique smooth solutions to the transport equations \eqref{delta.Psi1}--\eqref{delta.Psi2}.

Once $\delta\Psi_{{n}}$ is specified, we can obtain $\delta V_{{n}}$ from \eqref{good.NM}
and $(V_{{n}+1},\Psi_{{n}+1},\psi_{{n}+1})$ from \eqref{NM0}.
The error terms: $e_{{n}}'$, $e_{{n}}''$, $e_{{n}}'''$, $\tilde{e}_{{n}}'$, $\tilde{e}_{{n}}''$,
$\tilde{e}_{{n}}'''$, $\hat{e}_{{n}}'$, $\hat{e}_{{n}}''$, and $\hat{e}_{{n}}'''$
are computed from \eqref{decom1}--\eqref{decom2} and \eqref{decom3}.
Then $e_{{n}}$, $\tilde{e}_{{n}}$, and $\hat{e}_{{n}}$ are obtained from \eqref{e.e.tilde} and \eqref{e.hat}.

Using \eqref{effective.NM} and \eqref{source},
we sum \eqref{decom1} and \eqref{decom2.b} from ${n}=0$ to $m$, respectively, to obtain
\begin{align}
\label{conv.1}&\mathcal{L}(V_{m+1},\Psi_{m+1})=\sum_{{n}=0}^{m}f_{{n}}+E_{m+1}=S_{\theta_{m}}f^a+(I-S_{\theta_{m}})E_{m}+e_m,\\
\label{conv.2}&\mathcal{B}(V_{m+1}|_{x_2=0},\psi_{m+1})=\sum_{{n}=0}^{m}g_{{n}}+\widetilde{E}_{m+1}=(I-S_{\theta_{m}})\widetilde{E}_{m}+\tilde{e}_m.
\end{align}
Plugging \eqref{source2} into \eqref{decom3.c1}--\eqref{decom3.c2}, we utilize \eqref{B.E.relation} to deduce
\begin{align}\label{conv.3}
\small \left\{\begin{aligned}
&\mathcal{E}(V_{m+1}^-,\Psi_{m+1}^-)
=&\!\!\!\!\!\!\!\!&\mathcal{R}_T\big(\big(\mathcal{B}(V_{{m}+1}|_{x_2=0},\psi_{{m}+1})\big)_2
-\big(\mathcal{B}(V_{{m}+1}|_{x_2=0},\psi_{{m}+1})\big)_1\big)\\
& &\!\!\!\!\!\!\!\!&+(I-S_{\theta_{m}})\big( \hat{E}_m^--\mathcal{R}_T\big(\widetilde{E}_{m,2}-\widetilde{E}_{m,1}\big) \big)\\
& &\!\!\!\!\!\!\!\!&+\hat{e}_{m}^--\mathcal{R}_T\big(\tilde{e}_{m,2}-\tilde{e}_{m,1}\big),\\
&\mathcal{E}(V_{m+1}^+,\Psi_{m+1}^+)=&\!\!\!\!\!\!\!\!&\mathcal{R}_T\big(\big(\mathcal{B}(V_{{n}+1}|_{x_2=0},\psi_{{n}+1})\big)_2\big)\\
&&\!\!\!\!&+(I-S_{\theta_{m}})\big( \hat{E}_m^+-\mathcal{R}_T\widetilde{E}_{m,2}\big)
+\hat{e}_{m}^+-\mathcal{R}_T\tilde{e}_{m,2}.
\end{aligned}\right.
\end{align}
Since $S_{\theta_{m}}\to I$ as $m\to \infty$, we can formally obtain the solution to problem \eqref{P.new}
from $$\mathcal{L}(V_{m+1},\Psi_{m+1})\to f^a,\quad \mathcal{B}(V_{m+1}|_{x_2=0},\psi_{m+1})\to 0,\quad \mathcal{E}(V_{m+1},\Psi_{m+1})\to 0,$$
provided that the error terms: $(e_m,\tilde{e}_m,\hat{e}_m)\to 0$.

\subsection{Inductive Hypothesis}

Given a constant $\varepsilon>0$ and an integer $\tilde{\alpha}$ that will be chosen later on,
we assume that {\bf (A-1)}--{\bf (A-3)} are satisfied and that the following estimate holds:
\begin{align} \label{small}
\big\|\tilde{U}^a\big\|_{H^{\tilde{\alpha}+3}_{\gamma}(\Omega_T)}+\big\|\tilde{\Phi}^a\big\|_{H^{\tilde{\alpha}+4}_{\gamma}(\Omega_T)}
+\big\|\varphi^a\big\|_{H^{\tilde{\alpha}+7/2}_{\gamma}(\Omega_T)}+\big\|f^a\big\|_{H^{\tilde{\alpha}+2}_{\gamma}(\Omega_T)}
\leq \varepsilon.
\end{align}
Given another integer $\alpha$, our inductive hypothesis reads:
\begin{align*}
(H_{{n}-1})\ \left\{\begin{aligned}
\textrm{(a)}\,\,  &\|(\delta V_k,\delta \Psi_k)\|_{H^{s}_{\gamma}(\Omega_T)}+\|\delta\psi_k\|_{H^{s+1}_{\gamma}(\omega_T)}\leq \varepsilon \theta_k^{s-\alpha-1}\Delta_k\\
&\quad \textrm{for all } k=0,\ldots,{n}-1\textrm{ and }s\in[3,\tilde{\alpha}]\cap\mathbb{N};\\
\textrm{(b)}\,\, &\|\mathcal{L}( V_k,  \Psi_k)-f^a\|_{H^{s}_{\gamma}(\Omega_T)}\leq 2 \varepsilon \theta_k^{s-\alpha-1}\\
&\quad \textrm{for all } k=0,\ldots,{n}-1\textrm{ and } s\in[3,\tilde{\alpha}-2]\cap\mathbb{N};\\
\textrm{(c)}\,\,  &\|\mathcal{B}( V_k|_{x_2=0},  \psi_k)\|_{H^{s}_{\gamma}(\omega_T)}\leq  \varepsilon \theta_k^{s-\alpha-1}\\
&\quad \textrm{for all } k=0,\ldots,{n}-1\textrm{ and } s\in[4,\alpha]\cap\mathbb{N};\\
\textrm{(d)}\,\,  &\|\mathcal{E}( V_k,  \Psi_k)\|_{H^{3}_{\gamma}(\Omega_T)}\leq  \varepsilon \theta_k^{2-\alpha}\quad\textrm{for all } k=0,\ldots,{n}-1,
\end{aligned}\right.
\end{align*}
where $\Delta_{k}:=\theta_{k+1}-\theta_k$ with $\theta_k$ defined by \eqref{theta}.
Notice that
\begin{align*}
\frac{1}{3\theta_k}\leq \Delta_{k}=\sqrt{\theta_k^2+1}-\theta_k\leq \frac{1}{2\theta_k}\qquad\textrm{for all }k\in\mathbb{N}.
\end{align*}
In particular, sequence $(\Delta_{k})$ is decreasing and tends to $0$.
Our goal is to show that, for a suitable choice of parameters $\theta_0\geq 1$ and $\varepsilon>0$, and for $f^a$ small enough,
($H_{{n}-1}$) implies ($H_{{n}}$) and that ($H_0$) holds.
Once this goal is achieved, we infer that ($H_{{n}}$) holds for all ${n}\in \mathbb{N}$,
which enables us to conclude the proof of Theorem {\rm\ref{thm}}.

From now on, we assume that ($H_{{n}-1}$) holds.
As in \cite{CS08MR2423311}, hypothesis ($H_{{n}-1}$) yields the following consequences:

\begin{lemma}\ \label{lem.triangle}
	If $\theta_0$ is large enough, then, for each $k=0,\ldots,{n}$ and each integer $s\in [3,\tilde{\alpha}]$,
	\begin{align}
	\label{tri1}&\|( V_k, \Psi_k)\|_{H^{s}_{\gamma}(\Omega_T)}+\|\psi_k\|_{H^{s+1}_{\gamma}(\omega_T)}
	\leq
	\left\{\begin{aligned}
	&\varepsilon \theta_k^{(s-\alpha)_+}\quad &&\textrm{if }s\neq \alpha,\\
	&\varepsilon \log \theta_k\quad &&\textrm{if }s= \alpha,
	\end{aligned}\right.\\
	\label{tri2}&\|( (I-S_{\theta_k})V_k, (I-S_{\theta_k})\Psi_k)\|_{H^{s}_{\gamma}(\Omega_T)}
	\leq C\varepsilon \theta_k^{s-\alpha}.
	\end{align}
	Furthermore, for each $k=0,\ldots,{n}$, and each integer $s\in [3,\tilde{\alpha}+4]$,
	\begin{align}
	\label{tri3}&\|( S_{\theta_k}V_k, S_{\theta_k}\Psi_k)\|_{H^{s}_{\gamma}(\Omega_T)}\leq
	\left\{\begin{aligned}
	&C\varepsilon \theta_k^{(s-\alpha)_+}\quad &&\textrm{if }s\neq \alpha,\\
	&C\varepsilon \log \theta_k\quad &&\textrm{if }s= \alpha.
	\end{aligned}\right.
	\end{align}
\end{lemma}
Estimates \eqref{tri2}--\eqref{tri3} follow directly
from \eqref{smooth.p1} and \eqref{tri1}.

\subsection{Estimate of the Error Terms}
To deduce ($H_{{n}}$) from ($H_{{n}-1}$), we need to estimate the quadratic error terms $e'_{k}$,
$\tilde{e}_{k}'$, and $\hat{e}_{k}'$,
the first substitution error terms $e_{k}''$, $\tilde{e}_{k}''$, and $\hat{e}_{k}''$,
the second substitution error terms $e_{k}'''$, $\tilde{e}_{k}'''$, and $\hat{e}_{k}'''$, and the last error term $D_{k+1/2} \delta\Psi_{k}$.
Recall from \eqref{decom1} that
 \begin{align*}
e_k'=\mathcal{L}(V_{k+1},\Psi_{k+1})-\mathcal{L}(V_{k},\Psi_{k})-\mathbb{L}'(U^a+V_{k},\Phi^a+\Psi_{k})(\delta V_{k},\delta\Psi_{k}),
\end{align*}
which can be rewritten as
\begin{align}
\notag e_k'=\int_{0}^{1}\mathbb{L}''(&\,U^a+V_{k}+\tau \delta  V_{k},\\
\label{error.1a} &\Phi^a+\Psi_{k}+\tau \delta \Psi_{k})
\big((\delta V_{k},\delta\Psi_{k}),(\delta V_{k},\delta\Psi_{k})\big) (1-\tau)\d\tau,
\end{align}
where operator $\mathbb{L}''$ is defined by
\begin{align*}
\mathbb{L}''(U,\Phi)\big((V',\Psi'),(V'',\Psi'')\big)
:=\left.\frac{\d}{\d \tau} \mathbb{L}'(U+\tau V'',\Phi+\tau \Psi'')(V',\Psi')\right|_{\tau=0},
\end{align*}
with operator $\mathbb{L}'$ given in \eqref{L.prime}.
We can also obtain a similar expression for $\tilde{e}_{k}'$ (resp.\;$\hat{e}_{k}'$)
defined by \eqref{decom2} (resp.\;\eqref{decom3}) in terms of the second derivative
operator $\mathbb{B}''$ (resp.\;$\mathcal{E}''$).

To control the quadratic error terms, we need the estimates for operators $\mathbb{L}''$, $\mathbb{B}''$, and $\mathcal{E}''$ (see \eqref{error.1a}).
They can be obtained from the explicit forms of $\mathbb{L}''$, $\mathbb{B}''$, and $\mathcal{E}''$
by applying the Moser-type and Sobolev embedding inequalities.
Omitting detailed calculations, we find that the explicit forms of operators $\mathcal{E}''(U,\Phi)$ and $\mathbb{B}''(U,\Phi)$
depend on state $(U,\Phi)$, which make the following estimates for $\mathcal{E}''$ and $\mathbb{B}''$ different from
those obtained in \cite[Proposition 5]{CS08MR2423311}. This difference is caused by the introduction of new primary
unknowns $(p,hw_1, hw_2)$.

\begin{proposition}\ \label{pro.tame2}
	Let $T>0$ and $s\in\mathbb{N}$ with $s\geq3$.
	Assume that $(\tilde{U},\tilde{\Phi})\in H^{s+1}_{\gamma}(\Omega_T)$, $\tilde{\Phi}|_{x_2=0}\in H^{s+1}_{\gamma}(\omega_T)$,
	and $\|(\tilde{U},\tilde{\Phi} )\|_{H_{\gamma}^{3}(\Omega_T)}\leq \widetilde{K}$ for all $\gamma\geq1$.
	Then there exist two positive constants $\widetilde{K}_0$ and $C$, which are independent of $T$ and $\gamma$, such that,
	if $\widetilde{K}+\varepsilon \leq \widetilde{K}_0$, $\gamma\geq1$, $(V_1,\Psi_1),(V_2,\Psi_2)\in H^{s+1}_{\gamma}(\Omega_T)$,
	and $(W_1,\psi_1),(W_2,\psi_2)\in H^{s+1}_{\gamma}(\omega_T)\times H^{s+1}_{\gamma}(\omega_T)$, then
	\begin{align}
	\notag&\big\|\mathbb{L}''\big(U^a+\tilde{U},\Phi^a+\tilde{\Phi} \big)\big((V_1,\Psi_1),(V_2,\Psi_2) \big)\big\|_{H^{s}_{\gamma}(\Omega_T)}\\
	\notag &+\big\|
	\mathcal{E}''\big(\tilde{U},\tilde{\Phi} \big)\big((V_1,\Psi_1),(V_2,\Psi_2) \big)\big\|_{H^{s}_{\gamma}(\Omega_T)}\\
	\notag&\quad \leq C\Big\{ \|(V_1,\Psi_1)\|_{W^{1,\infty}(\Omega_T)}\|(V_2,\Psi_2)\|_{W^{1,\infty}(\Omega_T)}
	\big\|\big(\tilde{U}+\tilde{U}^a,\tilde{\Phi}+\tilde{\Phi}^a\big)\big\|_{H^{s+1}_{\gamma}(\Omega_T)}\\
	\notag&\quad\quad\quad\ \   +\sum_{i\neq j}\|(V_i,\Psi_i)\|_{H^{s+1}_{\gamma}(\Omega_T)} \|(V_j,\Psi_j)\|_{W^{1,\infty}(\Omega_T)} \Big\},\\
	\notag&\big\|\mathbb{B}''\big(U^a+\tilde{U},\Phi^a+\tilde{\Phi} \big)\big((W_1,\psi_1),(W_2,\psi_2) \big)\big\|_{H^{s}_{\gamma}(\omega_T)}\\
	\notag&\quad \leq C\Big\{ \|(W_1,\psi_1)\|_{W^{1,\infty}(\omega_T)}\|(W_2,\psi_2)\|_{W^{1,\infty}(\omega_T)}
	\big\|\big(\tilde{U}+\tilde{U}^a,\tilde{\Phi}+\tilde{\Phi}^a\big)\big\|_{H^{s+1}_{\gamma}(\omega_T)}\\
	\notag&\quad\quad\quad\ \  +\sum_{i\neq j}\|(W_i,\psi_i)\|_{H^{s+1}_{\gamma}(\omega_T)}
	\|(W_j,\psi_j)\|_{W^{1,\infty}(\omega_T)} \Big\}.
	\end{align}
\end{proposition}

Using Proposition \ref{pro.tame2}, we obtain the following result.

\begin{lemma}[Estimate of the quadratic error terms] \label{lem.quad}
	Let $\alpha\geq4$. Then there exist $\varepsilon>0$ sufficiently small and $\theta_0\geq 1$ sufficiently large
	such that, for all $k=0,\ldots,{n}-1$, and all integers $s\in [3,\tilde{\alpha}-1]$,
	\begin{align}\notag
	\|(e_k', \hat{e}_k')\|_{H^{s}_{\gamma}(\Omega_T)}+\|\tilde{e}_k'\|_{H^{s}_{\gamma}(\omega_T)}
	\leq C\varepsilon^2 \theta_k^{L_1(s)-1}\Delta_k,
	\end{align}
	where  $L_1(s):=\max\{(s+2-\alpha)_++4-2\alpha,s+3-2\alpha \}$.
\end{lemma}

\begin{proof}\
	In light of $(H_{{n}-1})$ and \eqref{tri1}, we have
	\begin{align*}
	\sup_{0\leq \tau\leq 1}\|(V_{k}+\tau \delta  V_{k},\Psi_{k}+\tau \delta \Psi_{k})\|_{H^{3}_{\gamma}(\Omega_T)}
	\leq C\varepsilon.
	\end{align*}
	For $\varepsilon$ small enough, we can apply Proposition \ref{pro.tame2} and use the Sobolev embedding inequality, \eqref{small}, and $(H_{{n}-1})$
	to obtain
	\begin{align*}
	&\|e_k'\|_{H^{s}_{\gamma}(\Omega_T)}\\
	&\quad\leq C\left\{\varepsilon^2\theta_k^{4-2\alpha}\Delta_k^2\big(\varepsilon+\|(V_k,\Psi_k,\delta V_k,\delta\Psi_k)\|_{H^{s+1}_{\gamma}(\Omega_T)}\big)+\varepsilon^2\theta_{k}^{s+2-2\alpha}\Delta_k^2\right\}
	\end{align*}
	for $s\in [3,\tilde{\alpha}-1]$.
	If $s+1\neq \alpha$, then we obtain from \eqref{tri1} and the inequality $2\theta_k\Delta_k\leq 1$ that
	\begin{align*}
	\|e_k'\|_{H^{s}_{\gamma}(\Omega_T)}\leq C \varepsilon^2 \Delta_k^2\big(\theta_k^{(s+1-\alpha)_+
		+4-2\alpha}+\theta_k^{s+2-2\alpha} \big)\leq C \varepsilon^2 \theta_k^{L_1(s)-1 }\Delta_k.
	\end{align*}
	If $s+1= \alpha$, then using \eqref{tri1} and $\alpha\geq 4$ yields
	\begin{align*}
	\|e_k'\|_{H^{\alpha-1}_{\gamma}(\Omega_T)}
	&\leq C \varepsilon^2 \Delta_k^2\left\{(\varepsilon+\varepsilon \log\theta_k+\varepsilon\theta_k^{-1}\Delta_{k}) +\theta_k^{1-\alpha}\right\}\\
	&\leq C \varepsilon^2 \Delta_k^2 \theta_k^{1-\alpha}\leq C \varepsilon^2 \theta_k^{L_1(\alpha-1)-1 }\Delta_k.
	\end{align*}
	The estimates for $\hat{e}_k'$ and $\tilde{e}_k'$ are similar and follow by applying Proposition \ref{pro.tame2} and the trace theorem.
	This completes the proof.
\qed\end{proof}

Now we estimate the first substitution error terms $e_{k}''$, $\tilde{e}_{k}''$, and $\hat{e}_{k}''$ given in \eqref{decom1}--\eqref{decom2},
and \eqref{decom3}
by rewriting them in terms of $\mathbb{L}''$, $\mathbb{B}''$, and $\mathcal{E}''$.
For instance,  $\tilde{e}_k''$ can be rewritten as
\begin{align}\label{tilde.e''}
\notag\tilde{e}_k''=\int_{0}^{1}&\mathbb{B}''\big(U^a+S_{\theta_k}V_k+\tau(I-S_{\theta_k})V_k,\Phi^a+S_{\theta_k}\Psi_k
+\tau(I-S_{\theta_k})\Psi_k\big)\\
&\big((\delta V_k|_{x_2=0},\delta \psi_k), ((I-S_{\theta_k})V_k|_{x_2=0},(I-S_{\theta_k})\Psi_k|_{x_2=0}) \big)\d\tau.
\end{align}
Then we have the following lemma.

\begin{lemma}[Estimate of the first substitution error terms] \label{lem.1st}
	Let $\alpha\geq4$. Then there exist $\varepsilon>0$ sufficiently small and $\theta_0\geq 1$ sufficiently large such that,
	for all $k=0,\ldots,{n}-1$, and all integers $s\in [3,\tilde{\alpha}-2]$,
	\begin{align}\notag
	\|(e_k'', \hat{e}_k'')\|_{H^{s}_{\gamma}(\Omega_T)}
	+\|\tilde{e}_k''\|_{H^{s}_{\gamma}(\omega_T)}\leq C\varepsilon^2 \theta_k^{L_2(s)-1}\Delta_k,
	\end{align}
	where  $L_2(s):=\max\{(s+2-\alpha)_++6-2\alpha,s+5-2\alpha \}$.
\end{lemma}

\begin{proof}\
	It follows from \eqref{tri2} and \eqref{tri3} that
	\begin{align*}
	\sup_{0\leq \tau\leq 1}\|(S_{\theta_k}V_k+\tau(I-S_{\theta_k})V_k,S_{\theta_k}\Psi_k+\tau(I-S_{\theta_k})\Psi_k)\|_{H^3_{\gamma}(\Omega_T)}
	\leq C\varepsilon.
	\end{align*}
	For $\varepsilon$ sufficiently small, we can apply Proposition \ref{pro.tame2} to estimate $\mathbb{B}''$ in \eqref{tilde.e''}.
	Employ the trace and embedding theorems to obtain
	\begin{align*}
&	\|\tilde{e}_k''\|_{H^{s}_{\gamma}(\omega_T)}\\
 &	\lesssim \big\|\big(\delta V_k,\delta\Psi_k\big)\big\|_{H^{3}_{\gamma}(\Omega_T)}
	\big\|\big((I-S_{\theta_k})V_k,(I-S_{\theta_k})\Psi_k\big)\big\|_{H^{3}_{\gamma}(\Omega_T)}\\
	&\quad \times \big\|\big(\tilde{U}^a+S_{\theta_k}V_k+\tau(I-S_{\theta_k})V_k,\tilde{\Phi}^a
	+S_{\theta_k}\Psi_k+\tau(I-S_{\theta_k})\Psi_k\big)\big\|_{H^{s+2}_{\gamma}(\Omega_T)}\\
	&\quad +\big\|\big(\delta V_k,\delta\Psi_k\big)\big\|_{H^{s+2}_{\gamma}(\Omega_T)}
	\big\|\big((I-S_{\theta_k})V_k,(I-S_{\theta_k})\Psi_k\big)\big\|_{H^{3}_{\gamma}(\Omega_T)}\\
	&\quad   +\big\|\big((I-S_{\theta_k})V_k,(I-S_{\theta_k})\Psi_k\big)\big\|_{H^{s+2}_{\gamma}(\Omega_T)}
	\big\|\big(\delta V_k,\delta\Psi_k\big)\big\|_{H^{3}_{\gamma}(\Omega_T)}.
	\end{align*}
	Using estimates \eqref{small}, $(H_{{n}-1})$, and \eqref{tri2}--\eqref{tri3}, we obtain that, for $s+2\neq \alpha$ and $s+2\leq \tilde{\alpha}$,
	\begin{align*}
	&\|\tilde{e}_k''\|_{H^{s}_{\gamma}(\omega_T)}\\
	&\quad \leq C\Big\{\varepsilon^2 \theta_k^{5-2\alpha}\Delta_k \big(\varepsilon+\varepsilon\theta_k^{(s+2-\alpha)_+}\big)
	+\varepsilon^2\theta_{k}^{s+4-2\alpha}\Delta_k\Big\}
	\leq C\varepsilon^2 \theta_k^{L_2(s)-1}\Delta_k.
	\end{align*}
	For $s+2=\alpha$, we obtain
	\begin{align*}
	\|\tilde{e}_k''\|_{H^{s}_{\gamma}(\omega_T)}
	&\leq C\left\{\varepsilon^2 \theta_k^{5-2\alpha}\Delta_k \big(\varepsilon+\varepsilon\log\theta_k\big)+\varepsilon^2\theta_{k}^{2-\alpha}\Delta_k\right\}\\
	&\leq C\varepsilon^2\Delta_k\left( \theta_k^{5-2\alpha}\log \theta_k + \theta_k^{2-\alpha}\right) \leq C\varepsilon^2 \theta_k^{L_2(\alpha-2)-1}\Delta_k,
	\end{align*}
	owing to $\alpha\geq 4$.
	The estimate for  $e_{k}''$ and $\hat{e}_{k}''$ can be deduced in the same way.
\qed\end{proof}

Now we estimate the second substitution error terms $e_{k}'''$, $\tilde{e}_{k}'''$, and $\hat{e}_{k}'''$
given in \eqref{decom1}--\eqref{decom2} and \eqref{decom3} by rewriting them in terms of $\mathbb{L}''$, $\mathbb{B}''$, and $\mathcal{E}''$.
For instance, $\hat{e}_k'''$ can be rewritten as
 \begin{align}
\notag \hat{e}_k'''=\int_{0}^{1}\textrm{\small $\mathcal{E}''\big(V_{k+1/2}+\tau(S_{\theta_k}V_k-V_{k+1/2}),\Psi_{k+1/2}\big)
\big((\delta V_k,\delta \Psi_k), (S_{\theta_k}V_k-V_{k+1/2},0) \big)$}\d\tau.
\end{align}
Here we have used relation $\Psi_{k+1/2}=S_{\theta_{k}}\Psi_{k}$ ({\it cf.}\;\eqref{modified}).
Then we have the following result.

\begin{lemma}[Estimate of the second substitution error terms]  \label{lem.2nd}
	Let $\alpha\geq4$. Then there exist $\varepsilon>0$ sufficiently small and $\theta_0\geq 1$ sufficiently large
	such that, for all $k=0,\ldots,{n}-1$, and all integers $s\in [3,\tilde{\alpha}-1]$,
	\begin{align}\notag
	\|(e_k''', \hat{e}_k''')\|_{H^{s}_{\gamma}(\Omega_T)}
	+\|\tilde{e}_k'''\|_{H^{s}_{\gamma}(\omega_T)}\leq C\varepsilon^2 \theta_k^{L_3(s)-1}\Delta_k,
	\end{align}
	where  $L_3(s):=\max\{(s+2-\alpha)_++8-2\alpha,s+6-2\alpha \}$.
\end{lemma}

\begin{proof}\
	Omitting detailed derivation, we can use the inductive assumption $(H_{{n}-1})$, definition \eqref{modified},
	and the properties of $S_{\theta}$ and $\mathcal{R}_T$ to obtain
	\begin{align}\label{CS.e1}
	\|S_{\theta_k}V_k-V_{k+1/2}\|_{H^{s}_{\gamma}(\Omega_T)}\leq C\varepsilon \theta_k^{s+1-\alpha}
	\end{align}
	for all $k=0,\ldots,{n}-1$ and all integers $s\in[3,\tilde{\alpha}+3]$.
	We refer to \cite[Proposition 7]{CS08MR2423311} for the proof of \eqref{CS.e1}.
	It follows from \eqref{tri3} and \eqref{CS.e1} that
	\begin{align} \label{CS.e3}
	\|V_{k+1/2}\|_{H^{s}_{\gamma}(\Omega_T)}\leq C\varepsilon \theta_k^{(s-\alpha)_++1}\qquad \textrm{for }s\in[3,\tilde{\alpha}+3].
	\end{align}
	Thus, we have
	\begin{align}
\notag 	\|(\tilde{U}^a+V_{k+1/2}+\tau(S_{\theta_k}V_k-V_{k+1/2}),\tilde{\Phi}^a+\Psi_{k+1/2})\|_{H^{s+1}_{\gamma}(\Omega_T)}\quad \\
\label{CS.e2}	\leq C\varepsilon \theta_k^{(s+1-\alpha)_++1}.
	\end{align}
	For $\varepsilon$ small enough, one may apply Proposition \ref{pro.tame2}
	and use $(H_{{n}-1})$, \eqref{CS.e1}, and \eqref{CS.e2} to deduce
	\begin{align*}
	\|\hat{e}_k'''\|_{H^{s}_{\gamma}(\Omega_T)}
	\leq C\Big\{\varepsilon^2 \theta_k^{6-2\alpha}\Delta_k \varepsilon\theta_k^{(s+1-\alpha)_++1}+\varepsilon^2\theta_{k}^{s+4-2\alpha}\Delta_k\Big\}
	\leq C\varepsilon^2 \theta_k^{L_3(s)-1}\Delta_k.
	\end{align*}
	The estimate for  $e_{k}'''$ and $\tilde{e}_{k}'''$ can be deduced in a similar way by using the trace theorem.
\qed\end{proof}

We now estimate the last error term \eqref{error.D}:
\begin{align*}
D_{k+1/2}\delta\Psi_k=\frac{\delta\Psi_k}{\p_2(\Phi^a+\Psi_{k+1/2})}R_k,
\end{align*}
where $R_k:=\p_2\mathbb{L}(U^a+V_{k+1/2},\Phi^a+\Psi_{k+1/2})$.
This error term results from the introduction of the good unknown in decomposition \eqref{decom1}.
Note from \eqref{modified}, \eqref{small}, and \eqref{tri3} that
\begin{align*}
|\p_2(\Phi^a+\Psi_{k+1/2})|=\big|\p_2\widebar{\Phi}+\p_2\big(\tilde{\Phi}^a+\Psi_{k+1/2}\big)\big|
\geq \frac{1}{2},
\end{align*}
provided that $\varepsilon$ is small enough.
Then we have the following estimate.

\begin{lemma}\  \label{lem.last}
	Let $\alpha\geq 5$ and $\tilde{\alpha}\geq \alpha+2$.
	Then there exist $\varepsilon>0$ sufficiently small and $\theta_0\geq 1$ sufficiently large such that,
	for all $k=0,\ldots,{n}-1$, and for all integers $s\in [3,\tilde{\alpha}-2]$, we have
	\begin{align} \label{last.e0}
	\|D_{k+1/2}\delta\Psi_k\|_{H^{s}_{\gamma}(\Omega_T)}\leq C\varepsilon^2 \theta_k^{L(s)-1}\Delta_k,
	\end{align}
	where  $L(s):=\max\{(s+2-\alpha)_++8-2\alpha,(s+1-\alpha)_++9-2\alpha,s+6-2\alpha\}$.
\end{lemma}

\begin{proof}\
	The proof follows from the arguments as in \cite{A89MR976971,CS08MR2423311}.
	Let $\Omega_T^+:=(0,T)\times\mathbb{R}^2_+$.
	Since $\delta\Psi_k$ vanishes in the past,
	using the Moser-type inequality, we obtain
	\begin{align}
	\notag& \|D_{k+1/2}\delta\Psi_k\|_{H^{s}_{\gamma}(\Omega_T)}=\|D_{k+1/2}\delta\Psi_k\|_{H^{s}_{\gamma}(\Omega_T^+)}\\
	&\leq C\big\{\|\delta\Psi_k\|_{L^{\infty}(\Omega_T^+)}\big(\|R_k\|_{H^{s}_{\gamma}(\Omega_T^+)}
	+\|R_k\|_{L^{\infty}(\Omega_T^+)}\|\tilde{\Phi}^a+\Psi_{k+1/2}\|_{H^{s+1}_{\gamma}(\Omega_T^+)}\big)\notag\\
	&\qquad\,\,\, + \|\delta\Psi_k\|_{H^{s}_{\gamma}(\Omega_T^+)}\|R_k\|_{L^{\infty}(\Omega_T^+)}\big\}.
	\label{last.p1}
	\end{align}
	To estimate $R_k$, we introduce the following decomposition for $t>0$:
	\begin{align}
	\notag&\mathbb{L}(U^a+V_{k+1/2},\Phi^a+\Psi_{k+1/2})-\mathcal{L}(V_k,\Psi_k)+f^a\\[1.5mm]
	\notag& =\mathbb{L}(U^a+V_{k+1/2},\Phi^a+\Psi_{k+1/2})-\mathbb{L}(U^a+V_{k},\Phi^a+\Psi_{k})\\
	\notag&=\int_{0}^1\mathbb{L}'\big(U^a+V_{k}+\tau(V_{k+1/2}-V_k),\\
	\label{decom4}&\qquad\qquad \Phi^a+\Psi_{k}+\tau(\Psi_{k+1/2}-\Psi_{k})\big)(V_{k+1/2}-V_k,\Psi_{k+1/2}-\Psi_{k})\d\tau.
	\end{align}
	If $s\leq \tilde{\alpha}-3$, the inductive assumption $(H_{{n}-1})$ implies
	\begin{align} \label{last.e1}
	\|\mathcal{L}(V_k,\Psi_k)-f^a\|_{H^{s+1}_{\gamma}(\Omega_T)}\leq 2\varepsilon\theta_k^{s-\alpha}.
	\end{align}
	Since we can obtain an estimate for $\mathbb{L}'$ similar to that for $\mathbb{L}''$ (see Proposition \ref{pro.tame2}),
	using Lemma \ref{lem.triangle} and \eqref{CS.e1} leads to
	\begin{align}
	\notag&\|\mathbb{L}(U^a+V_{k+1/2},\Phi^a+\Psi_{k+1/2})-\mathbb{L}(U^a+V_{k},\Phi^a+\Psi_{k})\|_{H^{s+1}_{\gamma}(\Omega_T)}\\
	\label{last.e2}&\,\leq C\varepsilon\big(\theta_k^{s+3-\alpha}+\theta_k^{(s+2-\alpha)_++5-\alpha} \big).
	\end{align}
	Plugging \eqref{last.e1}--\eqref{last.e2} into \eqref{decom4} yields
	\begin{align}\label{R.k}
	\|R_{k}\|_{H^{s}_{\gamma}(\Omega_T)} \leq C\varepsilon\big(\theta_k^{s+3-\alpha}+\theta_k^{(s+2-\alpha)_++5-\alpha} \big)
	\qquad \mbox{for $s\in[3,\tilde{\alpha}-3]$}.
	\end{align}
	If $s=\tilde{\alpha}-2\geq \alpha$, then we use \eqref{tri3} and \eqref{CS.e3} to obtain
	\begin{align*}
	\|R_{k}\|_{H^{s}_{\gamma}(\Omega_T)} &\leq\|\mathbb{L}(U^a+V_{k+1/2},\Phi^a+\Psi_{k+1/2})\|_{H^{s+1}_{\gamma}(\Omega_T)} \\
	&\leq C\|(\tilde{U}^a+V_{k+1/2},\tilde{\Phi}^a+\Psi_{k+1/2})\|_{H^{s+2}_{\gamma}(\Omega_T)}\\
	&\leq C\varepsilon \theta_k^{s+3-\alpha}.
	\end{align*}
	Thus, we obtain estimate \eqref{R.k} for $s\in [3,\tilde{\alpha}-2]$.
	Thanks to $(H_{{n}-1})$ and \eqref{R.k}, we utilize the embedding inequality to find
	\begin{align*}
	\|\delta\Psi_{k}\|_{L^{\infty}(\Omega_T)}\leq C\varepsilon\theta_k^{2-\alpha}\Delta_k,\qquad
	\|R_{k}\|_{L^{\infty}(\Omega_T)}\leq C\varepsilon\theta_k^{6-\alpha}.
	\end{align*}
	Using these bounds and plugging \eqref{R.k}, $(H_{{n}-1})$, and \eqref{tri3} into \eqref{last.p1}
	yield \eqref{last.e0}.
\qed\end{proof}

From Lemmas \ref{lem.quad}--\ref{lem.last},
we can immediately obtain the following estimate for $e_k$, $\tilde{e}_k$, and $\hat{e}_k$
defined in \eqref{e.e.tilde} and \eqref{e.hat}.

\begin{lemma}\  \label{lem.sum1}
	Let $\alpha\geq 5$. Then there exist $\varepsilon>0$ sufficiently small and $\theta_0\geq 1$ sufficiently large
	such that, for all $k=0,\ldots,{n}-1$, and for all integers $s\in [3,\tilde{\alpha}-2]$, we have
	\begin{align} \label{es.sum1}
	\|e_k\|_{H^{s}_{\gamma}(\Omega_T)}+\|\hat{e}_k\|_{H^{s}_{\gamma}(\Omega_T)}+\|\tilde{e}_k\|_{H^{s}_{\gamma}(\omega_T)}
	\leq C\varepsilon^2 \theta_k^{L(s)-1}\Delta_k,
	\end{align}
	where  $L(s)$ is defined in Lemma {\rm \ref{lem.last}}.
\end{lemma}

Lemma \ref{lem.sum1} yields the estimate of the accumulated error terms $E_k$, $\widetilde{E}_k$, and $\hat{E}_k$
that are defined in \eqref{E.E.tilde} and \eqref{e.hat}.

\begin{lemma}\ \label{lem.sum2}
	Let $\alpha\geq 7$ and $\tilde{\alpha}=\alpha+4$.
	Then there exist $\varepsilon>0$ sufficiently small and $\theta_0\geq 1$ sufficiently large such that
	\begin{align}\label{es.sum2}
	\|(E_{{n}}, \hat{E}_{{n}})\|_{H^{\alpha+2}_{\gamma}(\Omega_T)}
	+\|\widetilde{E}_{{n}}\|_{H^{\alpha+2}_{\gamma}(\omega_T)}\leq C\varepsilon^2 \theta_{{n}}.
	\end{align}
\end{lemma}

\begin{proof}\
	Notice that $L(\alpha+2)\leq 1$ if $\alpha\geq 7$. From \eqref{es.sum1}, we have
	\begin{align*}
	&\|(E_{{n}}, \hat{E}_{{n}})\|_{H^{\alpha+2}_{\gamma}(\Omega_T)}
	+\|\widetilde{E}_{{n}}\|_{H^{\alpha+2}_{\gamma}(\omega_T)}\\
	&\, \leq \sum_{k=0}^{{n}-1}\big\{\|(e_{k},\hat{e}_{k})\|_{H^{\alpha+2}_{\gamma}(\Omega_T)}
	+\|\tilde{e}_{k}\|_{H^{\alpha+2}_{\gamma}(\omega_T)} \big\}\\
	&\, \leq \sum_{k=0}^{{n}-1} C\varepsilon^2 \Delta_k
	\leq C\varepsilon^2\theta_{{n}},
	\end{align*}
	provided that $\alpha\geq 7$ and $\alpha+2\in [3,\tilde{\alpha}-2]$.
	Thus, the minimal possible $\tilde{\alpha}$ is $\alpha+4$.
\qed\end{proof}

\subsection{Proof of Theorem {\rm\ref{thm}}}
To prove our main result, we first derive the estimates of the source terms $f_{{n}}$, $g_{{n}}$,
and $G_{{n}}^{\pm}$ defined in \eqref{source} and \eqref{source2}.

\begin{lemma}\  \label{lem.source}
	Let $\alpha\geq 7$ and $\tilde{\alpha}=\alpha+4$.
	Then there exist $\varepsilon>0$ sufficiently small and $\theta_0\geq 1$ sufficiently large such that,
	for all integers $s\in [3,\tilde{\alpha}+1]$,
	\begin{align}
	\label{es.fl}&\|f_{{n}}\|_{H^s_{\gamma}(\Omega_T)}
	\leq C \Delta_{{n}}\big\{\theta_{{n}}^{s-\alpha-2}(\|f^a\|_{H^{\alpha+1}_{\gamma}(\Omega_T)}
	+\varepsilon^2)+\varepsilon^2\theta_{{n}}^{L(s)-1}\big\},\\
	\label{es.gl}&\|g_{{n}}\|_{H^s_{\gamma}(\omega_T)}
	\leq C \varepsilon^2 \Delta_{{n}}\big(\theta_{{n}}^{s-\alpha-2}+\theta_{{n}}^{L(s)-1}\big),
	\end{align}
	and for all integers $s\in [3,\tilde{\alpha}]$,
	\begin{align}
	\label{es.Gl}\|G_{{n}}^{\pm}\|_{H^s_{\gamma}(\Omega_T)}\leq C\varepsilon^2 \Delta_{{n}}\big(\theta_{{n}}^{s-\alpha-2}+\theta_{{n}}^{L(s)-1}\big).
	\end{align}
\end{lemma}

\begin{proof}\
	It follows from \eqref{source} that
	\begin{align*}
	f_{{n}}=(S_{\theta_{{n}}}-S_{\theta_{{n}-1}})f^a-(S_{\theta_{{n}}}-S_{\theta_{{n}-1}})E_{{n}-1}-S_{\theta_{{n}}} e_{{n}-1}.
	\end{align*}
	Using \eqref{smooth.p1a}, \eqref{smooth.p1c}, \eqref{es.sum1}, and \eqref{es.sum2}, we obtain the estimates:
	\begin{align*}
	\|(S_{\theta_{{n}}}-S_{\theta_{{n}-1}})f^a\|_{H^s_{\gamma}(\Omega_T)}
	&\leq C \theta_{{n}-1}^{s-\alpha-2}\|f^a\|_{H^{\alpha+1}_{\gamma}(\Omega_T)}\Delta_{{n}-1},\\
	\|(S_{\theta_{{n}}}-S_{\theta_{{n}-1}})E_{{n}-1}\|_{H^s_{\gamma}(\Omega_T)}
&\leq C \theta_{{n}-1}^{s-\alpha-3}\|E_{{n}-1}\|_{H^{\alpha+2}_{\gamma}(\Omega_T)}\Delta_{{n}-1}
	\\&
	\leq C\varepsilon^2 \theta_{{n}-1}^{s-\alpha-2}\Delta_{{n}-1},\\
	\|S_{\theta_{{n}}}e_{{n}-1}\|_{H^s_{\gamma}(\Omega_T)} &\leq C \varepsilon^2  \theta_{{n}-1}^{L(s)-1}\Delta_{{n}-1}.
	\end{align*}
	Combining the above estimates with the inequalities: $\theta_{{n}-1}\leq \theta_{{n}}\leq \sqrt{2}\theta_{{n}-1}$
	and $\Delta_{{n}-1}\leq 3\Delta_{{n}}$, we derive \eqref{es.fl}.
	Similarly, we obtain \eqref{es.gl}.
	To prove \eqref{es.Gl}, we use \eqref{source2} to find
	\begin{align*}
	G_{{n}}^+=(S_{\theta_{{n}}}-S_{\theta_{{n}-1}})\big(\mathcal{R}_T\widetilde{E}_{{n}-1,2}-\hat{E}_{{n}-1}^+\big)
	+S_{\theta_{{n}}}\big(\mathcal{R}_T\tilde{e}_{{n}-1,2}-\hat{e}_{{n}-1}^+\big).
	\end{align*}
	Then we obtain the estimate for $G_{{n}}^+$ by using \eqref{es.sum1}--\eqref{es.sum2} as above.
	The estimate of $G_{{n}}^-$ is the same.
\qed\end{proof}

We are going to obtain the estimate of the solution to problem \eqref{effective.NM}
by employing the tame estimate \eqref{E6.1}.

\begin{lemma}\  \label{lem.Hl1}
	Let $\alpha\geq 7$. If $\varepsilon>0$ and $\|f^a\|_{H^{\alpha+1}_{\gamma}(\Omega_T)}/\varepsilon$ are sufficiently small,
	and if $\theta_0\geq1$ is sufficiently large, then, for all integers $s\in [3,\tilde{\alpha}]$,
	\begin{align} \label{Hl.a}
	\|(\delta V_{{n}},\delta\Psi_{{n}})\|_{H^{s}_{\gamma}(\Omega_T)}+\|\delta\psi_{{n}}\|_{H^{s+1}_{\gamma}(\omega_T)}
	\leq \varepsilon \theta_{{n}}^{s-\alpha-1}\Delta_{{n}}.
	\end{align}
\end{lemma}

\begin{proof}\
	Let us consider problem \eqref{effective.NM}, which can be solved, since $U^a+V_{{n}+1/2}$ and $\Phi^a+\Psi_{{n}+1/2}$
	satisfy the required constraints \eqref{modified.1}.
	Constraint \eqref{bas.c1} can be obtained by truncating the coefficients, $U^a+V_{{n}+1/2}$ and $\Phi^a+\Psi_{{n}+1/2}$,
	by a suitable cut-off function,
	while \eqref{bas.eq.2} can be obtained by taking $\varepsilon>0$ small enough.
	We can consider the coefficients with a fixed compact support.
	In order to apply Theorem {\rm\ref{thm.L}}, we obtain \eqref{H6.1}, by using the classical trace estimate,
	\eqref{small}, \eqref{tri3}, \eqref{CS.e1}, and $\alpha\geq 7$.
	Thus, we can employ the tame estimate \eqref{E6.1} to obtain
	\begin{align}\label{Hl1.p0}
	&\|\delta \dot{V}_{{n}}\|_{H^s_{\gamma}(\Omega_T)}+\|\delta \psi_{{n}}\|_{H^{s+1}_{\gamma}(\omega_T)}\notag\\
	& \leq C\big\{  \big(\|f_{{n}}\|_{H^4_{\gamma}(\Omega_T)}
	+\|g_{{n}}\|_{H^4_{\gamma}(\omega_T)}\big)\|(\tilde{U}^a+V_{{n}+1/2},\tilde{\Phi}^a+\Psi_{{n}+1/2})\|_{H^{s+3}_{\gamma}(\Omega_T)}\notag \\
	&\qquad\,\ +\|f_{{n}}\|_{H^{s+1}_{\gamma}(\Omega_T)}+\|g_{{n}}\|_{H^{s+1}_{\gamma}(\omega_T)}\big\}.
	\end{align}
	The particular case $s=3$ implies
	\begin{align} \label{Hl1.p1}
	\|\delta \dot{V}_{{n}}\|_{H^3_{\gamma}(\Omega_T)}
	\leq C\big(\|f_{{n}}\|_{H^4_{\gamma}(\Omega_T)}+\|g_{{n}}\|_{H^4_{\gamma}(\omega_T)}\big).
	\end{align}
	Given $\delta\psi_{{n}}$, we can compute $\delta\Psi_{{n}}$ from equations \eqref{delta.Psi1}--\eqref{delta.Psi2}.
	Performing the energy estimates for $\delta\Psi_{{n}}$, and using Lemma \ref{lem.triangle}, \eqref{CS.e1}, and the Sobolev embedding theorem,
	we derive
	\begin{align}
&	\notag \gamma \|\delta\Psi_{{n}}\|_{H^s_{\gamma}(\Omega_T)}\\
&\notag 	\leq C\big\{ \|g_{{n}}\|_{H^s_{\gamma}(\omega_T)}+ \|G_{{n}}\|_{H^s_{\gamma}(\Omega_T)}+\|\delta\dot{V}_{{n}}\|_{H^s_{\gamma}(\Omega_T)}+\|\delta\dot{V}_{{n}}\|_{H^3_{\gamma}(\Omega_T)}\\
	&\qquad\ \  \times\|\tilde{\Phi}^a+S_{\theta_{{n}}}\Psi_{{n}} \|_{H^{s+1}_{\gamma}(\Omega_T)}
	+\varepsilon \theta_{{n}}^{(s+2-\alpha)_+}\|\delta\Psi_{{n}}\|_{H^3_{\gamma}(\Omega_T)} \big\}
	\label{Hl1.p2}
	\end{align}
	for all integers $s\in [3,\tilde{\alpha}]$ and $\varepsilon$ small enough.
	For $s=3$, using \eqref{Hl1.p1}, we have
	\begin{align}\label{Hl1.p3}
	\|\delta \Psi_{{n}}\|_{H^3_{\gamma}(\Omega_T)}\leq C\big(\|f_{{n}}\|_{H^4_{\gamma}(\Omega_T)}+\|g_{{n}}\|_{H^4_{\gamma}(\omega_T)}+\|G_{{n}}\|_{H^3_{\gamma}(\Omega_T)}\big).
	\end{align}
	In view of \eqref{good.NM}, using estimates \eqref{Hl1.p0}, \eqref{Hl1.p2}--\eqref{Hl1.p3},
	and the Moser-type inequality, we obtain
	\begin{align}
	\notag &\|(\delta V_{{n}},\delta\Psi_{{n}})\|_{H^s_{\gamma}(\Omega_T)}+\|\delta \psi_{{n}}\|_{H^{s+1}_{\gamma}(\omega_T)}\\
	\notag &\leq C\big\{ \|f_{{n}}\|_{H^{s+1}_{\gamma}(\Omega_T)}+\|g_{{n}}\|_{H^{s+1}_{\gamma}(\omega_T)}+\|G_{{n}}\|_{H^s_{\gamma}(\Omega_T)}\\
	\notag &\quad \quad\,\, +(\|f_{{n}}\|_{H^4_{\gamma}(\Omega_T)}+\|g_{{n}}\|_{H^4_{\gamma}(\omega_T)}+\|G_{{n}}\|_{H^3_{\gamma}(\Omega_T)})\\
	\label{Hl1.p4}&\quad \quad\,\, \times \big( \|(\tilde{U}^a+V_{{n}+1/2},\tilde{\Phi}^a+\Psi_{{n}+1/2})\|_{H^{s+3}_{\gamma}(\Omega_T)}
	+\varepsilon \theta_{{n}}^{(s+2-\alpha)_+}\big)\big\}
	\end{align}
	for all integers $s\in [3,\tilde{\alpha}]$.
	Using Lemma \ref{lem.source}, \eqref{tri3}, and \eqref{CS.e1}, we obtain from \eqref{Hl1.p4} that
	\begin{align}
	\notag &\|(\delta V_{{n}},\delta\Psi_{{n}})\|_{H^s_{\gamma}(\Omega_T)}
	+\|\delta \psi_{{n}}\|_{H^{s+1}_{\gamma}(\omega_T)}\\
	\notag & \leq C\Delta_{{n}}\big\{\theta_{{n}}^{2-\alpha}(\|f^a\|_{H^{\alpha+1}_{\gamma}(\Omega_T)} +\varepsilon^2)
	+\varepsilon^2\theta_{{n}}^{9-2\alpha}\big\}\big(\varepsilon \theta_{{n}}^{(s+3-\alpha)_+}+\varepsilon \theta_{{n}}^{s+4-\alpha} \big)\\
	\label{Hl1.p5} &  \quad +C\Delta_{{n}}\big\{\theta_{{n}}^{s-\alpha-1}(\|f^a\|_{H^{\alpha+1}_{\gamma}(\Omega_T)}
	+\varepsilon^2)+\varepsilon^2\theta_{{n}}^{L(s+1)-1}\big\}.
	\end{align}
	Exactly as in \cite{CS08MR2423311}, we can obtain the following inequalities:
	\begin{align*}
	\left\{\begin{aligned}
	&L(s+1)\leq s-\alpha,\\
	&(s+3-\alpha)_++2-\alpha\leq s-\alpha-1,\\
	&(s+3-\alpha)_++9-2\alpha\leq s-\alpha-1,\\
	&s+6-2\alpha\leq s-\alpha-1,\\
	&s+13-3\alpha\leq s-\alpha-1,
	\end{aligned}\right.
	\end{align*}
	for $\alpha\geq 7$ and $s\in[3,\tilde{\alpha}]$.
	Thus, \eqref{Hl1.p5} yields
	\begin{align*}
	\|(\delta V_{{n}},\delta\Psi_{{n}})\|_{H^{s}_{\gamma}(\Omega_T)}+\|\delta\psi_{{n}}\|_{H^{s+1}_{\gamma}(\omega_T)}\leq
	C\big(\|f^a\|_{H^{\alpha+1}_{\gamma}(\Omega_T)}+\varepsilon^2\big)\theta_{{n}}^{s-\alpha-1}\Delta_{{n}},
	\end{align*}
	and \eqref{Hl.a} follows by taking $\varepsilon+\|f^a\|_{H^{\alpha+1}_{\gamma}(\Omega_T)}/\varepsilon$ small enough.
\qed\end{proof}

Estimate \eqref{Hl.a} is inequality (a) of $(H_{{n}})$. We now prove the other inequalities in $(H_{{n}})$.

\begin{lemma}\ \label{lem.Hl2}
	Let $\alpha\geq 7$. If $\varepsilon>0$ and $\|f^a\|_{H^{\alpha+1}_{\gamma}(\Omega_T)}/\varepsilon$ are sufficiently small,
	and if $\theta_0\geq1$ is sufficiently large,
	then, for all integers $s\in [3,\tilde{\alpha}-2]$,
	\begin{align}\label{Hl.b}
	\|\mathcal{L}( V_{{n}},  \Psi_{{n}})-f^a\|_{H^{s}_{\gamma}(\Omega_T)}\leq 2 \varepsilon \theta_{{n}}^{s-\alpha-1}.
	\end{align}
	Moreover, for all integers $s\in [4,\alpha]$,
	\begin{align}\label{Hl.c}
	\|\mathcal{B}( V_{{n}}|_{x_2=0},  \psi_{{n}})\|_{H^{s}_{\gamma}(\omega_T)}\leq  \varepsilon \theta_{{n}}^{s-\alpha-1}
	\end{align}
	and
	\begin{align}\label{Hl.d}
	\|\mathcal{E}( V_{{n}},  \Psi_{{n}})\|_{H^{3}_{\gamma}(\Omega_T)}\leq  \varepsilon \theta_{{n}}^{2-\alpha}.
	\end{align}
\end{lemma}

\begin{proof}\
	From \eqref{conv.1}, we have
	\begin{align*}
&	\|\mathcal{L}( V_{{n}},  \Psi_{{n}})-f^a\|_{H^{s}_{\gamma}(\Omega_T)}
	\leq \|(I-S_{\theta_{{n}-1}})f^a\|_{H^{s}_{\gamma}(\Omega_T)}\\
	&\qquad \quad + \|(S_{\theta_{{n}-1}}-I) E_{{n}-1}\|_{H^{s}_{\gamma}(\Omega_T)}+ \|e_{{n}-1}\|_{H^{s}_{\gamma}(\Omega_T)}.
	\end{align*}
	For $s\in [\alpha+1,\tilde{\alpha}-2]$, using \eqref{smooth.p1a} and \eqref{small}, we obtain
	\begin{align*}
	\|(I-S_{\theta_{{n}-1}})f^a\|_{H^{s}_{\gamma}(\Omega_T)}
	&\leq \theta_{{n}-1}^{s-\alpha-1}(\|f^a\|_{H^s_{\gamma}(\Omega_T)}+C\|f^a\|_{H^{\alpha+1}_{\gamma}(\Omega_T)})\\
	&\leq  \varepsilon  \theta_{{n}}^{s-\alpha-1}\Big(1+\frac{C\|f^a\|_{H^{\alpha+1}_{\gamma}(\Omega_T)}}{\varepsilon}\Big),
	\end{align*}
	while, for $s\in [3,\alpha+1]$, applying \eqref{smooth.p1b}, we have
	\begin{align*}
	\|(I-S_{\theta_{{n}-1}})f^a\|_{H^{s}_{\gamma}(\Omega_T)}\leq C\theta_{{n}-1}^{s-\alpha-1}\|f^a\|_{H^{\alpha+1}_{\gamma}(\Omega_T)}
	\leq C\theta_{{n}}^{s-\alpha-1}\|f^a\|_{H^{\alpha+1}_{\gamma}(\Omega_T)}.
	\end{align*}
	Lemma \ref{lem.sum2} and \eqref{smooth.p1b} imply
	\begin{align*}
	\|(I-S_{\theta_{{n}-1}})E_{{n}-1}\|_{H^{s}_{\gamma}(\Omega_T)}
	\leq C \theta_{{n}-1}^{s-\alpha-2}\|E_{{n}-1}\|_{H^{\alpha+2}_{\gamma}(\Omega_T)}
	\leq C\varepsilon^2 \theta_{{n}}^{s-\alpha-1}
	\end{align*}
	for $3\leq s\leq \alpha+2=\tilde{\alpha}-2$. It follows from \eqref{es.sum1} that
	\begin{align*}
	\|e_{{n}-1}\|_{H^{s}_{\gamma}(\Omega_T)}\leq C \varepsilon^2\theta_{{n}-1}^{L(s)-1}\Delta_{{n}-1}
	\leq C \varepsilon^2\theta_{{n}}^{L(s)-2}\leq C \varepsilon^2\theta_{{n}}^{s-\alpha-1}.
	\end{align*}
	By virtue of the above estimates, we choose $\varepsilon$ and $\|f^a\|_{H^{\alpha+1}_{\gamma}(\Omega_T)}/\varepsilon$  sufficiently small
	to obtain \eqref{Hl.b}.
	Similarly, using decompositions \eqref{conv.2}--\eqref{conv.3},
	we can prove estimates \eqref{Hl.c}--\eqref{Hl.d}.
\qed\end{proof}

In view of Lemmas \ref{lem.Hl1}--\ref{lem.Hl2}, we have obtained  $(H_{{n}})$ from $(H_{{n}-1})$,
provided that $\alpha\geq 7$, $\tilde{\alpha}=\alpha+4$,  \eqref{small} holds,
$\varepsilon>0$ and $\|f^a\|_{H^{\alpha+1}_{\gamma}(\Omega_T)}/\varepsilon$ are sufficiently small,
and $\theta_0\geq1$ is large enough.
Fixing constants $\alpha$, $\tilde{\alpha}$, $\varepsilon>0$ and $\theta_0\geq1$,
we now prove $(H_{0})$.

\begin{lemma}\ \label{lem.H0}
	If $\|f^a\|_{H^{\alpha+1}_{\gamma}(\Omega_T)}/\varepsilon$ is sufficiently small,
	then $(H_0)$ holds.
\end{lemma}

\begin{proof}\
	Recall from assumptions {(A-1)}--{(A-3)}
	that $(V_0,\Psi_0,\psi_0,g_0,G_{0}^{\pm})=0$ and $f_0=S_{\theta_0}f^a$.
	Then it follows from \eqref{modified} that $(V_{1/2},\Psi_{1/2})=0$.
	Thanks to \eqref{small} and the properties of the approximate solution in Lemma \ref{lem.app},
	we may apply Theorem {\rm\ref{thm.L}} to obtain $(\delta \dot{V}_0,\delta\psi_0)$
	as the unique solution of \eqref{effective.NM} for ${n}=0$, which satisfies
	\begin{align} \notag
	\|\delta \dot{V}_0\|_{H^{s}_{\gamma}(\Omega_T)}+\|\delta \psi_0\|_{H^{s+1}_{\gamma}(\Omega_T)}
	\leq C\|S_{\theta_0}f^a\|_{H^{s+1}_{\gamma}(\Omega_T)}.
	\end{align}
	Then we find $\delta\Psi_0^{\pm}$ from equations \eqref{delta.Psi1}--\eqref{delta.Psi2} with ${n}=0$.
	The standard energy estimates yield
	\begin{align*}
	\|\delta \Psi_0\|_{H^{s}_{\gamma}(\Omega_T)}
	\leq C\|\delta \dot{V}_0\|_{H^{s}_{\gamma}(\Omega_T)}
	\qquad \textrm{for }s\in [3,\tilde{\alpha}],
	\end{align*}
	which, combined with \eqref{good.NM} and \eqref{small}, implies
	\begin{align*}
	&\|(\delta V_0,\delta \Psi_0)\|_{H^{s}_{\gamma}(\Omega_T)}+\|\delta \psi_0\|_{H^{s+1}_{\gamma}(\Omega_T)}
	\\&\qquad \leq C\|S_{\theta_0}f^a\|_{H^{s+1}_{\gamma}(\Omega_T)}
	\leq C\theta_0^{(s-\alpha)_+}\|f^a\|_{H^{\alpha+1}_{\gamma}(\Omega_T)}.
	\end{align*}
	If $\|f^a\|_{H^{\alpha+1}_{\gamma}(\Omega_T)}/\varepsilon$ is suitably small,
	then we can obtain inequality (a) of $(H_0)$.
	The other inequalities of $(H_0)$ can be shown to hold
	by taking  $\|f^a\|_{H^{\alpha+1}_{\gamma}(\Omega_T)}$ small enough.
\qed\end{proof}

From \eqref{lem.H0} and Lemmas \ref{lem.Hl1}--\ref{lem.Hl2},
we derive that $(H_{{n}})$ holds for every ${n}\in\mathbb{N}$,
provided that the parameters are well-chosen and that $f^a$ is sufficiently small.
We are now in a position to conclude the proof of Theorem {\rm\ref{thm}}.

\vspace*{3mm}
\noindent  {\bf Proof of Theorem {\rm\ref{thm}}}\quad
	We consider the initial data $(U_0^{\pm},\varphi_0)$ satisfying all the assumptions of Theorem {\rm\ref{thm}}.
	Let $\tilde{\alpha}=\mu-2$ and $\alpha=\tilde{\alpha}-4\geq 7$.
	Then the initial data $U_0^{\pm}$ and $\varphi_0$ are compatible up to order $\mu=\tilde{\alpha}+2$.
	From \eqref{app3} and \eqref{app5}, we obtain \eqref{small} and all the requirements
	of Lemmas \ref{lem.Hl1}--\ref{lem.H0},
	provided that $(\tilde{U}_0^{\pm},\varphi_0)$ is sufficiently small
	in $H^{\mu+1/2}(\mathbb{R}^2_+)\times H^{\mu+1}(\mathbb{R})$
	with $\tilde{U}_0^{\pm}:=U_0^{\pm}-\widebar{U}^{\pm}$.
	Hence, for small initial data, property $(H_{{n}})$ holds
	for all integers ${n}$. In particular, we have
	\begin{align*}
	\sum_{k=0}^{\infty}\left(\|(\delta V_k,\delta \Psi_k)\|_{H^{s}_{\gamma}(\Omega_T)}+\|\delta\psi_k\|_{H^{s+1}_{\gamma}(\omega_T)} \right)
	\leq C\sum_{k=0}^{\infty}\theta_k^{s-\alpha-2} <\infty
	\end{align*}
	 for $s\in[3,\alpha-1]$.
	Thus, sequence $(V_{k},\Psi_{k})$ converges
	to some limit $(V,\Psi)$ in $H^{\alpha-1}_{\gamma}(\Omega_T)$,
	and sequence $\psi_{k}$ converges to some limit $\psi$ in $H^{\alpha}_{\gamma}(\Omega_T)$.
	Passing to the limit in \eqref{Hl.b}--\eqref{Hl.c} for $s=\alpha-1=\mu-7$,
	and in \eqref{Hl.d}, we obtain \eqref{P.new}.
	Therefore, $(U, \Phi)=(U^a+V, \Phi^a+\Psi)$ is a solution on $\Omega_T^+$ of the original
	problem \eqref{RE0} and \eqref{Phi.eq}. This completes the proof.
\qed

\medskip
\begin{appendices}
	\section[Symmetrization of the Relativistic Euler Equations]{Symmetrization of the Relativistic  Euler Equations} \label{AppA}

	Under assumption \eqref{p.con}, Makino--Ukai  \cite{MU95IIMR1346915} showed that there exists a strictly convex entropy function
	for the relativistic Euler equations \eqref{RE2},
	which yields a symmetrizer for \eqref{RE2} by following Godunov's symmetrization
	procedure in \cite{G61MR0131653}.
	By contrast, the symmetrizable hyperbolic system \eqref{RE4} is deduced
	by using a purely algebraic symmetrization of the relativistic Euler
	equations \eqref{RE2}; see Trakhinin \cite{T09MR2560044} for a different algebraic symmetrization.

	In order to derive \eqref{RE4}, we need to recover another conservation
	law (that is, conservation of particle number)
	from equations \eqref{RE2}.
	Denoting by $N$ the particle number density and by $e$ the specific internal energy,
	then
	\begin{align} \label{rho.N}
	\rho=N(1+\epsilon^2 e).
	\end{align}
	The particle number density $N$ was introduced by Taub \cite{T59MR0105998}.
	For a perfect fluid, $N$ and $e$ are functions of the two thermodynamic
	variables $\rho$ and $S$ (specific entropy).
	According to the first law of thermodynamics,
	the following differential relation holds:
	\begin{align} \label{Gibbs}
	T\mathrm{d}S=\mathrm{d}e+p\,\mathrm{d}N^{-1},
	\end{align}
	where $T$ is the absolute temperature.
	By virtue of \eqref{rho.N}--\eqref{Gibbs}, we have
	\begin{align}\label{Gibbs2}
	\frac{\p (\ln N)}{\p \rho}=\frac{1}{\rho+\epsilon^2 p},\qquad \frac{\p (\ln N)}{\p S}=-\epsilon^2NT.
	\end{align}
	In the case of barotropic fluids
	where pressure $p$ depends solely on $\rho$,
	it is natural to introduce the ``mathematical'' particle number density $N$
	as a function of $\rho$ only such that the first relation in \eqref{Gibbs2} holds.
	This motivates us to define $N=N(\rho)$ as \eqref{N.c.def}.
	
	Let $(\rho,v)$ be a  $C^1$--solution to \eqref{RE2}.
	It follows from \eqref{RE2.a} and $h=(\rho+\epsilon^2 p)/N$ that
	\begin{align*}
	-h\varGamma\left\{\partial_t(N\varGamma)+\partial_k(N\varGamma v_k)\right\}
	= N\varGamma\left\{\varGamma(\partial_t+v_k\partial_k)h+h(\partial_t+v_k\partial_k)\varGamma\right\}-\epsilon^2 \partial_tp.
	\end{align*}
	In view of \eqref{Gamma.v.1} and \eqref{RE2.b}, we obtain
	\begin{align*}
	&Nh\varGamma(\partial_t+v_k\partial_k)\varGamma
	=Nh\varGamma\epsilon^2v_j(\partial_t+v_k\partial_k)w_j\\
&	=-\epsilon^2N|w|^2(\partial_t+v_k\partial_k)h-\epsilon^2 h\varGamma^{-1}|w|^2\left\{\partial_t(N\varGamma)
	+\partial_k(N\varGamma v_k)\right\}-\epsilon^2v_j\partial_jp,
	\end{align*}
	which implies
	\begin{align*}
	&
	h \varGamma^{-1}(\epsilon^2|w|^2-\varGamma^2)\left\{\partial_t(N\varGamma)+\partial_k(N\varGamma v_k)
	\right\}
	\\&\quad
	=N(\varGamma^2-\epsilon^2|w|^2)(\partial_t+v_k\partial_k)h-\epsilon^2(\partial_t+v_k\partial_k)p.
	\end{align*}
	Thanks to \eqref{N.c.def} and \eqref{Gamma.v.1},
	$N(\rho)h'(\rho)=\epsilon^2p'(\rho)$ and $\varGamma^2-\epsilon^2|w|^2=1$.
	Then we obtain the conservation of particle number:
	\begin{align} \label{N.eq}
	\partial_t(N\varGamma)+\partial_k(N\varGamma v_k)=0.
	\end{align}
	Equations \eqref{hw.eq} then follow from \eqref{RE2.b} and \eqref{N.eq}.
	Using the identities: $\varGamma_t =\epsilon^2v\cdot\partial_t w$ and
	$N'(p)=\frac{N'(\rho)}{p'(\rho)}=\frac{1}{h(\rho)c^{2}(\rho)}$,
	we see from \eqref{N.eq} that
	\begin{align} \notag
	{\varGamma}(\partial_t+v\cdot\nabla_x)p+Nhc^2\left(\epsilon^2 v\cdot\partial_t w+\nabla_x\cdot w\right)=0.
	\end{align}
	We use the relations: $h'(p)=\epsilon^2/ N$ and $w=\varGamma v$
	to deduce
	\begin{align}
&\notag	\varGamma(1-\epsilon^4 c^2 |v|^2)\p_t p+\varGamma (1- \epsilon^2 c^2  )v\cdot \nabla_x p \\
 \label{RE3.b}&\quad 	+
	Nc^2\left(\epsilon^2 v\cdot\partial_t(hw)+\nabla_x\cdot (hw)\right)=0.
	\end{align}
	Set $U:=(p,hw_1,hw_2)^{\mathsf{T}}$.
	Then equations \eqref{hw.eq} and \eqref{RE3.b} can be written as
	\begin{align} \label{RE3}
	B_0(U)\p_t U+B_1(U)\p_1 U+B_2(U)\p_2 U=0,
	\end{align}
	where  the coefficient matrices are given by
	\begin{align} \label{B0}
	B_0(U)
	&:=\begin{pmatrix}
	\varGamma(1-\epsilon^4 c^2 |v|^2)& \epsilon^2 c^2 N v^{\mathsf{T}} \\
	0 & \varGamma I_2
	\end{pmatrix},\\[2mm]
	\label{B12}
	B_j(U)&:=\begin{pmatrix}
	\varGamma v_j(1-\epsilon^2 c^2 )& Nc^2 e_j^{\mathsf{T}}\\  N^{-1} e_j & \varGamma v_j I_2
	\end{pmatrix},\qquad  j=1,2,
	\end{align}
	Here we have set $e_j:=(\delta_{1j},\delta_{2j})^{\mathsf{T}}$ and $I_2:=(\delta_{ij})_{2\times 2}$
	with  $\delta_{ij}$ being the Kronecker symbol.
	Let us define
	\begin{align} \label{S1}
	S_1(U):=\begin{pmatrix}1 & \epsilon^2 N c^2\varGamma^{-1} v^{\mathsf{T}} \\
	0 & I_2-\epsilon^2 v\otimes v
	\end{pmatrix}.
	\end{align}
	Multiplying \eqref{RE3} by $S_1(U)$
	and using the identity: $\varGamma^2-\epsilon^2 c^2\varGamma^2+\epsilon^2 c^2=\varGamma^2(1-\epsilon^4c^2|v|^2)$,
	we obtain system \eqref{RE4}.
	Conversely, we can also deduce system \eqref{RE2} from \eqref{RE4} so that
	we derive the equivalence of these two systems in the region
	where the solutions are in $C^1$.

	It remains to show that system \eqref{RE4} is symmetrizable hyperbolic
	in region $\{\rho_{*}<\rho<\rho^{*},|v|<\epsilon^{-1}\}$.
	Let us set the \emph{Friedrichs symmetrizer}:
	\begin{align} \label{Sym}
	S_2(U):=\begin{pmatrix}1 & -2\epsilon^2 N c^2  \varGamma v^{\mathsf{T}} \\
	0 &  N^2c^2 I_2\end{pmatrix}.
	\end{align}
	After straightforward calculations,
	we derive that all matrices $S_2(U)A_j(U)$ are symmetric, and
	the eigenvalues of $S_2(U)A_0(U)$ are
	\begin{align*}
	\lambda_1=\varGamma(1-\epsilon^4 c^2 |v|^2),\qquad
	\lambda_2=\varGamma N^2 c^2,\qquad
	\lambda_3=\varGamma N^2 c^2 (1-\epsilon^2|v|^2).
	\end{align*}
	Assumption \eqref{p.con} yields that  $\lambda_1$,  $\lambda_2$,
	and $\lambda_3$  are all positive.
	Consequently, $S_2(U)A_0(U)$ is positive definite
	and system \eqref{RE4} is symmetrizable hyperbolic.
	
\end{appendices}

\begin{acknowledgements}
	\quad
	The research of Gui-Qiang G. Chen was supported in part by the UK Engineering and Physical Sciences Research Council Award EP/E035027/1 and EP/L015811/1, and the Royal Society--Wolfson Research Merit Award (UK).
	The research of Paolo Secchi was supported in part by the Italian research projects PRIN 2012L5WXHJ-004 and PRIN 2015YCJY3A-004.
	The research of Tao Wang was supported in part by NSFC grants \#11601398 and \#11731008, and the Italian research project PRIN 2012L5WXHJ-004.
	Tao Wang warmly thanks Prof.\;Alessandro Morando, Prof.\;Paolo Secchi, and Prof.\;Paola Trebeschi for support and hospitality during his postdoctoral stay at University of Brescia, and also expresses much gratitude to Prof.\;Huijiang Zhao for his continuous encouragement and constant support.
\end{acknowledgements}

\medskip


\end{document}